\documentclass[11pt]{article}
\usepackage{amssymb,amsmath,mathrsfs,latexsym,amscd}
\usepackage{exscale,latexsym,amsthm,graphics}

\usepackage{makeidx}
\makeindex
\usepackage{epsfig}
\usepackage{hyperref}
\usepackage{txfonts}
\usepackage{lmodern}
\usepackage{color}
\usepackage{ulem}

\usepackage{enumitem}
\newtheorem{defn}{Definition}[section]
\newtheorem{theorem}[defn]{Theorem}

\newtheorem{corollary}[defn]{Corollary}
\newtheorem{lemma}[defn]{Lemma}

\newtheorem{proposition}[defn]{Proposition}

\newtheorem{definition}[defn]{Definition}

\newtheorem{remark}[defn]{Remark}

\newtheorem{eg}[defn]{Example}

\def\cB{\mathcal B}
\def\cE{\mathcal E}

\def\R{\mathbb R}
\def\N{\mathbb N}
\def\P{\mathbb P}
\def\E{\mathbb E}
\def\M{\mathbb M}

\def\({\left (}
\def\){\right )}
\def\[{\left [}
\def\]{\right ]}

\def\vp{\varphi}
\def\ve{\varepsilon}
\numberwithin{equation}{section}
\allowdisplaybreaks

\begin{document}
\setcounter{secnumdepth}{6}
\setcounter{tocdepth}{6}
\noindent
{{\Large\bf Analytic theory of It\^o-stochastic differential equations with non-smooth coefficients}\footnote{The research of Haesung Lee was supported by the Basic Science Research Program through the National Research Foundation of Korea (NRF) funded by the Ministry of Education (2020R1A6A3A01096151), and  by the DFG through the IRTG 2235 \lq\lq Searching for the regular in the irregular: Analysis of singular and random systems.\rq\rq\ The research of Wilhelm Stannat was partially supported by the DFG through the Research Unit FOR 2402 \lq\lq Rough paths, stochastic partial differential equations and related topics.\rq\rq\ The research of Gerald Trutnau was supported by the Basic Science Research Program through the National Research Foundation of Korea (NRF) funded by the Ministry of Education (2017R1D1A1B03035632).}}

\bigskip
\noindent
{\bf Haesung Lee}, {\bf Wilhelm Stannat}, 
{\bf Gerald Trutnau}
\\

\noindent
{\small{\bf Abstract.}  We present a detailed analysis of non-degenerate time-homogeneous It\^o-stochastic differential equations with low local regularity assumptions on the coefficients. In particular the drift coefficient may only satisfy a local integrability condition. We discuss non-explosion, irreducibility, Krylov-type estimates, regularity of the transition function and resolvent, moment inequalities, recurrence, transience, long time behavior of the transition function, existence and uniqueness of invariant measures, as well as pathwise uniqueness, strong solutions and uniqueness in law. This analysis shows in particular that sharp explicit conditions for the various mentioned properties can be derived  similarly to the case of classical stochastic differential equations with local Lipschitz coefficients.\\ 



\newpage
\tableofcontents
\newpage
%
%
%
\section{Introduction}
\label{chapter_1} 
This monograph is devoted to the systematic analytic and probabilistic study of weak solutions to the 
stochastic differential equation (hereafter SDE)
\begin{eqnarray}
\label{intro:eq1}
X_t = x+ \int_0^t \sigma (X_s) \, dW_s +   \int^{t}_{0}   \mathbf{G}(X_s) \, ds, 
\quad 0\le  t <\zeta, x\in \R^d,
\end{eqnarray}
where $(W_{t})_{t\ge 0}$ is a $d$-dimensional Brownian motion, $A=(a_{ij})_{1\le i,j\le d}=\sigma\sigma^T$ 
with $\sigma=(\sigma_{ij})_{1\le i,j\le d}$ is locally uniformly strictly elliptic (see \eqref{eq:2.1.2}), 
$\mathbf{G}=(g_1, \dots, g_d)$, and $\zeta$ is the lifetime of $X$, under low regularity assumptions 
on the coefficients. \\
The classical approach to the solution of \eqref{intro:eq1} is a pathwise solution and global existence and uniqueness of solutions can be obtained 
under locally Lipschitz assumptions combined with a linear growth condition. 
However, typical finite-dimensional approximations of stochastic partial differential equations do not match these 
assumptions and also several applications in the natural and engineering sciences, see for example \cite{Car}, \cite{Naga}. Therefore the need for substantial generalization arises. \\
Another essential drawback of the pathwise approach is the, to a large extent, still open problem of a 
mathematical rigorous characterization of the generator 
\begin{eqnarray}
\label{intro:eq2}
Lf=\frac 12\sum_{i,j=1}^d a_{ij}\partial_{ij} f  
+ \sum_{i=1}^d g_i\partial_i f, \quad f\in C_0^{\infty}(\R^d)
\end{eqnarray}
of \eqref{intro:eq1}. More specifically, in order to investigate 
properties of the solution 
of \eqref{intro:eq1} with analytic tools on the state space, especially from PDE theory and functional analysis, 
it is necessary to uniquely characterize $w(x,t) = \mathbb{E} (f(X_t) \mid X_0 = x)$, $t \ge 0$, 
$x\in\R^d$, as a solution of a Cauchy problem 
\begin{equation} 
\label{intro:eq3} 
\partial_t \varw (x,t) = L \varw (x,t) \, ,t \ge 0 \, , x\in\R^d,  
\end{equation} 
with initial condition $w(x,0) = f(x)$, for some proper extension $L$, whose full domain will depend 
on the underlying function space and may in general not explicitly be characterized.  \\
In this monograph we will investigate a converse approach to the solution and further investigation of 
\eqref{intro:eq1}, by starting with the analysis of the Cauchy problem \eqref{intro:eq3} on $L^1$-spaces 
with weights and subsequently constructing a strong solution to \eqref{intro:eq1} via the 
Kolmogorov-type construction of an associated Markov process. The essential advantage of this approach, 
which we will describe in more detail in Section \ref{MethodsandResults}
below, is that at each stage of the 
construction we keep a rigorous analytic description of the associated Cauchy problem \eqref{intro:eq3} 
including its full generator $\overline{L}$. This allows us on the one hand to establish a rigorous mathematical 
connection between SDEs and related stochastic calculus, with regularity theory of partial differential equations (PDEs), potential and semigroup theory, and generalized 
Dirichlet form theory on the other. \\ 
As another advantage we can relax the local regularity assumptions on the coefficients of \eqref{intro:eq1} 
considerably. If for instance, for some $p\in (d,\infty)$, the components of  
$\sigma=(\sigma_{ij})_{1\leq i,j \leq d}$ have $H^{1,p}_{loc}$-regularity and $\mathbf{G}$ has 
$L^{p}_{loc}$-regularity, strong existence and pathwise uniqueness of a solution $X$ 
to \eqref{intro:eq1} holds for all times under any global condition that implies non-explosion. Various non-explosion conditions are given in Section 
\ref{subsec:3.2.1}.
Our main result then, 
Theorem \ref{theo:3.3.1.8}, provides a detailed analysis of the properties of the solution $X$, like 
strong Feller properties of the transition semigroup and resolvent, Krylov-type estimates, moment 
inequalities, transience, recurrence, ergodicity, and existence and uniqueness of invariant measures with 
sharp explicit conditions, similarly to the classical case of locally Lipschitz continuous coefficients. 
\\
In recent years stunning and important new results about pathwise uniqueness and existence of a strong 
solution to \eqref{intro:eq1}, when $\mathbf{G}$ merely fulfills some local integrability condition, were 
presented (\cite{GyMa}, \cite{KR}, \cite{Zh11}). All these works also cover the time-dependent case with 
some trade-off between the integrability assumptions in time and space, but struggle to provide a complete stochastic analysis as in Theorem \ref{theo:3.3.1.8}, without a drastic strengthening of the local regularity assumptions (cf. Remark \ref{rem:3.3.1}). \\
Instead, the crucial idea here is to construct weak solutions to \eqref{intro:eq1} by PDE techniques 
(\cite{ArSe} and \cite{St65}) and generalized Dirichlet form theory (\cite{WS99}, \cite{WSGDF}, 
\cite{Tr2}, \cite{Tr5}), and thus separately and independently from local pathwise uniqueness and 
probabilistic techniques. 
Following this approach, initiated in \cite{RoShTr} in the frame of sectorial Dirichlet forms, we finally 
only rely on a local pathwise uniqueness result (\cite[Theorem 1.1]{Zh11}), since it enables us 
to construct a Hunt process with a transition semigroup that has such a nice regularity that 
all presumably optimal classical conditions for the properties of a solution to \eqref{intro:eq1} above 
carry over to our situation of non-smooth and/or locally unbounded coefficients.
\subsection{Methods and results}\label{MethodsandResults}
Let us describe in more detail the respective stages in our approach to the analysis of 
\eqref{intro:eq1} and the main results obtained. 
\subsection*{I. The abstract Cauchy problem}
The starting point of our approach is in Section \ref{sec2.1}, the analysis of the Cauchy problem 
\eqref{intro:eq3} on a space $L^1 (\R^d , \mu)$, where $\mu$ is a measure having a regular density $\rho$ 
and which satisfies 
\begin{eqnarray} 
\label{intro:eq4} 
\int_{\R^d} Lf \, d\mu = 0  \quad \forall\, f\in C_0^{\infty} (\R^d). 
\end{eqnarray} 
We call such a measure an infinitesimally invariant measure for $(L, C_0^{\infty}(\R^d))$ and although the above property 
of $\mu$ is loosely linked with the concept of invariance of stochastic processes, our approach is not 
at all limited to SDEs that admit an invariant measure or even are 
ergodic. We emphasize that the existence of such a measure $\mu$ is much less restrictive, if at all, 
than it might seem at first sight, and in fact $\mu$ will not be a finite measure in general, let 
alone a probability.\\  \\
\textit{Semigroup approach to the Cauchy problem}\\ \\
In the first step we realize in Theorem \ref{theorem2.1.5} an extension of $(L,C_0^{\infty}(\R^d))$ as 
the infinitesimal generator of a sub-Markovian $C_0$-semigroup of contractions $(T_t)_{t>0}$, which then gives rise to solutions of the Cauchy problem \eqref{intro:eq3}.  A crucial 
step in our analysis is the decomposition of the drift coefficient $\mathbf{G}$ as
$$
\mathbf{G} = \beta^{\rho , A}+\mathbf{B}, 
$$
where $\beta^{\rho , A}=(\beta^{\rho, A}_1 , \ldots,\beta^{\rho, A}_d)$ is the logarithmic derivative of $\rho$ associated with $A$ (see \eqref{eq:2.1.5}), and $\mathbf{B}$ is 
a $\mu$-divergence zero vector field. This allows us to decompose the operator $L$ as
$$ 
Lf = L^0 f + \langle \mathbf{B}, \nabla f \rangle,    
$$ 
where 
$$ 
L^0 f =  \frac 12\sum_{i,j=1}^d a_{ij}\partial_{ij} f  
+ \sum_{i=1}^d \beta^{\rho, A}_i \partial_i f, \quad f\in C_0^{\infty}(\R^d)
$$ 
is symmetric on $L^2 (\R^d , \mu )$ and can be extended in a unique way to a self-adjoint generator of 
a symmetric Dirichlet form, which plays a crucial role in our analysis.\\ \\
\textit{Uniqueness, invariance and conservativeness} \\ \\
We then discuss in the abstract functional analytic setting uniqueness of such infinitesimal generators, which is 
linked to the uniqueness of solutions of the Cauchy problem, and its interrelations with invariance and 
conservativeness as global properties of the semigroup. The corresponding results can be found in 
Proposition \ref{prop2.1.9}, Proposition \ref{prop:2.1.10}, Remark \ref{rem:2.1.10a}, Corollaries 
\ref{cor:2.1.4.1} and \ref{cor2.1.1}.
\subsection*{II. Infinitesimally invariant measures $\mu$} 
The existence and further regularity properties of a measure $\mu$ satisfying \eqref{intro:eq4} needed 
for our approach is investigated in Section \ref{sec2.2}. It is shown in particular that if, for some 
$p\in (d,\infty)$, the components of $A$ have $H^{1,p}_{loc}$-regularity and $\mathbf{G}$ has 
$L^{p}_{loc}$-regularity, such a $\mu$ exists, having a strictly positive density $\rho\in H^{1,p}_{loc}(\R^d)$. Reformulating 
$L$ into divergence form, we can considerably relax the assumptions on $A$, see assumption {\bf (a)} of 
Section \ref{subsec:2.2.1} for precise conditions. From there onwards assumption {\bf (a)} will always be in force unless otherwise stated. The main result on existence and regularity of $\mu$ is 
given in Theorem \ref{theo:2.2.7}. 
\subsection*{III. Regular solutions of the Cauchy problem} 
In order to enable a Kolmogorov-type construction of a Markov process whose transition semigroup is 
given by the semigroup $(T_t)_{t>0}$, it is necessary to pass from $(T_t)_{t> 0}$ to kernels of 
sub-probability measures. To this end we further analyze the regularity properties of $(T_t)_{t>0}$ in 
Section \ref{sec2.3}, in particular the existence of a H\"older-continuous version $P_tf$ of $T_tf$ that 
gives rise to a transition function $(P_t)_{t > 0}$. The corresponding results are given in 
Theorem \ref{theo:2.6} 
using
\eqref{semidef}. We also discuss precise  
interrelations of our regularity results with the strong Feller property.\\ \\
\textit{Irreducibility} \\ \\
In Section \ref{sec2.4}, irreducibility of $(T_t)_{t>0}$ and of the associated transition function 
$(P_t)_{t>0}$, called irreducibility in the probabilistic sense, are obtained. See Proposition 
\ref{prop:2.4.2} for the corresponding result. This closes the analytic part of our approach. 
\subsection*{IV. Associated Markov processes} 
\textit{Construction and identification}\\ \\
Our first step on the probabilistic side is to construct in Section \ref{sec:3.1} a Hunt process $\M$ 
with transition semigroup $(P_t)_{t>0}$. The corresponding result is contained in Theorem \ref{th: 3.1.2}. 
The existence of $\M$ does not follow immediately from the general theory of Markov processes, since $(P_t)_{t>0}$ may fail to be Feller. Instead, we use a refinement of a construction method 
from \cite{AKR} that involves elements of generalized Dirichlet form theory. For this purpose a higher 
regularity of the resolvent is needed, which requires another assumption {\bf (b)} (see Section 
\ref{subsec:3.1.1}) in addition to assumption {\bf (a)}. From there onwards both assumptions {\bf (a)} and 
{\bf (b)} will be in force unless otherwise stated. Given the 
regularity properties of the resolvent and the transition semigroup, the identification of $\M$ as a 
weak solution to \eqref{intro:eq1} (cf. Definition \ref{def:3.48}(iv)) then follows standard lines. See 
Proposition \ref{prop:3.1.6} and Theorem \ref{theo:3.1.4} for the corresponding results.\\ \\
\textit{Krylov-type estimates} \\ \\
As a by-product of the improved resolvent regularity, we obtain in Theorem \ref{theo:3.3} a Krylov-type 
estimate that has an interest of its own (see Remark \ref{rem:ApplicationKrylovEstimates}). Its importance stems from the fact that probabilistic quantities like $\int_0^t g(X_s)ds$, which are related to the drift or to the quadratic variation of the local martingale part of $X$ in \eqref{intro:eq1}, can be controlled in terms of the $L^q$-norm of $g$ and thereby make solutions to \eqref{intro:eq1} more tractable.\\ \\
\textit{Non-explosion and conservativeness} \\ \\
Throughout Section \ref{sec:3.2}, we investigate global properties of the Hunt process $\M$ constructed in 
Theorem \ref{th: 3.1.2}, by analytic and by probabilistic methods. Since we already know that $\M$ has continuous 
sample paths on $[0,\zeta)$, where $\zeta$ denotes the lifetime,
the first important global property of $\M$ is non-explosion, i.e. $\zeta=\infty$. It guarantees that 
the weak solution of Theorem \ref{theo:3.1.4} exists for all times and is continuous on 
$[0,\infty)$. Due to the strong Feller property, conservativeness of $(T_t)_{t>0}$ is equivalent to non-explosion of $\M$. In Section \ref{subsec:3.2.1}, various qualitatively different sufficient non-explosion 
criteria for $\M$ are presented. See Proposition \ref{prop:3.2.8}, Lemma \ref{lem3.2.6}, Corollaries  
\ref{cor:3.2.2} and \ref{cor:3.1.3} and Proposition \ref{theo:3.2.8} for Lyapunov-type conditions for  
non-explosion. Since the drift coefficient does not need to be locally bounded, we also provide in 
Proposition \ref{theo:3.2.8} non-explosion criteria of a different nature than in the case 
of locally bounded coefficients (\cite{Pi}), which we further illustrate with examples
(Example \ref{exam:3.2.1.4}). We also present in Proposition \ref{prop:3.2.9}
volume growth conditions for non-explosion, which follow from generalized Dirichlet form techniques and 
are again of a different nature than classical non-explosion conditions.\\ \\
\textit{Transience and recurrence} \\ \\
In Section \ref{subsec:3.2.2}, we study transience and recurrence. Recurrence is an important concept as it implies stationarity of solutions w.r.t. $\mathbb{P}_{\mu}$ and that $\mu$ is the unique (however possibly infinite) invariant measure for the solution of \eqref{intro:eq1} (see \cite{LT21in}).
We establish in Theorem \ref{theo:3.3.6}
a well-known dichotomy between recurrence and transience (cf. for instance \cite[Theorem 3.2]{Bha} and 
\cite[Theorem 7.4]{Pi} for the case of locally bounded coefficients) and develop several sufficient analytic
criteria for recurrence. The corresponding results are given in Proposition \ref{theo:3.2.6}, Corollary \ref{cor:3.2.2.5}, and Proposition \ref{cor:3.3.2.6}.\\ \\
\textit{Uniqueness of invariant measures} \\ \\
Section \ref{subsec:3.2.3} deals with uniqueness of invariant measures and the long time behavior of 
$(P_t)_{t>0}$. Again, due to the regularity properties of $(P_t)_{t>0}$, Doob's ergodic 
theorem is applicable. Based on this, we develop several classical-like explicit criteria for ergodicity 
(see Proposition \ref{prop:3.3.12} and Corollary \ref{cor:3.2.3.7}). Example \ref{ex:3.8} provides a 
counterexample to uniqueness of invariant measures.
\subsection*{V. The stochastic differential equation} 
In the final stage of our approach we consider the stochastic differential equation \eqref{intro:eq1} 
and investigate in Section \ref{sec:3.3} two types of uniqueness of a solution.\\ \\
\textit{Pathwise uniqueness and strong solutions} \\ \\
The first type of uniqueness is pathwise uniqueness (cf. Definition \ref{def:3.48}(v)) and we explore 
the existence of a strong solution to \eqref{intro:eq1} (cf. Definition \ref{def:3.48}(ii)). Using the 
classical Yamada--Watanabe Theorem (\cite{YW71}) and a local uniqueness result from \cite{Zh11}, we 
obtain Theorem \ref{theo:3.3.1.8} both under the mere assumption of {\bf (c)} of Section \ref{sec:3.3} 
(which implies the two 
assumptions {\bf (a)} and {\bf (b)}) and the assumption that the constructed Hunt process $\M$ in 
Theorem \ref{th: 3.1.2} is non-explosive. 
This is one of the main achievements of our approach. 
It shows that SDEs with non-smooth coefficients, for instance those with locally unbounded drift, 
can be treated with classical-like methods and presents a real extension of the It\^{o} theory of locally 
Lipschitz coefficients and non-degenerate dispersion coefficients. In particular, our new approach allows us to 
close a partial gap in the existing literature and we refer the reader to the introduction of 
Section \ref{sec:3.3} and Remark \ref{rem:3.3.1}  for more details.\\ \\
\textit{Uniqueness in law} \\ \\
The second type of uniqueness is related to uniqueness in law under the conditions {\bf (a)} and {\bf (b)}. 
Since uniqueness in law in the classical sense as in Definition \ref{def:3.48}(vi) may not hold in the 
general class of coefficients (cf. for instance the introduction of \cite{LT19de}), here we consider a 
weaker form of uniqueness in law which is related to $L^1$-uniqueness (cf. Definition \ref{def:3.3.2.1}). 
The corresponding uniqueness result is contained in Proposition \ref{prop:3.3.1.16}.\\ 
\subsection{Organization of the book} 
The text is structured and divided into an analytic part (Chapter \ref{chapter_2}), a probabilistic part (Chapter \ref{chapter_3}), and a conclusion and outlook part (Chapter \ref{conclusionoutlook}). For a better orientation of the reader we start each section with a summary of its main contents and the assumptions that are in force.  We also provide historical remarks concerning specific aspects of our work, where we cite relevant related work and compare existing literature with our results in a detailed way (Remark \ref{rem:2.30new}, Remark \ref{rem:2.4.1}, Remark \ref{rem:2.3.3}, Remark \ref{rem:ApplicationKrylovEstimates}, Remark \ref{rem:3.3.1}). Additional information to existing  theories and results that are used for our analysis is provided throughout the text. In particular, Sections \ref{Comments2}, resp. \ref{Comments3}, provide a summary of techniques and results that we rely on in Chapters \ref{chapter_2}, resp. \ref{chapter_3}. 

%
%
%
\newpage

\section{The abstract Cauchy problem in $L^r$-spaces with weights} 
\label{chapter_2} 

\subsection{The abstract setting, existence and uniqueness}
\label{sec2.1}
We consider the Cauchy problem\index{Cauchy problem}  
\begin{equation} 
\label{eq:2.1.0} 
\partial_t \varw (x,t) = L \varw (x,t) \, ,t \ge 0 \, , x\in\R^d,  
\end{equation} 
where $L= \frac 12\sum_{i,j=1}^d a_{ij}\partial_{ij} + \sum_{i=1}^d g_i\partial_i $ is some locally  uniformly strictly elliptic partial differential operator of second order on $\R^d$ with domain $C_0^\infty (\R^d )$ and suitable initial condition 
$\varw(x,0) = f(x)$ on the space $L^1 (\R^d , \mu )$. Here, $\mu$ is a locally finite nonnegative 
measure that is infinitesimally invariant for $(L,C_0^{\infty}(\R^d))$ (see \eqref{eq:2.1.0d}). We explicitly construct in 
Section \ref{subsec:2.1.2a}, under minimal assumptions on the coefficients $(a_{ij})_{1\le i,j\le d}$ and 
$(g_i)_{1\le i\le d}$, extensions of $(L,C_0^{\infty}(\R^d))$ generating sub-Markovian $C_0$-semigroups on $L^1 (\R^d, \mu )$ (see 
Theorem \ref{theorem2.1.5} for the main result) and discuss in Section \ref{subsec:2.1.4} uniqueness of 
such extensions. The main result, contained in Corollary \ref{cor2.1.1}, establishes a link 
between uniqueness of maximal extensions and invariance of the infinitesimally invariant measure $\mu$ 
under the associated semigroup $(\overline{T}_t)_{t\ge 0}$. We discuss in Section \ref{subsec:2.1.4} the 
interrelations of invariance with conservativeness of $(\overline{T}_t)_{t\ge 0}$, resp. its dual semigroup, 
and provide in Proposition \ref{prop:2.1.10}, resp. Corollary \ref{cor:2.1.4.1} explicit sufficient 
conditions on the coefficients, including Lyapunov-type conditions, implying invariance resp. 
conservativeness. We also illustrate the scope of the results with some counterexamples.\\[3pt]
In view of the envisaged application to the analysis of weak solutions of stochastic differential 
equations, we will be in particular interested in the existence of solutions $\varw(x,t)$ to the Cauchy problem
\eqref{eq:2.1.0} that can be represented as an expectation w.r.t. some associated Markov process: 
\begin{equation} 
\label{eq:2.1.0a} 
\varw(x,t) = \mathbb{E} (f(X_t) \mid X_0 = x) \, , t\ge 0\, , x\in\R^d. 
\end{equation}
The classical linear semigroup theory (see \cite[Chapter 4]{Pa}) provides a solution to the abstract Cauchy problem 
\eqref{eq:2.1.0} 
in terms of $\varw(x,t)= \overline{T}_t f (x)$, where $\overline{T}_0=id$ and $(\overline{T}_t)_{t> 0}$ is a 
strongly continuous semigroup ($C_0$-semigroup) on a suitable function space $\mathcal{B}$, whose infinitesimal 
generator 
$$ 
\overline{L} f := \frac{d\overline{T}_t f}{dt}\Big|_{t = 0}  \quad \text{ on $\mathcal{B}$}
$$ 
with domain $D(\overline{L}) := \{f\in \mathcal{B}  : \frac{d\overline{T}_t f}{dt}\Big|_{t = 0} \text{ exists in } \mathcal{B} \}$  
is a closed extension of $(L,C_0^{\infty}(\R^d))$, i.e. $C_0^\infty (\R^d) \subset D(\overline{L})$ and 
$\overline{L}|_{C_0^\infty (\R^d)} = L$. Such extensions of the operator $L$ are called  
\textbf{maximal extensions}\index{operator ! maximal extension} of $(L,C_0^{\infty}(\R^d))$ on $\mathcal{B}$. \\
However, in order to be able to represent $\overline{T}_t f(x) = \varw(x,t) 
= \mathbb{E} (f(X_t) \mid X_0 = x)$ as the expectation of some Markov process, the semigroup has 
to be in addition \textbf{sub-Markovian}\index{semigroup ! sub-Markovian}, i.e. 
\begin{equation} 
\label{eq:2.1.0b} 
 0\le f\le 1 \quad \Rightarrow \quad 0\le \overline{T}_t f \le 1  \, , t > 0 . 
\end{equation} 
Using the maximum principle, the construction of such sub-Markovian semigroups associated with $L$ can be achieved within classical PDE theory under appropriate regularity assumptions on the coefficients $a_{ij}$ and $g_i$. On the other hand, the theory of stochastic processes provides the existence of 
Markov processes under much weaker regularity assumptions on the coefficients, for example with the help of SDEs and the precise mathematical characterization of their transition semigroups and (infinitesimal) generators. This has been intensively investigated in the past, but still leaves many challenging questions open.   \\
A very successful approach towards such a rigorous mathematical theory, connecting solutions of the abstract Cauchy problem \eqref{eq:2.1.0} with transition semigroups of Markov processes under minimal regularity assumptions, has been developed within the theory of symmetric Dirichlet forms (\cite{FOT}) in the particular case where the differential operator  $L$ becomes symmetric,
\begin{equation} 
\label{eq:2.1.0c} 
\int_{\R^d} Lu \varv\, d\mu = \int_{\R^d} u L\varv\, d\mu  \quad \forall\, u,\varv\in C_0^{\infty} (\R^d), 
\end{equation} 
 w.r.t. the inner product on the Hilbert space $L^2 (\R^d , \mu )$ induced by some locally finite nonnegative measure 
$\mu$. The measure $\mu$ is called a \textbf{symmetrizing measure} \index{measure ! symmetrizing} for 
the operator $L$ in this case. Using linear perturbation theory of symmetric operators, the scope of Dirichlet form theory had 
subsequently been successfully extended in a first generalization to the case where $L$ can be realized as 
a sectorial operator on some $L^2$-space (see \cite{MR}) and later to the fully non-symmetric case in \cite{WS99} (see also \cite{WSGDF}).  \\
The general theory developed in \cite{WS99} combines semigroup theory with Dirichlet form techniques 
in order to solve the abstract Cauchy problem \eqref{eq:2.1.0} in terms of a sub-Markovian semigroup 
$(\overline{T}_t)_{t > 0}$ on the Banach space $L^1 (\R^d , \mu )$, where $\mu$ is an \textbf{infinitesimally invariant measure} \index{measure ! infinitesimally invariant} 
for $(L, C_0^{\infty}(\R^d))$, i.e. a locally finite nonnegative measure satisfying $Lu\in L^1(\R^d,\mu)$ for all $u\in C_0^{\infty} (\R^d)$ and 
\begin{equation} 
\label{eq:2.1.0d} 
\int_{\R^d} Lu \, d\mu = 0  \quad \forall\, u\in C_0^{\infty} (\R^d)\, . 
\end{equation}  
Note that \textbf{symmetry \eqref{eq:2.1.0c} implies invariance \eqref{eq:2.1.0d}} by choosing a 
function $\chi\in C_0^\infty (\R^d)$ such that $\chi\equiv 1$ on the support of $u$, since 
\begin{equation*} 
\label{eq:2.1.0dd} 
\int_{\R^d} Lu \, d\mu = \int_{\R^d} L u \chi\, d\mu = \int_{\R^d} u L\chi\, d\mu  = 0\, , 
\end{equation*} 
because $L\chi \equiv 0$ on the support of $u$.\\
The existence (and uniqueness), as well as the analytic and probabilistic interpretation of (infinitesimally) invariant measures $\mu$, will be further analyzed thoroughly in subsequent 
sections (see in particular Sections \ref{sec2.2} and \ref{subsec:3.2.3}).  \\
Before stating the precise assumptions on the coefficients and the infinitesimally invariant measure in the next 
section let us discuss the most relevant functional analytic implications of assumption 
\eqref{eq:2.1.0d}.  

\begin{itemize} 
\item \textbf{(Beurling--Deny property)}\index{Beurling--Deny property} Let $\psi\in L^1_{loc} (\R )$, be monotone increasing. Then 
\begin{equation} 
\label{eq:2.1.0e}
\int_{\R^d} \psi (u)L u \, d\mu \le 0\quad \forall u\in C_0^\infty (\R^d)\, . 
\end{equation} 
Indeed, assume first that $\psi\in C^\infty (\R)$ is monotone increasing, hence $\psi' \ge 0$. Let $\Psi (t) := \int_0^t \psi (s)\, ds$. Then $\Psi (0) = 0$, hence 
$\Psi (u)\in C_0^\infty (\R^d)$ and using the ellipticity
$$ 
L \Psi (u) = \psi (u) Lu + \psi' (u) \frac 12 \sum_{i,j=1}^d a_{ij}\partial_i u \partial_j u  \ge \psi (u) Lu 
$$ 
hence integrating w.r.t. the infinitesimally invariant measure $\mu$ yields \eqref{eq:2.1.0e}. The general case then follows by straightforward approximation. 

\item As a consequence of the Beurling--Deny property we obtain that $(L , C_0^\infty (\R^d))$ is 
dissipative on $L^r (\R^d , \mu)$ for all $r\in [1, \infty)$ (see \cite[Lemma 1.8, p. 36]{Eberle}), 
which is a necessary condition for the existence of maximal extensions of $L$ generating a  
$C_0$-semigroup of contractions on $L^r (\R^d , \mu)$.  

\item Since $L$ is dissipative, it is in particular closable. Its closure in $L^r (\R^d, \mu )$ generates 
a $C_0$-semigroup $(T_t )_{t>  0}$, if and only if the following \textbf{range condition}
\index{operator ! range condition}  holds (\cite[Theorem 3.1]{LP61}):
\begin{equation}
\label{eq:2.1.0f}
\exists\, \lambda > 0 \text{ such that } (\lambda - L)(C_0^\infty (\R^d )) \subset L^r (\R^d , \mu ) 
\text{ dense.} 
\end{equation} 
In this case, the semigroup $(T_t )_{t> 0}$ is sub-Markovian (see \cite[pp. 36--37]{Eberle}). 
\end{itemize} 
\noindent
We will apply the range condition, in Section \ref{subsec:2.1.2} below, to some suitable, but still 
explicit, extension of $L$, to obtain, for any relatively compact open subset $V\subset \R^d$, the 
existence of a sub-Markovian semigroup $( \overline{T}^V_t )_{t> 0}$ on $L^1 (V, \mu )$ whose 
generator $(\overline{L}^V , D(\overline{L}^V))$ extends $(L, C_0^\infty (V))$ 
(Proposition \ref{prop:2.1}). The associated Markov process (also constructed in \cite{WS99}) is a 
stochastic process killed at the instant it reaches the boundary of $V$. It is therefore only natural to 
conjecture the following \textbf{domain monotonicity}\index{semigroup ! domain monotonicity}:  
\begin{equation}
\label{eq:2.1.0g}
\overline{T}^{V_1}_t  \le \overline{T}^{V_2}_t \text{ for any relatively compact open subsets } 
 V_1 \subset V_2. 
\end{equation} 
Here $\overline{T}^{V_1}_t \le\overline{T}^{V_2}_t$ means that $\overline{T}^{V_1}_t f  
\le \overline{T}^{V_2}_t f$ for all $f\in L^1 (V_1, \mu )$, $f\ge 0$.  
We give a rigorous purely analytic proof for this monotonicity in terms of the corresponding resolvents 
in Lemma \ref{lemma2.1.6} below.  \\
Having established \eqref{eq:2.1.0g}, we can consider in the next step the monotone limit 
\begin{equation}
\label{eq:2.1.0h}
\overline{T}_t f = \lim_{n\to\infty} \overline{T}^{V_n}_t f\, , t\ge 0\, , 
\end{equation} 
for an increasing sequence $(V_n)_{n\ge 1}$ of relatively compact open subsets satisfying 
$\overline{V}_n\subset V_{n+1}$, $n\ge 1$. It is quite easy to see that the monotone limit 
$( \overline{T}_t )_{t > 0}$ defines a sub-Markovian $C_0$-semigroup of contractions on $L^1 (\R^d , \mu )$. 
A remarkable fact of this construction is its \textbf{independence of the chosen exhausting sequence 
$(V_n)_{n\ge 1}$} (Theorem \ref{theorem2.1.5}).

\subsubsection{Framework and basic notations} 
\label{subsec:2.1.1}

Let us next introduce our precise mathematical framework and fix basic notations and assumptions used throughout 
up to the end of Section \ref{sec2.1}. 
We suppose that $\mu$  is a $\sigma$-finite (positive) measure on $\cB (\R^d)$ as follows:
\begin{eqnarray}\label{condition on mu}
\mu=\rho\,dx, \ \text{ where }\rho = \vp^2,\ \vp\in H_{loc}^{1,2}(\R^d), \ d\ge 1, \ \text{supp}(\mu )\equiv \R^d.
\end{eqnarray}

\noindent 
Let $V$ be an open subset of $\R^d$. If $\mathcal{A}\subset L^s (V, \mu )$, $s\in [1, \infty]$, is 
an arbitrary subset, denote by $\mathcal{A}_0$ the subspace of all elements $u\in \mathcal{A}$ such that 
$\text{supp} (|u|\mu )$ is a compact subset contained in $V$, and $\mathcal{A}_b$ the subspace of 
all bounded elements in $\mathcal{A}$. Finally, let $\mathcal{A}_{0,b} := \mathcal{A}_0\cap \mathcal{A}_b$.\\
Let us next introduce weighted Sobolev spaces that we are going to use in our analysis. 
Let $H_0^{1,2} (V, \mu )$ be the closure of $C_0^\infty (V)$ in $L^2 (V,\mu )$ w.r.t. the norm 
$$ 
\|u\|_{H_0^{1,2} (V,\mu )} 
:= \left( \int_V u^2\, d\mu + \int_V \|\nabla u\|^2\, d\mu \right)^{\frac 12}. 
$$
Finally, let $H_{loc}^{1,2} (V,\mu )$ be the space of all elements $u$ such that 
$u\chi\in H_0^{1,2}(V,\mu)$ for all $\chi\in C_0^\infty (V)$.

\noindent 
The precise assumptions on the coefficients of our differential operator 
that we want to analyze are as follows: let $A = (a_{ij})_{1\le i,j\le d}$ with 
\begin{equation} 
\label{eq:2.1.1}
a_{ji}= a_{ij}\in H_{loc}^{1,2}(\R^d ,\mu )\, , 1\le i,j\le d,  
\end{equation}  
be locally strictly elliptic, i.e., for all $V$ relatively compact  
there exists a constant $\nu_V > 0$ such that 
\begin{equation}
\label{eq:2.1.2}
\nu^{-1}_V\| \xi \|^2\le\langle A (x) \xi,\xi\rangle\le\nu_V \|\xi\|^2 \text{ for all } \xi \in\R^d , x\in V. 
\end{equation}
Let  
\begin{equation} 
\label{eq:2.1.3}
\mathbf{G}=(g_1,\ldots,g_d) \in L^2_{loc}(\R^d, \R^d, \mu),
\end{equation} 
i.e., $\int_V\|\mathbf{G}\|^2\, d\mu < \infty$ for all $V$ relatively compact in $\R^d$, 
and suppose that the measure $\mu$ is an {\bf infinitesimally invariant measure}
\index{measure ! infinitesimally invariant}  for $(L^A+\langle \mathbf{G}, 
\nabla \rangle, C_0^{\infty}(\R^d))$, i.e.
\begin{equation} 
\label{eq:2.1.4}
\int_{\R^d}\big ( L^{A} u  + \langle\mathbf{G} , \nabla u\rangle\big )\, d\mu = 0  \qquad  \forall u\in C_0^\infty (\R^d), 
\end{equation} 
where for $f\in C^2(\R^d)$ 
\begin{equation} \label{eq:2.1.3bis}
L^A f := \frac 12 \sum_{i,j =1}^d a_{ij}\partial_{ij} f.
\end{equation} 
Moreover, throughout this monograph, we shall let for $f\in C^2(\R^d)$ 
\begin{equation} \label{eq:2.1.3bis2}
L f := L^A f +\langle \mathbf{G},\nabla f\rangle=\frac 12 \sum_{i,j =1}^d a_{ij}\partial_{ij} f +\sum_{i=1}^dg_i\partial_i f.
\end{equation} 
We will provide in Theorem \ref{theo:2.2.7} of Section \ref{sec2.2} 
explicit sufficient conditions on $A$ and $\mathbf{G}$ such that an 
infinitesimally invariant measure $\mu$ with the required regularity \eqref{condition on mu} exists, and for which the 
assumptions \eqref{eq:2.1.1}--\eqref{eq:2.1.4} are satisfied (see in particular 
Theorem \ref{theo:2.2.7} and Remark \ref{rem:2.2.7}).\\
As mentioned in the previous section, \eqref{eq:2.1.4} implies that the operator $(L^A + \langle \mathbf{G}, \nabla \rangle, C_0^\infty (\R^d))$ is 
dissipative on the Banach space $L^1 (\R^d , \mu )$, which is necessary for the 
existence of a closed extension $(\overline{L}, D(\overline{L}))$ of  
$(L^A + \langle \mathbf{G}, \nabla \rangle, C_0^\infty (\R^d))$ generating a $C_0$-semigroup of contractions 
on $L^1 (\R^d , \mu )$. In general we cannot expect that the closure of 
$(L^A + \langle \mathbf{G}, \nabla\rangle, C_0^\infty (\R^d))$ will be already generating such a semigroup, in fact, in general there exist many maximal extensions and not all maximal extensions will generate 
sub-Markovian semigroups. Here we recall that a closed extension $(\overline{L}, D(\overline{L}))$ of $(L,
C_0^\infty (\R^d))$ is called a {\bf maximal extension}, if it is the generator of a $C_0$-semigroup
in $L^1 (\R^d , \mu )$.\\
To find the right maximal extension that meets our requirements for the analysis of 
associated Markov processes, we first need to extend the domain $C_0^\infty (\R^d )$ in a 
nontrivial, but nevertheless explicit way. To this end observe that we can rewrite 
\begin{equation} 
\label{eq:2.1.3b}
L^A  + \langle \mathbf{G}, \nabla \rangle  = L^0 +  \langle \mathbf{B}, \nabla \rangle \quad \text{ on }\ C^{\infty}_0(\R^d)
\end{equation} 
into the sum of some $\mu$-symmetric operator $L^0$ and a first-order perturbation given by some vector 
field $\mathbf{B}$, which is of $\mu$-divergence zero.   \\
Indeed, note that for $u,\varv\in C_0^\infty (\R^d)$, an integration by parts yields that 
\begin{equation} 
\label{eq:2.1.3c}
\int_{\R^d} \big( L^A u  + \langle\mathbf{G}, \nabla u\rangle\big ) \varv\, d\mu 
= - \frac 12 \int_{\R^d} \langle A\nabla u, \nabla \varv\rangle\, d\mu  
+ \int_{\R^d} \langle \mathbf{G} - \beta^{\rho, A} , \nabla u\rangle \varv\, d\mu 
\end{equation} 
with $\beta^{\rho , A} = (\beta_1^{\rho , A} , \ldots , \beta_d^{\rho , A} )\in L^2_{loc}  
(\R^d, \R^d, \mu )$ defined as 
\begin{equation} 
\label{eq:2.1.5}
\beta^{\rho , A}_i  
= \frac 12 \sum_{j=1}^d \Big ( \partial_j a_{ij} + a_{ij} \frac{\partial_j \rho}{\rho}\Big ),   
1\le i\le d.   
\end{equation} 
The symmetric positive definite bilinear form 
$$ 
\cE^0 (u,\varv) := \frac 12 \int_{\R^d}\langle A\nabla u, \nabla \varv\rangle\,d\mu,   u,\varv\in C_0^\infty (\R^d )
$$
can be shown to be closable on $L^2 (\R^d , \mu )$ by using results of \cite[Subsection II.2b)]{MR}.
Denote its closure by $(\cE^0,D(\cE^0))$, the associated self-adjoint generator 
by $(L^0, D(L^0))$ and the corresponding sub-Markovian semigroup by $(T_t^0)_{t>0}$. We let
$$
\cE^0_{\alpha}(\cdot,\cdot):=\cE^0(\cdot,\cdot)+\alpha(\cdot,\cdot)_{L^2(\R^d,\mu)}\, , \ \alpha>0.
$$
Recall that  
the domain of the generator is defined as 
$$ 
D(L^{0}) := \left\{ u\in D(\cE^0 )  
: \varv\mapsto \cE^0 (u,\varv) \text{ is continuous w.r.t } \|\cdot\|_{L^2 (\R^d ,\mu)} \; \text{on } D(\mathcal{E}^0) \right\} 
$$ 
and for $u\in D(L^{0})$, $L^{0} u$ is defined via the Riesz representation theorem (see \cite[Theorem 5.5]{BRE}) as the unique 
element in $L^2 (\R^d , \mu)$ satisfying 
\begin{equation} 
\label{DefGeneratorDF}
\cE^0 (u,\varv)  = - \int_{\R^d} L^{0} u \varv\, d\mu , \;\; \forall \varv\in D(\cE^0 ). 
\end{equation}  
\noindent
It is easy to see that our assumptions imply that $C_0^\infty (\R^d )\subset D(L^0)$ and 
\begin{equation} 
\label{eq:2.1.3d}
L^0 u = L^A u + \langle \beta^{\rho , A},\nabla u\rangle,  u\in C_0^\infty (\R^d ), 
\end{equation} 
so that 
\begin{equation} 
\label{eq:2.1.5a} 
\mathbf{B} = \mathbf{G} - \beta^{\rho , A}, 
\end{equation}
which is also contained in $L^2_{loc} (\R^d, \R^d, \mu )$. $\mathbf{B}$ now is in fact of 
$\mu$-divergence zero, i.e., 
\begin{equation} 
\label{eq:2.1.6}
\int_{\R^d} \langle\mathbf{B}, \nabla u \rangle\, d\mu  =  0 \quad \forall u\in C_0^\infty (\R^d),  
\end{equation} 
since $\int_{\R^d} \langle\mathbf{B}, \nabla u \rangle\, d\mu  = \int_{\R^d} L^A u + \langle \mathbf{G} , 
\nabla u\rangle\, d\mu - \int_{\R^d} L^0 u\,d\mu = 0$. \\
The decomposition \eqref{eq:2.1.3b} is crucial for our construction of a closed extension of 
$L^A + \langle \mathbf{G}, \nabla \rangle$ on $L^1 (\R^d , \mu )$ generating a  $C_0$-semigroup that 
is {\it sub-Markovian}. 
\begin{remark}\label{cnulltwocoincide}
{\it By the same line of argument as above one can show that \eqref{eq:2.1.3b} holds in fact also on $C_0^2(\R^d)$.}
\end{remark}

\subsubsection{Existence of maximal extensions on $\R^d$}
\label{subsec:2.1.2a}

\paragraph{Existence of maximal extensions on relatively compact subsets $V\subset\R^d$} 
\label{subsec:2.1.2}
Throughout this section we fix a relatively compact open subset $V$ in $\R^d$. 
Then all assumptions on the coefficients become global. In particular, the restriction of 
$A(x)$ is uniformly strictly elliptic, the restriction of $\mu$ is a finite measure and the vector fields $\mathbf{G}$, $\beta^{\rho , A}$ and $\mathbf{B}$ are in $L^2 (V, \R^d ,\mu )$. Our aim in this section is to construct a {\bf maximal extension} \index{operator ! maximal extension} $(\overline{L}^V, D(\overline{L}^V))$ of 
\begin{equation} 
\label{OperatorBoundedDomain} 
L^A u+ \langle\mathbf{G}, \nabla u\rangle = L^0 u + \langle\mathbf{B}, \nabla u\rangle\, ,  
u\in C_0^\infty (V) , 
\end{equation} 
on $L^1 (V, \mu )$, i.e.
$(\overline{L}^V, D(\overline{L}^V))$ is a closed extension of $(L^A + \langle \mathbf{G}, \nabla \rangle, C_0^{\infty}(V))$ on $L^1(V, \mu)$ that generates a $C_0$-semigroup of contractions on $L^1(V, \mu)$. \\
It is clear that we cannot achieve this by simply taking its closure on 
$C_0^\infty (V)$, since no boundary conditions are specified. However, we can impose Dirichlet 
boundary conditions as follows: let $(L^{0,V}, D(L^{0,V}))$ be the self-adjoint generator of the 
symmetric Dirichlet form $\cE^0 (u,\varv)$, $u,\varv\in H_0^{1,2}(V,\mu )$, which is characterized similar to 
the full domain case \eqref{DefGeneratorDF} as 
\begin{equation} 
\label{DefGeneratorDFBoundedDomain}
\cE^0 (u,\varv)  = - \int_V L^{0, V} u \varv\, d\mu, \quad \forall u \in D(L^{0,V}), \ \varv \in H^{1,2}_0(V, \mu).
\end{equation}  
Note that $C_0^2(V)\subset  D(L^{0,V})$ and that for $u\in D(L^{0,V}) \subset H_0^{1,2} (V, \mu )$, $\langle\mathbf{B}, \nabla u\rangle 
\in L^1 (V,\mu )$, so that in particular its restriction to bounded functions,   
$$ 
L^{0,V} u + \langle\mathbf{B}, \nabla u\rangle\, ,  u\in D(L^{0,V})_b , 
$$
is a well-defined extension of \eqref{OperatorBoundedDomain} on $L^1 (V, \mu )$. Note that the zero 
$\mu$-divergence of the vector field $\mathbf{B}$ (see \eqref{eq:2.1.6}) extends to all of  
$H_0^{1,2} (V, \mu )$ by simple approximation.  \\
The following proposition now states that this operator is closable and that its closure generates a sub-Markovian $C_0$-semigroup of contractions. In addition, the integration by parts \eqref{eq:2.1.3c} extends to all bounded functions in the domain of the closure.

\begin{proposition} 
\label{prop:2.1}
Let \eqref{condition on mu}--\eqref{eq:2.1.4}
be satisfied and $V$ be a relatively compact open subset in $\R^d$. 
Let $(L^{0,V}, D(L^{0,V}))$ be the generator of $(\cE^0 , H_0^{1,2} (V,\mu ))$ (see \eqref{DefGeneratorDFBoundedDomain}). Then: 
\item{(i)} The operator 
$$
L^V u := L^{0,V} u + \langle \mathbf{B} , \nabla u\rangle, \quad u\in D(L^{0,V})_b,
$$ 
is dissipative, hence in particular closable, on 
$L^1(V, \mu )$. The closure $(\overline{L}^V, D(\overline{L}^V))$ generates a 
sub-Markovian $C_0$-semigroup of contractions $(\overline{T}_t^V)_{t >0}$ on $L^1(V, \mu)$. In particular 
$(\overline{L}^V, D(\overline{L}^V))$ is a {\bf maximal extension}\index{operator ! maximal extension} of 
$$
(\frac 12 \sum_{i,j =1}^d a_{ij}\partial_{ij}  + \sum_{i=1}^d g_{i}\partial_{i},
C_0^\infty (V))
$$ 
(cf. \eqref{eq:2.1.3bis} and \eqref{OperatorBoundedDomain}).

\item{(ii)} $D(\overline{L}^V)_b\subset H_0^{1,2} (V,\mu )$ and 
\begin{equation} 
\label{eq:2.1.7}
\cE^0 (u,\varv) - \int_V\langle\mathbf{B}, \nabla u\rangle \varv\, d\mu = - \int_V
\overline{L}^V u\, \varv\, d\mu \, ,u\in D(\overline{L}^V)_b , \varv\in H_0^{1,2} 
(V, \mu )_b.   
\end{equation} 
In particular, 
\begin{equation} 
\label{eq:2.1.8}
\cE^0 (u,u) = -\int_{V}\overline{L}^V u\, u\, d\mu,  u\in D(\overline{L}^V)_b. 
\end{equation} 
\end{proposition} 

\begin{proof} 
The complete proof of Proposition \ref{prop:2.1} is given in \cite{WS99}. Let us only state its essential steps in the following. \\
(i)  {\bf Step 1:} To show that $(L^V, D(L^{0,V})_b)$ is dissipative, it suffices to show that 
$$
\int_V L^Vu \psi (u)\, d\mu \le 0 \, , u\in D(L^{0,V})_b , 
$$ 
with $\psi = 1_{(0, \infty)} - 1_{(-\infty , 0)}$, since $\|u\|_1 \psi (u) \in L^\infty (V,\mu ) 
= (L^1 (V, \mu ))'$ is a normalized tangent functional to $u$. Since $\psi  = 1_{(0, \infty)} 
- 1_{(-\infty , 0)}$ is monotone increasing, it therefore suffices to extend the Beurling--Deny property \eqref{eq:2.1.0e} to this setting. But this follows from the well-known fact that it holds for the generator $L^{0,V}$ of the symmetric Dirichlet form (\cite{BH}) and since 
$u\in H_0^{1,2} (V, \mu )$ implies $|u|\in H_0^{1,2} (V,\mu)$,  
$$ 
\int_V \langle \mathbf{B} , \nabla u\rangle \psi (u) \, d\mu = 
\int_V \langle \mathbf{B} , \nabla |u|\rangle \, d\mu = 0. 
$$ 
\noindent
{\bf Step 2:} In the next step one shows that the closure $(\overline{L}^V, D(\overline{L}^V))$ generates 
a $C_0$-semigroup of contractions $(\overline{T}^V_t)_{t> 0}$ on $L^1(\R^d, \mu)$. To this end by \cite[Theorem 3.1]{LP61}),  
verifies the range condition:  $(1-L^V)(D(L^{0,V})_b)\subset L^1 (V , \mu )$ dense.  
Indeed, let $h\in L^\infty (V, \mu )$ be such that $\int_V (1-L^V)u\, h\, d\mu = 0$ for all $u\in D(L^{0,V})_b$. Then $u\mapsto\int_V (1- L^{0,V})u\, h \, d\mu = \int_V \langle\mathbf{B} , \nabla u\rangle h\, d\mu$, $u\in D(L^{0,V})_b$, is continuous w.r.t. the norm on $H_0^{1,2}(V,\mu )$ which implies the existence of some element $v\in H_0^{1,2} (V, \mu )$ such that $\cE^0_1(u,v) = \int_V (1- L^{0,V})u\, h \, d\mu $ for all $u\in D(L^{0,V})_b$. It follows that 
$\int_V (1- L^{0,V})u (h-v) \, d\mu = 0$ for all $u\in D(L^{0,V})_b$. Since the semigroup generated by $(L^0 , D(L^{0,V}))$ is in particular $L^\infty$-contractive, we obtain that $(1-L^{0,V})(D(L^{0,V})_b) \subset L^1 (V,\mu )$ dense and consequently, $h = v$. In particular, $h\in 
H_0^{1,2} (V, \mu )$ and 
$$ 
\begin{aligned}
\cE^0_1(h,h) 
& = \lim_{t\to 0+}\cE^0_1(T_t^{0,V}h,h) =  \lim_{t\to 0+}\int_V (1-L^{0,V})T^{0,V}_th\,  h \, d\mu \\ 
& =  \lim_{t\to 0+}\int_V \langle \mathbf{B} , \nabla T^{0,V}_t h \rangle h \, d\mu 
=  \int_V \langle \mathbf{B} , \nabla h \rangle h \, d\mu = 0
\end{aligned} 
$$ 
by \eqref{eq:2.1.5} and \cite[Lemma 1.3.3(iii)]{FOT} and therefore $h = 0$. 
\noindent
{\bf Step 3:} $(\overline{T}^V_t)_{t > 0}$ is sub-Markovian. 
\noindent
This follows from the fact that the Beurling--Deny property \eqref{eq:2.1.0e} for $(L^V,D(L^{0,V})_b)$ extends to its closure. In particular,  
$$ 
\int_V\overline{L}^Vu\, 1_{\{u > 1\}}\, d\mu\le 0.
$$ 
It is well-known that this property now implies that the semigroup  $(\overline{T}^V_t)_{t> 0}$ 
is sub-Markovian. \\
(ii) In order to verify the integration by parts \eqref{eq:2.1.7} first note that it holds for 
$u\in D(L^{0,V})_b$ by the construction of $L^V$. It remains to extend it to bounded $u$ in the closure 
$u\in D(\overline{L}^{V})_b$. This is not straightforward, since convergence of $(u_n)_{n\ge 1} \subset 
D(L^{0,V})_b$ to $u\in D (\overline{L}^V)$ w.r.t. the graph norm does not immediately imply convergence in 
$H_0^{1,2} (V, \mu )$. One therefore needs to apply a suitable cutoff function $\psi\in C_b^2 (\mathbb R)$ 
such that $\psi (t) = t$ if $|t|\le\|u\|_{L^\infty (\R^d , \mu )} + 1$ and $\psi (t) = 0$ if $|t|\ge 
\|u\|_{L^\infty (\R^d , \mu )} + 2$, to pass to the uniformly bounded sequence $(\psi (u_n))_{n\ge 1} 
\subset D(L^{0,V})_b$. Clearly, 
$$ 
\overline{L}^V\psi (u_n) = \psi '(u_n) L^V u_n + \frac 12 \psi '' (u_n)\langle A\nabla u_n , \nabla u_n\rangle 
$$ 
and the essential step now is to verify that 
$$ 
\lim_{n\to\infty} \psi '' (u_n)\langle A\nabla u_n , \nabla u_n\rangle = 0 \text{ on } L^1 (V,\mu)\, , 
$$
since this then implies $\lim_{n\to\infty}\overline{L}^V\psi (u_n) =  \overline{L}^Vu$,  
$(\psi (u_n))_{n\ge 1} \subset H_0^{1,2} (V,\mu )$ bounded, hence $u\in H_0^{1,2} (V,\mu )$, and 
\eqref{eq:2.1.7} holds for the limit $u\in D(\overline{L}^V)_b$. 
\end{proof}

\begin{remark} 
\label{rem2.1.3}
{\it Let \eqref{condition on mu}--\eqref{eq:2.1.4} be satisfied and $V$ be a relatively compact open subset in $\R^d$. 
Since $- \left(\mathbf{G}- \beta^{\rho, A}\right)$ satisfies the same assumptions as 
$\mathbf{G} - \beta^{\rho, A}$, the closure $(\overline{L}^{V,\prime},  
D(\overline{L}^{V,\prime}))$ of $L^{0,V}u -\langle \mathbf{G} - \beta^{\rho, A}, 
\nabla u\rangle$, $u\in D(L^{0,V})_b$, on $L^1 (V, \mu)$ generates a sub-Markovian $C_0$-semigroup 
of contractions $(\overline{T}^{V,\prime}_t)_{t> 0}$, $D(\overline{L}^{V,\prime})_b 
\subset H_0^{1,2}(V,\mu )$ and 
$$ 
\cE^0 (u,\varv) + \int_V\langle\mathbf{G} - \beta^{\rho, A} , \nabla u\rangle\varv\, d\mu  
= - \int_V\overline{L}^{V,\prime}u\,\varv\, d\mu\, ,  
u\in D(\overline{L}^{V,\prime})_b, \varv\in H_0^{1,2}(V, \mu )_b. 
$$
If $(L^{V, \prime}, D(L^{V, \prime}))$ is the part of $(\overline{L}^{V,\prime},  
D(\overline{L}^{V,\prime}))$ on $L^2 (V, \mu )$ and $(L^V, D(L^V))$ is the part of 
$(\overline{L}^V, D(\overline{L}^V))$ on $L^2 (V, \mu )$, then 
\begin{equation} 
\label{eq:2.1.11} 
\begin{aligned} 
(L^Vu, & \varv )_{L^2 (V, \mu )}  = -\cE^0 (u,\varv) + \int_V\langle\mathbf{G}  
- \beta^{\rho, A} ,\nabla u \rangle\varv\, d\mu  \\ 
& = -\cE^0 (\varv ,u) - \int_V\langle\mathbf{G} - \beta^{\rho, A}, 
\nabla\varv \rangle u\, d\mu  
= (L^{V, \prime}\varv ,u )_{L^2 (V, \mu )} 
\end{aligned} 
\end{equation} 
for all $u\in D(L^V)_b$, $v\in D(L^{V,\prime})_b$. Since $(L^V, D(L^V))$ 
(resp. $(L^{V,\prime}, D(L^{V,\prime}))$) is the generator of a sub-Markovian $C_0$-semigroup, 
it follows that $D(L^V)_b\subset D(L^V)$ (resp. $D(L^{V, \prime})_b\subset D(L^{V,\prime})$) dense 
w.r.t. the graph norm, \eqref{eq:2.1.11} extends to all $u\in D(L^V)$, $\varv\in D(L^{V, \prime})$, 
which implies that the parts of $\overline{L}^V$ and $\overline{L}^{V, \prime}$ on $L^2 (V, \mu )$ 
are adjoint operators. }
\end{remark} 
\noindent
Note that the sub-Markovian $C_0$-semigroup of contractions 
$(\overline{T}^V_t)_{t> 0}$ on $L^1 (V, \mu )$ can be restricted to a 
semigroup of contractions on $L^r (V, \mu )$ for all $r\in [1, \infty)$ by 
the Riesz-Thorin interpolation theorem (cf. \cite[Theorem IX.17]{ReSi}) and that 
the restricted semigroup is strongly continuous on $L^r (V, \mu )$. The 
corresponding generator $(\overline{L}^V_r, D(\overline{L}^V_r))$ is the 
{\bf part of} $(\overline{L}^V, D(\overline{L}^V))$ on $L^r (V, \mu )$, i.e., 
$D(\overline{L}^V_r) = \{u\in D(\overline{L}^V)\cap L^r(V,\mu ):\overline
{L}^Vu\in L^r (V, \mu )\}$ and $\overline{L}^V_r u = \overline{L}^Vu$, 
$u\in D(\overline{L}^V_r)$.

\paragraph{Existence of maximal extensions on the full domain $\R^d$} 
\label{subsec:2.1.3}

We are now going to extend the previous existence result to the full domain. 
For any relatively compact open subset $V$ in $\R^d$ let
$(\overline{L}^V, D(\overline{L}^V))$ be the maximal extension of $(L, C_0^\infty (V))$ 
on $L^1 (V, \mu )$ constructed in Proposition \ref{prop:2.1} and $(\overline{T}_t^V)_{t> 0}$ 
be the associated sub-Markovian $C_0$-semigroup of contractions. Recall from linear semigroup theory that 
for $\alpha > 0$, the operator $(\alpha - \overline{L}^V , D(\overline{L}^V))$ is invertible 
with bounded inverse $\overline{G}^V_\alpha = (\alpha - \overline{L}^V)^{-1}$. 
$(\overline{G}^V_\alpha )_{\alpha > 0}$ is called the resolvent generated by $\overline{L}^V$ and it is given as the Laplace transform 
$$ 
\overline{G}^V_\alpha  = \int_0^\infty e^{-\alpha t} \overline{T}_t^V \, dt, \alpha > 0, 
$$ 
of the semigroup. The strong continuity of $(\overline{T}_t^V)_{t> 0}$ implies the strong continuity 
$\lim_{\alpha\to\infty} \alpha \overline{G}_\alpha^V f = f$ in $L^1 (V, \mu )$ of the resolvent and 
sub-Markovianity of the semigroup implies the same for $\alpha\overline{G}_\alpha^V$. \\
If we define 
$$ 
\overline{G}^V_\alpha f := \overline{G}^V_\alpha (f1_V), f\in L^1 
(\R^d , \mu ), \alpha > 0, 
$$ 
then $\alpha \overline{G}^V_\alpha$, $\alpha > 0$, can be extended to a 
sub-Markovian contraction on $L^1 (\R^d, \mu )$, which is, however, no longer 
strongly continuous in the usual sense, but still satisfies $\lim_{\alpha\to\infty} 
\alpha\overline{G}^V_\alpha f = f1_V$ in $L^1 (\R^d , \mu )$.\\
The crucial observation for the existence of an extension now is the following domain 
monotonicity:  

\begin{lemma}
\label{lemma2.1.6}  
Let \eqref{condition on mu}--\eqref{eq:2.1.4} be satisfied. Let $V_1$, $V_2$ be relatively compact open subsets in 
$\R^d$ and $V_1\subset V_2$. Let $u\in L^1 (\R^d , \mu )$, $u\ge 0$, and 
$\alpha > 0$. Then $\overline{G}^{V_1}_\alpha u\le\overline{G}^{V_2}_\alpha u$. 
\end{lemma} 

\begin{proof} Clearly, we may assume that $u$ is bounded. Let 
$w_\alpha := \overline{G}^{V_1}_\alpha u - \overline{G}^{V_2}_\alpha u $. 
Then $w_\alpha\in H_0^{1,2}(V_2, \mu )$ but also 
$w_\alpha^+\in H_0^{1,2}(V_1, \mu )$ since $w_\alpha^+\le\overline
{G}^{V_1}_\alpha u$ and $\overline{G}^{V_1}_\alpha u\in H_0^{1,2}
(V_1, \mu )$. Note that $\int_{\R^d} \langle\mathbf{B}, 
\nabla  w_\alpha\rangle w_\alpha^+\, d\mu = \int_{\R^d} \langle\mathbf{B}, \nabla  
w_\alpha^+\rangle w_\alpha^+\, d\mu  = 0$ and that $\cE^0 (w_\alpha^+, w_\alpha^-) 
\le 0$, since $(\cE^0 , H_0^{1,2}(V_2, \mu ))$ is a Dirichlet form. Hence 
by \eqref{eq:2.1.7}
$$
\begin{aligned} 
\cE_\alpha^0 (w_\alpha^+ , w_\alpha^+) & \le\cE^0_\alpha (w_\alpha , 
w_\alpha^+ ) - \int_{\R^d} \langle\mathbf{B},\nabla w_\alpha\rangle w_\alpha^+ \, d\mu \\ 
& = \int_{\R^d} (\alpha - \overline{L}^{V_1})\overline{G}^{V_1}_\alpha u\,  
w_\alpha ^+ \, d\mu - \int_{\R^d} (\alpha - \overline{L}^{V_2})\overline
{G}^{V_2}_\alpha u \, w_\alpha ^+ \, d\mu = 0. 
\end{aligned}  
$$  
Consequently, $w_\alpha^+ = 0$, i.e., $\overline{G}^{V_1}_\alpha u
\le\overline{G}^{V_2}_\alpha u$. 
\end{proof}

\begin{theorem}
\label{theorem2.1.5} 
Let \eqref{condition on mu}--\eqref{eq:2.1.4}
be satisfied and let $(L^0 , D(L^0 ))$ be the generator of 
$(\cE^0 , D(\cE^0))$ (see \eqref{DefGeneratorDF}). Then 
there exists a closed extension $(\overline{L}, D(\overline{L}))$ of 
\begin{eqnarray}\label{defLfirst}
Lu := L^0 u + \langle \mathbf{B}, \nabla u\rangle, \quad u\in D(L^0)_{0,b},
\end{eqnarray}  
on $L^1 (\R^d , \mu )$ satisfying the following properties: 

\item{(i)} $(\overline{L}, D(\overline{L}))$ generates a sub-Markovian 
$C_0$-semigroup of contractions $(\overline{T}_t)_{t> 0}$. In particular 
$(\overline{L}, D(\overline{L}))$ is a {\bf maximal extension}\index{operator ! maximal extension} 
of 
$$
(\frac 12 \sum_{i,j =1}^d a_{ij}\partial_{ij}  + \sum_{i=1}^d g_{i}\partial_{i},
C_0^\infty (\R^d))
$$
(cf. \eqref{eq:2.1.3bis} and \eqref{eq:2.1.3b}).

\item{(ii)} If $(V_n)_{n\ge 1}$ is an increasing sequence of 
relatively compact open subsets in $\R^d$ such that $\R^d = \bigcup_{n\ge 1} V_n $ then 
$$
\overline{G}_\alpha f:=\lim_{n\to\infty}\overline{G}^{V_n}_\alpha f = (\alpha - \overline
{L})^{-1}f
$$ 
in $L^1 (\R^d , \mu )$ for all $f\in L^1 (\R^d , \mu )$ and $\alpha > 0$. In particular, 
$(\overline{G}_\alpha)_{\alpha>0}$ is a sub-Markovian $C_0$-resolvent of contractions on $L^1(\R^d,\mu)$ and has 
$(\overline{L}, D(\overline{L}))$ as generator.
\item{(iii)} $D(\overline{L})_b\subset D(\cE^0)$ and 
$$ 
\cE^0 (u,\varv ) - \int_{\R^d}\langle\mathbf{B}, \nabla u\rangle\varv\, d\mu = - \int_{\R^d}
\overline{L}u\, \varv\, d\mu \, ,u\in D(\overline{L})_b , \varv\in H^{1,2}_0 (\R^d , \mu )_{0,b}. 
$$
Moreover, 
$$
\cE^0 (u,u) \le -\int_{\R^d}\overline{L}u\, u\, d\mu \, , u\in D(\overline{L})_b. 
$$
\end{theorem}

\begin{proof} 
The complete proof of Theorem \ref{theorem2.1.5} is given in \cite{WS99}. Let us again only 
state its essential steps in the following. \\
Let $(V_n)_{n\ge 1}$ be some increasing sequence of relatively compact open subsets in $\R^d$ such that $\overline{V}_n\subset V_{n+1}$, $n\ge 1$, and $\R^d = \bigcup_{n\ge 1}
V_n$. Let $f\in L^1 (\R^d , \mu )$, $f\ge 0$. Then $\lim_{n\to\infty}\overline
{G}^{V_n}_\alpha f =: \overline{G}_\alpha f$ exists $\mu$-a.e. by Lemma \ref{lemma2.1.6}.  Since  
$$ 
\int_{\R^d}\alpha\overline{G}^{V_n}_\alpha f\, d\mu\le\int_{\R^d} f1_{V_n}\, d\mu, 
\le \int_{\R^d} f\, d\mu 
$$ 
the sequence converges in $L^1 (\R^d , \mu )$, and 
\begin{equation}
\label{eq:2.1.12}
\int_{\R^d}\alpha\overline{G}_\alpha f\, d\mu\le\int_{\R^d} f\, d\mu, 
\end{equation} 
in particular $\alpha\overline{G}_\alpha$ is a linear contraction on $L^1 (\R^d , \mu )$. 
Since $\alpha\overline{G}^{V_n}_\alpha$ is sub-Markovian, the limit 
$\alpha\overline{G}_\alpha$ is sub-Markovian too. Also the resolvent equation follows 
immediately. \\
The strong continuity of $(\overline{G}_\alpha )_{\alpha > 0}$ is verified as follows. Let 
$u\in D(L^0)_{0,b}$, hence $u\in D(L^{0,V_n})_b$ for large $n$ and thus 
$u = \overline{G}^{V_n}_\alpha (\alpha - L^{V_n})u = \overline{G}^{V_n}_\alpha (\alpha - L)u$. 
Hence  
\begin{equation}
\label{eq:2.1.15}
u = \overline{G}_\alpha (\alpha - L)u. 
\end{equation}
In particular, 
$$ 
\begin{aligned} 
\|\alpha\overline{G}_\alpha u - u\|_{L^1 (\R^d ,\mu )} 
& = \|\alpha\overline{G}_\alpha u - \overline{G}_\alpha (\alpha - L)u\|_{L^1 (\R^d ,\mu )} 
= \|\overline{G}_\alpha Lu\|_{L^1 (\R^d ,\mu )} \\ 
& \le \frac 1\alpha \|Lu\|_{L^1 (\R^d ,\mu )}\to 0\,  , \alpha\to\infty ,
\end{aligned} 
$$ 
for all $u\in C_0^\infty (\R^d)$ and the strong continuity then follows by a $3 \ve$-argument.  \\
Let $(\overline{L}, D(\overline{L}))$ be the generator of $(\overline
{G}_\alpha )_{\alpha > 0}$. Then $(\overline{L}, D(\overline{L}))$ extends 
$(L,  D(L^0)_{0,b})$ by \eqref{eq:2.1.15}. By the Hille-Yosida Theorem $(\overline{L}, 
D(\overline{L}))$ generates a $C_0$-semigroup of contractions $(\overline
{T}_t)_{t> 0}$. Since $\overline{T}_t u = \lim_{\alpha\to\infty}\exp (t
\alpha (\alpha \overline{G}_\alpha - 1))u$ for all $u\in L^1 (\R^d, \mu )$ 
(cf. \cite[Chapter 1, Corollary 3.5]{Pa}) we obtain that $(\overline{T}_t)_{t> 0}$ is 
sub-Markovian. \\
To see that the construction of $(\overline{L}, D(\overline{L}))$ is actually independent 
of the exhausting sequence, let $(W_n)_{n\ge 1}$ be another increasing sequence of 
relatively compact open subsets in $\R^d$ such that $\R^d = \bigcup_{n\ge 1} W_n$. 
Compactness of $\overline{V}_n$ 
then implies that $V_n\subset W_m$ for some $m$, hence $\overline{G}^{V_n}_\alpha f 
\le \overline{G}^{W_m}_\alpha f$ by Lemma \ref{lemma2.1.6}, so $\overline{G}_\alpha f 
\le\lim_{n\to\infty}\overline{G}^{W_n}_\alpha f$. Similarly, $\lim_{n\to\infty} 
\overline{G}^{W_n}_\alpha f\le\overline{G}_\alpha f$, hence (i) is satisfied.   \\
Finally, the integration by parts (iii) is first verified for $u = \overline{G}_\alpha f$, 
$f\in L ^1 (\R^d , \mu )_b$. For such $u$ one first shows that $\lim_{n\to\infty}  
\overline{G}^{V_n}_\alpha f = u$ weakly in $D(\cE^0)$, since by \eqref{eq:2.1.8} 
\begin{equation*}  
\begin{aligned} 
\cE^0 (\overline{G}^{V_n}_\alpha f , \overline{G}^{V_n}_\alpha f) 
& = - \int_{V_n} \overline{L}^{V_n} \overline{G}^{V_n}_\alpha f \overline{G}^{V_n}_\alpha f \, d\mu \\
& = \int_{V_n} (f1_{V_n} - \alpha \overline{G}^{V_n}_\alpha f )\overline{G}^{V_n}_\alpha f \, d\mu \\
& \le \frac {1}{\alpha}\ \|f\|_{L^1 (\R^d ,\mu)}\|f\|_{L^\infty (\R^d ,\mu)}. 
\end{aligned} 
\end{equation*} 
We can therefore take the limit in \eqref{eq:2.1.7} to obtain the integration by parts (iii)  
in this case. To extend (iii) finally to all $u\in D(\overline{L})_b$, it suffices to consider 
the limit $u = \lim_{\alpha\to\infty} \alpha\overline{G}_\alpha u$ weakly in $D(\mathcal{E}^0)$. 
\end{proof}

\begin{remark} 
\label{remark2.1.7} 
{\it (i) Clearly, $(\overline{L}, D(\overline{L}))$ is uniquely determined by 
properties (i) and (ii) in Theorem \ref{theorem2.1.5}. 

\noindent 
(ii) Similarly to $(\overline{L}, D(\overline{L}))$ we can construct a closed 
extension $(\overline{L}^{\prime}, D(\overline{L}^{\prime}))$ of 
\begin{eqnarray}\label{defL'first}
L^{\prime}u:=L^0 u - \langle\mathbf{B}, 
\nabla u\rangle, \quad u\in D(L^0)_{0,b}, 
\end{eqnarray}
that generates a sub-Markovian  
$C_0$-semigroup of contractions $(\overline{T}'_t)_{t> 0}$. Since for 
all $V$ relatively compact in $\R^d$ by \eqref{eq:2.1.11}   
\begin{equation}
\label{eq:2.1.16}
\int_{\R^d}\overline{G}^V_\alpha u\, \varv\, d\mu = \int_{\R^d} u\overline{G}_\alpha
^{V,\prime} \varv\, d\mu\text{ for all }u,\varv\in L^1 (\R^d, \mu )_b, 
\end{equation} 
it follows that 
\begin{equation}
\label{eq:2.1.17}
\int_{\R^d}\overline{G}_\alpha u\, \varv\, d\mu = \int_{\R^d} u\,\overline{G}^{\prime}_\alpha \varv\, d\mu 
\text{ for all }u,\varv\in L^1 (\R^d, \mu )_b, 
\end{equation} 
where $\overline{G}^{\prime}_\alpha = (\alpha - \overline{L}^{\prime})^{-1}$. 

\noindent 
(iii) The construction of the maximal extension $\overline{L}$ can be extended to the case 
of arbitrary open subsets $W$ in $\R^d$ (see \cite[Theorem 1.5]{WS99} for further details). }
\end{remark}

\begin{definition} 
\label{definition2.1.7} 
Let \eqref{condition on mu}--\eqref{eq:2.1.4} be satisfied. 
$(\overline{T}_t)_{t>0}$ (see Theorem \ref{theorem2.1.5} (i)) restricted to 
$L^1(\R^d, \mu)_b$ can be extended to a sub-Markovian $C_0$-semigroup of contractions on 
$L^s(\R^d, \mu)$, $s\in (1, \infty)$ and to a sub-Markovian semigroup on 
$L^{\infty}(\R^d, \mu)$. These semigroups will all be denoted by $(T_t)_{t>0}$ and in order 
to simplify notations, $(T_t)_{t>0}$ shall denote the semigroup on $L^s(\R^d, \mu)$ for any 
$s \in [1, \infty]$ from now on, whereas $(\overline{T}_t)_{t>0}$ denotes the semigroup 
acting exclusively on $L^1(\R^d, \mu)$. Likewise, we define $(T'_t)_{t>0}$ acting on all 
$L^s(\R^d, \mu)$, $s \in [1, \infty]$, corresponding to $(\overline{T}'_t)_{t>0}$ as in 
Remark \ref{remark2.1.7}(ii) which acts exclusively on $L^1(\R^d, \mu)$. The resolvents 
corresponding to $(T_t)_{t>0}$ are also all denoted by $(G_{\alpha})_{\alpha>0}$, those 
corresponding to $(T'_t)_{t>0}$ by $(G'_{\alpha})_{\alpha>0}$.\\
Furthermore, we denote by $(L_s,D(L_s))$, $(L_s^{\prime},D(L_s^{\prime}))$, the generators 
corresponding to $(T_t)_{t>0}$, $(T'_t)_{t>0}$ defined on $L^s(\R^d, \mu)$,  
$s \in [1, \infty)$, so that in particular $(L_1,D(L_1))=(\overline{L}, D(\overline{L}))$, 
$(L_1^{\prime},D(L_1^{\prime}))=(\overline{L}' , D(\overline{L}'))$. \eqref{eq:2.1.17} 
implies that $L_2$ and $L_2^{\prime}$ are adjoint operators on $L^2 (\R^d , \mu )$, 
and that $(T_t)_{t > 0}$ is the adjoint semigroup of $(T'_t)_{t> 0}$ when considered on 
$L^2 (\R^d , \mu )$.
\end{definition}

\begin{lemma} \label{lemma2.1.4}
Let \eqref{condition on mu}--\eqref{eq:2.1.4} be satisfied and $(\overline{L}, D(\overline{L}))$ be as in Theorem \ref{theorem2.1.5}. 
The space $D(\overline{L})_b$ is an algebra, i.e., $u,\varv\in D(\overline{L})_b$ implies 
$u\varv\in D(\overline{L})_b$. Moreover, 
\begin{equation} 
\label{eq:2.1.17a} 
\overline{L} (u\varv) = \varv \overline{L} u + u \overline{L} \varv + \langle A\nabla u, 
\nabla \varv\rangle. 
\end{equation}  
\end{lemma}

\begin{proof} 
It suffices to prove that $u\in D(\overline{L})_b$ implies $u^2\in D(\overline{L})_b$ and 
$\overline{L} (u ^2 ) = g := 2u\overline{L}u + \langle A\nabla u, \nabla u\rangle$. To this end it is sufficient to show that  
\begin{equation}
\label{eq:2.1.18}
\begin{aligned} 
\int_{\R^d}\overline{L}'\varv\, u^2\, d\mu & = \int_{\R^d} g\varv\, d\mu \text{ for all }\varv = 
\overline{G}'_1h\, , h\in L^1 (\R^d , \mu )_b,  
\end{aligned} 
\end{equation} 
since then $\int_{\R^d}\overline{G}_1 (u^2 - g)h\, d\mu = \int_{\R^d} (u^2 - g)\overline
{G}_1 ' h\, d\mu = \int_{\R^d} u^2 (\overline{G}_1 ' h - \overline{L}' \overline
{G}_1 ' h) \, d\mu = \int_{\R^d} u^2 \, h\, d\mu$ for all $h\in L^1 (\R^d , \mu )_b$. 
Consequently, $u^2 = \overline{G}_1 (u^2 - g)\in D(\overline{L})_b$. 

\noindent 
For the proof of \eqref{eq:2.1.18}  fix $\varv = \overline{G}'_1h$, $h\in L^1 (\R^d, \mu )_b$, 
and suppose first that $u=\overline{G}_1f$ for some $f\in L^1 (\R^d , \mu )_b$.  
Let $u_n := \overline{G}_1^{V_n}f$ and $\varv_n = \overline{G}_1^{V_n ,\prime}h$, 
where $(V_n)_{n\ge 1}$ is as in Theorem \ref{theorem2.1.5}(ii). Then by Proposition 
\ref{prop:2.1} and Theorem \ref{theorem2.1.5}   
$$ 
\begin{aligned} 
& \int_{V_n}\overline{L}^{V_n ,\prime}\varv_n\, uu_n \, d\mu  = -\cE^0 (\varv_n , uu_n) 
- \int_{V_n}\langle\mathbf{B},\nabla \varv_n\rangle uu_n\, d\mu \\ 
& = -\cE^0 (\varv_n u_n , u) - \frac 12\int_{\R^d}\langle A\nabla \varv_n , \nabla u_n\rangle u\, 
d\mu + \frac 12\int_{\R^d}\langle A\nabla u_n, \nabla u\rangle \varv_n\, d\mu \\ 
& \qquad + \int_{\R^d}\langle\mathbf{B}, \nabla u\rangle \varv_n u_n\, d\mu + \int_{\R^d}\langle
\mathbf{B}, \nabla u_n \rangle \varv_nu\, d\mu \\ 
& = \int_{\R^d}\overline{L}u\, \varv_n u_n\, d\mu + \int_{V_n}\overline{L}^{V_n}u_n\,  \varv_n 
u\, d\mu + \frac 12\int_{\R^d}\langle A\nabla u_n, \nabla (\varv_n u)\rangle\, d\mu \\ 
& \qquad - \frac 12 \int_{\R^d}\langle A\nabla \varv_n , \nabla u_n\rangle u\, d\mu 
+ \frac 12 \int_{\R^d}\langle A\nabla u_n, \nabla u\rangle \varv_n\, d\mu \\ 
& = \int_{\R^d}\overline{L}u\, \varv_n u_n\, d\mu + \int_{V_n}\overline{L}^{V_n}u_n\, \varv_n u\, d\mu  
+ \int_{\R^d}\langle A\nabla u_n, \nabla u\rangle \varv_n\, d\mu. 
\end{aligned} 
$$ 
Note that $\lim_{n\to\infty}\int_{\R^d}\langle A\nabla u_n , \nabla u\rangle \varv_n \, 
d\mu = \int_{\R^d}\langle A\nabla u , \nabla u \rangle \varv\, d\mu $ since 
$\lim_{n\to\infty}u_n = u$ weakly in $D(\cE^0)$ and $\lim_{n\to\infty}
\langle A\nabla u , \nabla u\rangle \varv_n^2 = \langle A\nabla u , \nabla u
\rangle \varv^2$ (strongly) in $L^1 (\R^d , \mu )$. Hence 
$$ 
\begin{aligned} 
\int_{\R^d}\overline{L}'\varv\, u^2\, d\mu 
& = \lim_{n\to\infty}\int_{V_n} \overline{L}^{V_n ,\prime}\varv_n\, uu_n \, d\mu \\ 
& = \lim_{n\to\infty}\int_{\R^d} \overline{L}u\, \varv_nu_n\, d\mu 
+ \int_{V_n}\overline{L}^{V_n}u_n\, \varv_n u\, d\mu + \int_{\R^d}\langle A\nabla u_n,  
\nabla u\rangle \varv_n\, d\mu \\ 
& = \int_{\R^d} g\varv\, d\mu.   
\end{aligned} 
$$  
Finally, if $u\in D(\overline{L})_b$ arbitrary, let $g_\alpha := 2(\alpha
\overline{G}_\alpha u)\overline{L}(\alpha \overline{G}_\alpha u) + \langle 
A\nabla \alpha\overline{G}_\alpha u,\nabla \alpha\overline{G}_\alpha u
\rangle$, $\alpha > 0$. Note that by Theorem \ref{theorem2.1.5}(iii) 
$$ 
\begin{aligned}  
\cE^0 (\alpha\overline{G}_\alpha u - u , \alpha\overline{G}_\alpha u - u) 
& \le -\int_{\R^d}\overline{L}(\alpha\overline{G}_\alpha u - u ) 
(\alpha\overline{G}_\alpha u - u) \, d\mu \\ 
& \le 2\|u\|_{L^\infty (\R^d , \mu )}  \|\alpha\overline{G}_\alpha\overline{L}u 
- \overline{L}u \|_{L^1 (\R^d , \mu )} \to 0  
\end{aligned} 
$$   
if $\alpha\to\infty$, which implies that $\lim_{\alpha\to\infty}\alpha
\overline{G}_\alpha u = u$ in $D(\cE^0)$ and thus $\lim_{\alpha\to\infty} 
g_\alpha = g$ in $L^1 (\R^d , \mu )$. Since $u + (1-\alpha)\overline
{G}_\alpha u\in L^1 (\R^d , \mu )_b$ and $\overline{G}_1 (u + (1-\alpha )
\overline{G}_\alpha u) = \overline{G}_\alpha u$ by the resolvent 
equation it follows from what we have just proved that 
$$
\int_{\R^d}\overline{L}'\varv (\alpha\overline{G}_\alpha u)^2\, d\mu 
= \int_{\R^d} g_\alpha \varv\, d\mu 
$$ 
for all $\alpha > 0$ and thus, taking the limit $\alpha\to\infty$,  
$$ 
\int_{\R^d}\overline{L}'\varv\, u^2\, d\mu = \int_{\R^d} g\varv\, d\mu 
$$ 
and \eqref{eq:2.1.18} is shown. 
\end{proof}

\subsubsection{Uniqueness of maximal extensions on $\R^d$} 
\label{subsec:2.1.4}

Having established the existence of maximal extensions $(\overline{L}, D(\overline{L}))$ 
of $(L, D(L^0)_{0,b})$ in $L^1 (\R^d, \mu )$, where $L^0$ denotes the generator 
of the symmetric Dirichlet form $\cE^0$ (see \eqref{DefGeneratorDF}) and  
$D(L^0)_{0,b}$ the subspace of compactly supported bounded functions in $D(L^0 )$, 
we now discuss the uniqueness of $\overline{L}$ and the connections of the uniqueness 
problem with global properties of the associated semigroup $(\overline{T}_t)_{t> 0}$. \\[3pt]
The uniqueness of maximal extensions of $L$ is linked to the domain $D$ on which we  
consider the operator $L$. It is clear that there can be only one maximal extension 
of $(L, D)$ if $D\subset D(\overline{L})$ is dense w.r.t. the graph norm of 
$\overline{L}$, but in general such dense subsets are quite difficult to characterize. 
We will consider this problem in the following exposition for two natural choices: the 
domain $D(L^0)_{0,b}$ and the domain $C_0^\infty (\R^d )$ of compactly 
supported smooth functions. \\[3pt]
Let us first introduce two useful notations: 

\begin{definition}\label{def:2.1.1}
Let \eqref{condition on mu}--\eqref{eq:2.1.4} be satisfied. Let $(\overline{T}_t)_{t> 0}$, $(\overline{T}'_t)_{t>0}$ and $(T_t')_{t>0}$ be as in Theorem \ref{theorem2.1.5},  Remark \ref{remark2.1.7}(ii) and Definition \ref{definition2.1.7}.\\[3pt]
(i) Let $r\in [1,\infty )$ and $(A, D)$ be a 
densely defined operator on $L^r (\R^d , \mu )$. We say that $(A ,D)$ is $L^r(\R^d, \mu)$-{\bf unique} (hereafter written for convenience as \textbf{$L^r$-unique})\index{uniqueness ! $L^r$-unique}, if there is only one extension of $(A , D)$ on 
$L^r(\R^d , \mu )$ that generates a $C_0$-semigroup. It follows from \cite[Theorem A-II, 1.33]{Na86}, that if $(A, D)$ is $L^r$-unique and $(\overline A, \overline D)$ 
its unique extension generating a $C_0$-semigroup, then $D\subset\overline D$ 
dense w.r.t. the graph norm. Equivalently, $(A, D)$ is $L^r$-unique, if and only if the 
range condition\index{operator ! range condition}  $(\alpha -A)(D)\subset L^r (\R^d , \mu )$ dense holds for some $\alpha > 0$.   \\[3pt]
(ii) Let $(S_t)_{t> 0}$ be a sub-Markovian $C_0$-semigroup 
on $L^1 (\R^d , \nu )$. We say that $\nu$ is \textbf{$(S_t)_{t> 0}$-invariant} 
\index{measure ! invariant} 
(resp. \textbf{$\nu$ is $(S_t)_{t> 0}$-sub-invariant})\index{measure ! sub-invariant}, if $\int_{\R^d} S_t f\, d\nu  
= \int_{\R^d} f\, d\nu$  (resp. $\int_{\R^d} S_t f\, d\nu \le \int_{\R^d} f\, d\nu$) for all 
$f\in L^1 (\R^d , \nu )_b$ with $f\ge 0$ and $t>0$.\\[3pt]
In particular, $\mu$ is always $(\overline{T}_t)_{t> 0}$-sub-invariant, since for $f\in L^1 (\R^d , \mu )_b$, $f\ge 0$ and $t>0$, we have by the sub-Markov property $\int_{\R^d} \overline{T}_t f\, d\mu  
= \int_{\R^d} fT_t' 1_{\R^d}\, d\mu\le \int_{\R^d} f\, d\mu$. Likewise, $\mu$ is always $(\overline{T}'_t)_{t>0}$-sub-invariant.
\end{definition}

\paragraph{Uniqueness of $(L, D(L^0)_{0,b})$} 

\begin{proposition}  
\label{prop2.1.9} 
Let \eqref{condition on mu}--\eqref{eq:2.1.4} be satisfied. Let $(L^0 , D(L^0 ))$ be the generator of $(\cE^0 , D(\cE^0))$ (see \eqref{DefGeneratorDF}) and recall that as in Theorem \ref{theorem2.1.5}
$$
Lu := L^0 u + \langle \mathbf{B}, \nabla u\rangle, \quad u\in D(L^0)_{0,b}.
$$
The following statements are equivalent: 
\item{(i)} $(L, D(L^0)_{0,b})$ is $L^1$-unique. 
\item{(ii)} $\mu$ is $(\overline{T}_t)_{t>0}$-invariant. 
\item{(iii)} There exist $\chi_n\in H^{1,2}_{loc}(\R^d , \mu )$ and 
$\alpha > 0$ such that  
$(\chi_n - 1)^-\in H^{1,2}_0(\R^d , \mu )_{0,b}$, $\lim_{n\to\infty} 
\chi_n = 0$ $\mu$-a.e. and  
\begin{equation} 
\label{eq:2.1.20}
\cE^0_\alpha (\chi_n , \varv ) + \int_{\R^d}\langle \mathbf{B} , \nabla\chi_n\rangle\varv \, d\mu 
\ge 0 
\text{ for all }\varv\in H_0^{1,2} (\R^d , \mu )_{0,b}, \varv\ge 0.  
\end{equation} 
\end{proposition} 

\begin{proof} 
$(i)\Rightarrow (ii)$: Since $\int_{\R^d}\overline{L}u\, d\mu = 0$ for all $u\in 
D(L^0)_{0,b}$ we obtain that $\int_{\R^d}\overline{L}u\, d\mu = 0$ for all $u\in 
D(\overline{L})$ and thus   
$$ 
\int_{\R^d}\overline{T}_tu\, d\mu = \int_{\R^d} u\, d\mu + \int_0^t\int_{\R^d}\overline{L}\,
\overline{T}_su\, d\mu\, ds = \int_{\R^d} u\, d\mu 
$$ 
for all $u\in D(\overline{L})$. Since $D(\overline{L})\subset 
L^1 (\R^d , \mu )$ dense we obtain that $\mu$ is 
$(\overline{T}_t)_{t>0}$-invariant. 

\medskip
\noindent 
$(ii)\Rightarrow (iii)$: 
As a candidate for a sequence $\chi_n$, $n \geq 1$, of functions satisfying the conditions in 
(iii) consider  
$$
\chi_n := 1-\overline{G}_1^{V_n ,\prime}(1_{V_n}) \text{ for } V_n = B_n.
$$ 
Clearly, $\chi_n\in H^{1,2}_{loc}(\R^d ,\mu )$ and $(\chi_n - 1)^- \in H_0^{1,2} (\R^d ,\mu )_{0,b}$. 
Moreover, $(\chi_n)_{n\ge 1}$ is decreasing by Lemma \ref{lemma2.1.6} and therefore 
$\chi_\infty := \lim_{n\to\infty}\chi_n$ exists $\mu$-a.e. To see that $\chi_\infty = 0$ note that 
\eqref{eq:2.1.16} implies for $g\in L^1 (\R^d ,\mu )_b$ that 
$$ 
\begin{aligned}  
\int_{\R^d} g\chi_\infty\, d\mu & = \lim_{n\to\infty}\int_{\R^d} g\chi_n\, d\mu = 
\lim_{n\to\infty} \int_{\R^d} g\, d\mu - \int_{\R^d} g \overline{G}_1^{V_n ,\prime}
(1_{V_n})\, d\mu \\ 
& = \lim_{n\to\infty}\int_{\R^d} g\, d\mu - \int_{\R^d}\overline{G}_1^{V_n}g \, 
1_{V_n}\, d\mu \\ 
& = \int_{\R^d} g\, d\mu - \int_{\R^d} \overline{G}_1g \, d\mu = 0,  
\end{aligned} 
$$ 
since $\mu$ is $(\overline{T}_t)_{t>0}$-invariant, hence  
$$ 
\int_{\R^d}\overline{G}_1 g\, d\mu = \int_0^\infty \int_{\R^d} e^{-t}\overline{T}_t g\, d\mu dt 
= \int_{\R^d} g\, d\mu  . 
$$  
It remains to show that $\chi_n$ satisfies \eqref{eq:2.1.20}. To this end we have to consider 
the approximation $w_\beta := \beta \overline{G}'_{\beta +1}\overline{G}_1^{V_n ,\prime}(1_{V_n} )$, 
$\beta > 0$. Since $w_\beta\ge\beta
\overline{G}_{\beta + 1}^{V_n ,\prime}\overline{G}_1^{V_n ,\prime}(1_{V_n})$ 
and $\beta\overline{G}_{\beta + 1}^{V_n ,\prime}\overline{G}_1^{V_n ,\prime}
(1_{V_n}) = \overline{G}_1^{V_n ,\prime}(1_{V_n}) - \overline{G}_{\beta + 1}
^{V_n ,\prime}(1_{V_n})\ge \overline{G}_1^{V_n ,\prime}(1_{V_n}) 
- 1/(\beta + 1)$ by the resolvent equation, it follows that 
\begin{equation} 
\label{eq:2.1.23} 
w_\beta\ge\overline{G}_1^{V_n ,\prime}(1_{V_n}) -  1/(\beta + 1), 
\beta > 0.  
\end{equation}
Note that by Theorem \ref{theorem2.1.5}
$$
\begin{aligned} 
&\cE^0_1 (w_\beta , w_\beta ) \\
& \le\beta (\overline{G}_1^{V_n ,\prime}
(1_{V_n}) - w_\beta , w_\beta )_{L^2 (\R^d , \mu )} \\ 
& \le\beta (\overline{G}_1^{V_n ,\prime}(1_{V_n}) - w_\beta , \overline
{G}_1^{V_n ,\prime}(1_{V_n}))_{L^2 (\R^d , \mu )} \\ 
& = \cE^0_1 (w_\beta , \overline{G}_1^{V_n ,\prime}(1_{V_n})) + 
\int_{V_n}\langle\mathbf{B}, \nabla w_\beta\rangle \overline{G}_1^{V_n ,\prime} 
(1_{V_n})\, d\mu \\ 
& \le \cE^0_1 (w_\beta , w_\beta )^{\frac 12}(\cE^0_1 (\overline
{G}_1^{V_n ,\prime}(1_{V_n}),\overline{G}_1^{V_n ,\prime}(1_{V_n}))
^{\frac 12} + \sqrt{2\nu_{V_n}}\| \mathbf{B} 1_{V_n}\|_{L^2 (\R^d, \R^d, \mu )}).     
\end{aligned} 
$$ 
Consequently, $\lim_{\beta\to\infty} w_\beta = \overline{G}_1^{V_n ,\prime}
(1_{V_n})$ weakly in $D(\cE^0)$. Now \eqref{eq:2.1.23} implies for  
$u\in H^{1,2}_0 (\R^d , \mu )_{0,b}$, $u\ge 0$, 
$$ 
\begin{aligned} 
\cE^0_1 (\chi_n , u) & + \int_{\R^d}\langle\mathbf{B}, \nabla\chi_n\rangle u \, d\mu  
 = \lim_{\beta\to\infty}\Big (\int_{\R^d} u\, d\mu - \cE^0_1(w_\beta , u)  
 - \int_{\R^d}\langle\mathbf{B}, \nabla w_\beta\rangle u\, d\mu\Big )  \\ 
& = \lim_{\beta\to\infty} \Big (\int_{\R^d} u\, d\mu   
  - \beta\int_{\R^d} (\overline{G}_1^{V_n ,\prime}(1_{V_n}) - w_\beta ) u\, d\mu\Big )  \ge 0. 
\end{aligned} 
$$ 
$(iii)\Rightarrow (i)$: It is sufficient to show that if 
$h\in L^\infty (\R^d , \mu )$ is such that $\int_{\R^d} (\alpha - L)u\, h\, d\mu = 0$  
for all $u\in D(L^0)_{0,b}$ it follows that $h=0$. To this end let  
$\chi\in C_0^\infty (\R^d )$. If $u\in D(L^0)_b$ it is easy to see that 
$\chi u\in D(L^0)_{0,b}$ and $L^0 (\chi u) = \chi L^0 u + \langle A\nabla\chi ,  
\nabla u\rangle + u L^0\chi$. Hence   
\begin{equation} 
\label{eq:2.1.21} 
\begin{aligned} 
\int_{\R^d} (\alpha -L^0)u (\chi h)\, d\mu  
& = \int_{\R^d} (\alpha -L^0)(u\chi ) h\, d\mu 
+ \int_{\R^d}\langle A \nabla u, \nabla\chi\rangle h\, d\mu \\ 
& \qquad + \int_{\R^d} u L^0\chi\, h\,d\mu  \\ 
& = \int_{\R^d}\langle\mathbf{B} , \nabla (u\chi )\rangle h\,d\mu 
+ \int_{\R^d}\langle A \nabla u, \nabla \chi\rangle h\, d\mu 
+ \int_{\R^d} u L^0\chi\,  h\,d\mu.
\end{aligned} 
\end{equation} 
Since $\|\mathbf{B}\|\in L^2_{loc} (\R^d , \mu )$ we obtain that $u\mapsto \int_{\R^d}
(\alpha -L^0)u (\chi h)\, d\mu $, $u\in D(L^0)_b$, is continuous w.r.t. the 
norm on $D(\cE^0)$. Hence there exists some element $\varv\in D(\cE^0)$ such that 
$\cE_\alpha ^0 (u , \varv ) = \int_{\R^d} (\alpha -L^0)u (\chi h)\, d\mu $. Consequently, 
$\int_{\R^d} (\alpha -L^0) u (\varv - \chi h)\, d\mu = 0$ for all $u\in D(L^0)_b$, which 
now implies that $\varv = \chi h$. In particular, $\chi h\in D(\cE^0)$ and 
\eqref{eq:2.1.21} yields
\begin{equation} 
\label{eq:2.1.22} 
\begin{aligned}  
\cE^0_\alpha (u, \chi h) 
& = \int_{\R^d} \langle\mathbf{B} , \nabla (\chi u)\rangle h\, d\mu 
+ \int_{\R^d}\langle A \nabla u, \nabla \chi\rangle h\, d\mu  \\ 
& \qquad + \int_{\R^d} L^0\chi\, u h\,d\mu  
\end{aligned} 
\end{equation}
for all $u\in D(L^0)_b$ and subsequently for all $u\in D(\cE^0)$. From \eqref{eq:2.1.22} 
it follows that 
$$
\cE^0_\alpha (u,h) - \int_{\R^d}\langle\mathbf{B} ,\nabla u\rangle h\, d\mu = 0
\text{ for all } u\in H_0^{1,2} (\R^d , \mu )_0. 
$$ 
\noindent
Let $\varv_n := \|h\|_{L^\infty (\R^d , \mu )} \chi_n - h$. Then $\varv_n^-\in H_0^{1,2}
(\R^d , \mu )_{0,b}$ and 
$$
0\le\cE^0_\alpha (\varv_n , \varv_n^- ) - \int_{\R^d}\langle\mathbf{B} , \nabla \varv_n^-\rangle \varv_n\, 
d\mu \le -\alpha \int_{\R^d} (\varv_n^-)^2\, d\mu,   
$$   
since $\int_{\R^d}\langle\mathbf{B} , \nabla\varv_n^-\rangle\varv_n\, d\mu  
= \int_{\R^d}\langle\mathbf{B} , \nabla\varv_n^-\rangle\varv_n^-\, d\mu = 0$ and 
$\mathcal{E}^0(v_n^+,v_n^-)\leq0$.
Thus $\varv_n^- = 0$, i.e., $h\le\|h\|_{L^\infty (\R^d , \mu )}\chi_n$. Similarly,  
$-h\le\|h\|_{L^\infty (\R^d , \mu )}\chi_n$, 
hence $|h|\le\|h\|_\infty\chi_n$. Since $\lim_{n\to\infty}\chi_n = 0$ $\mu$-a.e., 
it follows that $h=0$. 
\end{proof}

\begin{remark} 
\label{rem:2.1.10}
{\it 
The proof of $(ii)\Rightarrow (iii)$ in Proposition \ref{prop2.1.9} shows that 
if $\mu$ is $(\overline{T}_t)_{t>0}$-invariant then there exists for {\it all} 
$\alpha > 0$ a sequence $(\chi_n)_{n\ge 1}\subset H^{1,2}_{loc}(\R^d , 
\mu )$ such that $(\chi_n - 1)^-\in H^{1,2}_0(\R^d , \mu )_{0,b}$, 
$\lim_{n\to\infty}\chi_n = 0$ $\mu$-a.e. and 
$$
\cE^0_\alpha (\chi_n , \varv ) + \int_{\R^d}\langle\mathbf{B} , \nabla\chi_n\rangle  
\varv\, d\mu \ge 0 
\text{ for all }\varv\in H_0^{1,2} (\R^d , \mu )_{0,b}, \varv\ge 0. 
$$
Indeed, it suffices to take $\chi_n := 1- \alpha\overline{G}_\alpha
^{V_n ,\prime}(1_{V_n})$, $n\ge 1$. }
\end{remark} 
\noindent 
Let us state sufficient conditions on $\mu$, $A$ and $\mathbf{G}$ that imply 
$(\overline{T}_t)_{t>0}$-invariance of $\mu$ and discuss its interrelation with the notion 
of conservativeness, that we define right below. 
\begin{definition} 
\label{def:3.2.1} 
$(T_t)_{t>0}$ as defined in Definition \ref{definition2.1.7} is called {\bf conservative}\index{semigroup ! conservative} if 
$$
T_t 1_{\R^d} = 1, \; \; \text{$\mu$-a.e. \quad for one (and hence all) $t>0$.}
$$
\end{definition}
\noindent
We can then state 
the following relations:

\begin{remark} 
\label{rem:2.1.10a}
{\it  Let \eqref{condition on mu}--\eqref{eq:2.1.4} be satisfied.\\
(i) The measure $\mu$ is $(\overline{T}_t)_{t> 0}$-invariant\index{measure ! invariant} (cf. Definition \ref{def:2.1.1}(ii)), if and only if the dual semigroup 
$(T'_t)_{t> 0}$ of $(\overline{T}_t)_{t>0}$, acting on $L^\infty (\R^d ,\mu )$, 
is conservative\index{semigroup ! conservative}. Indeed, if $\mu$ is $(\overline{T}_t)_{t>0}$-invariant, then 
for any $f \in C_0^{\infty}(\R^d)$, $t>0$, 
$$
\int_{\R^d} f d\mu =
\int_{\R^d} \overline{T}_t f d\mu  
=\lim_{n \rightarrow \infty}\int_{\R^d} f T'_t 1_{B_n} d\mu 
= \int_{\R^d} f T'_t 1_{\R^d} d\mu,
$$
hence $T'_t 1_{\R^d}=1$, $\mu$-a.e. The converse follows similarly. Likewise, $\mu$ is $(\overline{T}'_t)_{t> 0}$-invariant, if and only if $(T_t)_{t> 0}$ is conservative.
Since in the {\it symmetric} case \;(i.e., $\mathbf{G} = \beta^{\rho, A}$) 
$\overline{T}'_{t}|_{L^1 (\R^d , \mu )_b}$ coincides with $\overline{T}_{t}|_{L^1 (\R^d , \mu )_b}$ we obtain that both notions 
coincide in this particular case. Conservativeness in the symmetric case 
has been well-studied by many authors. We refer to \cite{Da85}, \cite{FOT}, \cite[Section 1.6]{Sturm94} and references therein. \\
(ii) Suppose that $\mu$ is {\it finite}. Then $\mu$ is $(\overline{T}_t)_{t> 0}$-invariant, 
if and only if $\mu$ is $(\overline{T}'_t)_{t> 0}$-invariant. Indeed, let $\mu$ be 
$(\overline{T}_t)_{t> 0}$-invariant. 
Then, since $1_{\R^d}\in L^1(\R^d,\mu)$, we obtain by the $(\overline{T}_t)_{t> 0}$-invariance
$\int_{\R^d} |1-T_t 1_{\R^d}|\, d\mu = \int_{\R^d} (1-\overline{T}_t 1_{\R^d})\, d\mu= 0$  for all $t>0$, i.e., 
$T_t 1_{\R^d} = 1$ $\mu$-a.e. for all $t>0$, 
which implies that $\mu$ is $(\overline{T}'_t)_{t> 0}$-invariant by (i). 
The converse is shown similarly. }
\end{remark} 
For $u\in C^2(\R^d)$, we define  (cf. \eqref{eq:2.1.3bis}, \eqref{eq:2.1.5}, \eqref{eq:2.1.5a}, and \eqref{eq:2.1.3})
\begin{eqnarray} \label{eq:2.1.3bis2'}
L^{\prime}u:= L^A u + \langle \beta^{\rho, A} - \mathbf{B}, 
\nabla u\rangle = L^A u + \langle 2\beta^{\rho, A} - \mathbf{G},\nabla u\rangle.
\end{eqnarray} 

\begin{remark}\label{cnulltwocoincideprime}
{\it Similarly to Remark \ref{cnulltwocoincide}, we have 
\begin{equation} 
\label{eq:2.1.3forprime}
L^A  + \langle \beta^{\rho, A} - \mathbf{B}, 
\nabla \rangle  = L^0 -  \langle \mathbf{B}, \nabla \rangle \quad \text{ on }\  C^{2}_0(\R^d).
\end{equation} 
Therefore the definitions \eqref{defL'first} and \eqref {eq:2.1.3bis2'} for $L'$ coincide on $D(L^0)_{0,b}\cap C^{2}(\R^d)=C^{2}_0(\R^d)$ and are therefore consistent.}
\end{remark}

\begin{proposition}
\label{prop:2.1.10}
Let \eqref{condition on mu}--\eqref{eq:2.1.4} be satisfied.
Each of the following conditions (i), (ii) and (iii) imply that $\mu$ is $(\overline{T}_t)_{t>0}$-invariant \index{measure ! invariant ! sufficient condition} (cf. Definition \ref{def:2.1.1}(ii)), or equivalently, by Remark \ref{rem:2.1.10a}(i), that $(T'_t)_{t>0}$ is conservative\index{semigroup ! conservative ! sufficient condition}:
\item{(i)} $a_{ij}, g_i - \beta^{\rho, A}_i \in L^1 (\R^d , \mu)$, $1\le i,j\le d$.  
\item{(ii)} There exist $u\in C^2 (\R^d )$ and $\alpha > 0$ such that 
$\lim_{\|x\|\to\infty} u(x) = \infty$ and $L'u = L^A u + \langle \beta^{\rho, A} - \mathbf{B}, 
\nabla u\rangle\le\alpha u$ $\mu$-a.e. \index{Lyapunov condition} 
\item{(iii)} There exists $M\ge 0$, such that $-\langle A(x)x,x\rangle/ (\|x\|^2 + 1) + \frac 12 {\rm trace}(A(x)) 
+ \langle (\beta^{\rho, A} - \mathbf{B})(x), x\rangle\le M(\|x\|^2 +1)\big ( \ln (\|x\|^2 + 1) + 1\big )$ 
for $\mu$-a.e. $x\in \R^d$.
\end{proposition}

\begin{proof}  
(i) By Proposition \ref{prop2.1.9} it is sufficient to show that $(L, D(L^0)_{0,b})$ 
is $L^1$-unique. But if $h\in L^\infty (\R^d , \mu )$ is such that 
$\int_{\R^d} (1-L)u\, h\, d\mu = 0$ for all $u\in D(L^0)_{0,b}$ we have seen in 
the proof of the implication $(i)\Rightarrow (ii)$ in Proposition \ref{prop2.1.9} 
that $h\in H^{1,2}_{loc}(\R^d , \mu )$ and 
\begin{equation} 
\label{eq:2.1.24} 
\cE^0_1 (u, h ) - \int_{\R^d}\langle \mathbf{B} , \nabla u\rangle h\,d\mu = 0
\text{ for all }  u\in H^{1,2}_0(\R^d , \mu )_0. 
\end{equation} 
Let $\psi_n\in C_0^\infty (\R^d )$ be such that $1_{B_n}\le
\psi_n\le 1_{B_{2n}}$ and $\|\nabla\psi_n\|_{L^\infty (\R^d , \R^d, \mu )}\le c/n$ for some $c>0$. Then  
\eqref{eq:2.1.24} implies that 
$$ 
\begin{aligned}  
\int_{\R^d} & \psi_n^2 h^2 \, d\mu  + \cE^0 (\psi_n h , \psi_n h ) = 
\cE^0_1 (\psi_n^2 h , h )  + \frac12 \int_{\R^d}\langle A\nabla\psi_n ,\nabla\psi_n
\rangle h^2\, d\mu \\ 
& \qquad  - \int_{\R^d}\langle\mathbf{B}, \nabla (\psi_n^2 h)\rangle h\, d\mu  + 
    \int_{\R^d}\langle\mathbf{B} , \nabla\psi_n\rangle \psi_n h^2\, d\mu \\ 
& \le \frac{c^2}{2n^2}\|h\|_{L^\infty (\R^d , \mu )}^2 \Big (\sum_{i,j=1}^d\int_{\R^d} |a_{ij}| \, d\mu \Big ) 
 +\frac{c}{n}  \|h\|_{L^\infty (\R^d , \mu )}^2 \Big (\sum_{i=1}^d\int_{\R^d} |g_i - \beta^{\rho, A}_i|\, d\mu\Big  ) 
\end{aligned} 
$$ 
and thus $\int_{\R^d} h^2\, d\mu = \lim_{n\to\infty}\int_{\R^d}\psi_n^2 h^2\, d\mu = 0$. 

\noindent 
(ii) Let $\chi_n := \frac {u}{n}$. Then $\chi_n\in H^{1,2}_{loc}(\R^d , 
\mu )$, $(\chi_n - 1)^-$ is bounded and has compact support, $\lim_{n\to
\infty}\chi_n = 0$ and 
$$ 
\begin{aligned} 
\cE^0_\alpha (\chi_n , \varv ) & + \int_{\R^d}\langle\mathbf{B}, \nabla\chi_n\rangle \varv\, d\mu \\ 
& = \frac 1n \int_{\R^d} (\alpha u - L^A u - \langle \beta^{\rho, A} -\mathbf{B}, \nabla u\rangle ) 
\varv \, d\mu \ge 0
\end{aligned} 
$$
for all $\varv\in H_0^{1,2} (\R^d , \mu )_0$, $\varv\ge 0$.  
By Proposition \ref{prop2.1.9} $\mu$ is $(\overline{T}_t)_{t>0}$-invariant.   

\noindent 
Finally, (iii) implies (ii) since we can take $u(x) = \ln (\|x\|^2 + 1) +  r$ for 
$r$ sufficiently large. 
\end{proof} 
\noindent 
As a direct consequence of Proposition \ref{prop:2.1.10} and Remark \ref{rem:2.1.10a}(i), we obtain the following result.
\begin{corollary} 
\label{cor:2.1.4.1}
Let \eqref{condition on mu}--\eqref{eq:2.1.4} be satisfied.
Each of the following conditions (i), (ii) and (iii) imply that $(T_t)_{t>0}$ is conservative\index{semigroup ! conservative ! sufficient condition}, or equivalently, by Remark \ref{rem:2.1.10a}(i), that $\mu$ is 
$(\overline{T}'_t)_{t>0}$-invariant\index{semigroup ! invariant ! sufficient condition}: 
\item{(i)} $a_{ij}, g_i - \beta^{\rho, A}_i \in L^1 (\R^d , \mu)$, $1\le i,j\le d$.  
\item{(ii)} There exist $u\in C^2 (\R^d )$ and $\alpha > 0$ such that 
$\lim_{\|x\|\to\infty} u(x) = \infty$ and $Lu \le\alpha u$ $\mu$-a.e., where $Lu=L^Au+\langle \beta^{\rho, A}+\mathbf{B}, \nabla  u \rangle = L^A u +\langle \mathbf{G},\nabla u \rangle$ (see \eqref{eq:2.1.3bis2} and \eqref{eq:2.1.5a}).  \index{Lyapunov condition} 
\item{(iii)} There exists $M\ge 0$, such that $-\langle A(x)x,x\rangle/ (\|x\|^2 + 1) 
+ \frac 12 {\rm trace}(A(x)) 
+ \langle \mathbf{G}(x), x\rangle\le M(\|x\|^2 +1)\big (\ln (\|x\|^2 + 1) + 1\big )$ 
for $\mu$-a.e. $x\in \R^d$.
\end{corollary}
\begin{remark}  
\label{remark2.1.11}
{\it
\noindent 
(i) Suppose that $\mu$ is {\it finite}, so that according to Remark \ref{rem:2.1.10a}(ii) 
$\mu$ is $(\overline{T}_t)_{t> 0}$-invariant if and only if $\mu$ is 
$(\overline{T}'_t)_{t> 0}$-invariant. In this case we replace $g_i - \beta^{\rho, A}_i$ 
(resp. $\beta^{\rho ,A} - \mathbf{B}$) in Proposition \ref{prop:2.1.10}(i) (resp.  
\ref{prop:2.1.10}(ii) and (iii)) by $g_i-\beta^{\rho,A}_i$ (resp. $\mathbf{G}$) and still obtain 
that $\mu$ is $(\overline{T}_t)_{t>0}$-invariant. \\
(ii) The criteria stated in part (iii) of Proposition \ref{prop:2.1.10} resp. Corollary \ref{cor:2.1.4.1} 
involve the logarithmic derivative $\beta^{\rho, A}$ of the density. This assumption can be replaced by 
volume growth conditions of $\mu$ on annuli (see Proposition \ref{prop:3.2.9} below).}
\end{remark}

\smallskip 

\begin{proposition} \label{prop:2.1.4.1.4}
Let \eqref{condition on mu}--\eqref{eq:2.1.4} be satisfied.
Suppose that there exist a bounded, nonnegative and nonzero function $u\in C^2 (\R^d )$  
and $\alpha > 0$, such that $L'u=  L^A u + \langle\beta^{\rho ,A} - \mathbf{B} , \nabla u\rangle \ge\alpha u$. 
Then $\mu$ is not $(\overline{T}_t)_{t > 0}$-invariant, or equivalently, by Remark \ref{rem:2.1.10a}(i), $(T'_t)_{t>0}$ is not conservative. In particular, if there exist  a bounded, nonnegative and nonzero function $u\in C^2 (\R^d )$  
and $\alpha > 0$ such that $Lu \geq \alpha u$, where $Lu = L^A u +\langle \beta^{\rho, A}+ \mathbf{B} ,\nabla u\rangle =
L^A u+\langle \mathbf{G}, \nabla u \rangle$ (see \eqref{eq:2.1.3bis2} and \eqref{eq:2.1.5a}), then $\mu$ is not $(\overline{T}'_t)_{t>0}$-invariant, or equivalently, by Remark \ref{rem:2.1.10a}(i), $(T_t)_{t>0}$ is not conservative. 
\end{proposition}
\begin{proof} We may suppose that $u\le 1$. If $\mu$ was 
$(\overline{T}_t)_{t>0}$-invariant, it would follow that there exist $\chi_n\in 
H^{1,2}_{loc}(\R^d , \mu )$, $n\ge 1$, such that $(\chi_n - 1)^-\in 
H^{1,2}_0(\R^d , \mu )_{0,b}$, $\lim_{n\to\infty}\chi_n = 0$ $\mu$-a.e. 
and $\cE^0_\alpha (\chi_n , \varv ) + \int_{\R^d}\langle\mathbf{B} , \nabla\chi_n\rangle 
\varv\, d\mu\ge 0$ for all $\varv\in H_0^{1,2} (\R^d , \mu )_{0,b}$, $\varv\ge 0$ 
(cf. the Remark \ref{rem:2.1.10}). Let $\varv_n := (\chi_n - u)$. Then $\varv_n^-\in 
H_0^{1,2} (\R^d , \mu )_{0,b}$ and 
$$
0\le\cE^0_\alpha (\varv_n , \varv_n^- ) - \int_{\R^d}\langle\mathbf{B} , \nabla \varv_n^-
\rangle \varv_n\, d\mu \le - \alpha \int_{\R^d} (\varv_n^-)^2\, d\mu,   
$$ 
since $\int_{\R^d}\langle\mathbf{B} , \nabla \varv_n^-\rangle \varv_n\, d\mu  
= \int_{\R^d}\langle\mathbf{B} , \nabla \varv_n^-\rangle \varv_n^-\, d\mu = 0$ and 
$\mathcal{E}^0(\varv_n^+, \varv_n^-) \leq 0$. 
Thus $\varv_n^- = 0$, i.e., $u\le\chi_n$. Since $\lim_{n\to\infty}\chi_n = 0$  
$\mu$-a.e. and $u\ge 0$, it follows that $u=0$, which is a contradiction to our 
assumption $u\neq 0$. 
The rest of the assertion follows by replacing $(\overline{T}_t)_{t>0}$ with $(\overline{T}'_t)_{t>0}$.
\end{proof}

\begin{remark} 
\label{rem:2.1.12} 
{\it 
Let us provide two examples illustrating the scope of our results. \\
(i) In the first example the measure $\mu$ is not 
$(\overline{T}_t)_{t > 0}$-invariant\index{measure ! invariant ! counterexample}.  
To this end let $\mu := e^{-x^2}\, dx$,
$\mathbf{G} (x) = -x -2e^{x^2}$,  $x \in \R$,
$$
Lu := \frac12 u'' + \mathbf{G}\cdot u', \;\; u\in C_0^\infty (\R),
$$ 
$(\overline{L}, D(\overline{L}))$ be the maximal extension having properties (i)--(iii) in Theorem \ref{theorem2.1.5} and 
$(\overline{T}_t)_{t> 0}$ be the associated semigroup. Let $h(x) := \int_{-\infty}^x e^{-t^2}\, dt $, 
$x\in\R$. Then  
$$ 
\begin{aligned} 
\frac12 h''(x) + (\beta^{\rho ,A} - \mathbf{B})(x) h'(x) 
& = \frac12 h'' (x) + (-x+ 2e^{x^2}) h'(x)  \\
& = - 2xe^{-x^2} + 2  \ge \frac{1}{\sqrt{\pi}} h(x) 
\end{aligned} 
$$
for all $x\in\R$. 
It follows from Proposition \ref{prop:2.1.4.1.4} that $\mu$ is not $(\overline{T}_t)_{t > 0}$-invariant. 
Since $\mu$ is finite, $\mu$ is also not $(\overline{T}'_t)_{t> 0}$-invariant according to Remark 
\ref{rem:2.1.10a}(ii) and thus both semigroups $(T_t)_{t > 0}$ and $(T'_t)_{t> 0}$ are not conservative according to Remark 
\ref{rem:2.1.10a}(i).   \\
(ii) in the second example $(T_t)_{t > 0}$ is conservative, but 
the dual semigroup $(T'_t)_{t > 0}$ of $(\overline{T}_t)_{t > 0}$ is not conservative\index{semigroup ! conservative ! counterexample}. Necessarily, 
the (infinitesimally 
invariant) measure $\mu$ must be infinite in this case. To this end let $\mu := e^x \, dx$, 
$\mathbf{G} (x) = \frac12 + \frac12 e^{-x}$,  \,$x \in \R$, $Lu := \frac12 u'' + \mathbf{G}\cdot u'$, 
$u\in C_0^\infty (\R)$,
$(\overline{L}, D(\overline{L}))$ be the maximal extension having properties (i)--(iii) in Theorem 
\ref{theorem2.1.5} and $(\overline{T}_t)_{t> 0}$ be the associated semigroup. Let $h(x) = 1 + x^2$,
$x\in\R$. Then 
$$ 
\frac12 h''(x) + (\beta^{\rho ,A} +\mathbf{B}) h'(x) = 1 + x +  e^{-x}x \le 2 (1 + x^2) = 2 h(x). 
$$
It follows from Proposition \ref{prop:2.1.10} that $\mu$ is $(\overline{T}'_t)_{t > 0}$-invariant, hence $(T_t)_{t > 0}$ is conservative according to Remark 
\ref{rem:2.1.10a}(i). To see that $\mu$ is not $(\overline{T}_t)_{t > 0}$-invariant, let 
$h(x) =  \Psi (e^{-x})$, $x\in\R$, for some bounded, nonzero and nonnegative function 
$\Psi\in C^2 ((0, \infty))$. Then 
\begin{equation} 
\label{eq:2.1.24a}
\frac12 h''(x) + (\beta^{\rho ,A} - \mathbf{B}) h'(x) = \frac12 h''(x) +(\frac12 -\frac12 e^{-x})h'(x) 
\ge \alpha h(x), \;\; \text{ for all $x \in \R$} 
\end{equation}
for some $\alpha > 0$ is equivalent with 
\begin{equation} 
\label{eq:2.1.24b}
\left( \Psi '' (y) + \Psi ' (y)\right)y^2 \ge 2\alpha \Psi (y)   
\end{equation}
for all $y>0$. An example of such a function $\Psi$ is given by 
$$ 
\Psi (y) = \begin{cases} 
y^2(6-y)  & \text{ if } 0<y  \le 3 \\ 
54 - \frac{81}{y}  & \text{ if } 3\le y . 
\end{cases} 
$$
Indeed, it follows that $\Psi(y)>0$ for all $y \in (0, \infty)$ with $\Psi \in C^2_b((0, \infty))$ and
$$ 
(\Psi'' (y)+\Psi'(y))y^2 = \begin{cases} 
y^2(-3y^2+6y+12) \ge 3y^2 \ge \frac12 \Psi(y)  & \text{ if } 0<y \leq 3\\ 
81-\frac{162}{y} \geq 27 \geq \frac12 \Psi(y)  & \text{ if } 3\le y . 
\end{cases} 
$$
Thus, $\Psi$ satisfies \eqref{eq:2.1.24b} with $\alpha = \frac14$, hence $h(x) = \Psi (e^{-x})$ satisfies 
\eqref{eq:2.1.24a} for the same $\alpha$.  
It follows from Proposition \ref{prop:2.1.4.1.4} that $\mu$ is not $(\overline{T}_t)_{t > 0}$-invariant, 
hence the dual semigroup $(T'_t)_{t > 0}$ is not conservative by Remark 
\ref{rem:2.1.10a}(i).\\
The intuition behind the example is as follows: the density of the measure $\mu$ is monotone 
increasing, its derivative as well. The drift of $L$ is bounded from above on $\mathbb{R}^+$, so the 
associated diffusion process will not explode to $+\infty$ in finite time. On $\mathbb{R}^-$ the drift  
becomes unbounded positive, but that excludes that the associated diffusion process can explode to   
$-\infty$ in finite time. This is exactly the opposite for the dual process: since the drift of the dual 
process becomes unbounded from below with exponential growth, the solution will explode in finite time 
to $-\infty$. }
\end{remark}

\paragraph{Uniqueness of $(L, C_0^\infty (\R^d ))$} 

We will now discuss the problem of uniqueness of the maximal extension 
$(L, C_0^\infty (\R^d ))$ on $L^1 (\R^d , \mu )$. To this end we make the following 
additional assumption on $A$: Suppose that for any compact $V$ there exist constants $M_V\ge 0$ 
and $\alpha_V\in (0, 1)$ such that 
\begin{equation} 
\label{AssumptionUniqueness}
|a_{ij} (x) - a_{ij}(y)|\le M_V \|x-y\|^{\alpha_V}\text{ for all }x,y\in V. 
\end{equation} 
The following regularity result is then crucial for our further investigations: 

\begin{theorem} 
\label{theorem2.1.3} 
Let \eqref{condition on mu}--\eqref{eq:2.1.4}
and \eqref{AssumptionUniqueness} be satisfied and $L$ be as in Theorem \ref{theorem2.1.5} (in particular $L$ can be expressed as in \eqref{defLfirst} and \eqref{eq:2.1.3bis2} on $C_0^{2}(\R^d)$). 
Let $h\in L^\infty (\R^d , \mu )$ be such that $\int_{\R^d} (1-L)u\, h\,d\mu = 0$ for all 
$u\in C_0^\infty (\R^d )$. Then $h\in H_{loc}^{1,2}(\R^d , \mu )$ and 
$\cE^0_1 (u,h)-\int_{\R^d}\langle\mathbf{B} ,\nabla u\rangle h\, d\mu = 0$ for all 
$u\in H_0^{1,2}(\R^d , \mu )_0$.  
\end{theorem} 

\begin{proof} 
First note that $C_0^2 (\R^d )\subset D(L^0)_{0,b}
\subset D(\overline{L})_{0,b}$ and that $\int _{\R^d}(1-L)u\, h\, d\mu = 0$ for all 
$u\in C_0^2 (\R^d )$. Let $\chi\in C_0^\infty (\R^d )$ and $r > 0$ 
be such that $\text{supp} (\chi )\subset B_r (0)$. We have to show that 
$\chi h\in H^{1,2}_0 (\R^d,\mu )$. Let $K \ge 0$ and $\overline{\alpha} \in (0,1)$ be constants, 
such that $|a_{ij} (x) - a_{ij}(y)|\le K\|x-y\|^{\overline{\alpha}}$ for all 
$x,y\in B_r (0)$ and define 
$$ 
\overline{a}_{ij}(x) := a_{ij}((\frac r{|x|}\wedge 1)x), x\in\R^d.   
$$ 
Then $\overline{a}_{ij}(x) := a_{ij}(x)$ for all $x\in B_r (0)$ and  
$|\overline{a}_{ij}(x)-\overline{a}_{ij}(y)|\le 2 K\|x-y\|^{\overline{\alpha}}$ for all 
$x,y\in\R^d$. Let $L^{\overline{A}} = \sum_{i,j=1}^d\overline{a}_{ij}
\partial_{ij}$. By \cite[Theorems 4.3.1 and 4.3.2]{Kr96}, there exists for 
all $f\in C_0^\infty (\R^d )$ and $\alpha > 0$ a unique function 
$\overline{R}_\alpha f\in C_b^2 (\R^d )$ satisfying $\alpha\overline
{R}_\alpha f - L^{\overline{A}}\overline{R}_\alpha f = f$ and $\|\alpha
\overline{R}_\alpha f\|_{C_b (\R^d)}\le \|f\|_{C_b (\R^d)}$. Moreover, $\alpha\overline 
{R}_\alpha f\ge 0$ if $f\ge 0$ by \cite[Theorem 2.9.2]{Kr96}.

\noindent 
Since $C_0^\infty (\R^d )\subset C_\infty (\R^d )$ dense, we obtain that $f\mapsto \alpha\overline{R}_\alpha f$, 
$f\in C_0^\infty (\R^d)$, can be uniquely extended to a positive linear 
map $\alpha\overline{R}_\alpha : C_\infty (\R^d)\to C_b (\R^d )$ 
such that $\|\alpha\overline{R}_\alpha f\|_{C_b (\R^d)} \le \|f\|_{C_b (\R^d)}$ for all 
$f\in C_\infty (\R^d )$. By the Riesz representation theorem there exists 
a unique positive measure $V_\alpha (x, \cdot )$ on $(\R^d , \cB 
(\R^d ))$ such that $V_\alpha f(x) := \int_{ \R^d     } f(y)\, V_\alpha 
(x, dy) = \overline{R}_\alpha f(x)$ for all $f\in C_\infty (\R^d )$, 
$x\in\R^d $. Clearly, $\alpha V_\alpha (\cdot , \cdot )$ is a kernel on 
$(\R^d , \cB (\R^d ))$ (cf. \cite[Chapter IX, Theorem 9]{DeM88}). Since $\alpha 
V_\alpha f = \alpha\overline{R}_\alpha f\le 1$ for all $f\in C_\infty 
(\R^d )$ such that $f\le 1$ we conclude that the linear operator 
$f\mapsto \alpha V_\alpha f$, $f\in \cB_b (\R^d )$, is sub-Markovian. 
 
\noindent 
Let $f_n\in C_0^\infty (\R^d )$, $n\ge 1$, such that $\|f_n\|_{C_b (\R^d )} \le 
\|h\|_{L^\infty (\R^d , \mu )}$ and $\tilde h := \lim_{n\to\infty} f_n$ is a $\mu$-version of 
$h$. Then $\lim_{n\to\infty}\alpha V_\alpha f_n (x) = \alpha V_\alpha 
\tilde h (x)$ for all $x\in\R^d$ by Lebesgue's theorem and $\|\alpha 
V_\alpha\tilde h\|_{C_b (\R^d )} \le \|h\|_{L^\infty (\R^d , \mu )} $. 
Then 
\begin{eqnarray}
\label{eq:2.1.2.2} 
&&\cE^0 (\chi\alpha V_\alpha f_n , \chi\alpha V_\alpha f_n )  = -\int_{\R^d} L^0 
(\chi\alpha V_\alpha f_n) \chi\alpha V_\alpha f_n\, d\mu \nonumber \\ 
& =&  -\int_{\R^d} \chi L^A \chi\, (\alpha V_\alpha f_n)^2 \, d\mu  - \int_{\R^d}\langle A
\nabla\chi , \nabla\alpha V_\alpha f_n\rangle\chi\alpha V_\alpha f_n\, d\mu\nonumber\\ 
& &\qquad -\int_{\R^d} L^{\overline{A}}(\alpha V_\alpha f_n )\, \chi^2\, \alpha 
V_\alpha f_n\, d\mu -\int_{\R^d}\langle \beta^{\rho, A},\nabla (\chi\alpha V_\alpha f_n)\rangle
\chi\alpha V_\alpha f_n\, d\mu \nonumber\\ 
& =& -\int_{\R^d}\chi L^A\chi\, (\alpha V_\alpha f_n)^2\, d\mu - \int_{\R^d}\langle A\nabla 
\chi ,\nabla (\chi\alpha V_\alpha f_n)\rangle\alpha V_\alpha f_n\, d\mu \nonumber\\ 
&& \qquad + \int_{\R^d} \langle A\nabla \chi , \nabla\chi\rangle (\alpha V_\alpha 
f_n)^2\,d\mu - \alpha\int_{\R^d} (\alpha V_\alpha f_n - f_n)\chi^2\,\alpha V_\alpha 
f_n\, d\mu \nonumber\\ 
& &\qquad - \int_{\R^d}\langle \beta^{\rho, A}, \nabla (\chi\alpha V_\alpha f_n)\rangle\chi\, 
\alpha V_\alpha f_n\, d\mu. \\ \nonumber
\end{eqnarray}
\noindent
Hence $\cE^0 (\chi\alpha V_\alpha f_n , \chi\alpha V_\alpha f_n )\le c\, 
\cE^0 (\chi\alpha V_\alpha f_n , \chi\alpha V_\alpha f_n )^{1/2} + M$ for 
some positive constants $c$ and $M$ independent of $n$. Consequently, 
$\sup_{n\ge 1}\cE^0 (\chi\alpha V_\alpha f_n , \chi\alpha V_\alpha f_n ) 
< + \infty$, hence $\chi\alpha V_\alpha\tilde h\in D(\cE^0)$ and 
$\lim_{n\to\infty}\chi\alpha V_\alpha f_n = \chi\alpha V_\alpha\tilde h$ 
weakly in $D(\cE^0)$. 

\noindent 
Note that 
\begin{eqnarray}
\label{eq:2.1.2.3}
&&-\ \alpha\int_{\R^d} (\alpha V_\alpha \tilde h  - \tilde h)\alpha V_\alpha\tilde h 
\chi^2\, d\mu\le - \alpha\int_{\R^d} (\alpha V_\alpha\tilde h - \tilde h )\tilde h 
\chi^2\, d\mu \nonumber\\ 
& =& \lim_{n\to\infty} - \alpha\int_{\R^d} (\alpha V_\alpha f_n - f_n )
\tilde h\chi^2\, d\mu \nonumber\\ 
& =& \lim_{n\to\infty} - \int_{\R^d} L^{\overline{A}} (\alpha V_\alpha f_n) \tilde h
\chi^2\,d\mu \nonumber\\ 
& = &\lim_{n\to\infty} \Big (- \int_{\R^d} L^A (\chi^2\alpha V_\alpha f_n )
\tilde h\,d\mu + 2\int_{\R^d}\langle A\nabla\chi, \nabla\alpha V_\alpha 
f_n\rangle\chi\tilde h \,d\mu \nonumber\\
& &\qquad + \int_{\R^d} L^A(\chi^2)\alpha V_\alpha f_n\,\tilde h\, d\mu\Big ) \nonumber\\
& =& \lim_{n\to\infty}  \Big (
- \int_{\R^d}\chi^2\alpha V_\alpha f_n\,\tilde h\, 
d\mu +  \int_{\R^d}\langle \mathbf{G} ,\nabla (\chi^2\alpha V_\alpha f_n)
\rangle \tilde h\, d\mu \nonumber\\ 
& &\qquad + \ 2 \int_{\R^d}\langle A\nabla\chi ,\nabla\alpha V_\alpha f_n
\rangle\chi\tilde h\, d\mu + \int_{\R^d} L^A (\chi^2)\alpha V_\alpha 
f_n\,\tilde h \, d\mu\Big ) \nonumber\\ 
& =& - \int_{\R^d}\chi^2(\alpha V_\alpha\tilde h)\,\tilde h\, d\mu + \int_{\R^d}
\langle \mathbf{G} ,\nabla (\chi\alpha V_\alpha \tilde h )\rangle
\chi\tilde h\, d\mu \nonumber\\ 
&& \qquad + \int_{\R^d}\langle \mathbf{G} , \nabla \chi\rangle\chi(\alpha V_\alpha\tilde h)
\, \tilde h\, d\mu + 2\int_{\R^d}\langle A\nabla\chi ,\nabla (\chi\alpha V_\alpha 
\tilde h)\rangle\tilde h\, d\mu \nonumber\\ 
&& \qquad -\  2 \int_{\R^d} \langle A\nabla\chi , \nabla\chi\rangle (\alpha V_\alpha
\tilde h)\,  \tilde h\, d\mu + \int_{\R^d} L^A(\chi^2)(\alpha V_\alpha\tilde h)\, 
\tilde h\, d\mu \nonumber\\
&\le& c\, \cE^0 (\chi\alpha V_\alpha\tilde h , \chi\alpha V_\alpha\tilde h )^{1/2} + M \\ \nonumber
\end{eqnarray}   
for some positive constants $c$ and $M$ independent of $\alpha$. Combining 
\eqref{eq:2.1.2.2} and \eqref{eq:2.1.2.3} we obtain that 
$$
\begin{aligned} 
\cE^0 (\chi\alpha V_\alpha\tilde h & , \chi\alpha V_\alpha\tilde h)\le
\liminf_{n\to\infty} \cE^0 (\chi\alpha V_\alpha\tilde f_n, \chi\alpha 
V_\alpha\tilde f_n) \\ 
& \le - \int_{\R^d}\chi L^A\chi (\alpha V_\alpha \tilde h)^2\, d\mu - \int_{\R^d}\langle A
\nabla\chi , \nabla (\chi\alpha V_\alpha \tilde h)\rangle\alpha V_\alpha 
\tilde h\, d\mu 
\\ 
& \qquad + \int_{\R^d}\langle A\nabla\chi , \nabla\chi\rangle (\alpha V_\alpha 
\tilde h)^2\, d\mu - \alpha\int_{\R^d} (\alpha V_\alpha\tilde h - \tilde h)\chi^2\, 
\alpha V_\alpha \tilde h\, d\mu \\ 
& \qquad - \int_{\R^d}\langle\beta^{\rho, A} , \nabla (\chi\alpha V_\alpha \tilde h)\rangle\chi
\alpha V_\alpha \tilde h\, d\mu \\ 
& \le\tilde c\, \cE^0 (\chi\alpha V_\alpha\tilde h , \chi\alpha V_\alpha
\tilde h )^{1/2} + \tilde M     
\end{aligned}
$$
for some positive constants $\tilde c$ and $\tilde M$ independent of $\alpha$. 
Hence $(\chi\alpha V_\alpha\tilde h )_{\alpha > 0}$ is bounded in 
$D(\cE^0)$.\\
If $u\in D(\cE^0)$ is the limit of some weakly convergent subsequence 
$(\chi\alpha_k V_{\alpha_k}\tilde h)_{k\ge 1}$ with $\lim_{k\to\infty}
\alpha_k = +\infty$ it follows for all $\varv\in C_0^\infty (\R^d )$ that 
\begin{eqnarray*}
\int_{\R^d} (u  - \chi \tilde h) \varv \, d\mu &= &\lim_{k\to\infty}\int_{\R^d}\chi (\alpha_k 
V_{\alpha_k}\tilde h - \tilde h)\varv\, d\mu \\ 
& = &\lim_{k\to\infty}\lim_{n\to\infty} \int_{\R^d}\chi (\alpha_k V_{\alpha_k}f_n - 
f_n ) \varv\, d\mu \\ 
& =& \lim_{k\to\infty}\lim_{n\to\infty} \int_{\R^d}\chi L^A (V_{\alpha_k}
f_n) \varv\, d\mu \\ 
& = &\lim_{k\to\infty}\lim_{n\to\infty} \Big (\int_{\R^d} V_{\alpha_k}f_n \, L^0 
(\chi \varv )\, d\mu -\int_{\R^d}\langle\beta^{\rho, A}  , \nabla V_{\alpha_k}f_n\rangle
\chi \varv\, d\mu \Big )\\ 
& = &\lim_{k\to\infty}\Big (\int_{\R^d} V_{\alpha_k}\tilde h \, L^0 (\chi \varv)\, 
d\mu - \int_{\R^d}\langle\beta^{\rho, A} , \nabla (\chi V_{\alpha_k}\tilde h)
\rangle \varv\, d\mu \\ 
&&\qquad \qquad + \int_{\R^d}\langle\beta^{\rho, A} , \nabla\chi\rangle V_{\alpha_k}
\tilde h\, \varv\, d\mu\Big ) \\ 
\end{eqnarray*} 
\newpage
\begin{eqnarray*}
& \le& \lim_{k\to\infty}\frac {1}{\alpha_k}\Big (\|h\|_{L^\infty (\R^d , \mu )}  
\|L^0 (\chi \varv)\|_{L^1 (\R^d , \mu )} \\ 
& &\quad \qquad \qquad+ \sqrt {2\nu }\|\|\beta^{\rho, A}  \|\varv\|_{L^2 (\R^d , \mu )}  
\cE^0 (\chi\alpha_k V_{\alpha_k} \tilde h , \chi\alpha_k V_{\alpha_k}\tilde h )^{1/2} \\ 
& &\quad \qquad \qquad\qquad + \sqrt {2\nu }\|h\|_{L^\infty (\R^d , \mu )} \|\beta^{\rho, A}  
\varv\|_{L^2 (\R^d , \R^d, \mu )} \cE^0 (\chi , \chi )^{1/2}\Big ) = 0.  
\end{eqnarray*} 
Consequently, $\chi\tilde h$ is a $\mu$-version of $u$. In particular, $\chi h\in H_0^{1,2}(\R^d , \mu )$. 

\noindent 
Let $u\in H_0^{1,2}(\R^d , \mu )$ with compact support, $\chi\in 
C_0^\infty (\R^d)$ such that $\chi \equiv 1$ on $\text{supp}(|u|\mu )$ 
and $u_n\in C_0^\infty (\R^d)$, $n\ge 1$, such that 
$\lim_{n\to\infty} u_n = u$ in $H_0^{1,2} (\R^d ,\mu )$. Then   
$$ 
\begin{aligned} 
\cE_1^0 (u,h) 
& - \int_{\R^d}\langle\mathbf{B}, \nabla u\rangle h\, d\mu 
= \lim_{n\to\infty}\cE_1^0 (u_n,h) -\int_{\R^d}\langle\mathbf{B}, 
\nabla u_n\rangle h\, d\mu \\  
& = \lim_{n\to\infty} \int_{\R^d} (1-L)u_n\, \chi h \, d\mu = 0. 
\end{aligned} 
$$ 
\end{proof} 

\begin{corollary} 
\label{cor2.1.1} 
Let \eqref{condition on mu}--\eqref{eq:2.1.4}
and \eqref{AssumptionUniqueness} be satisfied. 
Let $(\overline{L}, D(\overline{L}))$ be the maximal extension of 
$(L, C_0^\infty (\R^d ))$ satisfying (i)--(iii) in Theorem \ref{theorem2.1.5}  and 
$(\overline{T}_t)_{t> 0}$ the associated semigroup. Then 
$(L, C_0^\infty (\R^d ))$ is $L^1$-unique, if and only if 
$\mu$ is $(\overline{T}_t)_{t>0}$-invariant (see Definition \ref{def:2.1.1})
\index{uniqueness ! $L^1$-unique}\index{measure ! invariant}. 
\end{corollary}  

\begin{proof} 
Clearly, if $(L, C_0^\infty (\R^d ))$ is $L^1$-unique 
it follows that $(L, D(L^0)_{0,b})$ is $L^1$-unique. Hence $\mu$ is 
$(\overline{T}_t)_{t>0}$-invariant by Proposition \ref{prop2.1.9}.

\noindent 
Conversely, let $h\in L^\infty (\R^d, \mu )$ be such that $\int_{\R^d} 
(1-L)u\, h\, d\mu = 0$ for all $u\in C_0^\infty (\R^d)$. Then 
$h\in H_{loc}^{1,2}(\R^d , \mu )$ and $\cE^0_1 (u,h)-\int_{\R^d}\langle \mathbf{B} , 
\nabla u \rangle h\, d\mu = 0$ for all $u\in H_0^{1,2}(\R^d , \mu )_0$ 
by Theorem \ref{theorem2.1.3}. In particular, 
\begin{equation} 
\label{eq:2.1.2.4}
\int_{\R^d} (1-L)u\, h\, d\mu = \cE_1^0(u, h)-\int_{\R^d}\langle \mathbf{B} ,\nabla u\rangle h
\, d\mu = 0 \text{ for all }u\in D(L^0)_{0,b}. 
\end{equation} 
Since $\mu$ is $(\overline{T}_t)_{t>0}$-invariant it follows from Proposition \ref{prop2.1.9} 
that $(L, D(L^0)_{0,b})$ is $L^1$-unique and \eqref{eq:2.1.2.4} now implies that $h=0$. 
Hence $(L, C_0^\infty (\R^d ))$ is $L^1$-unique too.  
\end{proof} 
\noindent
In the particular symmetric case, i.e., we can reformulate 
Corollary \ref{cor2.1.1} as follows: 

\begin{corollary} \label{cor2.1.2} 
Let \eqref{condition on mu}--\eqref{eq:2.1.4} and \eqref{AssumptionUniqueness} be satisfied.
Let $\mathbf{G} = \beta^{\rho, A}$, i.e. $\mathbf{B}=0$  (cf. \eqref{eq:2.1.5a} and \eqref{eq:2.1.3}). 
Then $(L^0, C_0^\infty (\R^d ))$ (cf. \eqref{eq:2.1.3d}) is $L^1$-unique, if and only if the associated 
Dirichlet form $(\cE^0 , D(\cE^0))$ is conservative, i.e. $T_t^0 1_{\R^d}=1$ $\mu$-a.e.  for all $t>0$\index{uniqueness ! $L^1$-unique}\index{semigroup ! conservative}.
\end{corollary} 

\begin{proof} 
Clearly, $(\cE^0 , D(\cE^0))$ is conservative, if and only if 
$T^{\prime}_t 1_{\R^d} =1$  $\mu$-a.e. for all $t> 0$. 
But by Remark \ref{rem:2.1.10a}(i),
$T^{\prime}_t 1_{\R^d} =1$, $\mu$-a.e., for all $t>0$, if and only if $\int_{\R^d}\overline{T}_t u \, d\mu 
= \int_{\R^d} u\, d\mu$ for all $u\in L^1 (\R^d ,\mu )$ and $t>0$, i.e., $\mu$ is  
$(\overline{T}_t)_{t>0}$-invariant, which implies the result by Corollary \ref{cor2.1.1}.  
\end{proof}

\subsection{Existence and regularity of densities to infinitesimally invariant measures}
\label{sec2.2}
Since the abstract analysis on existence and uniqueness of solutions to the abstract Cauchy problem 
\eqref{eq:2.1.0} developed in Section \ref{sec2.1} requires the existence and certain regularity properties 
of an infinitesimally invariant measure $\mu$ for $(L, C_0^\infty (\R^d))$, i.e. a locally finite 
nonnegative measure satisfying \eqref{condition on mu}--\eqref{eq:2.1.3} and 
\begin{equation} 
\label{eq:2.2.0}
\int_{\R^d} \big (\frac 12 \sum_{i,j =1}^d a_{ij}\partial_{ij}f  + \sum_{i=1}^d g_i \partial_i f\big ) 
\, d\mu = 0,  \qquad  \forall f\in C_0^\infty (\R^d),  
\end{equation} 
we will first identify in Section \ref{subsec:2.2.1} a set of sufficient conditions on the coefficients 
$(a_{ij})_{1\le i,j\le d}$ and $(g_i)_{1\le i\le d}$ that imply the existence of such $\mu$. We will in particular obtain existence of a 
sufficiently regular density $\rho$, that allows us to apply Theorem \ref{theorem2.1.5} in order to obtain 
the existence of a 
closed extension of  $(L,C_0^{\infty}(\R^d))$ generating a sub-Markovian 
$C_0$-semigroup of contractions $(\overline{T}_t)_{t>0}$ on $L^1 (\R^d , \mu )$ with the further properties of Theorem \ref{theorem2.1.5}. 
As one major aim of this book is to understand $L$ as the generator of a solution to an SDE with 
corresponding coefficients, the class of admissible coefficients, i.e. the class of coefficients $(a_{ij})_{1\le i,j\le d}$ and $(g_i)_{1\le i\le d}$ for which \eqref{eq:2.2.0}  has a solution $\mu$ with nice density $\rho$, plays an important 
role.

\subsubsection{Class of admissible coefficients and the main theorem} 
\label{subsec:2.2.1}
In order to understand the class of admissible coefficients better, it will be suitable to write $L$ in divergence form. {\bf Throughout,  we let the dimension $d\ge 2$.} The case $d=1$ plays a special role since it allows for explicit and partly
elementary computations with strong regularity results. It is therefore
best treated separately elsewhere and will therefore not be considered
further from now on. Instead we included the case $d=1$ in the outlook (cf. Chapter \ref{conclusionoutlook} part 1.).
We then consider the following class of divergence form operators\index{operator ! divergence form} with respect to a possibly non-symmetric diffusion matrix and perturbation $\mathbf{H}=(h_1,\dots,h_d)$:
\begin{eqnarray}\label{eq:2.2.0first}
Lf & = & \frac12 \sum_{i,j=1}^{d} \partial_i((a_{ij}+c_{ij})\partial_j)f+\sum_{i=1}^{d}h_i\partial_i f, \quad f\in C^{2}(\R^d),
\end{eqnarray}
where the coefficients $a_{ij}$, $c_{ij}$, and $h_i$, satisfy the following {\bf assumption}\index{assumption ! {\bf (a)}}:
\begin{itemize}
\item[{\bf (a)}] \ \index{assumption ! {\bf (a)}}$a_{ji}= a_{ij}\in H_{loc}^{1,2}(\R^d) \cap C(\R^d)$, $1 \leq i, j \leq d$, $d\ge 2$, and $A = (a_{ij})_{1\le i,j\le d}$ satisfies \eqref{eq:2.1.2}.
$C = (c_{ij})_{1\le i,j\le d}$, with $-c_{ji}=c_{ij} \in H_{loc}^{1,2}(\R^d) \cap C(\R^d)$, $1 \leq i,j \leq d$, $\mathbf{H}=(h_1, \dots, h_d) \in L_{loc}^p(\R^d, \R^d)$ for some $p\in (d,\infty)$.
\end{itemize}
The anti-symmetry $-c_{ji}=c_{ij}$ in assumption {\bf (a)} is needed for the equivalence of infinitesimal invariance \eqref{eq:2.2.0} and variational equality \eqref{eq:2.2.0a}, to switch from divergence form \eqref{eq:2.2.0first} to non-divergence form \eqref{equation G with H}, and to obtain that $\beta^{\rho, C^T}$ has zero divergence in the weak sense with respect to $\mu$ (see Remark \ref{rem:2.2.4}).\\
Under assumption {\bf \text{(a)}},  $L$ as in \eqref{eq:2.2.0first} is written for $f\in C^{2}(\R^d)$ as
\begin{eqnarray}\label{equation G with H}
Lf &= & \frac12\mathrm{div}\big ( (A+C)\nabla f\big )+\langle\mathbf{H}, \nabla f\rangle \nonumber \\
&= & \frac12\mathrm{trace}\big ( A\nabla^2 f\big )+\big \langle\frac{1}{2}\nabla  (A+C^{T})+ \mathbf{H}, \nabla f\big \rangle,\\ \nonumber
\end{eqnarray}
where for a matrix $B=(b_{ij})_{1 \leq i,j \leq d}$ of functions
\begin{equation}\label{divergence of row}
\nabla B=((\nabla B)_1, \ldots, (\nabla B)_d)
\end{equation}
with
\begin{equation}\label{divergence of row i}
(\nabla B)_i=\sum_{j=1}^d\partial_j b_{ij}, \qquad 1 \leq i \leq d.
\end{equation}
{\bf From now on (unless otherwise stated), we assume always that $\mathbf{G}$ has under assumption {\bf \text{(a)}} the following form}:
\begin{eqnarray}\label{form of G}
\mathbf{G}=(g_1, \dots, g_d)=\frac{1}{2}\nabla \big (A+C^{T}\big )+ \mathbf{H},
\end{eqnarray}
where $A$, $C$, and $\mathbf{H}$ are as in assumption {\bf (a)}. Thus $L$ as in \eqref{eq:2.2.0first} is written as a non-divergence form operator\index{operator ! non-divergence form}
\begin{eqnarray*}
Lf & = & \frac12 \sum_{i,j=1}^{d}a_{ij}\partial_{ij}f+\sum_{i=1}^{d}g_i\partial_i f, \quad f \in C^{2}(\R^d),
\end{eqnarray*}
where
\begin{eqnarray}\label{defofL}
g_i=\frac12\sum_{j=1}^{d} \partial_{j} (a_{ij}+c_{ji})+h_i,\quad  1 \leq i \leq d.
\end{eqnarray}
\begin{remark}\label{rem:2.2.7}
{\it The class of admissible coefficients satisfying {\bf \text{(a)}} is quite large. It does not only allow us to consider fairly general divergence form operators.
Assumption {\bf (a)} allows us also to consider a fairly general subclass of non-divergence form operators $L$. Indeed, choose
$a_{ij} \in  H_{loc}^{1,p}(\R^d)\cap C(\R^d)$, $1\le i,j\le d$, for some $p\in (d,\infty)$, such that $A = (a_{ij})_{1\le i,j\le d}$  satisfies \eqref{eq:2.1.2}, $C\equiv 0$,
and 
$$
\mathbf{H}=\mathbf{\widetilde{H}}-\frac12 \nabla A, \text{ with arbitrary }\mathbf{\widetilde{H}}\in L_{loc}^p(\R^d, \R^d).
$$
Putting $\mathbf{\widetilde{H}}=\mathbf{G}$, this leads to any non-divergence form operator $L$, such that for any $f \in C^{2}(\R^d)$
\begin{eqnarray*} 
Lf &=& \frac12 \sum_{i,j=1}^{d} a_{ij}\partial_{ij}f+\sum_{i=1}^{d}g_i\partial_i f\ =\ \frac12\mathrm{trace}(A\nabla^2 f)+\langle \mathbf{G}, \nabla f \rangle \label{eq:2.2.1} \\
\end{eqnarray*}
with the following {\bf assumption} on the coefficients
\begin{itemize}
\item[{\bf (a$^{\prime}$)}] \index{assumption ! {\bf (a$^{\prime}$)}}for some $p\in (d,\infty)$, $a_{ji}= a_{ij}\in H_{loc}^{1,p}(\R^d) \cap C(\R^d)$, $1 \leq i, j \leq d$, $A = (a_{ij})_{1\le i,j\le d}$ satisfies \eqref{eq:2.1.2} and $\mathbf{G}=(g_1,\ldots,g_d) \in L_{loc}^p(\R^d,\R^d)$.
\end{itemize}
}
\end{remark}
\medskip
\noindent
If assumption {\bf (a)} holds and $\rho \in H^{1,2}_{loc}(\R^d)$, \eqref{eq:2.2.0} is by integration by parts equivalent (cf. \eqref{equation G with H}--\eqref{defofL}) to the following variational equality\index{variational equality}:
\begin{equation}\label{eq:2.2.0a}
\int_{\R^d} \langle \frac12 (A+C^T) \nabla \rho - \rho \mathbf{H}, \nabla f \rangle dx  = 0, \quad \forall f \in C_0^{\infty}(\R^d).
\end{equation}
In the next section, we show how the variational equation \eqref{eq:2.2.0a} can be adequately solved using classical tools from PDE theory and that for the measure $\rho\,dx$, where $\rho$ is the solution to \eqref{eq:2.2.0a}, \eqref{eq:2.1.1} and \eqref{eq:2.1.3} are satisfied. In particular, replacing $\hat{A}$ and $\hat{\mathbf{H}}$ in Theorem \ref{Theorem2.2.4} below with $\frac12 (A+C^T)$ and $\mathbf{H}$, respectively, we obtain  the {\bf following main theorem} of this section.

\begin{theorem}\label{theo:2.2.7}
Under assumption {\bf (a)} (see the beginning of Section \ref{subsec:2.2.1}), there exists $\rho \in H^{1,p}_{loc}(\R^d) \cap C(\R^d)$ with $\rho(x)>0$ for all $ x\in \R^d$ such that with $\mu=\rho dx$, and $L$ as in \eqref{equation G with H} (see also \eqref{form of G} and Remark \ref{rem:2.2.7}), it holds that
\begin{equation} \label{eq:2.2.8}
\int_{\R^d} Lf d\mu = 0, \quad \text{ for all } f \in C_0^{\infty}(\R^d).
\end{equation}
In particular, $\mu$ as given above satisfies the assumption \eqref{condition on mu} on $\mu$ at the beginning of Section \ref{subsec:2.1.1} and moreover as a simple consequence of assumption {\bf (a)}, the assumptions \eqref{eq:2.1.1}--\eqref{eq:2.1.4} are satisfied and therefore Theorem \ref{theorem2.1.5} applies.
\end{theorem}
\medskip
By Remark \ref{rem:2.2.7}, the first part of Theorem \ref{theo:2.2.7} is a generalization of \cite[Theorem 1]{BRS} (see also \cite[Theorem 2.4.1]{BKRS}), where the existence of a density $\rho$ with the same properties as in Theorem \ref{theo:2.2.7} is derived under 
assumption {\bf (a$^{\prime}$)}.

\subsubsection{Proofs} 
\label{subsec:2.2.2}

This section serves to provide the missing ingredients for the proof of Theorem \ref{theo:2.2.7}, in particular Theorem \ref{Theorem2.2.4} below.

\begin{lemma} \label{Lemma2.2.1}   
Let $\hat{A}:=(\hat{a}_{ij})_{1 \leq i,j \leq d}$ be a (possibly non-symmetric) matrix of bounded measurable functions on an open ball $B$ and suppose there is a constant $\lambda>0$ such that 
\begin{equation} \label{eq:2.2.3}
\lambda \|\xi\|^2 \leq \langle \hat{A}(x) \xi, \xi \rangle, \quad \forall \xi \in \R^d, x\in B.
\end{equation}
Let $\hat{\mathbf{H}} \in L^{d\vee (2+\varepsilon)}(B, \R^d)$ for some $\varepsilon>0$.
Then for any $\Phi \in H^{1,2}_0(B)'$, there exists a unique $u \in H^{1,2}_0(B)$ such that
$$
\int_B \langle \hat{A} \nabla u+ u\hat{\mathbf{H}}, \nabla \varphi \rangle dx = \[\Phi, \varphi \], \quad \forall \varphi \in H^{1,2}_0(B),
$$ 
where $[\cdot,\cdot]$ denotes the dualization between $H^{1,2}_0(B)'$ and $H^{1,2}_0(B)$, i.e.$\[\Phi, \varphi \]=\Phi(\varphi)$.
\end{lemma}
\begin{proof}
For $\alpha \geq 0$, define a bilinear form $\mathcal{B}_{\alpha}: H^{1,2}_0(B) \times H^{1,2}_0(B)  \rightarrow \R$ by
$$
\mathcal{B}_{\alpha}(u,v) = \int_B \langle \hat{A} \nabla u+ u\hat{\mathbf{H}} , \nabla v \rangle dx + \alpha \int_{B} uv dx.
$$
Then by \cite[Lemme 1.5, Th\'eor\`eme 3.2]{St65}, there exist constants $M, \gamma, \delta>0$ such that
$$
|\mathcal{B}_{\gamma}(u,v)| \leq M \|u\|_{H_0^{1,2}(B)} \|v\|_{H_0^{1,2}(B)}\, ,\quad \forall u,v \in H^{1,2}_0(B)
$$
and
\begin{equation} \label{eq:2.2.4}
|\mathcal{B}_{\gamma}(u,u)| \geq \delta \|u\|^2_{H^{1,2}_0(B)}.
\end{equation}
Let $\Psi \in H^{1,2}_0(B)'$ be given.
Then by the Lax--Milgram theorem \cite[Corollary 5.8]{BRE}, there exists a unique $S(\Psi) \in H^{1,2}_0(B)$ such that
$$
\mathcal{B}_{\gamma}(S(\Psi), \varphi) = \[\Psi, \varphi \], \quad \varphi \in H^{1,2}_0(B).
$$
By \eqref{eq:2.2.4}, it follows that the map $S: H_0^{1,2}(B)' \rightarrow H^{1,2}_0(B)$ is a bounded linear operator. 
Now define $J: H^{1,2}_0(B) \rightarrow H^{1,2}_0(B)'$ by 
$$
[ J(u),v ] = \int_{B} u v dx, \qquad u, v \in H^{1,2}_0(B).
$$
By the weak compactness of balls in $H^{1,2}_0(B)$, $J$ is a compact operator, hence $S \circ J: H^{1,2}_0(B) \rightarrow H^{1,2}_0(B)$
is also a compact operator. In particular 
$$
\exists v \in H^{1,2}_0(B) \ \text{  with }\ \mathcal{B}_0(v,\varphi) = [ \Psi, \varphi ],\quad \forall \varphi \in  H_0^{1,2}(B),
$$ 
if and only if
$$
\exists v \in H^{1,2}_0(B) \ \text{  with } \ \left(I-\gamma S\circ J\right)(v) = S(\Psi),
$$
where $I:H^{1,2}_0(B) \rightarrow H^{1,2}_0(B)$ is the identity map.
By the maximum principle \cite[Theorem 4]{T77}, $\left(I-\gamma S \circ J\right)(v) =S(0)$ if and only if $v=0$.
Now let $\Phi \in H^{1,2}_0(B)'$ be given. Using the Fredholm-alternative \cite[Theorem 6.6(c)]{BRE}, we can see that there exists a unique $u \in H^{1,2}_0(B)$ such that
$$
\left(I-\gamma S\circ J\right)(u) = S(\Phi),
$$
as desired.
\end{proof}

The following Theorem \ref{Theorem2.2.2} originates from \cite[Theorem 1.8.3]{BKRS}, where $\hat{A}$ is supposed to be symmetric, but it straightforwardly extends to non-symmetric $\hat{A}$ by \cite[Theorem 1.2.1]{BKRS} (see \cite[Theorem 2.8]{Kr07} for the original result), since \cite[Theorem 2.8]{Kr07} holds for non-symmetric $\hat{A}$. 
\begin{theorem}  \label{Theorem2.2.2}
Let $B$ be an open ball in $\R^d$ and $\hat{A}=(\hat{a}_{ij})_{1 \leq i,j \leq d}$ be a (possibly non-symmetric) matrix of continuous functions on $\overline{B}$ that satisfies \eqref{eq:2.2.3}. 
Let $\mathbf{F} \in L^p(B, \R^d)$ for some $p\in (d,\infty)$. Suppose $u \in H^{1,2}(B)$ satisfies
$$
\int_{B} \langle \hat{A} \nabla u,\nabla \varphi\rangle dx = \int_{B} \langle \mathbf{F}, \nabla \varphi \rangle dx, \quad \forall \varphi \in C_0^{\infty}(B).
$$
Then  $u \in H^{1,p}(V)$ for any open ball $V$ with $\overline{V} \subset B$.
\end{theorem}

\begin{theorem} \label{Theorem2.2.4}
Let $\hat{A}:=(\hat{a}_{ij})_{1 \leq i,j \leq d}$ be a (possibly non-symmetric) matrix of locally bounded measurable functions on $\R^d$. 
Assume that for each open ball $B$  there exists a constant $\lambda_B>0$ such that 
$$
\lambda_B \|\xi\|^2 \leq \langle \hat{A}(x) \xi, \xi \rangle, \quad \xi \in \R^d, x\in B.
$$ 
Let $\hat{\mathbf{H}}\in L^p_{loc}(\R^d, \R^d)$  for some $p\in (d,\infty)$. Then it holds that:\\[3pt]
(i)  There exists $\rho \in H_{loc}^{1,2}(\R^d) \cap C(\R^d)$ with $\rho(x)>0$ for all $x \in \R^d$
such that 
$$
\int_{\R^d} \langle \hat{A} \nabla \rho+ \rho \hat{\mathbf{H}}, \nabla \varphi \rangle dx = 0, \quad \forall \varphi \in C^{\infty}_0(\R^d).
$$ 
(ii) If additionally $\hat{a}_{ij} \in C(\R^d)$, $1 \leq i,j \leq d$, then $\rho \in H^{1,p}_{loc}(\R^d)$.
\end{theorem}
\begin{proof}
(i) Let $n \in \N$. By Lemma  \ref{Lemma2.2.1}   and \cite[Corollary 5.5]{T77}, there exists $v_n \in H_0^{1,2}(B_n) \cap C(B_n)$ such that
\begin{equation*}
\int_{B_{n}}  \langle \hat{A} \nabla v_n+ v_n \hat{\mathbf{H}} , \nabla \varphi  \rangle  dx = \int_{B_n} \langle-\hat{\mathbf{H}}, \nabla \varphi \rangle dx, \quad \text{ for all } \varphi \in C_0^{\infty}(B_n).
\end{equation*}
Let $u_n:= v_n+1$.  Then $T(u_n)=1$ on $\partial B_n$, where $T$ is the trace operator from $H^{1,2}(B_n)$ to $L^2(\partial B_n)$. Moreover,
\begin{equation} \label{eq:2.2.5}
\int_{B_{n}}  \langle \hat{A} \nabla u_n+ u_n \hat{\mathbf{H}} , \nabla \varphi  \rangle  dx = 0, \quad \text{ for all } \varphi \in C_0^{\infty}(B_n).
\end{equation} 
Since $0 \leq u_n^{-}\le v_n^-$, we have $u^-_n \in H^{1,2}_0(B_n)$.  Thus by 
\cite[Lemma 3.4]{LT19}, we get 
\begin{equation*}
\qquad \quad \int_{B_n} \langle \hat{A} \nabla u_n^{-} +u_n^{-} \,\hat{\mathbf{H}}, \nabla \varphi  \rangle dx \leq 0, \quad  \varphi \in C_0^{\infty}(B_n) \text{ with } \varphi \geq 0.
\end{equation*}
By the maximum principle \cite[Theorem 4]{T77}, we have $u_n^- \leq 0$, hence $u_n \geq 0$. Suppose there exists $x_0 \in B_n$ such that $u_n(x_0)=0$. Then applying the Harnack inequality of \cite[Corollary 5.2]{T73} to $u_n$ of \eqref{eq:2.2.5}, $u_n(x)=0$  for all $x \in B_n$, hence $T(u_n)=0$ on $\partial B_n$, which is contradiction. Therefore, $u_n$ is strictly positive on $B_n$. Now let $\rho_n(x):= u_n(0)^{-1}u_n(x)$, $x\in B_n, n\in \N$. Then $\rho_n(0)=1$ and
$$
\int_{B_n}  \langle  \hat{A} \nabla \rho_n +\rho_n \hat{\mathbf{H}} , \;\nabla \varphi  \rangle dx=0 \;\; \text{ for all } \varphi \in C_0^{\infty}(B_n). 
$$
Fix $r>0$. Then, by \cite[Corollary 5.2]{T73}
$$
\sup_{x \in B_{2r}}\rho_n(x) \leq C_1 \inf_{x \in B_{2r}}\rho_n(x) \leq C_1 \; \text{ for all } n>2r,
$$
where $C_1>0$ is independent of $\rho_n$, $n>2r$. 
By \cite[Lemma 5.2]{St65}, 
$$
\| \rho_n \|_{H^{1,2}(B_{r})} \leq C_2 \|\rho_n \|_{L^2(B_{2r})} \leq C_1 C_2 \, dx(B_{2r}), \; \text{ for all } n>2r,
$$
where $C_2$ is independent of $(\rho_n)_{n>2r}$. By \cite[Corollary 5.5]{GilbargTrudinger}
$$
\| \rho_n \|_{C^{0,\gamma}(\overline{B}_r)} \leq  C_3 \sup_{B_{2r}}\|\rho_n \| \leq C_1 C_3,
$$
where $\gamma \in (0,1)$ and $C_3>0$ are independent of $(\rho_n)_{n>2r}$.
By weak compactness of balls in $H_0^{1,2}(B_r)$ and the Arzela--Ascoli theorem, there exists $(\rho_{n,r})_{n \geq 1} \subset (\rho_n)_{n >2r}$ and $\rho_{(r)} \in H^{1,2}(B_r) \cap C(\overline{B}_r)$
such that as $n\to \infty$
$$
\rho_{n,r} \rightarrow \rho_{(r)} \;\text{ weakly in }  H^{1,2}(B_r), \qquad \rho_{n,r} \rightarrow \rho_{(r)} \;\text{ uniformly on } \overline{B}_r.
$$
Choosing $(\rho_{n,k})_{n \geq 1} \supset (\rho_{n,k+1})_{n \geq 1}$, $k\in \N$, we get $\rho_{(k)}=\rho_{(k+1)}$ on $B_{k}$, hence we can well-define $\rho$ as 
$$
\rho := \rho_{(k)} \text{ on } B_{k}, k\in \N.
$$
Finally, applying the Harnack inequality of \cite[Corollary 5.2]{T73}, it holds that $\rho(x)>0$ for all $x \in \R^d$.\\
(ii) Let $R>0$. Then
$$
\int_{B_{2R}} \langle \hat{A} \nabla \rho, \nabla \varphi \rangle dx = - \int_{B_{2R}} \langle \rho \hat{\mathbf{H}}, \nabla \varphi \rangle dx, \quad \forall \varphi \in C^{\infty}_0(B_{2R}).
$$ 
Since $\rho \hat{\mathbf{H}} \in L^p(B_{2R}, \R^d)$, we obtain $\rho \in H^{1,p}(B_R)$ by Theorem \ref{Theorem2.2.2}.
\end{proof}

\begin{proof}{\it  of Theorem \ref{theo:2.2.7}}  \;  By Theorem \ref{Theorem2.2.4} applied with $\hat{A}=A+C^T$, there exists $\rho \in H^{1,p}_{loc}(\R^d) \cap C(\R^d)$ with $\rho(x)>0$ for all $x \in \R^d$ such that the variational equation \eqref{eq:2.2.0a} holds.
Using integration by parts, we obtain from \eqref{eq:2.2.0a}
$$
-\int_{\R^d} \left(\frac12\mathrm{trace}\big ( A\nabla^2 f\big )+\big \langle\frac{1}{2}\nabla  (A+C^{T})+ \mathbf{H}, \nabla f\big \rangle \right) \rho dx=0 , \quad \forall f \in C_0^{\infty}(\R^d).
$$
Letting $\mu=\rho dx$, \eqref{eq:2.2.8} follows. Since $\rho \in H^{1,p}_{loc}(\R^d) \cap C(\R^d)$ and $\rho(x)>0$ for all $x \in \R^d$, we obtain $\sqrt{\rho} \in H^{1,p}_{loc}(\R^d) \cap C(\R^d)$ with the help of the chain rule (\cite[Theorem 4.4(ii)]{EG15}). Moreover, $a_{ij} \in H^{1,2}_{loc}(\R^d)=H^{1,2}_{loc}(\R^d, \mu)$ for all $1 \leq i,j \leq d$ and $\mathbf{G}=\frac{1}{2}\nabla  (A+C^{T})+ \mathbf{H} \in L^2_{loc}(\R^d, \R^d) = L_{loc}^2(\R^d, \R^d, \mu)$.
\end{proof}

\subsubsection{Discussion} 
\label{subsec:2.2.3}

The converse problem of constructing and analyzing a partial differential operator $(L, C_0^\infty (\R^d))$ 
with suitable coefficients, given a prescribed infinitesimally invariant measure, appears in applications 
of SDEs, e.g. to the sampling of probability distributions (see \cite{Hwang}) 
or more generally to ergodic control problems (see \cite{Borkar}). In the following remark we will briefly 
discuss the applicability of our setting to this problem.

\begin{remark} \label{rem:2.2.4}
{\it In Theorem \ref{theo:2.2.7}, we derived under the assumption {\bf (a)} the existence of a nice density 
$\rho$ such that \eqref{eq:2.2.8} holds. Conversely, if $\rho \in H^{1,p}_{loc}(\R^d) \cap C(\R^d)$ for some $p\in (d,\infty)$ with 
$\rho(x)>0$ for all $x\in \R^d$ is explicitly given, we can construct a large class of partial differential 
operators $(L, C_0^{\infty}(\R^d))$ as in \eqref{eq:2.2.0first} satisfying condition {\bf (a)} and such 
that $\mu=\rho dx$ is an infinitesimally invariant measure for $(L, C_0^{\infty}(\R^d))$, i.e. \eqref{eq:2.2.8} 
holds. \\
More specifically, for any $A=(a_{ij})_{1 \leq i,j \leq d}$ and $C=(c_{ij})_{1 \leq i,j \leq d}$ 
satisfying condition {\bf (a)} of Section \ref{subsec:2.2.1} and any $\mathbf{\overline{B}} \in L_{loc}^p(\R^d, \R^d)$ satisfying 
$$
\int_{\R^d} \langle \mathbf{\overline{B}}, \nabla \varphi \rangle \rho dx=0, \quad \text{ for all } 
\varphi \in C_0^{\infty}(\R^d) 
$$
it follows that $A$, $C$ and $\mathbf{H}:=\frac{(A+C^T) \nabla \rho}{2\rho}+\mathbf{\overline{B}}$ satisfy 
condition {\bf (a)}. In particular (cf. \eqref{eq:2.1.5}, \eqref{eq:2.1.5a} and \eqref{form of G}) $\mathbf{B} =\mathbf{G}-\beta^{\rho,A} 
= \beta^{\rho, C^T}+\mathbf{\overline{B}} \in L_{loc}^2(\R^d, \R^d, \mu)$, where
$$
\beta^{\rho, C^T} = \frac{1}{2} \nabla C^T +C^T \frac{\nabla \rho}{2 \rho}
$$ 
(see \eqref{divergence of row} and \eqref{divergence of row i} for the definition of $\nabla C^T$),
and $\rho dx$ is an 
infinitesimally invariant measure for $(L, C_0^{\infty}(\R^d))$, since by integration by parts
$$
\int_{\R^d} \langle \beta^{\rho, C^T}, \nabla \varphi \rangle \rho dx 
= \frac12\int_{\R^d} \sum_{i,j=1}^{d}\rho c_{ij} \partial_i \partial_j \varphi  dx=0, \quad \text{ for all } 
\varphi \in C_0^{\infty}(\R^d),
$$
so that
$$
\int_{\R^d} \langle \mathbf{B}, \nabla \varphi \rangle \rho dx = 0, \quad \text{ for all } \varphi \in C_0^{\infty}(\R^d).
$$
In particular, \eqref{condition on mu}--\eqref{eq:2.1.4} hold, so that the results of Section \ref{sec2.1} are applicable.
}
\end{remark}

\subsection{Regular solutions to the abstract Cauchy problem}
\label{sec2.3}
In this section, we investigate the regularity properties of $(T_t)_{t>0}$ as defined in Definition \ref{definition2.1.7}, as well as regularity properties of the corresponding resolvent. The semigroup regularity will play an important role in Chapter \ref{chapter_3} to construct an associated Hunt process that can start from every point in $\R^d$. The resolvent regularity will be used to derive a Krylov-type estimate for the associated Hunt process in Theorem \ref{theo:3.3}. 
Throughout this section, we let
$$
\mu=\rho\,dx
$$ 
be as in Theorem \ref{theo:2.2.7} or as in  Remark \ref{rem:2.2.4}.\\
Here to obtain the $L^{s}(\R^d,\mu)$-strong Feller property, $s\in [1, \infty]$, including the strong Feller property of $(P_t)_{t>0}$ (for both definitions see Definition \ref{def:2.3.1} below), we only need condition {\bf (a)} of Section \ref{subsec:2.2.1}. The conservativeness of $(T_t)_{t>0}$ is not needed. Our main strategy is to use H\"{o}lder regularity results and Harnack inequalities for variational solutions to  elliptic and parabolic PDEs of divergence type. Indeed, we show that given a sufficiently regular function $f$, $\rho G_{\alpha} f$ and $\rho T_{\cdot} f$ are the variational solutions to elliptic and parabolic PDEs of divergence type, respectively, so that the results of \cite{St65} and \cite{ArSe} apply. \\
To obtain the regularity of $(T_t)_{t>0}$ in our case, it is notable that one could apply the result \cite[Theorem 4.1]{BKR2} based on  Sobolev regularity for parabolic equations involved with measures. But then it would be required that $a_{ij} \in H^{1,\widetilde{p}}_{loc}(\R^d)$ for all $1 \leq i,j \leq d$ and $\mathbf{G} \in L^{\widetilde{p}}_{loc}(\R^d, \R^d)$, $\widetilde{p}>d+2$ and the strong Feller property of the regularized version $(P_t)_{t>0}$ of $(T_t)_{t>0}$ may not directly be derived without assuming the conservativeness of $(T_t)_{t>0}$. Proceeding this way would hence be too restrictive.\\
At the end of this section we briefly 
discuss related work on regularity results in the existing literature.

\begin{theorem} \label{theorem2.3.1}
Let $q=\frac{pd}{p+d}$, $p\in (d,\infty)$. Suppose {\bf (a)} of Section \ref{subsec:2.2.1} holds and let $g \in 
\cup_{r\in [q,\infty]} L^r(\R^d,\mu)$ with $ g \geq 0$, $\alpha>0$. Then $G_{\alpha} g$ (see Definition \ref{definition2.1.7}) has a locally H\"{o}lder continuous version $R_{\alpha} g$ and for any open balls $U$, $V$ in $\R^d$, with $\overline{U} \subset V$, 
\begin{eqnarray} 
\| R_{\alpha}g \|_{C^{0, \gamma}(\overline{U})} \le c\Big (\| g \|_{L^q(V, \mu)} + \| G_{\alpha}g \|_{L^1(V, \mu)}\Big ),  \label{eq:2.3.37}
\end{eqnarray} 
where $c>0$ and $\gamma \in (0,1)$ are constants, independent of $g$.
\end{theorem}

\begin{proof}
Let $g \in C_0^{\infty}(\R^d)$ and $\alpha >0$. Then for all $\varphi \in C_0^{\infty}(\R^d)$,
\begin{eqnarray}\label{eq:2.3.38b}
\int_{\R^d} (\alpha- L_2' ) \varphi \cdot \big(G_{\alpha} g\big) \,d\mu = \int_{\R^d} G'_{\alpha}(\alpha-L_2') \varphi  \cdot g \,d\mu=\int_{\R^d} \varphi  g \,d\mu, \vspace{-0.4em}
\end{eqnarray}
and it follows from \eqref{eq:2.1.5a}, Definition \ref{definition2.1.7}, \eqref{eq:2.1.3bis2'}, and \eqref{equation G with H}, that
\begin{eqnarray*}
L_2' \varphi&=& \frac{1}{2}\text{trace}(A \nabla^2 \varphi)+ \langle 2 \beta^{\rho,A}-\mathbf{G}, \nabla \varphi \rangle \\
&=& \frac12 \text{div}((A+C^T) \nabla \varphi )+ \langle -\frac12 \nabla (A+ C)+2 \beta^{\rho,A}-\mathbf{G}, \nabla \varphi \rangle \\
&=&  \frac12 \text{div}((A+C^T) \nabla \varphi )+ \langle \frac{A \nabla \rho}{\rho}- \mathbf{H}, \nabla \varphi \rangle.
\end{eqnarray*}
Since  by Theorem \ref{theorem2.1.5}, $G_{\alpha} g \in D(\overline{L})_b \subset D(\mathcal{E}^0) \subset H^{1.2}_{loc}(\R^d)$, 
applying integration by parts to the left hand side of \eqref{eq:2.3.38b}, for any $\varphi \in C_0^{\infty}(\R^d)$,
$$
\int_{\R^d} \big \langle \frac12 (A+C) \nabla (\rho G_{\alpha} g) + (\rho G_{\alpha} g)(\mathbf{H}-\frac{A \nabla \rho}{\rho}), \nabla \varphi \big \rangle +\alpha(\rho G_{\alpha}g) \varphi dx = \int_{\R^d} (\rho g) \varphi dx.
$$
Suppose now that $g \geq 0$. Then since $\frac{1}{\rho}$ is locally H\"{o}lder continuous, by \cite[
Th\'eor\`eme 7.2, 8.2]{St65}, $G_{\alpha} g$ has a locally H\"{o}lder continuous version $R_{\alpha}g$ on $\R^d$ and there exists a constant $\gamma \in (0, 1-d/p)$, independent of $g$, such that
\begin{eqnarray*}
\| \rho R_{\alpha}g\|_{C^{0, \gamma}(\overline{U})} &\leq& c_1 \left( \|\rho G_{\alpha} g\|_{L^2(V)} + \| \rho  g\|_{L^q(V)}  \right) \\
&\leq& c_1 \left( c_2\inf_{V} (\rho R_{\alpha} g) +c_2\|\rho g\|_{L^q(V)} +\|\rho g\|_{L^q(V)}   \right) \\
&\leq& c_3 \Big( \|\rho G_{\alpha} g\|_{L^1(V)}+ \| \rho g\|_{L^q(V)}  \Big),
\end{eqnarray*}
where $c_1, c_2, c_3>0$ are constants, independent of $g$.
Since $\rho \in L^{\infty}(V)$ and $\frac{1}{\rho} \in C^{0,\gamma}(\overline{U})$, \eqref{eq:2.3.37} follows for all $g \in C_0^{\infty}(\R^d)$ with $g \geq 0$.\\
\noindent
Moreover, for such $g$, using the $L^r(\R^d, \mu)$-contraction property of $\alpha G_{\alpha}$ for $r \in [q, \infty)$ and H\"{o}lder's inequality,
\begin{eqnarray} 
&&\| R_{\alpha}g \|_{C^{0, \gamma}(\overline{U})} \le c\left (\| g \|_{L^q(V, \mu)} + \| G_{\alpha}g \|_{L^1(V, \mu)}\right )  \label{eq:2.50new} \\
&\le&  c \left( \|\rho\|^{\frac{1}{q}-\frac{1}{r}}_{L^1(V)}
\| g \|_{L^r(\R^d, \mu)} +\|\rho \|^{\frac{r-1}{r}}_{L^1(V)} \|G_{\alpha} g\|_{L^r(\R^d, \mu)} \right) \nonumber \\
&\le& c(\|\rho\|^{\frac{1}{q}-\frac{1}{r}}_{L^1(V)} \vee \frac{1}{\alpha}\|\rho \|_{L^1(V)}^{\frac{r-1}{r}}) \|g\|_{L^r(\R^d, \mu)}. \label{eq:2.51new}
\end{eqnarray} 
\noindent
Now, suppose $g \in L^r(\R^d, \mu)$ for some $r \in [q, \infty)$ and $g \geq 0$. Choose $(g_n)_{n \geq 1} \subset C_0^{\infty}(\R^d) \cap \mathcal{B}^+(\R^d)$ with $\lim_{n \rightarrow \infty} g_n= g$ in $L^r(\R^d, \mu)$. Using a Cauchy sequence argument together with \eqref{eq:2.51new}, there exists $u^g \in C^{0, \gamma}(\overline{U})$ such that 
\begin{equation} \label{eq:2.52new}
\lim_{n \rightarrow \infty} R_{\alpha} g_n =u^g \quad \text{ in $C^{0, \gamma}(\overline{U})$}.
\end{equation}
Since $U$ is an arbitrary open ball in $\R^d$, we can well-define
\begin{equation} \label{eq:2.53def}
R_{\alpha} g:= u^g \quad \text{ on $\R^d$},
\end{equation}
i.e. $R_{\alpha} g$ is the same for any chosen sequence $(g_n)_{n \geq 1}$ as above. Moreover, $R_{\alpha} g$ is a continuous version of $G_{\alpha} g$ by \eqref{eq:2.52new} and it follows from \eqref{eq:2.50new}  that
\begin{eqnarray} \label{eq:2.54new}
\| R_{\alpha}g \|_{C^{0, \gamma}(\overline{U})} &\le& c\Big (\| g \|_{L^q(V, \mu)} + \| R_{\alpha} g \|_{L^1(V, \mu)}\Big ).
\end{eqnarray}
Finally, let $g \in L^{\infty}(\R^d, \mu)$ with $g \geq 0$ and $g_n:=1_{B_n} \cdot g \in L^1(\R^d, \mu)_b \subset L^{q}(\R^d, \mu)$, $n \geq 1$. Then $\lim_{n \rightarrow \infty} g_n=g$, a.e. By the sub-Markovian property of $(G_{\alpha})_{\alpha>0}$ and the continuity of $z \mapsto R_{\alpha} g_n(z)$ on $\R^d$, $(R_{\alpha} g_n(z))_{n \geq 1}$ is  for each $z \in \R^d$ a uniformly bounded and increasing sequence in $[0,1 ]$. Applying Lebesgue's theorem, $(R_{\alpha}g_n)_{n \geq 1}$ is a Cauchy sequence in $L^1(V, \mu)$ and $(g_n)_{n \geq 1}$ is a Cauchy sequence in $L^q(V, \mu)$. By using a Cauchy sequence argument together with \eqref{eq:2.54new},  we can well-define $R_{\alpha}g$ on $\R^d$ as we did in \eqref{eq:2.52new} and \eqref{eq:2.53def}. Hence $R_{\alpha}g$ is a continuous version of $G_{\alpha}g$ and \eqref{eq:2.3.37} holds for all $g \in L^{\infty}(\R^d, \mu)$ with $g \geq 0$ as desired.
\end{proof}
\medskip
\noindent
Let $g \in L^r(\R^d, \mu)$ for some $r \in [q, \infty]$ and $\alpha>0$. By splitting $g=g^+-g^-$, we define
\begin{equation}\label{resoldef}
R_{\alpha} g:=R_{\alpha}g^+ -R_{\alpha} g^-\quad \text{ on \,$\R^d$.}
\end{equation}
Then $R_{\alpha}g$ is a continuous version of $G_{\alpha}g$ and it follows from \eqref{eq:2.3.37} and the $L^r(\R^d, \mu)$-contraction property of $\alpha G_\alpha$ that
\begin{equation} \label{eq:2.3.38a}
\| R_{\alpha}g \|_{C^{0, \gamma}(\overline{U})} \leq c_4 \|g\|_{L^r(\R^d, \mu)},
\end{equation}
where $c_4>0$ is a constant, independent of $g$. Finally, let $f \in D(L_r)$ for some $r \in [q, \infty)$. Then $f = G_1 (1-L_r) f$, and $f$ has a 
locally H\"{o}lder continuous version on $\R^d$ by Theorem \ref{theorem2.3.1}. Moreover, for any open ball $U$, \eqref{eq:2.3.38a} implies
\begin{equation} \label{eq:2.3.39}
\|f\|_{C^{0,\gamma}(\overline{U})} \leq c_4  \| (1-L_r) f\|_{L^r(\R^d, \mu)} \leq c_4  \|f\|_{D(L_r)}.
\end{equation}
Since also $T_t f \in D(L_r)$, $T_t f$ has a continuous $\mu$-version, say $P_{t}f$, and it follows from \eqref{eq:2.3.39} and the $L^r(\R^d, \mu)$-contraction property of $(T_t)_{t>0}$ that
\begin{equation} \label{eq:2.3.40}
\|P_t f\|_{C^{0, \gamma}(\overline{U})} \leq c_4 \|T_t f \|_{D(L_r)} \leq c_4 \|f\|_{D(L_r)}.
\end{equation}

\begin{lemma} \label{eq:2.3.39a}
Let {\bf (a)} of Section \ref{subsec:2.2.1} be satisfied. Then for any $f\in \bigcup_{r\in [q,\infty)} D(L_r)$ the map
$$
(x,t)\mapsto P_t f(x)
$$
is continuous on $ \R^d\times [0,\infty)$, where $P_0 f:=f$ and $q=\frac{pd}{p+d}$, $p\in (d,\infty)$.
\end{lemma}
\begin{proof}
Let $f\in D(L_r)$ for some $r\in [q,\infty)$ and let $\left ((x_n,t_n)\right )_{n\ge 1}$ be a sequence in $\R^d\times [0,\infty)$ that converges to $(x_0,t_0)\in \R^d\times [0,\infty)$. 
Let $B$ be an open ball, such that $x_n \in B$ for all $n \geq 1$. Then by \eqref{eq:2.3.39}
for all $n \geq 1$ 
\begin{eqnarray*}
\left | P_{t_n} f(x_n)-P_{t_0} f(x_0) \right | &\leq& \| P_{t_n} f- P_{t_0}f \,\|_{C(\overline{B})} +\left | P_{t_0} f(x_n)-P_{t_0} f(x_0) \right | \\[3pt]
&\leq& c_4\| P_{t_n} f- P_{t_0}f \,\|_{D(L_r)}+\big | P_{t_0} f(x_n)-P_{t_0} f(x_0) \big |.
\end{eqnarray*}
Using the $L^r(\R^d, \mu)$-strong continuity of $(T_t)_{t>0}$ and the continuity of $P_{t_0} f$ at $x_0$, the assertion follows.
\end{proof}

\begin{theorem} \label{theo:2.6}
Suppose {\bf (a)} of Section \ref{subsec:2.2.1} holds and that $f \in \cup_{s \in [1, \infty]}L^s(\R^d, \mu)$, $f \geq 0$. Then $T_t f$, $t>0$ (see Definition \ref{definition2.1.7}) has a continuous $\mu$-version $P_t f$ on $\R^d$ and $P_{\cdot}f(\cdot)$ is locally parabolic H\"{o}lder continuous on $\R^d \times (0, \infty)$. Moreover, for any bounded open sets $U$, $V$ in $\R^d$ with $\overline{U} \subset V$ and $0<\tau_3<\tau_1<\tau_2<\tau_4$, we have the following estimate: 
\begin{equation} \label{eq:2.3.41}
\|P_{\cdot} f(\cdot)\|_{C^{\gamma; \frac{\gamma}{2}}(\overline{U} \times [\tau_1, \tau_2])} \leq  C_4 \| P_{\cdot} f(\cdot) \|_{L^1( V \times (\tau_3, \tau_4), \mu\otimes dt) },
\end{equation}
where $C_4>0$, $\gamma \in (0,1)$ are constants, independent of $f$. 
\end{theorem}
\begin{proof}
First assume $f \in D(\overline{L})_b \cap D(L_2) \cap D(L_q)$ with $f \geq 0$ and $q=\frac{pd}{p+d}$, $p\in (d,\infty)$. Set $u(x,t):=\rho(x) P_t f(x)$. Then by Lemma \ref{eq:2.3.39a}, $u \in C(\R^d \times [0, \infty))$. Let $B$ be an open ball in $\R^d$ and $T>0$. 
Using Theorem \ref{theorem2.1.5}, one can see $u \in H^{1,2}(B \times (0,T))$. Let $\phi \in C_0^{\infty}(\R^d)$, $\psi \in C^{\infty}_0((0,T))$ and $\varphi:=\phi \psi$. Then
$$
\frac{d}{dt} \int_{\R^d} \phi T_t f d\mu=
\int_{\R^d} \phi L_2 T_t f d\mu =\int_{\R^d} L_2' \phi \cdot T_t f d\mu,
$$
hence using integration by parts,
\begin{eqnarray} \label{eq:2.3.42}
0=-\int _0^T\int_{\R^d}  \left ( \partial _t \varphi +L'_2\varphi  \right ) u\,  dxdt. 
\end{eqnarray}
By $C^2$-approximation with finite linear combinations $\sum\phi_i \psi_i$, \eqref{eq:2.3.42} extends to all $\varphi \in C_0^{\infty}(\R^d \times (0,T))$. Applying integration by parts to \eqref{eq:2.3.42}, for all $\varphi \in C_0^{\infty}(\R^d \times (0,T))$ (see proof of Theorem \ref{theorem2.3.1}),
\begin{equation} \label{eq:2.3.45}
0=\int_0^T\int_{\R^d} \left ( \frac{1}{2}\langle (A+C) \nabla u, \nabla \varphi \rangle + u \langle \mathbf{H}-\frac{A \nabla \rho}{\rho}, \nabla \varphi \rangle -u\partial_t\varphi \right ) dxdt.
\end{equation}
Let $\bar{x} \in \R^d$ and $\bar{t} \in (0,T)$. Take a sufficiently small $r>0$ so that  $\bar{t}-(3r)^2>0$. Then by \cite[Theorems 3 and 4]{ArSe},
\begin{eqnarray*}
\|u\|_{C^{\gamma;\frac{\gamma}{2}}(\bar{R}_{\bar{x}}(r) \times [\bar{t}-r^2, \bar{t}])} &\leq& C_1\sup\big \{u(z)\, :\,z\in R_{\bar{x}}(3r) \times (\bar{t}-(3r)^2, \bar{t})\big \} \\
&\leq& C_1 C_2  \inf \big \{u(z)\, : \,z\in R_{\bar{x}}(3r) \times \big (\bar{t}+6(3r)^2, \bar{t}+7(3r)^2\big )\big \}\\[3pt]
&\leq & C_1 C_2 C_3 \|u\|_{L^1\big(R_{\bar{x}}(3r) \times  (\bar{t}+6(3r)^2, \bar{t}+7(3r)^2) \big)},
\end{eqnarray*}
where $\gamma \in (0, 1-d/p]$, $C_1, C_2, C_3>0$ are constants,  independent of $u$. Using a partition of unity and $\frac{1}{\rho} \in C^{0, \gamma}(\bar{R}_{\bar{x}}(3r))$, \eqref{eq:2.3.41} holds for all $f \in D(\overline{L})_b \cap D(L_2) \cap D(L_q)$ with $f \geq 0$. 
Moreover, using the $L^1(\R^d, \mu)$-contraction property of $(T_t)_{t>0}$, for all $f \in D(\overline{L})_b \cap D(L_2) \cap D(L_q)$ with $f \geq 0$, $q=\frac{pd}{p+d}$, it holds that
\begin{eqnarray}
&&\|P_{\cdot} f(\cdot)\|_{C^{\gamma; \frac{\gamma}{2}}(\overline{U} \times [\tau_1, \tau_2])} \leq  C_4 \| P_{\cdot} f(\cdot) \|_{L^1( V \times (\tau_3, \tau_4), \mu\otimes dt) }  \nonumber \\
&&\leq C_4 \int_{\tau_3}^{\tau_4} \| P_t f \|_{L^1(V, \mu)}  dt \leq C_4(\tau_4-\tau_3) \|f\|_{L^1(\R^d, \mu)}. \label{eq:2.56new}
\end{eqnarray}
Now let $f \in L^1(\R^d, \mu)_b$ with $f \geq 0$. Then $f_n:=nG_n f \in D(\overline{L})_b \cap D(L_2) \cap D(L_q)$, $n\geq 1$, $f_n \geq 0$ by the sub-Markovian property of $(G_{\alpha})_{\alpha>0}$ and $\lim_{n \rightarrow \infty} f_n=f$ in $L^1(\R^d, \mu)$ by the $L^1(\R^d, \mu)$-strong continuity of $(G_{\alpha})_{\alpha>0}$.\\
Using a Cauchy sequence argument together with \eqref{eq:2.56new}, there exists $u^f \in C^{\gamma;\frac{\gamma}{2}}( \overline{U} \times [\tau_1, \tau_2])$ such that 
\begin{equation} \label{eq:2.57a}
\lim_{n \rightarrow \infty} P_{\cdot} f_n(\cdot)=u^f \;\; \text{ in } \;C^{\gamma; \frac{\gamma}{2}}(\overline{U} \times [\tau_1, \tau_2]).
\end{equation}
Since $U \times [\tau_1, \tau_1]$ is arbitrarily chosen in $\R^d \times (0, \infty)$, given  $t>0$
we can define 
\begin{equation} \label{eq:2.57b}
P_t f:=u^f(\cdot, t), \quad \text{ on $\R^d$}.
\end{equation}
Then $P_t f$ is a continuous version of $T_t f$ by \eqref{eq:2.57a} and it follows from \eqref{eq:2.56new} that \eqref{eq:2.3.41} holds for all $f \in L^1(\R^d, \mu)_b$. Moreover, for $r \in [1, \infty)$, using  the $L^r(\R^d, \mu)$-contraction property of $(T_t)_{t>0}$ and H\"{o}lder's inequality, we get
\begin{eqnarray}
&&\|P_{\cdot} f(\cdot)\|_{C^{\gamma; \frac{\gamma}{2}}(\overline{U} \times [\tau_1, \tau_2])} \leq  C_4 \| P_{\cdot} f(\cdot) \|_{L^1( V \times (\tau_3, \tau_4), \mu\otimes dt) }   \label{eq:2.58a} \\
&&  \leq C_4 \int_{\tau_3}^{\tau_4} \| P_t f \|_{L^r(V, \mu)} \| \rho\|^{\frac{r-1}{r}}_{L^1(V)} dt \leq C_4(\tau_4-\tau_3)  \| \rho\|^{\frac{r-1}{r}}_{L^1(V)}  \|f\|_{L^r(\R^d, \mu)}. \qquad \;\; \text{} \label{eq:2.58new}
\end{eqnarray}
\\
Now let $f \in L^r(\R^d, \mu)$ with $f \geq 0$ and $r \in [1, \infty)$. Then there exists $(f_n)_{n \geq 1} \subset L^1(\R^d, \mu)_b \cap \mathcal{B}^+(\R^d)$ such that $\lim_{n \rightarrow \infty} f_n = f$ in $L^r(\R^d, \mu)$. By using a Cauchy sequence argument together with \eqref{eq:2.58new}, we can well-define $P_t f$ on $\R^d$ as we did in \eqref{eq:2.57a} and \eqref{eq:2.57b}, so that $P_t f$ is a continuous version of $T_t f$ and \eqref{eq:2.3.41} holds for all $f \in \cup_{r \in [1, \infty)}L^r(\R^d, \mu)$ with $f \geq 0$ by \eqref{eq:2.58a}.
\\
Finally, let $f \in L^{\infty}(\R^d, \mu)$ with $f \geq 0$ and $f_n:=1_{B_n} \cdot f \in L^1(\R^d, \mu)_b$ for $n \geq 1$. Then $\lim_{n \rightarrow \infty} f_n=f$, a.e. By the sub-Markovian property of $(T_t)_{t>0}$ and the continuity of $z \mapsto P_t f_n(z)$ on $\R^d$ for each $t>0$, $(P_t f_n(z))_{n \geq 1}$ is a uniformly bounded and increasing sequence in $[0,1]$ for each $t>0$ and $z \in \R^d$. Therefore, applying Lebesgue's theorem, $(P_{\cdot} f_n(\cdot))_{n \geq 1}$ is a Cauchy sequence in $L^1( V \times (\tau_3, \tau_4), \mu\otimes dt)$. By using a Cauchy sequence argument together with \eqref{eq:2.58a},  we can define $P_t f$ on $\R^d$ as we did in \eqref{eq:2.57a} and \eqref{eq:2.57b}. Then $P_t f$ is a continuous version of $T_t f$ and \eqref{eq:2.3.41} holds for all $f \in L^{\infty}(\R^d, \mu)$ with $f \geq 0$ as desired.
\end{proof}
\medskip
\noindent
For $f\in L^s(\R^d, \mu)$ with $s \in [1, \infty]$ and $t>0$, by splitting $f=f^+-f^-$, we define
\begin{equation} \label{semidef}
P_t f:= P_{t} f^+ - P_{t}f^- \quad \text{ on\, $\R^d$.}
\end{equation}
Then by Theorem \ref{theo:2.6}, $P_t f$ is a continuous version of $T_t f$ and for any bounded open subset $U$ of $\R^d$ and $0<\tau_1< \tau_2<\infty$, $P_{\cdot} f(\cdot) \in C^{\gamma; \frac{\gamma}{2}}(\overline{U} \times [\tau_1, \tau_2])$, where $\gamma \in (0,1)$ is a constant as in Theorem \ref{theo:2.6}. Moreover, applying the $L^s(\R^d, \mu)$-contraction property of $(T_t)_{t>0}$ for $s \in [1, \infty]$ and H\"{o}lder's inequality to \eqref{eq:2.3.41}, for any open subset $V$ of $\R^d$ with $\overline{U} \subset V$, $0<\tau_3<\tau_1<\tau_2<\tau_4<\infty$ and $t \in [\tau_1, \tau_2]$, it follows that
\begin{eqnarray}\label{eq:2.3.46}
\|P_{t} f \|_{C^{0,\gamma}(\overline{U})} & \leq & 2 C_4 (\tau_4-\tau_3) \|\rho \|_{L^1(V)}^{\frac{s-1}{s}} \cdot \|f\|_{L^s(\R^d, \mu) },
\end{eqnarray}
where $C_4>0$ is the constant of Theorem \ref{theo:2.6} and $\frac{s-1}{s}:=1$ if $s=\infty$ (cf. \eqref{eq:2.58new}).
The H\"older exponent $\gamma$ in \eqref{eq:2.3.46} may depend on the domains and may hence vary for different domains. But the important fact that we need for further considerations is that for a given domain, the constant $\gamma \in (0,1)$ and the constant in front of  $\|f\|_{L^s(\R^d, \mu)}$ in  \eqref{eq:2.3.46} are independent of $f$.
\\ \\
In a final remark, we discuss some previously derived and related regularity results. 
In order to fix some terminologies used there, we first give a definition.
\begin{definition}\label{def:2.3.1}
(i) Let $r\in [1,\infty]$. A family of positive linear operators $(S_t)_{t>0}$ defined on $L^r(\R^d,\mu)$ is said to be
{\bf $L^r(\R^d,\mu)$-strong Feller}\index{$L^r(\R^d,\mu)$-strong Feller}, if $S_t(L^{r}(\R^d,\mu)) \subset C(\R^d)$ for any $t>0$.\\[3pt]
(ii) A family of positive linear operators $(S_t)_{t>0}$ defined on $\mathcal{B}_b(\R^d)$ is said to be {\bf strong Feller}\index{strong Feller}, if $S_t (\mathcal{B}_b(\R^d)) \subset C_b(\R^d)$ for any $t>0$.
In particular, the $L^{\infty}(\R^d,\mu)$-strong Feller property implies the strong Feller property.\\[3pt]
(iii) A family of positive linear operators $(S_t)_{t\ge 0}$ defined on $C_{\infty}(\R^d)=\{f \in C_b(\R^d): \exists\lim_{\|x\| \rightarrow \infty} f(x)= 0  \}$ with $S_0=id$, where $C_{\infty}(\R^d)$ is equipped with the sup-norm $\|\cdot\|_{C_b(\R^d)}$, is called a {\bf Feller semigroup}\index{semigroup ! Feller}, if: 
\begin{itemize} 
\item[(a)] $\|S_t f\|_{C_b(\R^d)}\le \|f\|_{C_b(\R^d)}$ for any $t>0$,
\item[(b)]  $\lim_{t\to 0}S_t f=f$ in $C_{\infty}(\R^d)$ for any $f\in C_{\infty}(\R^d)$,
\item[(c)]  $S_t(C_{\infty}(\R^d)) \subset C_{\infty}(\R^d)$ for any $t>0$.
\end{itemize}
\end{definition}
\medskip
\noindent
If $(S_t)_{t\ge 0}$ is a Feller semigroup\index{semigroup ! Feller}, then by \cite[Chapter III. (2.2) Proposition]{RYor} and \cite[(9.4) Theorem]{BlGe} there exists a Hunt process (see Definition \ref{def:3.1.1}(ii)) whose  transition semigroup  is determined by $(S_t)_{t\ge 0}$.

\begin{remark}\label{rem:2.30new}
{\it In \cite{AKR}, \cite{BGS13}, and \cite{ShTr13a}, regularity properties of the resolvent and semigroup associated with a symmetric Dirichlet form are studied. For instance, if one considers a symmetric Dirichlet form defined as the closure of
\begin{equation} \label{eq:2.3.48}
\frac12  \int_{\R^d} \langle \nabla f, \nabla g \rangle d\mu, \quad f, g \in C_0^{\infty}(\R^d),
\end{equation}
then, provided $\rho$ has enough regularity, the drift coefficient of the associated generator has the form $\mathbf{G}=\nabla \phi$, where $\phi=\frac12 \ln \rho$. In \cite{AKR}, \cite{BGS13} using Sobolev regularity for elliptic equations involved with measures, $L^r(\R^d,\mu)$-strong Feller properties of the corresponding resolvent are shown, where $r \in (d, \infty]$.  In those cases, $L^s(\R^d,\mu)$-strong Feller properties of the associated semigroup, $s \in (d, \infty)$ immediately follow from the analyticity of symmetric semigroups. Conservativeness (see for instance \cite[Proposition 3.8]{AKR}) of the semigroup is assumed in order to derive the strong Feller property of the regularized semigroup $(P_t)_{t>0}$ (see Definition \ref{def:2.3.1}).
Similarly, in the sectorial case \cite{RoShTr}, analyticity and conservativeness of the semigroup are used to derive its $L^s(\R^d,\mu)$-strong Feller properties, $s \in (d, \infty]$ and in \cite[Section 3]{ShTr13a} the special properties of Muckenhoupt weights, which in particular imply conservativeness, lead to the strong Feller property of the semigroup using the joint continuity of the heat kernel and its pointwise upper bound.\\
We introduce three further references, where mainly analytical methods are used to construct a semigroup that has the strong Feller property. In  \cite{MPW02}, a sub-Markovian semigroup on $\mathcal{B}_b(\R^d)$ is constructed under the assumption that the diffusion and drift coefficients of the associated generator are locally H\"{o}lder continuous on $\R^d$ and the strong Feller property of the semigroup is derived in \cite[Corollary 4.7]{MPW02}  by interior Schauder estimates for parabolic PDEs of non-divergence type. Similarly, the strong Feller property is derived under the existence of an additional zero-order term in \cite[Proposition 2.2.12]{LB07}. In \cite[Theorem 1]{Ki18}, a sub-Markovian and analytic $C_0$-semigroup of contractions on $L^p(\R^d)$, where $p$ is in a certain open subinterval of $(d-1,\infty)$, $d \geq 3$, associated with the partial differential operator $\Delta+\langle  \sigma, \nabla \rangle$, where $\sigma$ is allowed to be in a certain nice class of measures, including absolutely continuous ones with drift components in $L^d(\R^d)+L^{\infty}(\R^d)$, is constructed and it is shown in \cite[Theorem 2]{Ki18} that the associated resolvent has the $L^p(\R^d)$-strong Feller property. Moreover in \cite[Theorem 2]{Ki18}, the semigroup is also shown to be Feller, so that the existence of an associated Hunt process follows (cf. Definition \ref{def:2.3.1}(iii)).
\\
In \cite[Section 2.3]{Scer}, some probabilistic techniques are used to show the strong Feller property of the semigroup, but the required conditions on the coefficients of the associated generator are quite restrictive. For instance, it is at least required that the diffusion coefficient is continuous and globally uniformly strictly elliptic and that the drift coefficient is locally Lipschitz continuous. We additionally refer to \cite{Bha}, where a possibly explosive diffusion process associated with $(L, C_0^{\infty}(\R^d))$ is constructed, where $A=(a_{ij})_{1 \leq i,j \leq d}$ satisfies \eqref{eq:2.1.2}, with $a_{ij} \in C(\R^d)$ for all $1 \leq i,j \leq d$ and $\mathbf{G} \in L^{\infty}_{loc}(\R^d, \R^d)$. In that case, the  strong Feller property is derived in \cite[Lemma 2.5]{Bha} under the assumption that the explosion time of the diffusion process is infinite (a.s.) for some initial condition $x_0 \in \R^d$.}
\end{remark}

\subsection{Irreducibility of solutions to the abstract Cauchy problem}
\label{sec2.4}
In order to investigate the ergodic behavior of the regularized semigroup $(P_t)_{t>0}$ in Section \ref{subsec:3.2.3}, the irreducibility in the probabilistic sense as defined in the following definition together with the strong Feller property
are important properties. Throughout this section, we let
$$
\mu=\rho\,dx
$$ 
be as in Theorem \ref{theo:2.2.7} or as in  Remark \ref{rem:2.2.4}. \\
\begin{definition} \label{def:2.4.4}
$(P_t)_{t>0}$ (see Theorem \ref{theo:2.6}) is said to be {\bf irreducible in the probabilistic sense}\index{irreducible ! in the probabilistic sense}, if for any $x \in \R^d$, $t>0$, $A \in \mathcal{B}(\R^d)$ with $\mu(A)>0$, we have  $P_t 1_A (x)>0$.
\end{definition}
\noindent
In this section, our main goal is to show the irreducibility in the probabilistic sense (Proposition \ref{prop:2.4.2}), which implies 
{\it irreducibility in the classical sense}, i.e. if for any $x \in \R^d$, $t>0$, $U \subset \R^d$ open, we have  $P_t 1_U(x)>0$.\\
To further explain the connections between different notions related to irreducibility in the literature 
and our work, let us introduce some notions related to generalized and symmetric Dirichlet form theory and 
in particular to our semigroup $(T_t)_{t>0}$. 
\begin{definition} 
\label{def:2.4.4bis}
$A\in \mathcal{B}(\R^d)$ is called a {\bf weakly invariant set} relative to $(T_t)_{t>0}$  (see Definition \ref{definition2.1.7}), if
$$
T_t (f \cdot 1_A ) (x)=0,\ \  \text{for } \mu \text{-a.e.}  \ \ x\in \R^d\setminus A,
$$
for any $t> 0$, $f\in L^2(\R^d,\mu)$. $(T_t)_{t>0}$ is said to be {\bf strictly irreducible}\index{irreducible ! strictly}, if for any 
weakly invariant set $A$ relative to $(T_t)_{t>0}$, we have $\mu(A)=0$ or $\mu(\R^d\setminus A)=0$.\\
\end{definition}
$A \in \mathcal{B}(\R^d)$ is called a {\it strongly invariant set} relative to $(T_t)_{t>0}$, if 
$$
T_t 1_A f = 1_A T_t f, \quad \text{$\mu$-a.e.}
$$
for any $t>0$ and $f \in L^2(\R^d, \mu)$. $(T_t)_{t>0}$ is said to be {\it irreducible}, if for any 
strongly invariant set $A$ relative to $(T_t)_{t>0}$, we have $\mu(A)=0$ or $\mu(\R^d\setminus A)=0$. One 
can check that $A \in \mathcal{B}(\R^d)$ is a strongly invariant set relative to $(T_t)_{t>0}$, if and only 
if $A$ and $\R^d\setminus A$ are weakly invariant sets relative to $(T_t)_{t>0}$. Therefore, if $(T_t)_{t>0}$ is 
strictly irreducible, then $(T_t)_{t>0}$ is irreducible. One can also check that $A \in \mathcal{B}(\R^d)$ 
is a weakly invariant set relative to $(T_t)_{t>0}$, if and only if $\R^d\setminus A$ is a weakly invariant set 
relative to $(T'_t)_{t>0}$. Hence, if $(T_t)_{t>0}$ is associated with a symmetric Dirichlet form, then the 
strict irreducibility of $(T_t)_{t>0}$ is equivalent to the irreducibility of $(T_t)_{t>0}$.
\begin{remark}\label{rem:2.4.1}
{\it In the symmetric case (see  \cite{FOT}), it is shown in \cite[Lemma 1.6.4]{FOT} that if $(T_t)_{t>0}$ is associated with a symmetric Dirichlet form and $(T_t)_{t>0}$ is irreducible, then $(T_t)_{t>0}$ is either recurrent or transient (see Definition \ref{def:3.2.2.2} below). Moreover, it is known from \cite[Exercise 4.6.3]{FOT}, that if $(T_t)_{t>0}$ is associated with a symmetric Dirichlet form and has the strong Feller property, then $(T_t)_{t>0}$ is irreducible. Since in our case the associated generator may be non-symmetric and non-sectorial, the above results dealing with symmetric Dirichlet form theory may not apply. Therefore, we use the stronger concept of strict irreducibility of $(T_t)_{t>0}$ covered in \cite{GT2}  and originally due to \cite{Ku11}. In \cite[Section 3.2.3]{GT2}, under the assumption that $\mu$ is a Muckenhoupt $\mathcal{A}_{\beta}$-weight, $\beta \in [1, 2]$, and that $(T_t)_{t>0}$ is associated to a symmetric Dirichlet form defined as the closure of \eqref{eq:2.3.48}, the pointwise lower bound of the associated heat kernel leads to the strict irreducibility of $(T_t)_{t>0}$. }
\end{remark}
\noindent
Here,  the strict irreducibility of $(T_t)_{t>0}$, merely follows under assumption {\bf (a)} of Section \ref{subsec:2.2.1}. Namely, we show the irreducibility in the probabilistic sense in Lemma \ref{lem:2.7}, which implies the strict irreducibility by Lemma \ref{prop:2.2}. As in the case of Section \ref{sec2.3}, for a sufficiently regular function $f$, $\rho T_{\cdot} f$ is a variational solution to a parabolic PDE of divergence type. We may hence apply the pointwise parabolic Harnack inequality of \cite[Theorem 5]{ArSe}, which is a main ingredient to derive our results.

\begin{lemma} \label{prop:2.2}
Suppose {\bf (a)} of Section \ref{subsec:2.2.1} holds. If $(P_t)_{t>0}$ is irreducible in the probabilistic sense, then $(T_t)_{t>0}$ is strictly irreducible.
\end{lemma}
\begin{proof}
Let $t_0>0$ and $A \in \mathcal{B}(\R^d)$ be a weakly invariant set relative to $(T_t)_{t>0}$. Let $f_n:=1_{B_n} \in L^2(\R^d, \mu)$. Then $T_{t_0}(f_n 1_A)(x)=0$ for $\mu$-a.e. $x \in \R^d \setminus A$, for all $n \in \N$. Since $f_n \nearrow 1_{\R^d}$, we have $T_{t_0} (f_n 1_A) \nearrow T_{t_0} 1_A$ $\mu$-a.e. Thus, $T_{t_0} 1_A(x)=0$ for $\mu$-a.e. $x \in \R^d \setminus A$, so that $P_{t_0} 1_A(x)=0$ for  $\mu$-a.e. $x \in \R^d \setminus A$.
\\
Now suppose that $\mu(A)>0$ and $\mu(\R^d \setminus A)>0$. Then there exists $x_0 \in \R^d \setminus A$ such that $P_{t_0}1_A(x_0)=0$, which is contradiction since $(P_t)_{t>0}$ is irreducible in the probabilistic sense. Therefore, we have $\mu(A)=0$ or $\mu(\R^d \setminus A)=0$, as desired.
\end{proof}

\begin{lemma}\label{lem:2.7} 
Suppose {\bf (a)} of Section \ref{subsec:2.2.1} holds. 
\begin{itemize}
\item [(i)] Let $A \in \mathcal{B}(\R^d)$ be such that $P_{t_0} 1_A (x_0)=0$ for some $t_0>0$ and $x_0\in \R^d$. Then $\mu(A)=0$.
\item [(ii)] Let $A \in \mathcal{B}(\R^d)$ be such that $P_{t_0} 1_A (x_0)=1$ for some $t_0>0$ and $x_0\in \R^d$.  Then $P_t 1_{A}(x)=1$ \;for all $(x,t) \in \R^d \times (0,\infty)$.
\end{itemize}
\end{lemma}

\begin{proof}
(i) Suppose that $\mu(A)>0$. Choose $r>0$ so that 
$$
0<\mu(A \cap B_r(x_0))<\infty.
$$
Let $u:=\rho P_{\cdot}1_{A \cap B_r(x_0)}$. Then $0 \leq u(x_0, t_0) \leq \rho(x_0) P_{t_0}1_A(x_0)=0$. Let $f_n:=nG_n 1_{A \cap B_r(x_0)}$ and $u_n:=\rho P_{\cdot} f_n$. Note that $f_n \in D(\overline{L})_b \cap D(L_2) \cap D(L_q)$ and $\lim_{n \rightarrow \infty} f_n  =1_{A \cap B_r(x_0)}$ in $L^1(\R^d, \mu)$. Let $U$ be a bounded open set in $\R^d$ and $\tau_1, \tau_2 \in (0, \infty)$ with $\tau_1<\tau_2$. By Theorem \ref{theo:2.6}, 
\begin{equation}\label{eq:2.4.46}
\lim_{n \rightarrow \infty} u_n  = u \;\; \text{ in } C(\overline{U} \times [\tau_1, \tau_2]).
\end{equation} 
 Fix $T>t_0$ and $U\supset \overline{B}_{r+1}(x_0)$. Then by \eqref{eq:2.3.45}, $u_n \in H^{1,2}(U \times (0,T))$ satisfies
for all $\varphi \in C_0^{\infty}(U\times (0, T))$
\begin{eqnarray*}
\int_0^T\int_{U} \left ( \frac{1}{2}\langle (A+C) \nabla u_n, \nabla \varphi \rangle + u_n \langle \mathbf{H}-\frac{A \nabla \rho}{\rho}, \nabla \varphi \rangle -u\partial_t\varphi \right ) dxdt=0.
\end{eqnarray*}
Take arbitrary but fixed $(x,t) \in B_{r}(x_0) \times (0, t_0)$.  By \cite[Theorem 5]{ArSe}
\begin{equation}\label{eq:2.4.47}
0 \leq u_n(x,t) \leq u_n(x_0, t_0)\;\exp \left ( C_1 \Big(\frac{\|x_0-x\|^2}{t_0-t}+ \frac{t_0-t}{\min(1,t)} +1 \Big)\right),
\end{equation}
where $C_1>0$ is a constant independent of $n \in \N$. Applying \eqref{eq:2.4.46} with $\overline{U} \times [\tau_1, \tau_2] \supset \overline{B}_{r+1}(x_0) \times [t,t_0]$ to \eqref{eq:2.4.47}, we have $u(x, t)=0$. Thus, $ P_t 1_{A\cap B_{r}(x_0)}(x)=0$ for all $(x,t) \in B_{r}(x_0) \times (0, t_0)$, so that by strong continuity inherited from $(T_t)_{t>0}$ (see Theorem \ref{theo:2.6} and Definition \ref{definition2.1.7})
$$
0=\int_{\R^d} 1_{A \cap B_{r}(x_0)} P_t 1_{A\cap B_{r}(x_0)} d\mu \underset{\text{as } t \rightarrow 0+}{ \; \longrightarrow} \mu(B_r(x_0) \cap A)>0,
$$
which is contradiction. Therefore, we must have $\mu(A)=0$.\\
(ii) Let $y \in \R^d$ and $0<s<t_0$ be arbitrary but fixed, $r:=2\|x_0-y\|$ and let $B_m$ be an open ball in $\R^d$ with $A \cap B_m \neq \emptyset$. Let $g_n:= nG_n 1_{A \cap B_m}$. Then $g_n \in D(\overline{L})_b \cap D(L_2) \cap D(L_q)$ and $\lim_{n \rightarrow \infty} g_n = 1_{A \cap B_m}$ in $L^1(\R^d, \mu)$. By Theorem \ref{theo:2.6}, 
\begin{equation}
\lim_{n \rightarrow \infty} P_{\cdot} g_n  = P_{\cdot} 1_{A \cap B_m} \;\; \text{ in } C(\overline{B}_r(x_0) \times [s/2, 2t_0]).
\end{equation}
Now fix $T>t_0$ and $U \supset \overline{B}_{r+1}(x_0)$. Using integration by parts and \eqref{eq:2.2.0a}, for all $\varphi \in C_0^{\infty}(U \times (0,T))$,
\begin{eqnarray}
&&\int_0^T\int_{U} \left ( \frac{1}{2}\langle (A+C) \nabla \rho, \nabla \varphi \rangle + \rho \langle \mathbf{H}-\frac{A \nabla \rho}{\rho}, \nabla \varphi \rangle -\rho\partial_t\varphi \right ) dxdt \nonumber \\
&&=-\int_0^T \int_{U} \langle \frac12 (A+C^T) \nabla \rho - \rho \mathbf{H}, \nabla \varphi \rangle dx dt = 0. \label{eq:2.4.49}
\end{eqnarray}
By \eqref{eq:2.3.45}, $\rho P_{\cdot} g_n \in H^{1,2}(U \times (0,T))$ satisfies for all $\varphi \in C_0^{\infty}(U\times (0, T))$
\begin{equation} \label{eq:2.4.50}
\int_0^T\int_{U} \left ( \frac{1}{2}\langle (A+C) \nabla (\rho P_{\cdot} g_n), \nabla \varphi \rangle + (\rho P_{\cdot} g_n) \langle \mathbf{H}-\frac{A \nabla \rho}{\rho}, \nabla \varphi \rangle -u\partial_t\varphi \right ) dxdt=0.
\end{equation}
Now let $u_n(x,t):=\rho(x) \left(1- P_t g_n (x) \right)$. Then $u_n \in H^{1,2}(U \times (0,T))$ and $u_n \geq 0$. Subtracting $\eqref{eq:2.4.50}$ from $\eqref{eq:2.4.49}$  implies
\begin{equation*} 
\int_0^T\int_{U} \left ( \frac{1}{2}\langle A \nabla u_n, \nabla \varphi \rangle + u_n \langle \mathbf{H}-\frac{A \nabla \rho}{\rho}, \nabla \varphi \rangle -u_n\partial_t\varphi \right ) dxdt=0.
\end{equation*}
Thus, by \cite[Theorem 5]{ArSe}
$$
0 \leq u_n(y,s) \leq u_n(x_0, t_0)\; \exp \left ( C_2 \Big(\frac{\|x_0-y\|^2}{t_0-s}+ \frac{t_0-s}{\min(1,s)} +1 \Big) \right),
$$
where $C_2>0$ is a constant independent of $n \in \N$. Letting $n \rightarrow \infty$ and $m \rightarrow \infty$, we obtain 
$P_{s}1_{A}(y)=1$. Since $(y, s) \in \R^d \times (0, t_0)$ was arbitrary, we obtain $P_{\cdot}1_A=1$ on $\R^d \times (0, t_0]$ by continuity. Then by the sub-Markovian property, $P_{t_0}1_{\R^d}(y)=1$ for any $y \in \R^d$. Now let $t \in (0, \infty)$ be given. Then there exists $k \in \N \cup \left \{0 \right \}$ such that
$$
k t_0<t \leq  (k+1) t_0
$$
and so $P_t 1_{A} = P_{kt_0+ (t-kt_0)} 1_{A}= \underbrace{P_{t_0} \circ \cdots \circ P_{t_0}}_{k\text{-times}} \circ P_{t-kt_0} 1_{A} =1$.
\end{proof}
\noindent
The following results are immediately derived by Lemma \ref{lem:2.7}(i) through contraposition and by Lemma \ref{prop:2.2}.
\begin{proposition}\label{prop:2.4.2}
Suppose {\bf (a)} of Section \ref{subsec:2.2.1} holds and let $(P_t)_{t>0}$ be as in Theorem \ref{theo:2.6}. Then:
\begin{itemize}
\item[(i)] $(P_t)_{t>0}$ is irreducible in the probabilistic sense (Definition \ref{def:2.4.4}).
\item[(ii)] $(T_t)_{t>0}$ is strictly irreducible (Definition \ref{def:2.4.4bis}).
\end{itemize}
\end{proposition}
We close this section with two remarks. The first is on a generalization of our results up to now to open sets and the second on related previous work.
\begin{remark}\label{rem:2.4.3}
{\it It is possible to generalize everything that has been achieved so far in Sections \ref{sec2.2}, \ref{sec2.3}, \ref{sec2.4} to general open sets $W\subset\R^d$. For this let $(W_n)_{n \geq 1}$ be a family of bounded and open sets in $\R^d$ with Lipschitz boundary $\partial  W_n$ for all $n \geq 1$, such that
$$
\overline{W}_n \subset W_{n+1}, \; \forall n \geq 1 \;\; \text{ and } \;\; W = \cup_{n \geq 1}W_n.
$$
Let $(p_n)_{n \geq 1}$ be a sequence in $\R$, such that $p_n \geq p_{n+1}>d$,\; $\forall n \geq 1$ and
$$
\lim_{n \rightarrow \infty} p_n =d,
$$
and assume that the coefficients $(a_{ij})_{1 \leq i,j \leq d}$, $(c_{ij})_{1 \leq i,j \leq d}$, and $(h_i)_{1 \leq i \leq d}$, satisfy for each $n\ge 1$:
\begin{itemize}
\item[] $a_{ji}= a_{ij}\in H^{1,2}(W_n) \cap C(W_n)$, $1 \leq i, j \leq d$ and $A = (a_{ij})_{1\le i,j\le d}$ satisfies \eqref{eq:2.1.2} on $W_n$, $C = (c_{ij})_{1\le i,j\le d}$, with $-c_{ji}=c_{ij} \in H^{1,2}(W_n) \cap C(W_n)$, $1 \leq i,j \leq d$, $\mathbf{H}=(h_1, \dots, h_d) \in L^{p_n}(W_n, \R^d)$.
\end{itemize}
Then taking into account Remark \ref{remark2.1.7}(iii) and adapting the methods of Sections \ref{sec2.2}, \ref{sec2.3}, \ref{sec2.4}, one can derive all results of Section \ref{sec2.2}, \ref{sec2.3}, \ref{sec2.4}, where $\R^d$ is replaced by $W$. 
}

\end{remark}

\begin{remark}\label{rem:2.3.3}
{\it We can mention at least two references \cite{Scer}, \cite{ZhXi16}, in which mainly probabilistic methods are employed  to derive irreducibility in the classical sense.
In \cite[Section 2.3]{Scer}, irreducibility in the classical sense is shown under the same assumptions as those which are used to show the strong Feller property. \\
In \cite{ZhXi16}, to obtain the strong Feller property and irreducibility in the classical sense of the semigroup associated with a diffusion process, restrictive conditions on the coefficients are imposed. The merit is that some time-inhomogeneous cases are covered 
in \cite{ZhXi16}, but the results are far from being optimal in the time-homogeneous case (see the discussion in the introduction of \cite{LT18}).\\
In \cite[Corollary 4.7]{MPW02}  the irreducibility of the semigroup in the classical sense is shown analytically by using the strict positivity of the associated heat kernel in \cite[Theorem 4.4]{MPW02} (see also \cite[Theorem 2.2.12 and Theorem 2.2.5]{LB07} for the case where there is an additional zero-order term).}
\end{remark}

\subsection{Comments and references to related literature}\label{Comments2}
Chapter \ref{chapter_2} is based on techniques from functional analysis and PDE theory that can be found in textbooks, for instance \cite{BRE}, \cite{EG15}, \cite{Ev11}. We further apply direct variational methods and  make use of standard results 
from semigroup, potential and operator theory.
In Section \ref{sec2.1}, the Lumer--Phillips theorem (\cite[Theorem 3.1]{LP61}) is used to derive that the closure of a dissipative operator generates a $C_0$-semigroup of contractions. In Section \ref{sec2.2},  the Lax--Milgram theorem (\cite[Corollary 5.8]{BRE}), the maximum principle (\cite[Theorem 4]{T77}), the Fredholm-alternative (\cite[Theorem 6.6(c)]{BRE}), and the elliptic Harnack inequality of \cite[Corollary 5.2]{T73} are mainly used to show existence of an infinitesimally invariant measure for $(L, C_0^{\infty}(\R^d))$. \\
Concerning more recent sources, beyond the classical ones, for the $H^{1,p}_{loc}$-regularity of the density of the infinitesimally invariant measure, \cite[Theorem 1.2.1]{BKRS} and \cite[Theorem 2.8]{Kr07} are used. 
In Section \ref{sec2.3}, the elliptic and parabolic H\"{o}lder regularity results (\cite[Th\'eor\`eme 7.2, 8.2]{St65}), \cite[Theorems 3 and 4]{ArSe}), are used to obtain regularized versions of the resolvent and the semigroup, respectively.
In Section \ref{sec2.4}, the irreducibility of the semigroup is derived by the pointwise parabolic Harnack inequality (\cite[Theorem 5]{ArSe}).\\
The content of Section \ref{sec2.1} is taken from \cite[Part I, Sections 1 and 2]{WS99}. Detailed explanations on the construction of the Markovian semigroup have been added, as well as the new example Remark \ref{rem:2.1.12}(ii). Sections \ref{sec2.2}--\ref{sec2.4} (and Chapter \ref{chapter_3}) originate roughly from \cite{LT18} and \cite{LT19}, but we recombined, reorganized, refined and further developed the results of \cite{LT18} and \cite{LT19}. In particular, the contents of Section \ref{sec2.2} are a refinement of \cite[Theorem 3.6]{LT19}. Some proofs on elliptic regularity (\cite[Lemma 3.3, 3.4]{LT19}) are omitted in this book and the interested reader may check the original source for the technical details.

%
%
%
\newpage
\section{Stochastic differential equations}
\label{chapter_3} 
 
\subsection{Existence}
\label{sec:3.1}
In Section \ref{sec:3.1} we show that the regularized semigroup $(P_t)_{t>0}$ from Theorem \ref{theo:2.6} and \eqref{semidef} determines the transition semigroup of a Hunt process  $(X_t)_{t\geq 0}$ with nice sample paths and that $(R_{\alpha})_{\alpha >0}$ determines its resolvent. For the construction of the Hunt process  $(X_t)_{t\geq 0}$, crucially the existence of a Hunt process $(\tilde{X}_t)_{t \ge 0}$ deduced from generalized Dirichlet form theory for a.e. starting point (Proposition \ref{prop:3.1.3}) is needed, and additionally to  assumption {\bf (a)} of Section \ref{subsec:2.2.1}, assumption {\bf (b)} of Section \ref{subsec:3.1.1}, which provides a higher resolvent regularity. Since $(\tilde{X}_t)_{t \ge 0}$ has continuous sample paths on the one-point-compactification $\R^d_{\Delta}$, the same is then true for $(X_t)_{t\geq 0}$. From Remark \ref{rem:3.1.1}  of Section \ref{subsec:3.1.1} on we assume assumptions {\bf (a)} and {\bf (b)} to hold, if not stated otherwise.
As a by-product of the existence of $(X_t)_{t\geq 0}$ and the resolvent regularity derived in Theorem \ref{theorem2.3.1} by PDE theory, we obtain Krylov-type estimates (see Remark \ref{rem:ApplicationKrylovEstimates}). The identification of $(X_t)_{t\geq 0}$ as a weak solution (cf. Definition \ref{def:3.48} (iv))  to \eqref{intro:eq1} then follows standard lines by representing continuous local martingales as stochastic integrals with respect to Brownian motion through the knowledge of their quadratic variations. \\ 
\subsubsection{Regular solutions to the abstract Cauchy problem as transition functions}\label{subsec: 3.1.1first}
Throughout this section we will assume that {\bf (a)} of Section \ref{subsec:2.2.1} holds, and that
$$
\mu=\rho\,dx
$$ 
is as in Theorem \ref{theo:2.2.7} or as in  Remark \ref{rem:2.2.4}. \\

\begin{proposition} 
\label{prop:3.1.1}
Assume {\bf (a)} of Section \ref{subsec:2.2.1} holds.
Let $(P_t)_{t>0}$ be as in Theorem \ref{theo:2.6}  and \eqref{semidef}. Let $(x,t) \in \R^d \times (0, \infty)$. Then:
\begin{itemize}
\item[(i)] $P_{t}(x, \cdot)$ defined through
$$
P_t(x,A):= P_t 1_{A}(x), \qquad A \in \mathcal{B}(\R^d)
$$
is a sub-probability measure on $\mathcal{B}(\R^d)$, i.e. $P_t(x,\R^d)\le 1$, and equivalent to $\mu$. 
\item[(ii)] \ We have
\begin{eqnarray} 
\label{eq:3.1.1}
P_t f (x)= \int_{\R^d} f(y) P_t(x,dy), \qquad \forall f\in \bigcup_{s\in [1,\infty]}L^s(\R^d,\mu).
\end{eqnarray}
In particular, \eqref{eq:3.1.1} extends by linearity to all $f\in L^1(\R^d,\mu)+L^\infty(\R^d,\mu)$, and for such $f$, $P_t f$ is continuous by Theorem \ref{theo:2.6} and \eqref{semidef}.
\end{itemize}
\end{proposition}
\begin{proof}
(i) That $P_{t}(x, \cdot)$ defines a measure is obvious by the properties of $(T_t)_{t>0}$ on $L^\infty(\R^d,\mu)\supset\mathcal{B}_b(\R^d)$ and since $P_t 1_{A}$ is a continuous version of $T_t 1_A$. In  particular $P_{t}(x, \cdot)$ defines a sub-probability measure since by the sub-Markov property $T_t 1_{\R^d}\le 1$ $\mu$-a.e. hence by continuity $P_{t}(x, \R^d)=P_t 1_{\R^d}(x)\le 1$ for every $x\in \R^d$. If $N\in \mathcal{B}(\R^d)$ is such that $\mu(N)=0$, then clearly  $P_{t}(x,N)=P_t 1_N(x)=0$ and if  $P_{t}(x,N)=P_t 1_N(x)=0$ then $\mu(N)=0$ by Lemma \ref{lem:2.7}(i).\\
(ii) For any $(x,t) \in  \R^d \times (0,\infty)$, we have
$$
P_t f (x)= \int_{\R^d} f(y) P_t(x,dy)
$$
for $f=1_A$, $A \in \mathcal{B}(\R^d)$ which extends to any $f\in \bigcup_{s\in [1,\infty]}L^s(\R^d,\mu)$ in view of \eqref{eq:2.3.46}.
\end{proof}
\begin{proposition}\label{prop:3.1.2}
Assume {\bf (a)} of Section \ref{subsec:2.2.1} holds.
Let $(R_{\alpha})_{\alpha>0}$ be as in Theorem \ref{theorem2.3.1} and \eqref{resoldef}. Let $(x, \alpha) \in \R^d \times (0, \infty)$. Then:
\begin{itemize}
\item[(i)]
$\alpha R_{\alpha}(x, \cdot)$, where 
$$
R_{\alpha}(x, A):= R_{\alpha} 1_{A}(x), \qquad A \in \mathcal{B}(\R^d)
$$
is a sub-probability measure on $\mathcal{B}(\R^d)$, absolutely continuous with respect to $\mu$.
\item[(ii)]\ We have
\begin{eqnarray} \label{eq:3.1.2} 
R_{\alpha}g(x)=\int_{\R^d}g(y) R_{\alpha}(x,dy), \qquad  \forall g\in \bigcup_{r\in [q,\infty]} L^r(\R^d,\mu),
\end{eqnarray}
where  $q=\frac{pd}{p+d}$, $p\in (d,\infty)$. In particular, \eqref{eq:3.1.2} extends by linearity to all $g\in L^q(\R^d,\mu)+L^\infty(\R^d,\mu)$, and for such $g$, $R_{\alpha}g$ is continuous by Theorem \ref{theorem2.3.1} and \eqref{resoldef}.
\end{itemize}
\end{proposition}
\begin{proof}
In view of \eqref{eq:2.3.38a} the proof is similar to the corresponding proof for Proposition \ref{prop:3.1.1} and we therefore omit it.
\end{proof}
Define
$$
P_0:=id.
$$
\begin{theorem}\label{th:3.1.1}
Assume {\bf (a)} of Section \ref{subsec:2.2.1} holds.
For $(x, \alpha) \in \R^d \times (0, \infty)$, it holds that
$$
R_{\alpha}g(x)=\int_{0}^{\infty} e^{-\alpha t}P_t g(x)dt, \qquad  g\in \bigcup_{r\in [q,\infty]} L^r(\R^d,\mu),
$$
where $q=\frac{pd}{p+d}$ and $p\in (d,\infty)$.
\end{theorem}
\begin{proof} Let first $g\in C_0^2(\R^d)$ and let $x_n\to x\in \R^d$ as $n \rightarrow \infty$. Then by Theorem \ref{theo:2.6} (see also \eqref{semidef}), $P_tg(x_n), P_tg(x)$
are continuous functions in $t\in (0,\infty)$ and $P_tg(x_n)\to P_tg(x)$ as $n \rightarrow \infty$ for any $t\in (0,\infty)$. Since further $\sup_{n\in\N}|P_tg(x_n)|\le \sup_{y\in \R^d}|g(y)|<\infty$ for any $t\in (0,\infty)$, Lebesgue's theorem implies that $\int_{0}^{\infty} e^{-\alpha t}P_t g dt$ is a continuous function on $\R^d$.  By Theorem \ref{theorem2.3.1}, $R_{\alpha}g$ is continuous. Since $(G_{\alpha})_{\alpha>0}$ is the Laplace transform of $(T_t)_{t>0}$ on $L^2(\R^d,\mu)$, the two continuous functions $R_{\alpha}g$ and $\int_{0}^{\infty} e^{-\alpha t}P_t g dt$ coincide $\mu$-a.e. hence everywhere on $\R^d$. 
Therefore, it holds that
$$
\int_{\R^d}g(y) R_{\alpha}(x,dy)=\int_0^{\infty}\int_{\R^d} g(y)P_t(x,dy)e^{-\alpha t}dt, \qquad \forall x\in \R^d,
$$
for any $g\in C_0^2(\R^d)$. Since the $\sigma$-algebra generated by  $C_0^2(\R^d)$ equals $\mathcal{B}(\R^d)$,
by a monotone class argument the latter extends to all $g\in \mathcal{B}_b(\R^d)$. Finally, splitting $g=g^+-g^-$ in positive and negative parts, using linearity and monotone approximation through $\mathcal{B}_b(\R^d)$ functions, the assertion follows for $g\in \bigcup_{r\in [q,\infty]} L^r(\R^d,\mu)$.
\end{proof}

\begin{remark}\label{rem:3.1equivalence ralpha}
{\it As a direct consequence of Theorem \ref{th:3.1.1}, the sub-probability measures $\alpha R_{\alpha}(x, dy)$ on $\mathcal{B}(\R^d)$ are equivalent to $\mu$ for all $\alpha>0$ and $x \in \R^d$. Indeed, by Proposition \ref{prop:3.1.2}(i), $\alpha R_{\alpha}(x, dy) \ll \mu$ for all $x \in \R^d$ for $\alpha>0$. For the converse, let $\alpha>0$, $x \in \R^d$ be given and assume that $A \in \mathcal{B}(\R^d)$ satisfies $\alpha R_{\alpha}(x, A)=0$. Then by Theorem \ref{th:3.1.1}, $P_t 1_A(x)=0$ for $dt$-a.e. $t \in (0, \infty)$, hence $\mu(A)=0$ by Lemma \ref{lem:2.7}(i), as desired.}
\end{remark}


With the definition $P_0=id$ and Proposition \ref{prop:3.1.1} from above, $(P_t)_{t \ge 0}$ determines a {\bf (temporally homogeneous) sub-Markovian transition function} on $(\mathbb{R}^d,\mathcal{B}(\mathbb{R}^d))$, i.e.:
\begin{itemize}
\item for all $x\in \mathbb{R}^d$, $t\ge 0$:\quad $A\in  \mathcal{B}(\mathbb{R}^d)\mapsto P_t(x,A)$ is a sub-probability measure;
\item for all $t\ge 0$, $A\in  \mathcal{B}(\mathbb{R}^d)$:\quad  $x\in \mathbb{R}^d \mapsto P_t(x,A)$ is $\mathcal{B}(\mathbb{R}^d)$-measurable;
\item for all $x\in \mathbb{R}^d$, $A\in  \mathcal{B}(\mathbb{R}^d)$, the Chapman--Kolmogorov equation
$$
P_{t+s}(x,A)=\int_{\R^d}P_s(y,A)P_t(x,dy), \qquad \forall t,s\ge 0
$$
holds.
\end{itemize}
Here the Chapman-Kolmogorov equation can be rewritten as $P_{t+s}1_A=P_{t}P_{s}1_A$ and therefore holds, since both sides are equal $\mu$-a.e. as $T_{t+s}1_A=T_{t}T_{s}1_A$, and $P_{t+s}1_A$ as well as $P_{t}P_{s}1_A$ are continuous functions by Theorem \ref{theo:2.6}, if either $t\not=0$ or $s\not=0$, and by the definition $P_0=id$, if $t=s=0$.\\
Since $(P_t)_{t \ge 0}$ only defines a  sub-Markovian transition function we will extend it to make it Markovian (i.e. conservative). For this let $\mathbb{R}^d_{\Delta}:=\mathbb{R}^d\cup \{\Delta\}$ be the one-point-compactification of $\R^d$ 
with the point at infinity \lq\lq$\Delta$\rq\rq, and
$$
\mathcal{B}(\mathbb{R}^d_{\Delta}):=\{A\subset \mathbb{R}^d_{\Delta} : A\in \mathcal{B}(\mathbb{R}^d) \text{ or } A=A_0\cup\{\Delta\}, \ A_0\in \mathcal{B}(\mathbb{R}^d)\}.
$$
Any function $f$ originally defined on $\mathbb{R}^d$ is extended to $\mathbb{R}^d_{\Delta}$ by setting $f(\Delta)=0$. Likewise any measure $\nu$ originally defined on $\mathcal{B}(\mathbb{R}^d)$  is extended to $\mathcal{B}(\mathbb{R}^d_{\Delta})$ by setting $\nu(\{\Delta\})=0$. For instance, $P_t(x, \{\Delta \})=0$ for all $x \in \R^d$. Now for $t\ge 0$,
$$
P_t^{\Delta}(x, dy) =
\begin{cases}
\big[1- P_t(x,\mathbb{R}^d)\big] \delta_{\Delta} (dy) + P_t(x, dy), \quad \text{if} \ x \in \mathbb{R}^d\\
\delta_{\Delta} (dy), \quad \text{if} \ x = \Delta
\end{cases}
$$
determines a (temporally homogeneous) Markovian transition function $(P_t^{\Delta})_{t\ge 0}$ on $(\mathbb{R}^d_{\Delta},\mathcal{B}(\mathbb{R}^d_{\Delta}))$. 

\subsubsection{Construction of a Hunt process}
\label{subsec:3.1.1} 
Throughout this section we will assume that {\bf (a)} of Section \ref{subsec:2.2.1} holds (except for Proposition \ref{prop:3.1.3}). Furthermore, we shall {\bf assume} that
\begin{itemize}
\item[{\bf (b)}] \ \index{assumption ! {\bf (b)}}$\mathbf{G}=(g_1,\ldots,g_d)=\frac{1}{2}\nabla \big (A+C^{T}\big )+ \mathbf{H} \in L_{loc}^q(\R^d, \R^d)$ (cf. \eqref{equation G with H} and \eqref{form of G}), where $q=\frac{pd}{p+d}$ and $p\in (d,\infty)$,
\end{itemize}
holds. Assumption {\bf (b)} will be needed from Remark \ref{rem:3.1.1} below on and implies $C^2_0(\R^d)\subset D(L_q)$, which is crucial for the construction of the Hunt process in Theorem \ref{th: 3.1.2} below. \\
By the results of Section \ref{sec:3.1}, $(P_t^{\Delta})_{t\ge 0}$ is a (temporally homogeneous) Markovian transition function on $(\mathbb{R}^d_{\Delta},\mathcal{B}(\mathbb{R}^d_{\Delta}))$. Restricting $(P_t^{\Delta})_{t\ge 0}$ to the positive dyadic rationals $S:= \bigcup_{n \in \mathbb{N}} S_n$, $S_n : = \{ k2^{-n}  : k \in \mathbb{N} \cup \{0\}\}$, we can hence construct a Markov process 
$$
\mathbb{M}^{0} = (\Omega , \mathcal{F}^0, (\mathcal{F}^0_s)_{s \in S}, (X_s^0)_{s \in S} , (\mathbb{P}_x)_{x \in \mathbb{R}^d_{\Delta}})
$$
by Kolmogorov's  method (see \cite[Chapter III]{RYor}). Here 
$
\Omega : = (\mathbb{R}^d_{\Delta})^S
$ 
is equipped with the product $\sigma$-field $\mathcal{F}^0$, $X_s^0 : (\mathbb{R}^d_{\Delta})^S \to \mathbb{R}^d_{\Delta}$ are coordinate maps and $\mathcal{F}_s^0 : = \sigma( X_r^0 \ | \ r \in S, r \le s)$. \\

\begin{definition}\label{def:3.1.1}
(i) $\widetilde{\M}= (\widetilde{\Omega},\widetilde{\mathcal{F}},(\widetilde{X}_t)_{t\ge
0},(\widetilde{\P}_x)_{x\in \mathbb{R}^d_\Delta})$ is called a {\bf strong Markov process} (resp. a {\bf right process})\index{process ! strong Markov}\index{process ! right}
with state space $\mathbb{R}^d$, lifetime $\widetilde{\zeta}$, and corresponding
filtration $(\widetilde{\mathcal{F}}_t)_{t\ge 0}$, if (M.1)--(M.6) (resp. (M.1)--(M.7)) below are fulfilled:

\begin{itemize}

\item[(M.1)] $\widetilde{X}_t : \widetilde{\Omega} \to \mathbb{R}^d_\Delta$ is $\widetilde{\mathcal{F}}_t/{\cal B}(\R^d_\Delta)$-
measurable for all $t \ge 0$, and $\widetilde{X}_t(\omega) = \Delta
\Leftrightarrow t \ge \widetilde{\zeta}(\omega)$ for all $\omega \in \widetilde{\Omega}$, where 
$(\widetilde{\mathcal{F}}_t)_{t\ge 0}$ is a filtration on $(\widetilde{\Omega}, \widetilde{\mathcal{F}} )$ and 
$\widetilde{\zeta} : \widetilde{\Omega} \to [ 0,\infty ]$.

\item[(M.2)]  For all $t \ge 0$ there exists a map $\vartheta_t : \widetilde{\Omega}
\to \widetilde{\Omega}$ such that $\widetilde{X}_s \circ \vartheta_t = \widetilde{X}_{s+t}$ for all $s \ge
0$.

\item[(M.3)] $(\widetilde{\P}_x)_{x\in \mathbb{R}^d_\Delta}$ is a family of probability
measures on $(\widetilde{\Omega},\widetilde{\mathcal{F}})$, such that $x \mapsto \widetilde{\P}_x (B)$ is ${\cal
 B}(\R^d_\Delta)^*$--measurable (here  ${\cal
 B}(\R^d_\Delta)^*$  denotes the universially measurable sets) for all $B \in \widetilde{\mathcal{F}}$ and 
${\cal B}(\mathbb{R}^d_\Delta)$-measurable for all $B \in \sigma (\widetilde{X}_t|t \ge 0)$ and
$\widetilde{\P}_\Delta(\widetilde{X}_0=\Delta) = 1$.

\item[(M.4)]   (Markov property) For all $A \in {\cal B}(\mathbb{R}^d_\Delta), s, t \ge 0$, and $x \in
\mathbb{R}^d_\Delta$
$$
   \widetilde{\P}_x(\widetilde{X}_{t+s} \in A|\widetilde{\mathcal{F}}_t) = \widetilde{\P}_{\widetilde{X}_t}(\widetilde{X}_s \in A)\, , \quad
     \widetilde{\P}_x\mbox{-a.s.}
$$
\item[(M.5)]  (Normal property) $\widetilde{\P}_x(\widetilde{X}_0 = x) = 1$ for all $x \in \mathbb{R}^d_\Delta$.

\item[(M.6)]   (Strong Markov property) $(\widetilde{\mathcal{F}}_t)_{t\ge0}$ is right continuous (see \eqref{defrightcon})
 and for any $\nu \in {\cal
P}(\mathbb{R}^d_\Delta):=\{\nu:\nu \text{ is a probability measure on }\R^d_{\Delta}\}$ and $(\widetilde{\mathcal{F}}_t)_{t\ge0}$--stopping time $\tau$
$$
  \widetilde{\P}_\nu(\widetilde{X}_{\tau+s} \in A| \widetilde{\mathcal{F}}_\tau) = \widetilde{\P}_{\widetilde{X}_\tau}(\widetilde{X}_s \in A)\, , \quad \widetilde{\P}_\nu\mbox{-a.s.}
$$
for all $A \in {\cal B}(\mathbb{R}^d_\Delta)$, $s \ge 0$, where for a positive measure $\nu$ 
on $(\mathbb{R}^d_\Delta,{\cal B}(\mathbb{R}^d_\Delta))$ we set $\widetilde{\P}_\nu(\cdot) := \int_{\R^d} \widetilde{\P}_x(\cdot) \,\nu(dx)$.

\item[(M.7)]   (Right continuity) $t \mapsto \widetilde{X}_t(\omega)$ is right continuous on
$[0,\infty)$ for all $\omega \in \widetilde{\Omega}$.
\end{itemize}
\noindent
(ii) A right process $\widetilde{\M}$ is said to be a {\bf Hunt process}\index{process ! Hunt}, if additionally to (M.1)--(M.7), (M.8)--(M.9) below are fulfilled:
\begin{itemize}
\item[(M.8)]  (left limits on $[0,\infty)$) $\widetilde{X}_{t-} := \lim_{{s\uparrow t}\atop{s<t}} \widetilde{X}_s$ exists in
  $\mathbb{R}^d_\Delta$ for all $t\in(0,\infty)\ \widetilde{\P}_\nu$-a.s. for all $\nu \in \mathcal{P}(\R^d_\Delta)$.

\item[(M.9)]   (quasi-left continuity on $[0,\infty)$) for all $\nu \in \mathcal{P}(\R^d_\Delta)$, we have 
  $\lim_{n\to\infty} \widetilde{X}_{\tau_n} = \widetilde{X}_\tau \ \widetilde{\P}_\nu$--a.s. on
  $\{\tau < \infty\}$  for every increasing sequence
  $(\tau_n)_{n\ge1}$ of $(\widetilde{\mathcal{F}}^{\P_\nu}_t)_{t\ge0}$-stopping times
  with limit $\tau$,  where for a
sub-$\sigma$-algebra $\mathcal{G} \subset \widetilde{\mathcal{F}}$ we let $\mathcal{G}^{\widetilde{\P}_\nu}$ be
its $\widetilde{\P}_\nu$-completion in $\widetilde{\mathcal{F}}$. 
\end{itemize}
A strong Markov process $\widetilde{\M}$ is said to have {\bf continuous sample paths on the one-point-compactification 
$\R^d_{\Delta}$ of $\R^d$}, if 
\begin{itemize}
\item[(M.10)] $\widetilde{\P}_x(t\mapsto \widetilde{X}_t \text{ is continuous in } t\in [0,\infty) \text{ on } \R^d_{\Delta})=1 \quad \text{for any } x\in \R^d_{\Delta}.$
\end{itemize}
Here the continuity is of course w.r.t. the topology of $\R^d_{\Delta}$.
In particular, if $(M.1)-(M.6)$ and $(M.10)$ hold, then $\widetilde{\M}$ has automatically left limits on $[0,\infty)$ and is quasi-left continuous on $[0,\infty)$, and therefore $\widetilde{\M}$ is a {\bf Hunt process (with continuous sample paths on the one-point-compactification 
$\R^d_{\Delta}$ of $\R^d$)}.\\
\end{definition}
\noindent
In what follows, we will need the following result, deduced from generalized Dirichlet form theory.
\begin{proposition}\label{prop:3.1.3}
Assume \eqref{condition on mu}--\eqref{eq:2.1.4} hold (which is the case if for instance condition {\bf (a)} of Section \ref{subsec:2.2.1} holds, see Theorem \ref{theo:2.2.7} and also Remark \ref{rem:2.2.4}). Then,
there exists a Hunt process 
$$
\tilde{\mathbb{M}} = (\tilde{\Omega}, \tilde{\mathcal{F}}, (\tilde{\mathcal{F}})_{t \ge 0}, (\tilde{X}_t)_{t \ge 0}, (\tilde{\mathbb{P}}_x)_{x \in \R^d \cup \{ \Delta \} })
$$ 
with state space $\R^d$, lifetime $\tilde\zeta:=\inf\{t\ge 0\,:\,\tilde{X}_t=\Delta\}$ and cemetery $\Delta$ such that 
for any $f\in  L^2(\R^d,\mu)_b$ and $\alpha>0$
$$
\tilde{\mathbb{E}}_{x}\Big [\int_0^{\infty}e^{-\alpha t}f(\tilde{X}_t)dt\Big ] =G_{\alpha}f(x)\qquad \text{for } \mu\text{-a.e. } x\in \R^d
$$
where $\tilde{\mathbb{E}}_{x}$ denotes the expectation with respect to $\tilde{\mathbb{P}}_x$ and 
$G_{\alpha}$ is as in Definition \ref{definition2.1.7}.
Moreover, $\tilde{\mathbb{M}}$ has continuous sample paths on the one-point-compactification $\R^d_{\Delta}$ of $\R^d$, i.e. we may assume that 
\begin{equation}\label{contipath}
\tilde{\Omega} = \{\omega = (\omega (t))_{t \ge 0} \in C([0,\infty),\R^d_{\Delta}) \, : \, \omega(t) = \Delta \quad \forall t \ge \tilde{\zeta}(\omega) \}
\end{equation}
and
$$
\tilde{X}_t(\omega) = \omega(t), \quad t \ge 0.
$$
\end{proposition}

\begin{proof}
Using in  particular Lemma \ref{lemma2.1.4} it is shown in \cite[proof of Theorem 3.5]{WS99} that the generalized Dirichlet form $\mathcal{E}$ associated with $(L_2,D(L_2))$ (cf. \cite[I.4.9(ii)]{WSGDF}) is quasi-regular and by \cite[IV. Proposition 2.1]{WSGDF} and Lemma \ref{lemma2.1.4} satisfies the structural condition D3 of \cite[p. 78]{WSGDF}. Thus by the theory of generalized Dirichlet forms \cite[IV. Theorem 2.2]{WSGDF}, there exists a standard process $\tilde{\tilde{\mathbb{M}}}$ properly associated with $\mathcal{E}$. Using in a crucial way the existence of $\tilde{\tilde{\mathbb{M}}}$ and Lemma \ref{lemma2.1.4} 
it is shown in \cite[Theorem 6]{Tr5} that the generalized Dirichlet form $\mathcal{E}$ is strictly quasi-regular and satisfies the structural condition SD3. Thus the existence of the Hunt process $\tilde{\mathbb{M}}$ follows by generalized Dirichlet form theory from \cite{Tr5}.\\
In order to show that $\tilde{\mathbb{M}}$ can be assumed to have continuous sample paths on the 
one-point-compactification $\R^d_{\Delta}$ of $\R^d$, it is enough to show that this holds for strictly $\mathcal{E}$-quasi-every starting point $x\in \R^d$. Indeed the complement of those points can be assumed to be a trap for $\tilde{\mathbb{M}}$.
Due to the properties of smooth measures with respect to $\text{cap}_{1,{\widehat{G}}_1\varphi}$ in \cite[Section 3]{Tr5} one can consider the work \cite{Tr2} with cap$_{\varphi}$ (as defined in \cite{Tr2}) replaced by $\text{cap}_{1,{\widehat{G}}_1\varphi}$. In particular \cite[Lemma 3.2, Theorem 3.10 and Proposition 4.2]{Tr2} apply with respect to the strict capacity $\text{cap}_{1,{\widehat{G}}_1\varphi}$. More precisely, in order to show that $\tilde{\mathbb{M}}$ has continuous sample paths on the one-point-compactification $\R^d_{\Delta}$ of $\R^d$ for strictly $\mathcal{E}$-quasi-every starting point $x\in \R^d$ one has to adapt three main arguments from \cite{Tr2}. The first one is related to the no killing inside condition \cite[Theorem 3.10]{Tr2}. In fact \cite[Theorem 3.10]{Tr2} which holds for $\mathcal{E}$-quasi-every starting point $x\in \R^d$ under the existence of an associated standard process and standard co-process and the quasi-regularity of the original and the co-form, holds with exactly the same proof for strictly $\mathcal{E}$-quasi-every starting point $x\in \R^d$ if we assume the existence of an associated Hunt process and associated Hunt co-process and the strict quasi-regularity for the original and the co-form. \cite[Lemma 3.2]{Tr2}, which holds under the quasi-regularity of $\mathcal{E}$  and the existence of an associated standard process for $\mathcal{E}$-quasi-every starting point $x\in \R^d$ and all $\nu \in \hat S_{00}$, holds in exactly the same way under the strict quasi-regularity of $\mathcal{E}$  and the existence of an associated Hunt process for strictly $\mathcal{E}$-quasi-every starting point $x\in \R^d$ and all $\nu \in \hat S_{00}^{str}$ (as defined in \cite[Section 3]{Tr5}). Finally, \cite[Proposition 4.2]{Tr2} also holds in exactly the same way for the Hunt process and its lifetime and the Hunt co-process and its lifetime for strictly $\mathcal{E}$-quasi-every starting point $x\in \R^d$. In particular all the mentioned statements then hold for $\mu$-a.e. starting point $x\in \R^d$.
\end{proof}
Let $\nu : = gd\mu$, where $g \in L^1(\R^d,\mu)$, $g>0$ $\mu$-a.e. and $\int_{\R^d} g \,d\mu =1$. For instance, we can choose $g(x):=\frac{e^{-\|x\|^2}}{( \pi)^{d/2} \rho(x)}$, $x\in \R^d$. Set
$$
\tilde{\mathbb{P}}_{\nu}(\cdot) := \int_{\R^d} \tilde{\mathbb{P}}_x (\cdot) \ g(x) \ \mu(dx), \qquad \mathbb{P}_{\nu}(\cdot) := \int_{\R^d} \mathbb{P}_x (\cdot) \ g(x) \ \mu(dx)  .
$$
Recall the definition of $S$ at the beginning of Section \ref{subsec:3.1.1}. Consider the one-to-one map 
$$
G : \tilde{\Omega} \to \Omega,\quad G(\omega): = \omega |_S.
$$
Then $G$ is $\tilde{\mathcal{F}}^0 / \mathcal{F}^0$ measurable and $\tilde{\Omega} \in \tilde{\mathcal{F}}^0$, where $\tilde{\mathcal{F}}^0 : = \sigma(\tilde{X}_s \ | \ s \in S)$ and using Proposition \ref{prop:3.1.3} exactly as in \cite[Lemma 4.2 and 4.3]{AKR} we can show that 
$$
\tilde{\mathbb{P}}_{\nu} |_{\tilde{\mathcal{F}}^0} \circ G^{-1} = \mathbb{P}_{\nu},\quad G(\tilde{\Omega}) \in \mathcal{F}^0, \quad \text{and} \quad\mathbb{P}_{\nu}(G(\tilde{\Omega}))=1.
$$ 
In particular
\begin{eqnarray}\label{eq:3.4}
\mathbb{P}_x(\Omega \setminus G(\tilde{\Omega}))=0, \quad \text{ for }\mu\text{-a.e. } x\in \R^d. 
\end{eqnarray}
Now, the following holds:
\begin{lemma}\label{lem:3.1.1}
Assume {\bf (a)} of Section \ref{subsec:2.2.1} holds.
Let
$$
\Omega_1 : = \bigcap_{s > 0, s \in S} \vartheta_s^{-1} (G(\tilde{\Omega})),
$$
where $\vartheta_s : \Omega \to \Omega$, $\vartheta_s(\omega) : = \omega(\cdot +s)$, $s \in S$, is the canonical shift operator. Then 
\begin{equation}\label{omega1}
\mathbb{P}_x(\Omega_1) = 1
\end{equation}
for all $x \in \R^d$.
\end{lemma}

\begin{proof}
Using the Markov property, we have for $x \in \R^d$, $s\in S$, $s>0$
\begin{eqnarray*}
&&\mathbb{P}_x\big (\Omega\setminus\vartheta_s^{-1} (G(\tilde{\Omega}))\big )
\ = \   \E_x\big [\E_{X_s^0}[1_{\Omega\setminus G(\tilde{\Omega})} ]\big ]\ =\ P_s^{\Delta}\big (\E_{\cdot}[1_{\Omega\setminus G(\tilde{\Omega})} ]\big )(x)\\
&&=\big[1- P_s(x,\mathbb{R}^d)\big] \int_{\R^d_{\Delta}}\E_{y}[1_{\Omega\setminus G(\tilde{\Omega})} ]\delta_{\Delta} (dy)\ +\  \int_{\R^d_{\Delta}}\E_{y}[1_{\Omega\setminus G(\tilde{\Omega})} ]P_s(x, dy)\\
\end{eqnarray*}
Now
$$
\int_{\R^d_{\Delta}}\E_{y}[1_{\Omega\setminus G(\tilde{\Omega})} ]P_s(x, dy)=\int_{\R^d}\E_{y}[1_{\Omega\setminus G(\tilde{\Omega})} ]P_s(x, dy)=0
$$
by \eqref{eq:3.4} since $P_s(x, dy)$ doesn't charge $\mu$-zero sets, and
$$
\int_{\R^d_{\Delta}}\E_{y}[1_{\Omega\setminus G(\tilde{\Omega})} ]\delta_{\Delta} (dy)=\mathbb{P}_{\Delta}(\Omega\setminus G(\tilde{\Omega}))=0,
$$
since for the constant path $\Delta$, we have $\Delta\in G(\tilde{\Omega})$ and $\mathbb{P}_{\Delta}(\Omega\setminus \{\Delta\})=0$. Thus $\mathbb{P}_x\big (\Omega\setminus \vartheta_s^{-1} (G(\tilde{\Omega}))\big )=0$ and the assertion follows.
\end{proof}

\begin{lemma} \label{lem:3.2}
Assume {\bf (a)} of Section \ref{subsec:2.2.1} holds.
Let $(P_t)_{t>0}$ and $(R_{\alpha})_{\alpha>0}$ be as in Theorems \ref{theorem2.3.1} and  \ref{theo:2.6}. Let $\alpha, t>0$, $x \in \R^d$ and $f \in \cup_{r \in [q, \infty]} L^r(\R^d, \mu)$ with $f \geq 0$, where $q=\frac{pd}{p+d}$ and $p\in (d,\infty)$. Then:
\begin{itemize}
\item[(i)]
$P_t R_{\alpha} f (x) = R_{\alpha} P_t f(x)  = e^{\alpha t} \int_t^{\infty} e^{-\alpha u} P_u f(x) du$. 
\item[(ii)]
$(\Omega , \mathcal{F}^0, (\mathcal{F}^0_s)_{s \in S}, \left ( e^{-\alpha s} R_{\alpha} f(X^0_s)\right )_{s \in S}, \P_x)$ is a positive supermartingale.
\end{itemize}
\end{lemma}

\begin{proof}
(i) Since $T_t G_{\alpha} f = G_{\alpha} T_t f$ $\mu$-a.e. and $P_t R_{\alpha}f$, $R_{\alpha} P_t f \in C(\R^d)$, it holds that
$$
P_t R_{\alpha} f (x) = R_{\alpha} P_t f(x).
$$
By Theorem \ref{th:3.1.1}, 
$$
R_{\alpha} P_t f(x)  = \int_0^{\infty} e^{-\alpha u} P_{t+u}f(x) du=e^{\alpha t} \int_t^{\infty} e^{-\alpha u} P_u f(x) du.
$$
(ii) Let $s \in S$. Since $R_{\alpha} f$ is continuous by Theorem \ref{theorem2.3.1}, $\left ( e^{-\alpha s} R_{\alpha} f(X^0_s)\right )_{s \in S}$ is adapted. Moreover, since  $R_\alpha f \in \cup_{r \in [q, \infty]} L^r(\R^d, \mu)$, it follows from Proposition \ref{prop:3.1.1}(i) that
\begin{eqnarray*}
\E_x\left[ |e^{-\alpha s} R_{\alpha}f(X^0_s)| \right]  &=& e^{-\alpha s} P^{\Delta}_s |R_{\alpha} f| (x) =e^{-\alpha s} \int_{\R^d} |R_{\alpha} f|(y) P_s(x, dy) < \infty.
\end{eqnarray*}
Let $s' \in S$ with $s' \geq s$. Then by the Markov property, (i) and Theorem \ref{th:3.1.1},
\begin{eqnarray*}
&&\E_x \left[ e^{-\alpha s'} R_{\alpha} f(X^0_{s'})  \vert \mathcal{F}_s \right] \ = \ \E_{X^0_s} \left[  e^{-\alpha s'} R_{\alpha} f (X^0_{s'-s})\right] \ = \ e^{-\alpha s'} P_{s' -s} R_{\alpha} f(X^0_s) \\
&=& e^{-\alpha s} \int_{s'-s}^{\infty} e^{-\alpha u} P_u f(X^0_s) du  \ \leq  \ e^{-\alpha s} R_{\alpha}  f(X^0_s).
\end{eqnarray*}
\end{proof}
\noindent
$\Omega_1$ defined in Lemma \ref{lem:3.1.1} consists of paths in $\Omega$ which have unique continuous extensions to $(0,\infty)$, which still lie in $\R^d_{\Delta}$, and which stay in $\Delta$ once they have hit $\Delta$. In order to handle the limits at $s=0$ the properties presented in the following remark are crucial. 
\begin{remark}\label{rem:3.1.1} 
{\it Assume that {\bf (a)} of Section \ref{subsec:2.2.1} and that {\bf (b)} of the beginning of this section
hold. Then, in view of Theorems \ref{theorem2.3.1} and \ref{theo:2.6}, and Lemma \ref{eq:2.3.39a} and \eqref{resoldef}}, one can find $\{ u_n  :  n \ge 1 \} \subset C_0^2(\R^d)\subset D(L_q)$, satisfying:
\begin{itemize}
\item[(i)] for all $\varepsilon \in \mathbb{Q} \cap (0,1)$ and
$y \in D$, where $D$ is any given countable dense set in $\R^d$, there exists $n \in \N$ such that $u_n (z) \ge 1$, for all $z \in \overline{B}_{\frac{\varepsilon}{4}}(y)$ and $u_n \equiv 0$ on $\R^d \setminus B_{\frac{\varepsilon}{2}}(y)$;
\item[(ii)] $R_1\big( [(1 -L) u_n]^+ \big)$, $R_1\big( [(1 -L) u_n]^- \big)$, $R_1 \big( [(1-L)u_n^2]^+ \big)$, $R_1 \big( [(1-L)u_n^2]^- \big)$ are continuous on $\R^d$ for all $n \ge 1$;
\end{itemize}
and moreover it holds that:
\begin{itemize}
\item[(iii)] $R_1 C_0(\R^d) \subset C(\R^d)$;
\item[(iv)] for any $x \in \R^d$ and $u\in C_0^2(\R^d)$, the maps $t \mapsto P_t u(x)$ and $t \mapsto P_t (u^2)(x)$ are continuous on $[0,\infty)$.
\end{itemize}

\end{remark}
\medskip
Define
$$
\Omega_0 : = \{ \omega \in \Omega_1 :  \lim_{s \searrow 0} X_s^0(\omega) \ \text{exists in }  \R^d\}.
$$

\begin{lemma}\label{akrlemma}
Assume {\bf (a)} of Section \ref{subsec:2.2.1} and {\bf (b)} of Section \ref{subsec:3.1.1} hold.
We have 
\begin{equation}\label{normal}
\lim_{\begin{subarray}{1} s \searrow 0 \\ s \in S \end{subarray} } X_s^0 =x \quad \mathbb{P}_x\text{-a.s.} \quad \text{for all } \  x \in \R^d.
\end{equation}  
In particular $\mathbb{P}_x (\Omega_0) = 1$ for any $x \in \R^d$.
\end{lemma}
\begin{proof}
Let $x \in \R^d$, $n \ge 1$. Then the processes with time parameter $s\in S$
$$
\big( e^{-s} R_1 \big( [(1-L) u_n]^+ \big) (X_s^0), \mathcal{F}_s^0, \mathbb{P}_x  \big) \quad \text{and} \quad \big( e^{-s} R_1 \big( [(1-L) u_n]^- \big) (X_s^0), \mathcal{F}_s^0, \mathbb{P}_x  \big)
$$
are positive supermartingales by Lemma \ref{lem:3.2}(ii). Then by \cite[1.4 Theorem 1]{CW} for any $t \ge 0$
$$
\exists \lim_{\begin{subarray} \ s \searrow t \\ s \in S  \end{subarray}} e^{-s} \  R_1 \big( [(1-L) u_n]^{\pm}\big) (X_s^0) \quad \mathbb{P}_x\text{-a.s.}
$$
thus
\begin{equation}\label{existslim}
\exists \lim_{\begin{subarray} \ s \searrow 0 \\ s \in S \end{subarray}} u_n(X_s^0) \quad \mathbb{P}_x\text{-a.s.}
\end{equation}
We have $u_n = R_1 \big((1-L)u_n \big)$ and $u_n^2 = R_1 \big((1-L)u_n^2 \big)$ $\mu$-a.e., but since both sides are respectively continuous by Remark \ref{rem:3.1.1}(ii), it follows that the equalities hold pointwise on $\R^d$. Therefore
$$
\E_x \big[\big(u_n(X_s^0) - u_n(x)   \big)^2 \big] = P_s R_1\big( (1-L)u_n^2 \big)(x) - 2u_n(x) P_s R_1\big((1-L)u_n\big)(x) + u_n^2(x)
$$
and so
\begin{equation}\label{separ}
\lim_{\begin{subarray} \ s \searrow 0 \\ s \in S  \end{subarray}} \E_x\big[ \big(u_n(X_s^0) -u_n(x) \big)^2 \big] =0
\end{equation}
by  Remark \ref{rem:3.1.1}(iv). \eqref{existslim} and \eqref{separ} now imply that 
\begin{equation}\label{separating}
\lim_{\begin{subarray} \ s \searrow 0 \\ s \in S  \end{subarray}} u_n(X_s^0 (\omega)) = u_n(x) \quad \text{for all} \ \omega \in \Omega_x^n,
\end{equation}
where $\Omega_x^n \subset \Omega_1$ with $\mathbb{P}_x (\Omega_x^n) = 1$. Let $\omega \in \Omega_x^0 : = \bigcap_{n \ge 1} \Omega_x^n$. Then $\mathbb{P}_x (\Omega_x^0) = 1$. Suppose that $X_s^0(\omega)$ does not converge to $x$ as $s \searrow 0$, $s \in S$. Then there exists $\varepsilon_0 \in \mathbb{Q}$ and a subsequence $(X_{s_k}^0(\omega))_{k \in \N}$ such that $\|X_{s_k}^0(\omega)-x\| > \varepsilon_0$ for all $k \in \N$. For $\varepsilon_0 \in \mathbb{Q}$ we can find $y \in D$ and $u_n$ as in Remark \ref{rem:3.1.1}(i) such that $\|x-y\| \le \frac{\varepsilon_0}{4}$ and $u_n(z) \ge 1 $, $z \in  \overline{B}_{\frac{\varepsilon_0}{4}}(y)$  and $u_n(z) = 0$, $z \in \R^d \setminus B_{\frac{\varepsilon_0}{2}}(y)$. 
Then
$\|X^0_{s_k}(\omega)-y\| >\frac{3}{4}\varepsilon_0$, and so $u_n(X_{s_k}^0(\omega))\equiv 0$ cannot converge to $u_n(x)=1$ as $k \to \infty$. This is a contradiction.
\end{proof}
Now we define for $t \ge 0$
$$ 
X_t(\omega) : = 
\begin{cases}
\lim_{\begin{subarray}{1} s \searrow t \\  s \in S \end{subarray} } X_s^0(\omega) \quad \text{if} \ \omega \in \Omega_0, \\
0\in \R^d \qquad \text{if} \ \omega \in \Omega \setminus \Omega_0.
\end{cases}
$$
Then by Remark \ref{rem:3.1.1}(iv) for any $t \ge 0$, $f \in C_0^2(\R^d)$ and $x \in \R^d$ 
$$
\E_x[f(X_t)] = P_t f(x),
$$
which extends to $f\in C_0(\R^d)$ using a uniform approximation of $f\in C_0(\R^d)$ through functions in $C_0^2(\R^d)$.
Since the $\sigma$-algebra generated by $C_0(\R^d)$ equals $\mathcal{B}(\R^d)$, it follows by a monotone class argument that 
$$
\mathbb{M} = (\Omega, \mathcal{F}, (\mathcal{F}_t)_{t\ge 0}, (X_t)_{t\geq0} , (\mathbb{P}_x)_{x \in \R^d_{\Delta} }),
$$
where $(\mathcal{F}_t)_{t\ge 0}$ is the natural filtration, is a normal Markov process (cf. Definition \ref{def:3.1.1}), such that $\E_x[f(X_t)] = P_t f(x)$ for any $t \ge 0$, $f \in \mathcal{B}_b(\R^d)$ and $x \in \R^d$. Moreover, $\mathbb{M}$ has continuous sample paths up to infinity on $\R^d_{\Delta}$. The strong Markov property of $\mathbb{M}$ follows from \cite[Section I. Theorem (8.11)]{BlGe} using Remark \ref{rem:3.1.1}(iii). Hence $\mathbb{M}$ is a strong Markov process with continuous sample paths on $\R^d_{\Delta}$, and has the transition function $(P_t)_{t \ge 0}$ as transition semigroup. In particular $\mathbb{M}$ is also a Hunt process (see Definition \ref{def:3.1.1}(ii)). Making a statement out of the these conclusions we formulate the following theorem.

\begin{theorem}
\label{th: 3.1.2}
Assume {\bf (a)} of Section \ref{subsec:2.2.1} and {\bf (b)} of Section \ref{subsec:3.1.1} hold. Then, 
there exists a Hunt process
$$
\M =  (\Omega, \mathcal{F}, (\mathcal{F}_t)_{t \ge 0}, (X_t)_{t \ge 0}, (\mathbb{P}_x)_{x \in \R^d\cup \{\Delta\}}   )
$$
with state space $\R^d$ and lifetime 
$$
\zeta=\inf\{t\ge 0\,:\,X_t=\Delta\}=\inf\{t\ge 0\,:\,X_t\notin \R^d\}, 
$$
having the transition function $(P_t)_{t \ge 0}$ (cf. Proposition \ref{prop:3.1.1} and the paragraph right after Remark \ref{rem:3.1equivalence ralpha}) as transition semigroup, 
i.e. for every $t\geq 0$, $x \in \R^d$ and $f\in \mathcal{B}_b(\R^d)$ it holds that
\begin{equation}\label{eq:3.1semigroupequal-a.e.}
P_tf(x)=\E_x[f(X_t)],
\end{equation}
where $\E_x$ denotes the expectation w.r.t $\P_x$.
Moreover, $\M$ has continuous sample paths
 on the one point compactification $\R^d_{\Delta}$ of  $\R^d$ with the cemetery $\Delta$ as a point at infinity, 
i.e. we may assume that 
\begin{equation*}
\Omega = \{\omega = (\omega (t))_{t \ge 0} \in C([0,\infty),\R^d_{\Delta}) \, : \, \omega(t) = \Delta \quad \forall t \ge \zeta(\omega) \}
\end{equation*}
and
$$
X_t(\omega) = \omega(t), \quad t \ge 0.
$$

\end{theorem}
\subsubsection{Krylov-type estimate}
\label{subsec:3.1.2} 
Throughout this section we will assume that {\bf (a)} of Section \ref{subsec:2.2.1} holds and that assumption {\bf (b)} of Section \ref{subsec:3.1.1}  holds  (except in the case of Proposition \ref{prop:3.1.5}). Let $(P_t)_{t>0}$ be as in Theorem \ref{theo:2.6} and \eqref{semidef}, and $\M$ be as in Theorem \ref{th: 3.1.2}, and we let
$$
\mu=\rho\,dx
$$ 
be as in Theorem \ref{theo:2.2.7} or as in  Remark \ref{rem:2.2.4}. \\
\begin{proposition}\label{prop:3.1.4}
Assume {\bf (a)} of Section \ref{subsec:2.2.1} and {\bf (b)} of Section \ref{subsec:3.1.1} hold.
Let $x\in \R^d$, $\alpha, t>0$. Then (cf. Proposition \ref{prop:3.1.1})
\begin{equation} \label{eq:3.12}
P_t f(x)=\int_{\R^d} f(y) P_t(x, dy)=\E_x\left [f(X_t)\right ], 
\end{equation}
for any $f\in L^1(\R^d,\mu)+L^\infty(\R^d,\mu)$ and (cf. Proposition \ref{prop:3.1.2})
\begin{equation} \label{eq:3.13}
R_{\alpha}g(x)=\int_{\R^d} g(y) R_{\alpha}(x, dy)=\E_x \left [\int_0^\infty e^{-\alpha s} g(X_s)ds\right ],  
\end{equation}
for any $g\in L^q(\R^d,\mu)+L^\infty(\R^d,\mu)$, $q=\frac{pd}{p+d}$ and $p\in (d,\infty)$. \\
In particular, integrals of the form 
$\int_0^\infty e^{-\alpha s} h(X_s)ds$, $\int_0^t h(X_s)ds$, $t\ge 0$ are for any $x\in \R^d$, whenever they are well-defined,  $\P_x$-a.s.  independent of the measurable $\mu$-version chosen for $h$. 
\end{proposition}
\begin{proof}
Using Theorem \ref{th: 3.1.2} and linearity, \eqref{eq:3.12} first holds for simple functions and extends to  $f \in \cup_{r \in [1, \infty]}L^r(\R^d, \mu)$ with $f \geq 0$ through monotone integration. Then \eqref{eq:3.12} follows by linearity. Using \eqref{eq:3.12} and Theorem \ref{th:3.1.1}, \eqref{eq:3.13} follows by Fubini's theorem.
\end{proof}
\begin{proposition} 
\label{prop:3.1.5}
Assume {\bf (a)} of Section \ref{subsec:2.2.1} holds.
Let $(P_t)_{t>0}$ and $(R_{\alpha})_{\alpha>0}$ be as in Theorems \ref{theorem2.3.1} and  \ref{theo:2.6}, and let $q=\frac{pd}{p+d}$, $p\in (d,\infty)$.
\begin{itemize}
\item[(i)] 
Let $f \in L^r(\R^d, \mu)$ for some $r \in [q, \infty]$, $B$ be an open ball in $\R^d$ and $t>0$. Then
$$
\sup_{x \in \overline{B}}\int_0^t P_s |f|(x) ds <  e^t c_{B,r} \|f\|_{L^r(\R^d, \mu)},
$$
where $c_{B,r}>0$ is a constant independent of $f$ and $t$.
\item[(ii)]
Let $\alpha>0$, $t>0$ and  $g \in D(L_r) \subset C(\R^d)$ for some $r \in [q, \infty)$. Then
$$
R_{\alpha} (\alpha-L_r) g (x) =g (x), \quad \forall x \in \R^d.
$$
\item[(iii)] 
Let $t>0$ and  $g \in D(L_r) \subset C(\R^d)$ for some $r \in [q, \infty)$. Then
$$
P_t g(x)-g(x) = \int_0^t  P_s L_r  g(x) ds, \quad \forall x \in \R^d.
$$
\end{itemize}
\end{proposition}

\begin{proof}
(i) By Theorem \ref{th:3.1.1} and \eqref{eq:2.3.38a},
$$
\sup_{x \in \overline{B}} \int_0^t P_s |f|(x) ds  \leq  e^t \sup_{ x \in \overline{B}} R_1 |f|(x) \leq e^t c_{B,r} \|f\|_{L^r(\R^d, \mu)},
$$
where $c_{B,r}>0$ is a constant independent of $f$ and $t$.\\
(ii)
We have $G_{\alpha} (\alpha-L_r) g(x) = g(x)$ for $\mu$-a.e. $x \in \R^d$. Since $R_{\alpha} (\alpha -L_r) g$ is a continuous $\mu$-version of $G_{\alpha} (\alpha-L_r) g$ and $g$ is continuous, the assertion follows. \\
(iii)  Let $f:=(1-L_r) g$. Then $R_1 f = g \in L^r(\R^d, \mu)$. For $x \in \R^d$, $s\geq 0$, it follows by Theorem \ref{th:3.1.1} that
$$
e^{-s} P_s R_1 f(x) = e^{-s} R_1 P_s f(x) =\int_0^{\infty} e^{-(s+u)} P_{s+u} f(x) du = \int_s^{\infty} e^{-u} P_u f(x) du,
$$
hence by Theorem \ref{th:3.1.1} again,
\begin{equation} \label{eq:3.14}
e^{-s} P_s R_1 f(x) - R_1 f(x) = \int_0^s -e^{-u} P_u f (x) du.
\end{equation}
For $s \in [0,t]$, let $\ell(s):=e^{-s} P_s R_1 f(x)$. Then by \eqref{eq:3.14}, $\ell$ is absolutely continuous on $[0,t]$ and has a weak derivative $\ell' \in L^1([0,t])$ satisfying
$$
\ell'(s)= -e^{-s} P_s f (x), \; \text{ for a.e. } s \in [0,t].
$$
Let $k(s):=e^s$, $s \in [0,t]$. Using the product rule and the fundamental theorem of calculus,
$$
P_t g(x)-g(x)=k(t)\ell(t) -k(0)\ell(0) = \int_0^t k'(s)\ell(s)+k(s)\ell'(s) ds = \int_0^t P_s L_r g(x) ds.
$$
\end{proof}
\noindent
Using Proposition \ref{prop:3.1.5}(i) and Fubini's theorem, we obtain the following theorem.
\begin{theorem} 
\label{theo:3.3}
Assume {\bf (a)} of Section \ref{subsec:2.2.1} and {\bf (b)} of Section \ref{subsec:3.1.1} hold.
Let $\M$ be as in  Theorem \ref{th: 3.1.2}. Let $r \in [q, \infty]$, with $q=\frac{pd}{p+d}$ and $p\in (d,\infty)$, $t>0$ and $B$ be an open ball in $\R^d$.
\begin{itemize}
\item[(i)]
Then for any $f\in L^r(\R^d,\mu)$,
\begin{equation}
\label{KrylovEstimate} 
\sup_{x\in \overline{B}}\E_x\left [ \int_0^t |f|(X_s) \, ds \right ] \le  e^t c_{B,r} \|f\|_{L^r(\R^d, \mu)},
\end{equation}
where $c_{B, r}>0$ is independent of $f$ and $t>0$. In particular, if $\rho \in L^{\infty}(\R^d)$, then for any $f \in L^{r}(\R^d)$,
\begin{equation}\label{kryest1}
\sup_{x\in \overline{B}}\E_x\left [ \int_0^t |f|(X_s) \, ds \right ] \le  e^t c_{B,r} \|\rho\|_{L^{\infty}(\R^d)} \|f\|_{L^r(\R^d)}.
\end{equation}
\item[(ii)] Let $V$ be an open ball in $\R^d$. Then for any $f \in L^q(\R^d)$ with $\text{supp}(f) \subset V$,
\begin{equation} \label{kryest2}
\sup_{x\in \overline{B}}\E_x\left [ \int_0^t |f|(X_s) \, ds \right ] \le e^t c_{B,q} \|\rho\|_{L^{\infty}(V)}  \, \|f\|_{L^q(\R^d)},
\end{equation}
where $c_{B, q}>0$ is a constant as in (i).
\end{itemize}
\end{theorem}
\begin{remark} 
\label{rem:ApplicationKrylovEstimates}
{\it The Krylov-type estimate \eqref{KrylovEstimate} and in particular its localization to Lebesgue integrable 
functions in Theorem \ref{theo:3.3}(ii) is an important tool in the derivation of tightness 
results for solutions of SDEs. Such an estimate is often applied in the approximation of
SDEs by SDEs with smooth coefficients (see, e.g.,  
\cite{GyKr}, \cite{Mel}, \cite{MiKr}, \cite{GyMa}
and \cite[p. 54, 4. Theorem]{Kry} for the 
original Krylov estimate involving conditional expectation). \\
A priori \eqref{KrylovEstimate} only holds for the Hunt process $\M$ constructed here. However, if 
uniqueness in law holds for the SDE solved by $\M$ with certain given coefficients (for instance in the 
situation of Theorem \ref{theo:3.3.1.8} and Propositions \ref{prop:3.3.1.15} and \ref{prop:3.3.1.16} 
below), then \eqref{KrylovEstimate} and its localization to Lebesgue integrable functions hold generally 
for any diffusion with the given coefficients. This may then  lead to an improvement in the order of integrability $r=q>\frac{d}{2}$ in Theorem \ref{theo:3.3} in comparison to $d$ in \cite[p. 54, 4. Theorem]{Kry}. In fact, the mentioned improvement in the order of integrability can already be observed in an application of Theorem \ref{theo:3.3} to the moment inequalities derived in Proposition \ref{theo:3.2.8}.\\
Estimate \eqref{KrylovEstimate} becomes particularly useful when the density $\rho$ is explicitly known, 
which holds for a large class of time-homogeneous generalized Dirichlet forms (see Remark 
\ref{rem:2.2.4}). As a particular example consider the non-symmetric divergence form case, i.e. the case 
where $\mathbf{H}, \mathbf{\overline{B}}\equiv 0$, in Remark \ref{rem:2.2.4}. Then the explicitly given 
$\rho\equiv 1$ defines an infinitesimally invariant measure. In this case $\mu$ in \eqref{KrylovEstimate} 
can be replaced by the Lebesgue measure}.
\end{remark}

\subsubsection{Identification of the stochastic differential equation}
\label{subsec:3.1.3}
Throughout this section we will assume that {\bf (a)} of Section \ref{subsec:2.2.1} holds and that assumption {\bf (b)} of Section \ref{subsec:3.1.1}  holds. And we let
$$
\mu=\rho\,dx
$$ 
be as in Theorem \ref{theo:2.2.7} or as in  Remark \ref{rem:2.2.4}. 
\begin{definition}\label{stopping times}
Consider $\M$ of Theorem \ref{th: 3.1.2} and let $A\in \mathcal{B}(\R^d)$. Let $B_n:=\{y\in \R^d : \|y\|<n\}, n\ge 1$. We define the following stopping times:
$$
\sigma_{A}:=\inf\{t>0\,:\, X_t\in A\},\qquad \sigma_n:=\sigma_{\R^d\setminus B_n},n\ge 1,
$$
and
$$
D_A:=\inf\{t\ge0\,:\, X_t\in  A\}, \qquad D_n:=D_{\R^d\setminus B_n}, n\ge 1.
$$
\end{definition}
\begin{lemma}\label{lem:3.1.4}
Assume {\bf (a)} of Section \ref{subsec:2.2.1} and {\bf (b)} of Section \ref{subsec:3.1.1} hold.
Let $x\in \R^d$, $t\ge 0$, $q=\frac{pd}{p+d}$ and $p\in (d,\infty)$. Let $\M$ be as in  Theorem \ref{th: 3.1.2}.
Then we have:
\begin{itemize}
\item[(i)] Let $\sigma_n, n\in \N$, be as in Definition \ref{stopping times}.
$$
\P_x \Big(\lim_{n \rightarrow \infty} \sigma_{n} = \zeta \Big)=1.
$$
\item[(ii)]
$$
\P_x\left (\int_0^t |f|(X_s)ds<\infty\right )=1,  \text{ if }\ \ f\in \bigcup_{r \in [q, \infty]} L^r(\R^d,\mu).
$$
\item[(iii)]
$$
\P_x \Big (\Big \{\int_0^t |f|(X_s)ds<\infty\Big \} \cap \{t<\zeta\}\Big )=\P_x \left (\{t<\zeta\}\right ), \text{ if } \  \ f\in L^q_{loc}(\R^d,\mu),
$$
i.e.
$$
\P_x \Big (1_{\{t<\zeta \}}\int_0^t |f|(X_s)ds<\infty \Big )=1, \;\text{ if } \  \ f\in L^q_{loc}(\R^d,\mu),
$$
\end{itemize}
\end{lemma}
\begin{proof}
(i) Fix $x \in \R^d$. By the $\P_x$-a.s. continuity of $(X_t)_{t \geq 0}$ on $\R^d_{\Delta}$, it follows that $\sigma_n \leq \sigma_{n+1} \leq \zeta$ for all $n \geq 1$, $\P_x$-a.s. Define $\zeta':=\lim_{n\to\infty}\sigma_n$.
Then $\sigma_n \leq \zeta' \leq \zeta$, for all $n \in \N$, $\P_x$-a.s. Now suppose that $\P_x(\zeta' < \zeta )>0$. Then $\P_x\left(X_{\zeta'} \in \R^d, \zeta'<\infty \right)=\P_x(\zeta'<\zeta)>0$. Let $\omega \in \{X_{\zeta'} \in \R^d, \zeta'<\infty\}$. By the $\P_x$-a.s. continuity of $(X_{t})_{t\geq 0}$, we may assume that $t \mapsto X_t(\omega)$ is continuous on $\R^d_{\Delta}$ and $\sigma_n(\omega) \leq \zeta'(\omega)$ for all $n\in \N$. Then there exists $N_{\omega} \in \N$ such that $\{X_{t}(\omega): 0 \leq t \leq \zeta'(\omega)\} \subset B_{N_{\omega}}$, hence $\zeta'(\omega) <\sigma_{N_{\omega}}(\omega)$, which is a contradiction. Thus, $\P_x(\zeta' \geq \zeta) = 1$ and since $x \in \R^d$ was arbitrary, the assertion follows. \\
(ii) follows from Theorem \ref{theo:3.3}(i). \\
(iii) Let $x \in \R^d$ and $f \in L_{loc}^q(\R^d, \mu)$. Then there exists $N_0\in \N$ with $x\in B_{N_0}$ and for any $n\ge N_0$, $X_s\in B_n$ for all $s\in [0,t]$ with $t<\sigma_{n}$, $\P_x$-a.s. By Theorem \ref{theo:3.3}(i),
$$
\E_x \Big [1_{\left \{t< \sigma_{n}\right \}}\int_0^t |f|(X_s)ds \Big ]\le 
\E_x \Big [\int_0^t |f|1_{B_n}(X_s)ds \Big ]<\infty, \ \ \ \forall n\ge N_0.
$$
Thus, we obtain
$$
\P_x \Big (1_{\{t<\sigma_n \}}\int_0^t |f|(X_s)ds<\infty \Big )=1, \;\; \; \forall n\ge N_0,
$$
so that
\begin{equation} \label{eq:3.16}
\P_x \Big (\Big \{\int_0^t |f|(X_s)ds<\infty\Big \} \cap \{t<\sigma_n \}\Big )=\P_x \big (\{t<\sigma_n \}\big ).
\end{equation}
Letting $n \rightarrow \infty$ in \eqref{eq:3.16}, the assertion follows from (i).
\end{proof}

\begin{proposition}\label{prop:3.1.6}
Assume {\bf (a)} of Section \ref{subsec:2.2.1} and {\bf (b)} of Section \ref{subsec:3.1.1} hold.
Let $\M$ be as in  Theorem \ref{th: 3.1.2}. Let $u \in D(L_r)$ for some $r \in [q, \infty)$ with $q=\frac{pd}{p+d}$ and $p\in (d,\infty)$, and define
$$
M_t ^u: = u(X_t) - u(x) - \int_0^t L_r u(X_s) \, ds , \quad t \ge 0.
$$ 
Then $(M_t^u)_{t \geq 0}$ is an $(\mathcal{F}_t)_{t \ge 0}$-martingale under $\P_x$  for any $x \in \R^d$. In particular, if $u \in C_0^{2}(\R^d)$, then 
$(M_t^u)_{t \geq 0}$ is a continuous $(\mathcal{F}_t)_{t \ge 0}$-martingale under $\P_x$  for any $x \in \R^d$, i.e. $\P_x$ solves the martingale problem associated with $(L, C_0^2(\R^d))$ for every $x \in \R^d$. 
\end{proposition}

\begin{proof}
Let  $x \in \R^d$, $u \in D(L_r)$ for some $r \in [q, \infty)$.  Then $\E_x[|M^u_t|]<\infty$ for all $t>0$ by Theorem \ref{theo:3.3}(i). Let $t \geq s \geq 0$. Then using the Markov property and Proposition \ref{prop:3.1.5}(iii),
\begin{eqnarray*}
\E_x \left[ M_t^u-M_s^u \vert \mathcal{F}_s \right]  &=& \E_x [u(X_t) \vert \mathcal{F}_s] - u(X_s) - \E_x \Big[ \int_s^t L_rf(X_{\varv})  d\varv \, \big \vert\,  \mathcal{F}_s \Big] \\
&=& \E_{X_s} [u(X_{t-s})] -u(X_s) - \E_{X_s} \Big[ \int_s^t L_r f(X_{\varv-s})d\varv \Big]\\
&=&  P_{t-s} u (X_s)-u(X_s)  - \int_s^t P_{\varv-s} L_r f (X_s) d\varv = 0.
\end{eqnarray*}
Let $u \in C_0^{2}(\R^d) \subset D(L_r)  \cap C_{\infty}(\R^d)$. Then $t \mapsto u(X_t)$ is continuous on $[0, \infty)$,
hence $(M_t^u)_{t \geq 0}$ is a continuous $(\mathcal{F}_t)_{t \ge 0}$-martingale under $\P_x$.
\end{proof}

\begin{proposition}\label{prop:3.1.7}
Assume {\bf (a)} of Section \ref{subsec:2.2.1} and {\bf (b)} of Section \ref{subsec:3.1.1} hold.
Let $\M$ be as in  Theorem \ref{th: 3.1.2}. Let $u \in C_0^2(\R^d)$, $t\ge 0$. Then the quadratic variation process $\langle M^u \rangle$ of the continuous martingale $M^u$ satisfies for any $x \in \R^d$
$$
\langle M^u \rangle_t=\int_0^t \langle A\nabla u, \nabla u\rangle(X_s)ds, \quad t\ge 0, \quad \P_x\text{-a.s.}
$$
In particular, by Lemma \ref{lem:3.1.4}(ii) $\langle M^u \rangle_t$ is $\P_x$-integrable for any $x \in \R^d$, $t\ge 0$  and  so $M^u$ is square integrable.
\end{proposition}
\begin{proof}
For $u\in C_0^2(\R^d)\subset D(L_q)$, where $q=\frac{pd}{p+d}$ and $p\in (d,\infty)$, we have $u^2\in C_0^2(\R^d)\subset D(L_q)$ and $L u^2 = \langle A\nabla u,\nabla u \rangle + 2 u L u$. Thus by Proposition \ref{prop:3.1.6} 
\begin{eqnarray*}
u^2(X_t) - u^2(x)= M_t^{u^2} +\int_0^t \left (\langle A\nabla u,\nabla u \rangle(X_s) + 2 u L u(X_s)\right ) ds.
\end{eqnarray*}
Applying It\^o's formula to the continuous semimartingale $(u(X_t))_{t\ge 0}$, we obtain
\begin{eqnarray*}
u^2(X_t) - u^2(x)= \int_0^t 2u(X_s)dM_s^{u} +\int_0^t 2 u L u(X_s)\,ds + \langle M^u \rangle_t.
\end{eqnarray*}
The last two equalities imply that $\big (\langle M^u \rangle_t-\int_0^t \langle A\nabla u, \nabla u\rangle(X_s)ds\big )_{t\ge 0}$ is 
a continuous $\P_x$-martingale of bounded variation for any $x\in \R^d$, hence constant. This implies the assertion.
\end{proof}
\noindent
For the following result, see for instance \cite[Theorem 1.1, Lemma 2.1]{ChHu}, which we can apply locally.
\begin{lemma}\label{lem:3.1.5}
Under the assumption {\bf (a)} of Section \ref{subsec:2.2.1}, there exists a symmetric non-degenerate matrix  of functions $\sigma=(\sigma_{ij})_{1\le i,j\le d}$ with $\sigma_{ij}\in C(\R^d)$ for all $1 \leq i, j \leq d$ such that
\begin{equation*}
A(x)=\sigma(x)\sigma(x), \ \  \forall x\in \R^d, 
\end{equation*} 
i.e. 
$$
a_{ij}(x)=\sum_{k=1}^d \sigma_{ik}(x)\sigma_{jk}(x), \ \  \forall x\in \R^d, \ 1\le i,j\le d
$$
and 
$$
\det(\sigma(x))>0, \ \  \ \forall x\in \R^d,
$$
where here $\det(\sigma(x))$ denotes the determinant of $\sigma(x)$.
\end{lemma}

\begin{definition}\label{non-explosive}
$\M$ (of Theorem \ref{th: 3.1.2}) is said to be {\bf non-explosive}\index{non-explosive}, if
$$
\P_x(\zeta=\infty)=1, \quad \text{ for all } x \in \R^d.
$$
\end{definition}

\begin{theorem}\label{theo:3.1.4}
Let $A=(a_{ij})_{1\leq i,j \leq d}$ and $\mathbf{G}=(g_1, \ldots, g_d)=\frac{1}{2}\nabla \big (A+C^{T}\big )+ \mathbf{H}$ (see\eqref{form of G}) satisfy the conditions {\bf (a)} of 
Section \ref{subsec:2.2.1} and {\bf (b)} of Section \ref{subsec:3.1.1}. 
Consider the Hunt process $\M$ from Theorem \ref{th: 3.1.2} with coordinates $X_t=(X_t^1,\ldots,X_t^d)$.
\begin{itemize}
\item[(i)] Suppose that  $\M$ is non-explosive.
Let $(\sigma_{ij})_{1 \le i,j \le d}$ be any matrix (possibly non-symmetric) consisting of locally bounded and measurable functions such that $\sigma \sigma^T =A$ (see for instance Lemma \ref{lem:3.1.5} for the existence of such a matrix). Then it holds that $\P_x$-a.s. for any $x=(x_1,\ldots,x_d)\in \R^d$, 
\begin{equation} \label{itosdeweakglo}
X_t = x+ \int_0^t \sigma (X_s) \, dW_s +   \int^{t}_{0}   \mathbf{G}(X_s) \, ds, \quad 0\le  t <\infty,     
\end{equation}
i.e. it holds that $\P_x$-a.s. for any $i=1,\ldots,d$
\begin{equation}\label{weaksolution}
X_t^i = x_i+ \sum_{j=1}^d \int_0^t \sigma_{ij} (X_s) \, dW_s^j +   \int^{t}_{0}   g_i(X_s) \, ds, \quad 0\le  t <\infty,
\end{equation}
where $W = (W^1,\dots,W^d)$ is a $d$-dimensional standard $(\mathcal{F}_t)_{t \geq 0}$-Brownian motion starting from zero.
\item[(ii)] Let $(\sigma_{ik})_{1 \le i \le d,1\le k \le l}$, $l\in \N$ arbitrary but fixed, be any matrix consisting of continuous functions  such that $\sigma_{ik}\in C(\R^d)$ for all $1\le i\le d,1\le k\le l$, and such that $A=\sigma\sigma^T$, i.e. 
$$
a_{ij}(x)=\sum_{k=1}^l \sigma_{ik}(x)\sigma_{jk}(x), \ \  \forall x\in \R^d, \ 1\le i, j \leq d.
$$
Then on a standard extension 
of $(\widetilde{\Omega}, \widetilde{\mathcal{F}}, (\widetilde{\mathcal{F}}_t)_{t\ge 0}, \widetilde{\P}_x )$, $x\in \R^d$, which we denote for notational convenience again 
by $(\Omega, \mathcal{F}, (\mathcal{F}_t)_{t\ge 0}, \P_x )$, $x\in \R^d$, there exists for every $n \in \N$ an $l$-dimensional  standard $(\mathcal{F}_t)_{t \geq 0}$-Brownian motion $(W_{n,t})_{t \geq 0} = \big((W_{n,t}^{1},\dots,W_{n,t}^{l})\big)_{t \geq 0}$ starting from zero such that $\P_x$-a.s. for any $x=(x_1,\ldots,x_d)\in \R^d$, $i=1,\dots,d$
\begin{equation*}
X_t^i = x_i+ \sum_{k=1}^l \int_0^t \sigma_{ik} (X_s) \, dW_{n,s}^{k} +   \int^{t}_{0}   g_i(X_s) \, ds, \quad 0\le  t \leq D_n,
\end{equation*}
where $D_n, n\in \N$, is as in Definition \ref{stopping times}.
Moreover, it holds that $W_{n,s} = W_{n+1, s}$ on $\{s \leq D_n\}$, hence with $W^k_s:=\lim_{n \rightarrow \infty} W^k_{n,s}$, $k=1, \ldots, l$ and $W_s:=(W^1_s, \ldots, W^l_s)$ on $\{ s < \zeta \}$ we get for $1 \leq i \leq d$,
\begin{equation*}
X_t^i = x_i+ \sum_{k=1}^l \int_0^t \sigma_{ij} (X_s) \, dW_s^{k} +   \int^{t}_{0}   g_i(X_s) \, ds, \quad 0\leq  t < \zeta,
\end{equation*}
$\P_x$-a.s. for any $x \in \R^d$. In particular, if $\M$ is non-explosive, then $(W_t)_{t \geq 0}$ is a standard $(\mathcal{F}_t)_{t \geq 0}$-Brownian motion.
\end{itemize}
\end{theorem}

\begin{proof}
(i) Since $\M$ is non-explosive, it follows from  Lemma \ref{lem:3.1.4}(i) that $D_n\nearrow \infty$ $\P_x$-a.s. for any $x\in \R^d$.  Let $\varv\in C^{2}(\R^d)$. Then we claim that
$$
M_t ^\varv: = \varv(X_t) - \varv(x) - \int_0^t \Big (\frac{1}{2}\sum_{i,j=1}^{d}a_{ij}\partial_i\partial_j \varv+\sum_{i=1}^{d}g_i\partial_i \varv\Big )(X_s) \, ds , \quad t \ge 0,
$$
is a continuous square integrable local $\P_x$-martingale with respect to the stopping times $(D_n)_{n\ge 1}$ for any $x\in \R^d$. Indeed, let $(\varv_n)_{n\ge 1}\subset C_0^2(\R^d)$ be such that $\varv_n=\varv$ pointwise on $\overline{B}_n$, $n\ge 1$. 
Then for any $n\ge 1$, we have $\P_x$-a.s
$$
M_{t\wedge D_n}^\varv=M_{t\wedge D_n}^{\varv_n}, \ \ t\ge 0,
$$
and $(M_{t\wedge D_n}^{\varv_n})_{t\ge 0}$ is a square integrable $\P_x$-martingale for any $x\in \R^d$ by Proposition \ref{prop:3.1.7}.
Now let $u_i \in C^{2}(\R^d)$, $i=1,\dots,d$, be the coordinate projections, i.e. $u_i(x)=x_i$. Then by Proposition \ref{prop:3.1.7}, polarization and localization with respect to $(D_n)_{n\ge 1}$, the quadratic covariation processes satisfy
$$
\langle M^{u_i}, M^{u_j} \rangle_t =  \int_0^t  a_{ij}(X_s) \, ds, \quad 1 \le i,j  \le d, \ t \ge 0.
$$
Using Lemma \ref{lem:3.1.5} we obtain by \cite[II. Theorem 7.1]{IW89} (see also \cite[IV. Proposition 2.1]{IW89}) that there exists a $d$-dimensional standard $(\mathcal{F}_t)_{t \geq 0}$-Brownian motion $(W_t)_{t \ge 0} = (W_t^1,\dots, W_t^d)_{t \ge 0}$ on $(\Omega, \mathcal{F}, (\mathcal{F}_t)_{t\ge 0}, \P_x )$, $x\in \R^d$,  such that
\begin{eqnarray}\label{sigma1}
M_t^{u_i} = \sum_{j=1}^{d}  \int_0^t \ \sigma_{ij} (X_s) \ dW_s^j, \quad 1 \le i \le d, \  t\ge 0.
\end{eqnarray}
Since for any $x\in \R^d$, $\P_x$-a.s.
\begin{eqnarray}\label{drift1}
M_t^{u_i}= X_t^i - x_i - \int_0^t g_i(X_s) \, ds , \quad t \ge 0.
\end{eqnarray}
(ii) Let $n\in \N$. Using the same notations and proceeding as in (i), we  obtain that
\begin{eqnarray*}
M^{i,n}_t:=M_{t\wedge D_n}^{u_i}= X_{t\wedge D_n}^i - x_i - \int_0^{t\wedge D_n} g_i(X_s) \, ds , \quad t \ge 0,
\end{eqnarray*}
is a continuous square integrable $\P_x$-martingale for any $x\in \R^d$ and it holds that
$$
\langle M^{i,n}, M^{j,n}\rangle_t =  \int_0^{t\wedge D_n}  a_{ij}(X_s) \, ds=\int_0^{t}  1_{[0,D_n]}(s)a_{ij}(X_s) \, ds, \quad 1\leq i,j \leq d, \ t \ge 0.
$$
Let $\Phi_{ij}(s)=a_{ij}(X_s)1_{[0,D_n]}(s)$, $1 \leq i, j \leq d$, $s \geq 0$, so that 
$$
\Phi_{ij}(s)=\sum_{k=1}^l \Psi_{ik}(s)\Psi_{jk}(s),  \ \text{ with }\ \ \Psi_{ik}(s)=\sigma_{ik}(X_s)1_{[0,D_n]}(s),\ \  
 1\le i, j \le d.
$$
Then for any $x \in \R^d$, $\P_x$-a.s. for all $1\leq i, j \leq d$, $1\leq k\leq l$,
\begin{eqnarray*}
\int_0^t |\Psi_{ik}(s)|^2 ds <\infty,  \;\; \int_0^t |\Phi_{ij}(s)| ds< \infty, \quad \text{ for all } t\geq 0.
\end{eqnarray*}
Then by \cite[II. Theorem 7.1']{IW89}, we obtain the existence of  an $l$-dimensional standard $(\mathcal{F}_t)_{t \geq 0}$-Brownian motion $(W_{n,t})_{t \geq 0} = \big((W_{n,t}^{1},\dots,W_{n,t}^{l})\big)_{t \geq 0}$ as in the assertion such that
\begin{eqnarray*}
M_t^{i,n} & = &\sum_{k=1}^{l}  \int_0^t \ \sigma_{ik} (X_s)1_{[0,D_n]}(s) \ dW_{n, s}^{k}\\
&=& \sum_{k=1}^{l}  \int_0^{t\wedge D_n} \ \sigma_{ik} (X_s) \  dW_{n,s}^{k}, \quad 1 \le i \le d, \  t\ge 0
\end{eqnarray*}
$\P_x$-a.s. for any $x\in \R^d$. Thus for $1\le i\le d$
\begin{equation*} 
X_{t\wedge D_n}^i =x_i +\sum_{k=1}^{l}  \int_0^{t\wedge D_n} \ \sigma_{ik} (X_s) \ dW_{n,s}^{k}+ \int_0^{t\wedge D_n} g_i(X_s) \, ds , \quad t \ge 0.
\end{equation*}
From the proof of \cite[II. Theorem 7.1']{IW89}, we can see the consistency $W_{n,s}=W_{n+1,s}$ on $\{ s \leq D_n \}$. This implies the remaining assertions.
\end{proof}

\subsection{Global properties}
\label{sec:3.2}
In this sectio, we investigate non-explosion, transience and recurrence, and invariant and sub-invariant measures of the Markov process $\M$, which is described as a weak solution to an SDE in Theorem \ref{theo:3.1.4}. Due to the strong Feller property, conservativeness of $(T_t)_{t>0}$ is equivalent to non-explosion of $\M$ (see Corollary \ref{cor:3.2.1}).\\
We first develop three sufficient criteria for non-explosion. The first type of such a criterion is related to the existence of a Lyapunov function, which implies a supermartingale property and provides explicit growth conditions on the coefficients given by a continuous function as upper bound (Proposition \ref{prop:3.2.8} and Corollaries \ref{cor:3.2.2} and \ref{cor:3.1.3}). The second type of non-explosion criterion is related to moment inequalities that are derived with the help of a Burkholder--Davis--Gundy inequality or Doob's inequality and finally a Gronwall inequality. Here the growth condition is given by the sum of a continuous function and an integrable function of some order as upper bound (Proposition \ref{theo:3.2.8}) and the growth condition is stated separately for the diffusion and the drift coefficients in contrast to the first type of non-explosion criterion. The third type of non-explosion criterion is a conservativeness criterion deduced from \cite{GT17}, which, in contrast to the first two types of non-explosion criteria, originates from purely analytical means and involves a volume growth condition on the infinitesimally invariant measure $\mu$. It is applicable, if the growth of $\mu$ on Euclidean balls is known, for instance if $\rho$ is explicitly known (see Proposition \ref{prop:3.2.9}).\\
In Section \ref{subsec:3.2.2}, we study transience and recurrence of the semigroup $(T_t)_{t>0}$ and of $\M$ in the probabilistic sense (see Definitions \ref{def:3.2.2.2} and Definition \ref{def:3.2.2.3}).  Since $(T_t)_{t>0}$ is strictly irreducible by Proposition \ref{prop:2.4.2}, we obtain in Theorem \ref{theo:3.3.6}(i) that $(T_t)_{t>0}$ is either recurrent or transient. Moreover, using the technique of \cite{Ge} and the regularity of the resolvent associated with $\M$, it follows that recurrence and transience of $(T_t)_{t>0}$ is equivalent to recurrence and transience of $\M$ in the probabilistic sense, respectively (see Theorem \ref{theo:3.3.6}).  We present in Proposition \ref{theo:3.2.6} a criterion to obtain the recurrence of $\M$ in the probabilistic sense and in the situation of Remark \ref{rem:2.2.4}, we present another type of recurrence criterion, Corollary \ref{cor:3.2.2.5}, which is a direct consequence of \cite[Theorem 21]{GT2} and Proposition \ref{prop:2.4.2}(ii). \\
In Section \ref{subsec:3.2.3}, we introduce the two notions, {\it invariant measure} and {\it sub-invariant measure} for $\M$, which are strongly connected to the notion of $(\overline{T}_t)_{t>0}$-invariance and sub-invariance respectively, introduced in Section \ref{subsec:2.1.4}. These will appear later in Section \ref{subsec:3.3.2}, for a result about uniqueness in law. We analyze further the long time behavior of the transition semigroup $(P_t)_{t>0}$ associated with $\M$, as well as uniqueness of invariant measures for $\M$ in the case where there exists a probability invariant measure for $\M$. For that, in Theorem \ref{theo:3.3.8}, the  strong Feller property (Definition \ref{def:2.3.1}) and the irreducibility in the probabilistic sense of $(P_t)_{t>0}$ (Definition \ref{def:2.4.4}) are essentially used to apply Doob's theorem, but we further complement it by using Lemma \ref{lem:2.7}(ii) and Remark \ref{remark2.1.11}(i). We show in Example \ref{ex:3.8} that there is a case where there is no unique invariant measure for $\M$ by presenting two distinct infinite invariant measures for $\M$, which cannot be  expressed as a constant multiple of each other.\\

\subsubsection{Non-explosion results and moment inequalities} \label{subsec:3.2.1}
Throughout this section, unless otherwise stated, we will assume that {\bf (a)} of Section \ref{subsec:2.2.1} holds and that assumption {\bf (b)} of Section \ref{subsec:3.1.1}  holds. Furthermnore, we let
$$
\mu=\rho\,dx
$$ 
be as in Theorem \ref{theo:2.2.7} or as in  Remark \ref{rem:2.2.4}. In fact, only at the end of this section for Proposition \ref{prop:3.2.9} and in Remark \ref{rem:3.2.1 1},  assumptions {\bf (a)} and {\bf (b)} and the assumption on $\mu$     may be omitted.
\\
Due to the strong Feller property, we have:
\begin{corollary} \label{cor:3.2.1}
Assume {\bf (a)} of 
Section \ref{subsec:2.2.1} and {\bf (b)} of Section \ref{subsec:3.1.1} hold.
$(T_t)_{t>0}$ is conservative\index{semigroup ! conservative} (Definition \ref{def:3.2.1}), if and only if $\M$ is non-explosive (Definition \ref{non-explosive}).
\end{corollary}
\begin{proof}
Assume that $(T_t)_{t>0}$ is conservative. Then by Theorem \ref{theo:2.6} and Proposition \ref{prop:3.1.4},
\begin{equation*}
\P_x(\zeta>t) =\P_x(X_t \in \R^d)=P_{t} 1_{\R^d}(x) = 1 \; \text{ for all $(x,t) \in \R^d \times (0, \infty)$}.
\end{equation*}
Letting $t \rightarrow \infty$, $\P_x(\zeta = \infty)=1$ for all $x \in \R^d$. Conversely, assume that $\M$ is non-explosive. Then 
$$
\P_x(X_t \in \R^d) = \P_x(\zeta>t) \geq \P_x(\zeta=\infty)=1\; \text{ for all $(x,t) \in \R^d \times (0, \infty)$.}
$$
Consequently, by Theorem \ref{theo:2.6} and Proposition \ref{prop:3.1.4}, $T_t 1_{\R^d} =1$, $\mu$-a.e. for all $t>0$. 
\end{proof}
\begin{remark}
{\it By Corollary \ref{cor:3.2.1} and Lemma \ref{lem:2.7}(ii), it follows that $\M$ is non-explosive, if and only if there exists $(x_0, t_0)\in \R^d \times (0, \infty)$:
$$
P_{t_0}1_{\R^d}(x_0)=\P_{x_0}(X_{t_0} \in \R^d) = \P_{x_0}(\zeta>t_{0})=1.
$$
Thus, $\M$ is non-explosive, if and only if $\P_{x_0}(\zeta=\infty)=1$ for some $x_0 \in \R^d$. This property is also derived in \cite[Lemma 2.5]{Bha} under the assumptions of a locally bounded drift coefficient and continuous diffusion coefficient. In comparison, our conditions {\bf (a)}, {\bf (b)} allow the drift coefficient to be locally unbounded but the diffusion coefficient has to be continuous with a suitable weak differentiability.}
\end{remark}
\medskip
\noindent
Consider the following {\bf condition}: \\
\\
{\bf (L)} \index{assumption ! {\bf (L)}}there exists  $\varphi \in C^2(\R^d)$, $\varphi \geq 0$ such that $\displaystyle \lim_{ r \rightarrow \infty} (\inf_{\partial B_r} \varphi)= \infty$ and
$$
L\varphi \leq M \varphi, \quad \text{ a.e. on } \R^d
$$
for some constant $M>0$. \\ \\
We will call a function $\varphi$ as in {\bf (L)} a {\bf Lyapunov function}\index{Lyapunov function}.  Under the assumption of {\bf (L)}, we saw an analytic method to derive conservativeness (hence non-explosion by Corollary \ref{cor:3.2.1}) in the proof of Proposition \ref{prop:2.1.10}(ii). The next proposition deals with a probabilistic method to derive the non-explosion of $\M$ under the assumption of {\bf (L)}. The method provides implicitly a moment inequality for $\varphi(X_t)$.
\begin{proposition} \label{prop:3.2.8}
Assume {\bf (a)} of 
Section \ref{subsec:2.2.1} and {\bf (b)} of Section \ref{subsec:3.1.1} hold.
Under the assumption of {\bf (L)} above, $\M$ is non-explosive (Definition \ref{non-explosive}) and for any $x \in \R^d$ it holds that
$$
\E_x \left[  \varphi(X_t) \right] \leq e^{Mt} \varphi(x), \;\; \;  t \geq 0.
$$
\end{proposition}

\begin{proof}
Let $x=(x_1, \ldots, x_d) \in \R^d$ and take $k_0 \in \N$ such that $x \in B_{k_0}$. Let  $X^{i,n}_t:=X^{i}_{t \wedge \sigma_n}$,\, $n \in \N$ with $n \geq k_0$, $i \in \{1, \ldots, d\}$, $t \geq 0$, where $\sigma_n, n\in \N$, is as in Definition \ref{stopping times}. Then by Theorem \ref{theo:3.1.4}(ii), $(X_t^{i, n})_{t \geq 0}$ is a continuous $\P_x$-semimartingale and $\P_x$-a.s.
$$
X^{i,n}_t= x_i+ \sum_{j=1}^l \int_0^{t \wedge \sigma_n} \sigma_{ij} (X_s) \, dW_s^j +   \int^{t \wedge \sigma_n}_{0}   g_i(X_s) \, ds, \;\; \quad 0 \leq t< \infty.
$$
For $j \in \{1, \ldots, d \}$, it follows that $\P_x$-a.s.
$$
\langle X^{i,n}, X^{j,n} \rangle_t = \int_0^{t \wedge \sigma_n} a_{ij} (X_s) ds, \quad 0 \leq t<\infty.
$$
Thus, by the time-dependent It\^{o} formula, $\P_x$-a.s.
\begin{eqnarray*}
&&e^{-Mt} \varphi(X_{t\wedge \sigma_n})  = \varphi(x) + \sum_{i=1}^{d}\sum_{j=1}^l   \int_0^{t \wedge \sigma_n} e^{-Ms} \partial_i  \varphi(X_s) \cdot \sigma_{ij} (X_s) \, dW_s^j   \\
&&+ \int_0^{t \wedge \sigma_n}  -Me^{-Ms}  \varphi(X_s) ds + \frac12 \sum_{i,j=1}^{d} \int_0^{t \wedge \sigma_n} e^{-Ms} \partial_i \partial_j  \varphi \cdot a_{ij}(X_s)ds   \\
&&  + \sum_{i=1}^{d} \int_0^{t \wedge \sigma_n} e^{-Ms}\partial_i  \varphi \cdot g_i(X_s) ds \\
&&=  \varphi(x)  + \int_0^{t \wedge \sigma_n} e^{-Ms}\nabla  \varphi \cdot \sigma(X_s) dW_s  + \int_0^{t \wedge \sigma_n} e^{-Ms}(L-M) \varphi(X_s) ds.
\end{eqnarray*}
Consequently, $\left(e^{-Mt} \varphi (X_{t\wedge \sigma_n})  \right)_{t \geq 0}$ is a positive continuous $\P_x$-supermartingale. Since $\M$ has continuous sample paths on the one-point-compactification $\R^d_{\Delta}$ of $\R^d$, it follows that
\begin{eqnarray*}
 \varphi(x) \geq \E_x \left[e^{-Mt} \varphi(X_{t\wedge \sigma_n})   \right] \geq \E_x\left[ e^{-Mt}  \varphi(X_{t \wedge \sigma_n})1_{\{\sigma_{n}\le t\}}\right] \ge e^{-M t} \inf_{\partial B_n}  \varphi  \cdot  \P_x(\sigma_{n}\le t). 
\end{eqnarray*}
Therefore, using Lemma \ref{lem:3.1.4}(i)
$$
\P_x(\zeta\le t)=\lim_{n\to \infty}\P_x(\sigma_{n}\le t)\leq \lim_{n \rightarrow \infty} \frac{e^{Mt} \varphi(x) }{\inf_{\partial B_n}  \varphi} =0.
$$
Letting $t \rightarrow \infty$, $\P_x(\zeta<\infty)=0$, hence $\M$ is non-explosive. Applying Lemma \ref{lem:3.1.4}(i), Fatou's lemma and the supermartingale property, for any $t \geq 0$
$$
\E_x \left[e^{-Mt}    \varphi(X_t)\right] =  \E_x \left[\liminf_{n \rightarrow \infty} e^{-Mt}    \varphi(X_{t\wedge \sigma_n})\right]  \leq \liminf_{n \rightarrow \infty}\E_x \left[ e^{-Mt}    \varphi(X_{t\wedge \sigma_n})\right]  \leq   \varphi(x),
$$
as desired.
\end{proof}
In the next lemma, we present a condition for {\bf (L)}, which is apparently weaker than {\bf (L)}.
\begin{lemma} \label{lem3.2.6}
Assume {\bf (a)} of 
Section \ref{subsec:2.2.1} and {\bf (b)} of Section \ref{subsec:3.1.1} hold.
Let $N_0 \in \N$. Let $g\in C^2(\R^d \setminus \overline{B}_{N_0}) \cap C(\R^d)$, $g \geq 0$, with
\begin{equation} \label{eq:3.19}
\lim_{r \rightarrow \infty} (\inf_{\partial B_r} g) = \infty.
\end{equation}
Assume that there exists a constant $M>0$ such that
$$
Lg\le M g \quad \text{a.e. on }\, \R^d\setminus \overline{B}_{N_0}.
$$ 
Then there exist a constant $K>0$, $N \in \N$ with $N \geq N_0+3$ and $\varphi \in C^2(\R^d)$ with $\varphi \geq K$, $\varphi(x)=g(x)+K$ for all $x \in \R^d \setminus B_{N}$ such that
$$
L \varphi \le M \varphi \quad \text{a.e. on } \, \R^d.
$$ 
In particular, $\M$ of Theorem \ref{th: 3.1.2} (see also Theorem \ref{theo:3.1.4}) is non-explosive by Proposition \ref{prop:3.2.8}.
\end{lemma}

\begin{proof}
We first show the following claim.\\
{\bf Claim:} If $g\in C^2(\R^d \setminus \overline{B}_{N_0}) \cap C(\R^d)$, $g \geq 0$ satisfies \eqref{eq:3.19}, then there exist $N_1 \in \N$ with $N_1 \geq N_0+2$ and $\psi \in C^2(\R^d)$ with $\psi \geq 0$ such that $\psi(x) = g(x)$ for all $x \in \R^d \setminus B_{N_1}$. \\
For the proof of the claim, let $\phi_1 \in C^2(\R)$ such that $\phi_1(t) \ge 0$ for all $t \in \R$ and 
$$
\phi_1(t)=\begin{cases} \ \sup_{B_{N_0+1}}g&  \text{ if } t\le \sup_{B_{N_0 +1}}g,\\
\ t\quad& \text{ if } t\ge 1+\sup_{B_{N_0 +1}}g.
\end{cases}
$$
Define $\psi:= \phi_1 \circ g$. Then $\psi \geq 0$, $\psi \equiv \sup_{B_{N_0+1}}g$ on $\overline{B}_{N_0+1}$ and $\psi \in C^2(\R^d \setminus \overline{B}_{N_0})$, hence $\psi \in C^2(\R^d)$. Let $A_1:=\{ x \in \R^d: |g(x)| \leq   1+\sup_{B_{N_0+1}}g \}$. Then $A_1$ is closed and bounded since $g \in C(\R^d)$ and \eqref{eq:3.19} holds. Thus, there exists $N_1 \in \N$ with $N_1 \geq N_0+2$ and  $A_1 \subset B_{N_1}$. In particular, $\psi(x)=g(x)$ for all $x \in \R^d \setminus B_{N_1}$, hence the claim is shown.\\
For the constructed $\psi \in C^2(\R^d)$ and $N_1 \in \N$ as in the claim above, it holds that
$$
L \psi \leq M \psi, \quad \text{ a.e. on $\R^d \setminus B_{N_1}$.}
$$
Let $\phi_2 \in C^2(\R)$ such that $\phi_2(t), \phi_2'(t) \geq 0$ for all $t \in \R$ and
$$
\phi_2(t)=\begin{cases} \ \sup_{B_{N_1}} \psi &  \text{ if } t\le \sup_{B_{N_1}} \psi,\\
\ t\quad& \text{ if } t\ge 1+\sup_{B_{N_1}} \psi.
\end{cases}
$$
Let $A_2:=\{ x: |\psi(x)| \leq  1+\sup_{B_{N_1}}\psi \}$. As above, there exists $N \in \N$ with $N \geq N_1+1$ and $A_2 \subset B_{N}$. 
Define 
$$
K:= \sup_{B_{N}} \left(\psi \cdot \phi_2'(\psi)\right) +\frac{\nu_{B_{N}}}{2M}\sup_{B_{N}} \left(|\phi_2''(\psi)| \| \nabla \psi\|^2\right),
$$
where $\nu_{B_{N}}$ is as in \eqref{eq:2.1.2} and $\varphi:= \phi_2 \circ \psi+K$.  Then $\varphi \in C^2(\R^d)$ with $\varphi \geq 0$. In particular, $\varphi(x)=\psi(x)+K=g(x)+K$ for all $x \in \R^d \setminus B_{N}$. Moreover,
$$
L\varphi=\begin{cases} \ 0 \leq M\varphi&  \text{ a.e. on } B_{N_1},\\
\ \phi'_2(\psi) L\psi+\frac12 \phi''_2(\psi) \langle A \nabla \psi, \nabla \psi \rangle \\
\quad \; \leq M \phi_2'(\psi) \psi + \frac12 \nu_{B_{N}} |\phi_2''(\psi)| \| \nabla \psi\|^2 \leq MK \leq M\varphi  \quad& \text{ a.e. on }B_{N} \setminus B_{N_1}, \\
\ L \psi \leq M \psi \leq M \varphi.  \quad& \text{ a.e. on  } \R^d \setminus B_{N}.
\end{cases}
$$
Finally, since  $\displaystyle \lim_{r \rightarrow \infty} (\inf_{\partial B_r} \varphi) =\infty$, $\M$ is non-explosive by Proposition \ref{prop:3.2.8}.
\end{proof}
\noindent
In view of Corollary \ref{cor:3.2.1}, the following result slightly improves the condition of Corollary \ref{cor:2.1.4.1}(iii).
\begin{corollary} \label{cor:3.2.2}
Assume {\bf (a)} of 
Section \ref{subsec:2.2.1} and {\bf (b)} of Section \ref{subsec:3.1.1} hold.
Assume that there exist a constant  $M> 0$ and $N_0\in \N$, such that 
\begin{eqnarray}\label{eq:3.20}
-\frac{\langle A(x)x, x \rangle}{ \left \| x \right \|^2 }+ \frac12\mathrm{trace}A(x)+ \big \langle \mathbf{G}(x), x \big \rangle \leq M\left  \| x \right \|^2 \big( {\rm ln}\left \| x \right \| +1      \big)
\end{eqnarray}
for a.e. $x\in \R^d\setminus \overline{B}_{N_0}$. Then $\M$ is non-explosive (Definition \ref{non-explosive}).
\end{corollary}

\begin{proof}
Define $g(x):= \ln(\|x\|^2 \vee N_0^2)+2$. Then $g \in C^{\infty}(\R^d \setminus \overline{B}_{N_0}) \cap C(\R^d)$ and
$$
Lg = -2 \frac{\langle A(x)x,x \rangle}{\|x\|^4}  + \frac{\text{trace}(A(x))}{\|x\|^2} + \frac{2 \langle  \mathbf{G}(x), x \rangle}{\|x\|^2 } \;\; \text{ on $\R^d \setminus \overline{B}_{N_0}$. }
$$
Since \eqref{eq:3.20} is equivalent to the fact that $Lg \leq Mg$ for a.e. on $\R^d \setminus \overline{B}_{N_0}$, the assertion follows from Lemma \ref{lem3.2.6}.
\end{proof}
\noindent
The following corollary allows (in the special case $d=2$) the diffusion coefficient to have an arbitrary growth in the 
case where the difference between the minimal and the maximal eigenvalue of the diffusion coefficient has 
quadratic-times-logarithmic growth.

\begin{corollary} 
\label{cor:3.1.3}
Assume {\bf (a)} of 
Section \ref{subsec:2.2.1} and {\bf (b)} of Section \ref{subsec:3.1.1} hold.
Let $d=2$, $N_0 \in \N$, $\Psi_1, \Psi_2 \in C(\R^2)$ with $\Psi_1(x), \Psi_2(x)>0$ 
for all $x \in \R^2$, and $Q=(q_{ij})_{1 \leq i,j \leq 2}$ be a matrix of measurable functions such that  
$Q^T(x)Q(x)=id$ for all $x\in \R^2\setminus \overline{B}_{N_0}$. 
Suppose that, in addition to the assumptions {\bf (a)} and {\bf (b)}, the diffusion coefficient has the form
$$
A(x) = Q^T(x) \begin{pmatrix}
\Psi_1(x)  & 0 \\ 
0 & \Psi_2(x)
\end{pmatrix}Q(x) \; \;\text{ for all $x\in \R^2\setminus \overline{B}_{N_0}$,}
$$
and that there exists a constant $M > 0$, such that
\begin{equation} 
\label{eq:3.2.1.20}
\frac{|\Psi_1(x)-\Psi_2(x)|}{2}  + \langle \mathbf{G}(x),x  \rangle \leq M  \left  \| x \right \|^2 \big( {\rm ln}\left \| x \right \| +1 \big)
\end{equation}
for a.e. $x\in \R^2\setminus \overline{B}_{N_0}$. Then $\M$ is non-explosive (Definition \ref{non-explosive}).
\end{corollary}

\begin{proof}
Let $x \in \R^2 \setminus \overline{B}_{N_0}$ and $y=(y_1,y_2):=Q(x)x$. Then 
$$
\|y\|^2=\langle Q(x)x, Q(x)x \rangle = \langle Q^T(x)Q(x)x,x\rangle=\|x\|^2
$$
and
$$
\langle A(x)x,x \rangle= \left\langle \begin{pmatrix}
\Psi_1(x)  & 0 \\ 
0 & \Psi_2(x)
\end{pmatrix} y, y \right \rangle  = \Psi_1(x) y^2_1+\Psi_2(x) y^2_2 \geq (\Psi_1(x) \wedge \Psi_2(x)) \|x\|^2.
$$  
Thus, 
$$
-\frac{\langle A(x)x, x \rangle}{ \left \| x \right \|^2 }+ \frac12\mathrm{trace}A(x) \leq -(\Psi_1(x) \wedge \Psi_2(x))+\frac{\Psi_1(x)+\Psi_2(x)}{2}=\frac{|\Psi_1(x)-\Psi_2(x)|}{2}.
$$
Now Corollary \ref{cor:3.2.2} implies the assertion.
\end{proof}

\begin{proposition} \label{theo:3.2.8}
Assume {\bf (a)} of 
Section \ref{subsec:2.2.1} and {\bf (b)} of Section \ref{subsec:3.1.1} hold.
Let $\sigma= (\sigma_{ij})_{1 \leq i \leq d,1\leq j \leq l}$ be as in Theorem \ref{theo:3.1.4}(ii). Let $h_1 \in L^p(\R^d, \mu)$, $h_2 \in L^q(\R^d, \mu)$ with $q=\frac{pd}{p+d}$ and $p\in (d,\infty)$, and $M>0$ be a constant.
\begin{itemize}
\item[(i)] If for a.e. $x \in \R^d$
\begin{eqnarray*} \label{lineargrowth-2}
\max_{1\leq i \leq d, 1 \leq j \leq l}|\sigma_{ij}(x)| \leq |h_1(x)|+ M(\sqrt{\|x\|}+1),
\end{eqnarray*}
and
$$
\max_{1 \leq i \leq d}|g_i(x)| \leq |h_2(x)| + M(\|x\|+1),
$$
then $\M$ is non-explosive   (Definition \ref{non-explosive}) and moreover, for any $T>0$ and any open ball $B$, there exist constants $D$, $E>0$ such that
$$
\sup_{x \in \overline{B}} \E_{x}\Big [\sup_{s \leq t} \|X_s \|\Big] \leq  D\cdot e^{E \cdot t}, \quad \forall t \in [0, T].
$$
\item[(ii)] If for a.e. $x \in \R^d$
\begin{eqnarray*} 
\max_{1\leq i \leq d, 1 \leq j \leq l} |\sigma_{ij}(x)|+\max_{1 \leq i \leq d}|g_i(x)| \leq |h_1(x)| + M(\|x\|+1),
\end{eqnarray*}
then $\M$ is non-explosive and moreover, for any $T>0$ and any open ball $B$, there exist constants $D, E>0$ such that
$$
\sup_{x \in \overline{B}} \E_{x}\Big[\sup_{s \leq t} \|X_s \|^2\Big] \leq D\cdot e^{E\cdot t}, \quad \forall t \in [0, T].
$$
\end{itemize}
\end{proposition}

\begin{proof}
(i) Let $x \in \overline{B}$, $T>0$ and $t \in [0, T]$, and $\sigma_n, n\in \N$, be as in Definition \ref{stopping times}. For any $i \in \{ 1, \ldots, d\}$ and $n \in \N$, it holds that $\P_x$-a.s.
\begin{eqnarray*} 
 \sup_{0 \leq s \leq t \wedge \sigma_{n}} |X_s^i|
\leq  |x_i| + \sum_{j=1}^l \sup_{\;0 \leq s \leq t \wedge \sigma_{n} } \Big|\int_0^{s} \sigma_{ij} (X_u) \, dW_u^j \Big | +  \int^{ t \wedge \sigma_{n}}_{0}   |g_i(X_u)| \, du.
\end{eqnarray*}
Using the Burkholder--Davis--Gundy inequality (\cite[IV. (4.2) Corollary]{RYor}) and Jensen's inequality, it holds that for any $i \in \{1, \ldots, d\}$ and $j \in \{1,\ldots, l \}$
\begin{eqnarray*} 
&& \E_x \Big [   \sup_{\;0 \leq s \leq t \wedge \sigma_{n} }\Big | \int_0^{s} \sigma_{ij} (X_u) \, dW_u^j \Big |   \Big ]  \\
&&  \leq  C_1   \E_x \Big [  \Big(\int_0^{t \wedge  \sigma_{n} } \sigma_{ij}^2(X_{u })  du \Big)^{1/2}  \Big ]  \leq   C_1   \E_x \Big [  \int_0^{t \wedge  \sigma_{n} } \sigma_{ij}^2(X_{u })  du   \Big ]^{1/2},
\end{eqnarray*}
where $C_1>0$ is a universal constant. Using Theorem \ref{theo:3.3}(i) and  
the inequalities $(a+b+c)^2\le 3(a^2+b^2+c^2)$, $\sqrt{a+b+c}\le \sqrt{a}+\sqrt{b}+\sqrt{c}$ and $\sqrt{a} \leq a +1/4$ which hold for $a,b,c\ge 0$,
\begin{eqnarray*}
&& \E_x \Big [  \int_0^{t \wedge  \sigma_{n} } \sigma_{ij}^2(X_{u })  du   \Big ]^{1/2}  \leq \E_{x} \Big [ 3 \int_0^{t \wedge \sigma_n} \Big (|h_1^2(X_u)| + M^2\|X_u\|+M^2\Big ) du \Big]^{1/2}\\
&& \leq  \sqrt{3} \E_x \Big [ \int_0^ T h_1^2(X_u) du\Big]^{1/2}+M\sqrt{3} \cdot \E_{x} \Big[ \int_0^{t \wedge \sigma_n}\|X_{u }\| du \Big]^{1/2} +M\sqrt{3T} \nonumber \\
&&\leq  \underbrace{ (3e^T c_{B, q})^{1/2} \|h_1 \|_{L^{2q}(\R^d, \mu)}+M\sqrt{3}\Big (\sqrt{T}+\frac{1}{4}\Big )}_{=:C_2} +M \sqrt{3}   \int_0^ t \E_{x} \Big[ \sup_{0 \leq s \leq u \wedge \sigma_n} \|X_s\| \Big] du. \nonumber
\end{eqnarray*}
Concerning the drift term, we have
\begin{eqnarray*}
&&\E_{x}\Big[\int^{ t \wedge \sigma_{n}}_{0}   |g_i(X_u)| \, du \Big] \leq  \E_{x}\Big[  \int_0^{ t \wedge \sigma_{n}} |h_2(X_u) | du\Big] +M\E_{x} \Big[  \int_0^{t \wedge \sigma_{n}} \big(\|X_{u }\|+1\big)  \,du \Big] \\
&&\leq \E_{x} \Big[ \int_0^T  |h_2(X_u)| du \Big]+MT+  M\E_{x} \Big[ \int_0^{t }  \sup_{0 \leq s \leq u \wedge \sigma_n} \|X_s\| du  \Big] \\
&& \leq  \underbrace{e^T c_{B, q} \|h_2 \|_{L^q(\R^d, \mu)}+MT}_{=:C_{3}}+  M\int_0^{t }\E_{x} \Big[  \sup_{0 \leq s \leq u \wedge \sigma_n} \|X_s\|\Big] du.
\end{eqnarray*}
Let $p_n(t):=\E_{x} \left[  \sup_{0 \leq s \leq t\wedge \sigma_{n}}\|X_{s}\| \right]$. Then
\begin{eqnarray*}
p_n(t) \leq \sum_{i=1}^{d} \sup_{0 \leq s \leq t \wedge \sigma_{n}} |X_s^i| \leq \underbrace{ \sqrt{d} \|x\| + d(lC_1C_2+C_{3})}_{=:D} + \underbrace{Md(\sqrt{3}lC_1+1)}_{=:E} \int_0^t p_n(u) du.
\end{eqnarray*}
By Gronwall's inequality, 
\begin{equation} \label{eq:3.21}
p_n(t) \leq D\cdot e^{E\cdot t}\;\;  \text{ $\forall t \in [0,T]$. } 
\end{equation}
Using Markov's inequality and \eqref{eq:3.21},
\begin{eqnarray*}
\P_x(\sigma_n \leq T) &\leq& \P_{x} \Big(\sup_{s \leq T \wedge \sigma_n}|X_{s} | \geq n   \Big)  \\
&\leq& \frac{1}{n} \E_{x} \Big[  \sup_{s \leq T \wedge \sigma_n}|X_{s} | \Big] \leq \frac{1}{n} D\cdot e^{E \cdot T} \rightarrow 0  \;\text{ as } n \rightarrow \infty.
\end{eqnarray*}
Therefore, letting $T \rightarrow \infty$, $\M$ is non-explosive by Lemma \ref{lem:3.1.4}(i). Applying Fatou's lemma to \eqref{eq:3.21} and taking the supremum over $\overline{B}$, the last assertion follows. \\
(ii) By Jensen's inequality, for $i \in \{1, \ldots, d \}$, $t \in [0, T]$ and $\sigma_n, n\in \N$, as in Definition \ref{stopping times}
\begin{eqnarray*} 
&&\sup_{0 \leq s \leq t \wedge \sigma_{n}} |X_s^i|^2 \\
&\leq& (l+2) \Big(x_i^2  + \sum_{j=1}^l  \Big(\sup_{\;0 \leq s \leq t \wedge \sigma_{n} } \Big | \int_0^{s} \sigma_{ij} (X_u) \, dW_u^j \Big |\Big)^2 + t\int^{ t \wedge \sigma_{n} }_{0}   |g_i(X_u)|^2 \, du \Big). 
\end{eqnarray*}
Using Doob's inequality (\cite[II. (1.7) Theorem]{RYor}), for any $j=1,\ldots,l$,
\begin{eqnarray*}
\E_x \Big [  \Big(\sup_{\;0 \leq s \leq t \wedge \sigma_{n} }  \Big |\int_0^{s} \sigma_{ij} (X_u) \, dW_u^j \Big |\Big)^2   \Big ] &\leq& 4 \E_x \Big [    \int_0^{t \wedge \sigma_n  } \sigma_{ij}^2(X_{u })  du   \Big ]. \\
\end{eqnarray*}
By Theorem \ref{theo:3.3}(i),
\begin{eqnarray*}
&& \E_x \Big[  \int_0^{t \wedge  \sigma_{n} } \sigma_{ij}^2(X_{u })  du   \Big]  \leq \E_{x} \Big[ 3 \int_0^{t \wedge \sigma_n} \big(|h_1^2(X_u)| + M^2\|X_u\|^2+M^2 \big)du \Big]  \\
&&\;\; \leq \underbrace{(3e^T c_{B,\frac{p}{2}}) \|h_1^2 \|_{L^{p/2}(\R^d, \mu)} + 3M^2T}_{=:c_1}+3M^2\int_0^{t }\E_{x} \Big[  \sup_{0 \leq s \leq u \wedge \sigma_n} \|X_s\|^2\Big] du.
\end{eqnarray*}
and 
\begin{eqnarray*}
&&\E_{x}\Big[\int_0^{ t \wedge \sigma_{n}} |g_i(X_u)|^2du\Big] \leq  c_1+  3M^2\int_0^{t }\E_{x} \Big[  \sup_{0 \leq s \leq u \wedge \sigma_n} \|X_s\|^2\Big] du.
\end{eqnarray*}
Let $p_n(t):=\E_{x} \Big[  \sup_{0 \leq s \leq t\wedge \sigma_{n}}\|X_{s}\|^2 \Big]$. Then
\begin{eqnarray*}
&&p_n(t) \leq \sum_{i=1}^{d} \E_x \Big[\sup_{0 \leq s \leq t \wedge \sigma_{n}} |X_s^i|^2\Big] \\
&&\leq \underbrace{(l+2) \Big( \|x\|^2 + 4c_1dl+c_1 dT \Big)}_{=:D} + \underbrace{3M^2d(l+2)(4l+T)}_{=:E} \int_0^t p_n(u) du.
\end{eqnarray*}
Using Gronwall's inequality, Markov's inequality and Jensen's inequality, the rest of the proof follows similarly to the proof of (i).
\end{proof}
\text{}\\
There are examples where the Hunt process $\M$ of Theorem \ref{th: 3.1.2} is non-explosive but \eqref{eq:3.20} is not satisfied and $\mathbf{G}$ has infinitely many singularities which form an unbounded set.
\begin{eg}\label{exam:3.2.1.4}
\begin{itemize}
\item[(i)]
Let $\eta \in C_0^{\infty}(B_{1/4})$ with $\eta\equiv 1$ on $B_{1/8}$ and define
$$
w(x_1, \dots, x_d):= \eta(x_1, \dots, x_d) \cdot  \int_{-\infty}^{x_1} \frac{1}{|t|^{1/d}} 1 _{[-1,1]}(t) dt, \;\; (x_1, \ldots, x_d) \in \R^d.
$$
Then $w \in H^{1,q}(\R^d) \cap C_0(B_{1/4})$ with $\partial_1 w \notin L_{loc}^{d}(\R^d)$. Let
$$
\varv(x_1, \dots, x_d):=\sum_{i=0}^{\infty} \frac{1}{2^i} w(x_1-i, \dots, x_d), \;\; (x_1, \ldots, x_d) \in \R^d.
$$
Then $\varv \in H^{1,q}(\R^d) \cap C(\R^d)$ with $\partial_1 \varv \notin L_{loc}^{d}(\R^d)$. Now define $P=(p_{ij})_{1 \leq i, j \leq d}$ as
$$
p_{1d}:=\varv, \;\; p_{d1}:=-\varv,\;\quad  p_{ij}:=0 \text{ if }\;(i,j) \notin \{(1,d), (d,1)\}.
$$
Let $Q=(q_{ij})_{1 \leq i, j \leq d}$ be a matrix of functions with $q_{ij}=-q_{ij} \in H^{1,2}_{loc}(\R^d) \cap C(\R^d)$ for all $1\leq i,j \leq d$ and assume there exists a constant $M>0$ such that
$$
\| \nabla Q \| \leq M( \|x\|+1), \quad \text{ for a.e. on } \R^d. 
$$
Let $A:=id$, $C:=(P+Q)^T$ and $\mathbf{H}:=0$. Then $\mu:=dx$ is an infinitesimally invariant measure for $(L, C_0^{\infty}(\R^d))$ and $\mathbf{G}=\frac12\nabla (A+C^T)= \frac12(\partial_1 \varv \,\mathbf{e}_1 + \nabla Q$). Observe that  
$\partial_1 \varv$ is unbounded in a neighborhood of  infinitely many isolated points that form an unbounded set and moreover $\nabla Q$ is a locally bounded vector field which has linear growth. By Proposition \ref{theo:3.2.8}(i), $\M$ is non-explosive.
\item[(ii)] \ Let $\gamma \in (0,1)$, $\psi(x):=\|x\|^{\gamma}$, $x \in B_{1/4}$ and $p:=\frac{d}{1-\frac{\gamma}{2}}>d$. Then since $p(1-\gamma)<d$, $\psi \in H^{1,p}(B_{1/4})$ with $\nabla \psi(x)=\frac{\gamma}{\|x\|^{1-\gamma}} \frac{x}{\|x\|}$. By  \cite[Theorem 4.7]{EG15}, $\psi$ can be extended to a function $\psi \in  H^{1,p}(\R^d)$ with $\psi \geq 0$ and $\text{supp}(\psi) \subset B_{1/2}$.
Let 
$$
\rho(x):=1+\sum_{k=0}^{\infty}\psi(x-k\mathbf{e}_1), \;\;\; x \in \R^d.
$$
Then $\rho \in H_{loc}^{1,p}(\R^d) \cap C_b(\R^d)$ with $\rho(x) \geq 1$ for all $x \in \R^d$ and $\|\nabla \rho\| \notin \cup_{r \in [1, \infty]}L^r(\R^d)$. Let $\mu:=\rho dx$. Since $\rho$ is bounded above, there exists $c>0$ such that $\mu(B_r) \leq cr^d$ for all $r>0$. Let $\mathbf{F} \in L^p_{loc}(\R^d, \R^d)$ be such that
$$
\int_{\R^d} \langle \mathbf{F}, \nabla \varphi \rangle dx = 0, \quad \text{ for all } \varphi \in C_0^{\infty}(\R^d)
$$
and for some $M>0$, $N_0 \in \N$, assume $\|\mathbf{F}(x)\| \leq M\|x\|(\ln\|x\|+1)$ for a.e. $x \in \R^d \setminus \overline{B}_{N_0}$. Let $A:=id$, $\mathbf{B}:=\frac{\mathbf{F}}{\rho}$ and $\mathbf{G}:=\beta^{\rho,A}+\mathbf{B}=\frac{\nabla \rho}{\rho}+\frac{\mathbf{F}}{\rho}$. Then $\frac{\nabla \rho}{\rho}$ is unbounded in a neighborhood of infinitely many isolated points, whose union forms an unbounded set and $|\langle\mathbf{B}(x),x\rangle| \leq M\|x\|^{2}(\ln\|x\|+1)$ for a.e. $x \in \R^d \setminus \overline{B}_{N_0}$. Therefore \eqref{eq:3.20} is not satisfied and $\mathbf{G}$ does also not satisfy the condition of Proposition \ref{theo:3.2.8}(i) and (ii). But by the following Proposition \ref{prop:3.2.9}(i), $(T_t)_{t>0}$ is conservative. Hence, $\M$ is non-explosive by Corollary \ref{cor:3.2.1}.
\end{itemize}
\end{eg}
\bigskip
So far, we proved non-explosion criteria by probabilistic means. The following proposition, which is an immediate consequence of \cite[Corollary 15(i) and (iii)]{GT17}, completes Proposition \ref{prop:2.1.10} in the sense that the conservativeness is proven by purely analytical means and that it is applied in the situation of  Section \ref{subsec:2.1.4}, where the density $\rho$ is explicitly given, in contrast to the situation of Theorem \ref{theo:2.2.7}, where $\rho$ is constructed and not known explicitly, except for its regularity properties.
\begin{proposition} 
\label{prop:3.2.9}
Suppose that the assumptions \eqref{condition on mu}--\eqref{eq:2.1.4}  of Section \ref{subsec:2.1.1} 
are satisfied and that the given density $\rho$ additionally satisfies $\rho>0$ a.e.
(Both hold for instance in the situation of Remark \ref{rem:2.2.4} which includes condition {\bf (a)} of Section \ref{subsec:2.2.1}). Suppose further that there exist constants $M, c>0$ and $N_0, N_1 \in \N$, such that either of the following holds:
\begin{itemize}
\item[(i)]
$$
\frac{\langle A(x)x,x\rangle}{\|x\|^2}+|\langle{\mathbf{B}(x)},x\rangle| \leq M\|x\|^{2} \ln(\|x\|+1),
$$
for a.e. $x\in \R^d\setminus \overline{B}_{N_0}$ and 
$$
\mu(B_{4n}\setminus B_{2n})\le (4n)^{c}\qquad \forall n\ge N_1.
$$ 
\item[(ii)] 
$$
\langle A(x)x,x\rangle+|\langle\mathbf{B} (x),x\rangle| \leq M\|x\|^{2},
$$
for a.e. $x\in \R^d\setminus \overline{B}_{N_0}$ and 
$$
\mu(B_{4n}\setminus B_{2n})\le e^{c(4n)^2}\qquad \forall n\ge N_1.
$$ 
\end{itemize}
Then $(T_t)_{t>0}$ and $(T'_t)_{t>0}$ are conservative (cf. Definitions \ref{definition2.1.7} and \ref{def:3.2.1}) and $\mu$ is $(\overline{T}_t)_{t>0}$-invariant and $(\overline{T}'_t)_{t>0}$-invariant (cf. Definition \ref{def:2.1.1}(ii)).
\end{proposition}

\begin{remark}\label{rem:3.2.1 1}
{\it Recall that the drift has the form $\mathbf{G} =\beta^{\rho , A}+\mathbf{B} $. Proposition \ref{prop:3.2.9} is a type of conservativeness result, where a growth condition on the logarithmic derivative $\beta^{\rho,A}$ is not explicitly required.
Instead, a volume growth condition on the infinitesimally invariant measure $\mu$ occurs.
Such types of conservativeness results have been studied systematically under more general assumptions on the coefficients.  
In the symmetric case in  \cite{Ta89} and \cite{Sturm94}, in the sectorial case in \cite{TaTr}, and in the possibly non-sectorial case in \cite{GT17}. 
In Proposition \ref{prop:3.2.9} there is an interplay between the growth conditions on $A$ and $\mathbf{B}$ and the growth condition of $\mu$ on annuli. The stronger conditions on $A$ and $\mathbf{B}$ in Proposition \ref{prop:3.2.9}(ii) allow for the weaker exponential growth condition of $\mu$ on annuli. In particular, an exponential growth condition as in Proposition \ref{prop:3.2.9}(ii) already appears in \cite{Ta89}.
The exponential growth condition of \cite{Ta89} can even be slightly relaxed in the symmetric case (see \cite[Remarks b), p. 185]{Sturm94}).
For instance, if according to Remark \ref{rem:2.2.4}, $A=id$, $C=0$, and $\mathbf{\overline{B}}=0$ (hence $\mathbf{B}=0$) and $\rho=e^{2\phi}$, with $\phi \in H^{1,p}_{loc}(\R^d)$ for some $p\in (d,\infty)$, then $\mathbf{G}=\nabla \phi$, $\mu=e^{2\phi}dx$, and by \cite[Theorem 4]{Sturm94}, $(T_t)_{t>0}$ is conservative and $\mu$ is $(\overline{T}_t)_{t>0}$-invariant, if there exists a constant $M>0$ and $N_0 \in \N$, such that
\begin{equation} \label{eq:3.2.27}
\phi (x) \leq M \|x\|^2 \ln(\|x\|+1), \;\quad \forall x \in \R^d \setminus \overline{B}_{N_0}.
\end{equation}
Indeed, in this case the intrinsic metric equals the Euclidean metric (see \cite[4.1 Theorem]{Sturm98} and its proof), so that \cite[{\bf Assumption (A)}]{Sturm94} is satisfied.
Moreover, $\mu(B_{r})\le e^{cr^{2} \ln(r+1)}$, for some constant $c$ and $r$ large, which further implies 
$$
\int_1^{\infty}\frac{r}{\ln \mu(B_r)} dr =\infty.
$$
Thus the result follows by \cite[Theorem 4]{Sturm94}.}
\end{remark}

\begin{eg}
We saw in Example \ref{exam:3.2.1.4}(ii) that the criterion Proposition \ref{prop:3.2.9}(i) was not covered by any other non-explosion or conservativeness result of this monograph. The same is true for Proposition \ref{prop:3.2.9}(ii) and the criterion \eqref{eq:3.2.27}.
Here we only show the latter. Hence let $A=id$, $C=0$, $\mathbf{\overline{B}}=0$.
Let $\phi_1(x)=1+\sum_{k=0}^{\infty}\psi(x-k\mathbf{e}_1)$, $x \in \R^d$, where $\psi$ is defined as in Example \ref{exam:3.2.1.4}(ii), $\phi_2(x)=(\|x\|+1)^2 \ln(\|x\|+1)$, $x \in \R^d$, $\phi=\phi_1+\phi_2$, and $\mu=\exp(2\phi)dx$. Then $\phi$ satisfies \eqref{eq:3.2.27} and the associated semigroup $(T_t)_{t>0}$ is conservative and $\M$ is non-explosive by Remark \ref{rem:3.2.1 1}. Indeed the volume growth of $\mu$ is as described at the end of in Remark  \ref{rem:3.2.1 1} and  the drift coefficient $\mathbf{G}$ consists of $\nabla \phi_1$, which has infinitely many singularities that form an unbounded set in $\R^d$ and of $\nabla \phi_2(x)=(\|x\|+1)(2\ln(\|x\|+1)+1) \frac{x}{\|x\|}$, which has linear times logarithmic growth. 
Hence, Corollary \ref{cor:3.2.2}, Proposition \ref{theo:3.2.8} and Proposition \ref{prop:3.2.9} cannot be used to determine the conservativeness of $(T_t)_{t>0}$.
\end{eg}

\subsubsection{Transience and recurrence}
\label{subsec:3.2.2} 
Throughout this section we will assume that {\bf (a)} of Section \ref{subsec:2.2.1} holds and that assumption {\bf (b)} of Section \ref{subsec:3.1.1}  holds. And we let
$$
\mu=\rho\,dx
$$ 
be as in Theorem \ref{theo:2.2.7} or as in  Remark \ref{rem:2.2.4}. \\ For $f \in L^1(\R^d, \mu)$ with $f \geq 0$, define through the following pointwise increasing limit
$$
Gf:=\int_0^{\infty} T_t f dt = \lim_{\alpha \rightarrow 0+} \int_{0}^{\infty} e^{-\alpha t} T_t f dt  = \lim_{\alpha \rightarrow 0+} G_{\alpha} f, \;\; \text{$\mu$-a.e.},
$$
where $(T_t)_{t>0}$ and $(G_{\alpha})_{\alpha>0}$ are defined in Definition \ref{definition2.1.7}.
\begin{definition} \label{def:3.2.2.2}
Assume {\bf (a)} of  Section \ref{subsec:2.2.1} holds. $(T_t)_{t>0}$ see Definition \ref{definition2.1.7}) is called {\bf recurrent}\index{recurrent ! semigroup}, if for any $f \in L^1(\R^d, \mu)$ with $f \geq 0$ \, $\mu$-a.e,
$$
G f(x) \in\{0,\infty\} \;\; \text{for $\mu$-a.e. $x \in \R^d$}.
$$
$(T_t)_{t>0}$ is called {\bf transient}\index{transient ! semigroup}, if there exists $g \in L^1(\R^d, \mu)$ with $g> 0$  $\mu$-a.e. such that
$$
G g(x) <\infty, \;\;  \text{for $\mu$-a.e. $x \in \R^d$.}
$$
\end{definition}
\noindent
Note that by \cite[Remark 3(a)]{GT2}, 
$(T_t)_{t>0}$ is transient, if and only if for any $f \in L^1(\R^d, \mu)$ with $f \geq 0$ $\mu$-a.e.
$$
G f(x) <\infty,\;\ \text{ for $\mu$-a.e. $x \in \R^d$. }
$$
For $x \in \R^d$ and $f \in L^1(\R^d, \mu)$ with $f \geq 0$, define for $(P_t)_{t>0}$ of Proposition \ref{prop:3.1.1} and $\M$ of Theorem \ref{th: 3.1.2} (see also Theorem \ref{theo:3.1.4}),  
\begin{eqnarray*}
R f(x):&=& \int_0^{\infty}P_t f(x) dt=\E_x \left[ \int_0^{\infty} f(X_t) dt  \right] = \lim_{n \rightarrow \infty} \E_x \left[ \int_0^{\infty} (f \wedge n)(X_t) dt  \right] \\
&=&\lim_{n \rightarrow \infty} \lim_{\alpha \rightarrow 0+} \E_x \left[ \int_0^{\infty} e^{-\alpha t}(f \wedge n)(X_t) dt  \right]
= \lim_{n \rightarrow \infty} \lim_{\alpha \rightarrow 0+} \Big( R_{\alpha} (f \wedge n) (x) \Big). 
\end{eqnarray*}
Since $Rf$ is the pointwise increasing limit of lower semi-continuous functions, $R f$ is lower semi-continuous on $\R^d$ by Theorem \ref{theorem2.3.1}. In particular, for any $f, g \in L^1(\R^d, \mu)$ with $f=g\ge 0$, $\mu$-a.e. it holds that $R f(x)= R g(x)$ for all $x \in \R^d$. Moreover, 
$$
Rf(x) = Gf(x), \;\; \text{ for $\mu$-a.e. $x \in \R^d$.}
$$
Define the last exit time $L_A$ from $A\in \mathcal{B}(\R^d)$ by
\begin{equation}\label{lastexittime}
L_A:=\sup \{ t\geq 0: X_t\in A\},\ \ (\sup \emptyset:=0).
\end{equation}
\begin{definition} \label{def:3.2.2.3}
Assume {\bf (a)} of 
Section \ref{subsec:2.2.1} and {\bf (b)} of Section \ref{subsec:3.1.1} hold.
$\M$ (see Theorem \ref{th: 3.1.2}  and also Theorem \ref{theo:3.1.4})  is called {\bf recurrent in the probabilistic sense}\index{recurrent ! in the probabilistic sense}, if for any non-empty open set $U$ in $\R^d$
\begin{eqnarray*}
\P_x(L_U=\infty)=1,\ \ \forall x\in \R^d.
\end{eqnarray*}
$\M$ is called {\bf transient in the probabilistic sense}\index{transient ! in the probabilistic sense}, if for any compact set $K$ in $\R^d$,
\begin{equation*}
\P_x(L_K<\infty)=1,\ \ \forall x\in \R^d.
\end{equation*}
\end{definition}

\begin{proposition} \label{prop:3.2.2.11} 
Assume {\bf (a)} of 
Section \ref{subsec:2.2.1} and {\bf (b)} of Section \ref{subsec:3.1.1} hold.
$\M$ is transient in the probabilistic sense, if and only if 
\begin{equation} \label{eq:3.25}
\P_x(\lim_{t \rightarrow \infty}X_t=\Delta)=1,\ \ \; \forall x\in \R^d.
\end{equation}
In particular, if $\M$ is transient, then 
\begin{equation*} \label{eq:3.25*}
\lim_{t \rightarrow \infty} P_t f(x) =0
\end{equation*}
for any $x \in \R^d$ and $f \in \mathcal{B}_b(\R^d)_0 + C_{\infty}(\R^d)$. 
\end{proposition}

\begin{proof}
Let $x \in \R^d$ and $K_n:=\overline{B}_n(x)$, $n \in \N$.  Let $\Omega_0:=\cap_{n \in \N} \{\omega \in \Omega:   L_{K_n}(\omega)<\infty\}$. Then it follows that
\begin{equation*}
\Omega_{0}= \{\omega \in \Omega:  \lim_{t \rightarrow \infty}X_t(\omega)=\Delta   \},\;\; \text{ $\P_x$-a.s. }
\end{equation*}
Assume that $\M$ is transient in the probabilistic sense. Then $\P_x(\Omega_0)=1$, hence \eqref{eq:3.25} holds. Conversely, assume \eqref{eq:3.25} holds. Then $\P_x(L_{K_n}<\infty)$ for all $n \in \N$. Let $K$ be a compact set in $\R^d$. Then there exists $N \in \N$ such that $K \subset K_N$, hence $\P_x(L_K <\infty)=1$. Thus, $\M$ is transient. \\
Now assume that $\M$ is transient. Let $x \in \R^d$ and $f \in  \mathcal{B}_b(\R^d)_0+C_{\infty}(\R^d)$. Then, $P_t f(x)=\E_x[f(X_t)]$ by Proposition \ref{prop:3.1.4}. Since $f$ is bounded and $\lim_{t \rightarrow \infty} f(X_t)=0$ \,$\P_x$-a.s. by \eqref{eq:3.25}, it follows from Lebesgue's theorem that 
$$
\lim_{t \rightarrow \infty} \E_x[f(X_t)] =0,
$$  
as desired.
\end{proof}

\begin{lemma} \label{lem:3.2.6}
Assume {\bf (a)} of 
Section \ref{subsec:2.2.1} and {\bf (b)} of Section \ref{subsec:3.1.1} hold and let $\M$ be as in Theorem \ref{th: 3.1.2}.
Assume that $\Lambda \in \mathcal{F}$ is $\vartheta_t$-invariant for some $t>0$, i.e. $\Lambda=\vartheta^{-1}_t(\Lambda)$.  
Then $x \mapsto \P_x(\Lambda)$ is continuous on $\R^d$.
\end{lemma}
\begin{proof}
By the Markov property,
$$
\mathbb{P}_x(\Lambda)  =  \mathbb{P}_x(\vartheta_t^{-1}(\Lambda))= \mathbb{E}_x[\mathbb{E}_x[ 1_{\Lambda}\circ \vartheta_t \,|\, \mathcal{F}_t]]
=  \mathbb{E}_x[\mathbb{E}_{X_t}[ 1_{\Lambda}]]=P_t \mathbb{P}_{\cdot}(\Lambda)(x).
$$
Since  $x \mapsto \P_x(\Lambda)$ is bounded and measurable, the assertion follows by Theorem \ref{theo:2.6}.
\end{proof}

\begin{theorem}\label{theo:3.3.6}
Assume {\bf (a)} of 
Section \ref{subsec:2.2.1} and {\bf (b)} of Section \ref{subsec:3.1.1} hold.
We have the following:
\begin{itemize}
\item[(i)] $(T_t)_{t>0}$ is either recurrent or transient (see Definition \ref{def:3.2.2.2}).
\item[(ii)] $(T_t)_{t>0}$ is transient, if and only if $\M$ is  transient in the probabilistic sense (see Definition \ref{def:3.2.2.3}).
\item[(iii)] $(T_t)_{t>0}$ is recurrent, if and only if $\M$ is recurrent in the probabilistic sense.
\item[(iv)] $\M$ is either recurrent or transient in the probabilistic sense.
\item[(v)] If $(T_t)_{t>0}$ is recurrent, then $(T_t)_{t>0}$ is conservative.
\end{itemize}
\end{theorem}

\begin{proof}
(i) Since $(T_t)_{t>0}$ is strictly irreducible by Proposition \ref{prop:2.4.2}(ii), $(T_t)_{t>0}$ is either recurrent or transient by \cite[Remark 3(b)]{GT2}.\\
(ii) If $(T_t)_{t>0}$ is transient, then by  \cite[Lemma 6]{GT2} there exists $g \in L^{\infty}(\R^d, \mu)$ with $g(x)>0$ for all $x\in \R^d$ such that 
$$
R g =\E_{\cdot}\Big [\int_0^{\infty} g(X_t) dt\Big ]\in L^{\infty}(\R^d, \mu).
$$
Since $Rg$ is lower-semicontinuous
$$
V :=\big \{Rg-\|Rg\|_{L^{\infty}(\R^d, \mu)}>0\big \}
$$
is open. Since $\mu(V)=0$ and $\mu$ has full support, we must have that $V=\varnothing$. It follows that $Rg\le\|Rg\|_{L^{\infty}(\R^d, \mu)}$ pointwise so that the adapted process $t\mapsto Rg(X_t)$ is $\mathbb{P}_x$-integrable for any $x\in \R^d$. Using the Markov property,
for any $0 \leq s <t$ and $x\in \R^d$
\begin{eqnarray*}
\E_x \left[ R g(X_{t})  \vert \mathcal{F}_s \right] &=& \E_{X_s} \left[ R g(X_{t-s})\right] \\
&=&   P_{t -s} R g(X_s)\  =\  \int_{t-s}^{\infty} P_u g(X_s) du  \ \leq \  R g(X_s)
\end{eqnarray*}
and moreover since $\M$ is a normal Markov process with right-continuous sample paths, we obtain $Rg(x)>0$ for any $x\in \R^d$. Thus, $(\Omega , \mathcal{F}, (\mathcal{F}_t)_{t \geq 0}, \left ( Rg(X_t)\right )_{t \geq 0}, \P_x)$ is a positive supermartingale for all $x\in \R^d$. \\
Let $U_n:=\{Rg>\frac{1}{n}\}$, $n \in \N$. Since $Rg$ is lower-semicontinuous, $U_n$ is open in $\R^d$ and since $Rg>0$ everywhere, $\R^d = \cup_{n \in \N} U_n$. \\
Let $K\subset \R^d$ be an arbitrary compact set. Since $\{U_n\cap B_n\}_{n\in \N}$ is an open cover of $K$, there exists $N\in \N$ with $K\subset V_N:=U_N\cap B_N$. Since $\overline{V}_N$ is compact and  $\{U_n\}_{n\in \N}$ is an open cover of  $\overline{V}_N$ there exists $M\in \N$, with 
$$
K\subset  V_N   \subset \overline{V}_N \subset U_M. 
$$
By the optional stopping theorem for positive supermartingales, for $t\ge 0$ and $x \in \R^d$ and $\sigma_{V_N}$ as in Definition \ref{stopping times}
\begin{eqnarray*}
P_t Rg(x)&=&\E_x\left[Rg(X_t)\right] \geq \E_x\big [ Rg(X_{t+\sigma_{V_N} \circ \vartheta_t}) \big ] \\
&\geq& \E_x\big [ Rg(X_{t+\sigma_{V_N}  \circ \vartheta_t}) 1_{\{ t+\sigma_{V_N}  \circ \vartheta_t<\infty \}} \big] \\
&\geq& \frac{1}{M} \cdot \P_x\left( t+\sigma_{V_N}  \circ \vartheta_t<\infty \right),
\end{eqnarray*}
and the last inequality holds since $X_{t+\sigma_{V_N}  \circ \vartheta_t} \in \overline{V}_N$, $\P_x$-a.s. on $\{ t+\sigma_{V_N}  \circ \vartheta_t<\infty \}$. Consequently,
$$
\lim_{t\rightarrow \infty}\P_x\left( t+\sigma_{V_N}  \circ \vartheta_t<\infty \right)\leq M \cdot \lim_{t\rightarrow \infty} P_t R g(x)=0,
$$
hence $\P_x\left( t+\sigma_{V_N}  \circ \vartheta_t<\infty  \text{ for all } t>0 \right)=0$ for any $x \in \R^d$. Therefore,
$$
1=\P_x\left( t+\sigma_{V_N}  \circ \vartheta_t=\infty  \text{ for some } t>0 \right)=\P_x(L_{V_N} <\infty).
$$
Since $L_K \leq L_{V_N}<\infty$ \,$\P_x$-a.s. for all $x\in \R^d$ (cf. \eqref{lastexittime} for the definition of $L_A$), we obtain the transience of $\M$ in the probabilistic sense.\\ 
Conversely, assume that $\M$ is transient in the probabilistic sense. Then condition (8) of \cite[Proposition 10]{GT2} holds with $B_n$ being the Euclidean ball of radius $n$ about the origin. Consequently, by  \cite[Proposition 10]{GT2}, there exists
$g \in L^1(\R^d, \mu)$ with $g>0$ $\mu$-a.e., such that $Rg(x)<\infty$ for $\mu$-a.e. $x\in \R^d$. Since $Rg$ is a $\mu$-version of $Gg$, we obtain that $(T_t)_{t>0}$ is transient.\\
(iii) Assume that $(T_t)_{t>0}$ is recurrent. Let $U$ be a nonempty open set in $\R^d$. Then $U$ is not $\mu$-polar and finely open. Thus by \cite[Proposition 11(d)]{GT2}, $\P_x(L_U<\infty)=1$ for $\mu$-a.e. $x \in \R^d$. Since $\{L_U<\infty \}\in \mathcal{F}$ is $\vartheta_t$-invariant for all $t>0$, it follows from Lemma \ref{lem:3.2.6} that
$$
\P_x(L_U<\infty)=1, \; \text{ for all } x \in \R^d
$$
as desired. \\
Conversely, if $\M$ is recurrent in the probabilistic sense, then $\M$ cannot be transient in the probabilistic sense. Thus $(T_t)_{t>0}$ cannot be transient by (ii). Therefore $(T_t)_{t>0}$ is recurrent by (i).\\
(iv) The assertion follows from (i), (ii) and (iii).\\
(v) This follows from \cite[Corollary 20]{GT2}.
\end{proof}

\begin{lemma} \label{lem:3.2.8}
Assume {\bf (a)} of 
Section \ref{subsec:2.2.1} and {\bf (b)} of Section \ref{subsec:3.1.1} hold. Let $\M$ be as in Theorem \ref{th: 3.1.2} (see also Theorem \ref{theo:3.1.4})
For any $x \in \R^d$ and $N\in \N$, it holds that $\P_x(\sigma_{N}<\infty)=1$, where $\sigma_N$, $N\in \N$, is as in Definition \ref{stopping times}.
\end{lemma}
\begin{proof}
Let $x \in \R^d$ and $N \in \N$. If $x \in \R^d \setminus \overline{B}_N$, then $\P_x(\sigma_N=0)=1$. Assume that $x \in \overline{B}_{N}$. Since $\M$ is either recurrent or transient in the probabilistic sense by Theorem \ref{theo:3.3.6}, it follows that $\P_x(L_{\R^d \setminus \overline{B}_{N}}=\infty)=1$ or $\P_x(L_{\overline{B}_{N}}<\infty)=1$ (cf. \eqref{lastexittime} for the definition of $L_A$), hence the assertion follows.
\end{proof}
\noindent
The following criterion to obtain the recurrence of $\M$ in the probabilistic sense is proven by a well-known technique which involves stochastic calculus (see for instance \cite[Theorem 1.1, Chapter 6.1]{Pi}), but we ultimately use our results, Lemma \ref{lem:3.2.8}, Theorem \ref{theo:3.3.6}(iv) and the claim of Lemma \ref{lem3.2.6}, so that, 
in contrast to \cite{Pi}, also the case of a locally unbounded drift coefficient can be treated.
\begin{proposition} \label{theo:3.2.6}
Assume {\bf (a)} of 
Section \ref{subsec:2.2.1} and {\bf (b)} of Section \ref{subsec:3.1.1} hold.
Let $N_0 \in \N$. Let $g\in C^2(\R^d \setminus \overline{B}_{N_0}) \cap C(\R^d)$, $g \geq 0$, with
\begin{equation*} 
\lim_{r \rightarrow \infty} (\inf_{\partial B_r} g) = \infty.
\end{equation*}
and assume that
$$
Lg\le 0 \quad \text{a.e. on }\, \R^d\setminus \overline{B}_{N_0}.
$$ 
Then $\M$ is recurrent in the probabilistic sense (see Definition \ref{def:3.2.2.3}).
\end{proposition}
\begin{proof}
By the claim of Lemma \ref{lem3.2.6}, there exists $N_1 \in \N$ with $N_1 \geq N_0+2$ and $\psi \in C^2(\R^d)$ with $\psi(x) \geq 0$ for all $x \in \R^d$, $\psi(x) =g(x)$ for all $x \in \R^d\setminus B_{N_1}$, such that
\begin{eqnarray}\label{eq:3.2.28a}
L \psi \leq 0 \quad \text{a.e. on }\, \R^d\setminus \overline{B}_{N_1}.
\end{eqnarray}
In particular, $\M$ is non-explosive by Lemma \ref{lem3.2.6}. We first show the following claim.\\
{\bf Claim:} Let $n \geq N_1$ and $x \in \R^d \setminus \overline{B}_n$ arbitrary. Then $\P_x(\sigma_{B_n}<\infty)=1$  (for $\sigma_{B_n}$ see Definition \ref{stopping times}). \\
To show the claim, choose any  $N \in \N$, with $x \in B_N$. By It\^{o}'s formula and Theorem \ref{theo:3.1.4}(i), $\P_x$-a.s. for any $t \in [0, \infty)$
$$
\psi(X_{t \wedge \sigma_{B_n}  \wedge \sigma_N})- \psi(x)  =\int_0^{t \wedge\sigma_{B_n}  \wedge \sigma_N} \nabla \psi \cdot \sigma(X_s) dW_s +\int_0^{t \wedge\sigma_{B_n}  \wedge \sigma_N}  L \psi (X_s) ds,
$$
where $\sigma=(\sigma_{ij})_{1 \leq i,j \leq d}$ is as in Lemma \ref{lem:3.1.5} and $\sigma_{N}$ as in Definition \ref{stopping times}). Taking expectations and using \eqref{eq:3.2.28a}
$$
\E_x \left[ \psi(X_{t \wedge \sigma_{B_n}  \wedge \sigma_N}) \right] \leq \psi(x).
$$
Since $\P_x(\sigma_N<\infty)=1$ by Lemma \ref{lem:3.2.8}, using Fatou's lemma, we obtain that
\begin{eqnarray*}
(\inf_{\partial B_N}\psi)\cdot \P_x(\sigma_{B_{n}}=\infty)&\le& \E_x[\psi(X_{\sigma_N})1_{\{\sigma_{B_{n}}=\infty\}}] \le \E_x[\psi(X_{\sigma_{B_{n}}\wedge\sigma_N})]\\
&\le& \liminf_{t \rightarrow \infty} \E_x \left[ \psi(X_{t \wedge \sigma_{B_n}  \wedge \sigma_N}) \right] \leq  \psi(x).
\end{eqnarray*}
Letting $N \rightarrow \infty$, we obtain $\P_x(\sigma_{B_{n}}=\infty)=0$ and the claim is shown.\\
Now let $x \in \R^d$ and $N_2:=N_1+1$. If $x \in \R^d \setminus \overline{B}_{N_2}$, then $\P_x(\sigma_{B_{N_2}}<\infty)=1$ by the claim. If $x \in B_{N_2}$, then $\P_x(\sigma_{B_{N_2}}=0)=1$, by the continuity and normal property of $\M$. Finally if $x \in \partial B_{N_2}$, then by the claim again $\P_x(\sigma_{B_{N_1}}<\infty)=1$, hence $\P_x(\sigma_{B_{N_2}}<\infty)=1$ since $\sigma_{B_{N_2}}< \sigma_{B_{N_1}}$, $\P_x$-a.s. Therefore, we obtain 
\begin{equation} \label{eq:3.2.29}
\P_x(\sigma_{B_{N_2}}<\infty)=1, \;\; \text{ for all $x \in \R^d$}.
\end{equation}
For $n\in \N$, define
$$
\Lambda_n:= \{ \omega \in \Omega:X_t(\omega)\in B_{N_2} \text{ for some } t\in[n,\infty) \}.
$$
Then $\Lambda_n=\{\omega \in \Omega:\sigma_{B_{N_2}}\circ \vartheta_n<\infty\}$, $\P_x$-a.s. Using the Markov property and that $\M$ is non-explosive, it follows by \eqref{eq:3.2.29} that
\begin{eqnarray*}
\P_x\big (  \sigma_{B_{N_2}}\circ \vartheta_n<\infty \big )  &=& \E_x\big[ 1_{\{\sigma_{B_{N_2}}<\infty\}}  \circ \vartheta_n \big] =  \E_x\big[\E_x\big[ 1_{\{\sigma_{B_{N_2}}<\infty\}} \circ \vartheta_n \big\vert \mathcal{F}_n\big]\big] \\
&=& \E_x\big[  \P_{X_n} \big(  \sigma_{B_{N_2}}<\infty \big) \big]=1.
\end{eqnarray*}
Therefore, $1=\P_x(\cap_{n \in \mathbb{N}} \Lambda_n)=\P_x(L_{B_{N_2}}=\infty)$, hence $\M$ is not transient. By Theorem \ref{theo:3.3.6}(iv), $\M$ is recurrent.
\end{proof}
\noindent
Choosing $g(x):= \ln(\|x\|^2 \vee N_0^2)+2$ as in the proof of Corollary \ref{cor:3.2.2} and using the same method as in the proof of Corollary \ref{cor:3.1.3},  the following result is a direct consequence of Proposition \ref{theo:3.2.6}.
\begin{corollary} \label{cor:3.2.2.5}
Assume {\bf (a)} of 
Section \ref{subsec:2.2.1} and {\bf (b)} of Section \ref{subsec:3.1.1} hold.
Assume that there exists $N_0\in \N$, such that 
\begin{eqnarray*}
-\frac{\langle A(x)x, x \rangle}{ \left \| x \right \|^2 }+ \frac12\mathrm{trace}A(x)+ \big \langle \mathbf{G}(x), x \big \rangle \leq 0
\end{eqnarray*}
for a.e. $x\in \R^d\setminus \overline{B}_{N_0}$. Then $\M$ is recurrent. In particular, if $d=2$ and $\Psi_1$, $\Psi_2$, $Q$, $A$ are as in Corollary \ref{cor:3.1.3}, and
$$
\frac{|\Psi_1(x)-\Psi_2(x)|}{2}  + \langle \mathbf{G}(x),x  \rangle \leq 0
$$
for a.e. $x\in \R^2\setminus \overline{B}_{N_0}$, then $\M$ is recurrent in the probabilistic sense (see Definition \ref{def:3.2.2.3}).
\end{corollary}
\noindent
Using Theorem \ref{theo:3.3.6} we obtain the following corollary of \cite[Theorem 21]{GT2}. 
\begin{proposition} \label{cor:3.3.2.6}
Consider the situation of Remark \ref{rem:2.2.4}. Define for $r\ge 0$,
\begin{equation*}
\varv_1(r):= \int_{B_r} \frac{\langle A(x)x, x \rangle}{\|x\|^2} \mu(dx), \ \ \ \varv_2(r):=\int_{B_r}   \big | \big \langle (\beta^{\rho, C^T}+ \overline{\mathbf{B}})(x),x \big \rangle \big | \mu(dx),
\end{equation*}
and let
$$
\varv(r):=\varv_1(r)+\varv_2(r), \ \ a_n:=\int_1^n \frac{r}{\varv(r)}dr, \ \ n\ge 1.
$$
Assume that
$$
\lim_{n\rightarrow \infty}a_n=\infty \ \ \ \text{and} \ \ \  \lim_{n\rightarrow \infty} \frac{\ln(\varv_2(n)\vee 1)}{a_n}=0.
$$
Then $(T_t)_{t>0}$ and $(T'_t)_{t>0}$ are recurrent (cf. Definitions \ref {def:3.2.2.2} and \ref{definition2.1.7}) 
and $\mu$ is $(\overline{T}_t)_{t>0}$-invariant (cf. Definition \ref{def:2.1.1}).
Moreover, if $\nabla (A+C^T) \in L^q_{loc}(\R^d, \R^d)$, then $\M$ is recurrent in the probabilistic sense (see Definition \ref{def:3.2.2.3}).
\end{proposition}

\begin{proof}
By \cite[Theorem 21]{GT2}  applied with $\rho(x)=\|x\|$ (the $\rho$ of \cite{GT2} is different from our density $\rho$ of $\mu$ defined here), $(T_t)_{t>0}$ is not transient. Hence by Theorem \ref{theo:3.3.6}(i), $(T_t)_{t>0}$ is recurrent. The same applies to $(T'_t)_{t>0}$ by replacing $\beta^{\rho, C^T}+ \overline{\mathbf{B}}$ with $-(\beta^{\rho, C^T}+ \overline{\mathbf{B}})$, hence we obtain that $(T'_t)_{t>0}$ is also recurrent. In particular, $(T'_t)_{t>0}$ is conservative by Theorem \ref{theo:3.3.6}(v), hence $\mu$ is $(\overline{T}_t)_{t>0}$-invariant by Remark \ref{rem:2.1.10a}(i). If $\nabla (A+C^T) \in L^q_{loc}(\R^d, \R^d)$, then $A$, $C$, $\mathbf{H}$ defined in Remark \ref{rem:2.2.4} satisfy {\bf (a)} and  {\bf (b)} of Sections \ref{subsec:2.2.1} and \ref{subsec:3.1.1}. Hence $\M$ is recurrent in the probabilistic sense by Theorem \ref{theo:3.3.6}(iii).
\end{proof}

\subsubsection{Long time behavior: Ergodicity, existence and uniqueness of invariant measures, examples/counterexamples}
\label{subsec:3.2.3} 
Throughout this section we will assume that {\bf (a)} of Section \ref{subsec:2.2.1} holds and that assumption {\bf (b)} of Section \ref{subsec:3.1.1}  holds. Let 
$$
\mu=\rho\,dx
$$ 
be as in Theorem \ref{theo:2.2.7} or as in  Remark \ref{rem:2.2.4}.

\begin{definition}\label{def:invariantforprocess}
Consider a right process
$$
\widetilde{\M} =  (\widetilde{\Omega}, \widetilde{\mathcal{F}}, (\widetilde{\mathcal{F}}_t)_{t \geq 0}, (\widetilde{X}_t)_{t \ge 0}, (\widetilde{\P}_x)_{x \in \R^d\cup \{\Delta\}})
$$  
with state space $\R^d$ (cf. Definition \ref{def:3.1.1}). A $\sigma$-finite measure $\widetilde{\mu}$ on $(\R^d, \mathcal{B}(\R^d))$  is called an {\bf invariant measure} for $\widetilde{\M}$\index{measure ! invariant for $\widetilde{\M}$}, if for any $t \geq 0$
\begin{eqnarray} \label{eq:3.29}
\int_{\R^d} \widetilde{\mathbb{P}}_x(\widetilde{X}_t \in A) \widetilde{\mu}(dx)=  \widetilde{\mu}(A), \quad \text{ for any } A \in \mathcal{B}(\R^d).
\end{eqnarray}
$\widetilde{\mu}$ is called a {\bf sub-invariant measure} for $\widetilde{\M}$\index{measure ! sub-invariant for $\widetilde{\M}$}, if \eqref{eq:3.29} holds with 
\lq\lq$=$\rq\rq \ replaced by \lq\lq$\leq$\rq\rq.
\end{definition}
 
\begin{remark} \label{rem:3.2.3.3} 
{\it Using monotone approximation by simple functions, $\widetilde{\mu}$ is an invariant measure for $\widetilde{\M}$, if and only if 
\begin{eqnarray} \label{eq:3.30}
\int_{\R^d} \widetilde{\mathbb{E}}_x\big[f(\widetilde{X}_t)\big] \widetilde{\mu}(dx)=  \int_{\R^d} f d \widetilde{\mu}, \quad \text{ for any $f \in \mathcal{B}^+_b(\R^d)$},
\end{eqnarray}
where $\widetilde{\E}_x$ denotes the expectation with respect to $\widetilde{\P}_x$. Likewise, $\widetilde{\mu}$ is a sub-invariant measure for $\widetilde{\M}$, if and only if \eqref{eq:3.30}  holds with \lq\lq$=$\rq\rq \ replaced by \lq\lq$\leq$\rq\rq. By the $L^1(\R^d, \mu)$ contraction property of $(T_t)_{t>0}$, $\mu$ (as at the beginning of this section) is always a sub-invariant measure for $\M$. Moreover, $\mu$ is $(\overline{T}_t)_{t>0}$(-sub)-invariant, if and only if $\mu$ is  a (sub-)invariant measure for $\M$ (cf. Definition \ref{def:2.1.1}(ii), Theorem \ref{theo:2.6}, \eqref{semidef} and \eqref{eq:3.1semigroupequal-a.e.})}.
\end{remark}

\begin{lemma} \label{lem:3.2.9}

Assume {\bf (a)} of Section \ref{subsec:2.2.1} and {\bf (b)} of Section \ref{subsec:3.1.1} hold.
$(P_t)_{t \geq 0}$ (cf. Proposition \ref{prop:3.1.1}) is stochastically continuous, i.e.
$$
\lim_{t \rightarrow 0+}P_t(x, B_{r}(x))=1, \quad \text{ for all  $r>0$ and $x \in \R^d$.}
$$
Moreover, for each $t_0>0$, $(P_{t})_{t>0}$ is $t_0$-regular, i.e. for all $x \in \R^d$, the sub-probability measures $P_{t_0}(x, dy)$ are mutually equivalent. 
\end{lemma}
\begin{proof}
By Lebesgue's theorem, for any $r>0$ and $x \in \R^d$ it holds that for $\M$ of Theorem \ref{th: 3.1.2}, 
$$
\lim_{t \rightarrow 0+}P_t(x, B_{r}(x))=\lim_{t \rightarrow 0+} \E_x\left[ 1_{B_r(x)}(X_t)\right]=1.
$$
By Proposition \ref{prop:3.1.1}(i), $(P_{t})_{t>0}$ is $t_0$-regular for any $t_0>0$.
\end{proof}
\noindent
The following theorem is an application of our results combined with those of \cite{DPZB}.
\begin{theorem} \label{theo:3.3.8}
Assume {\bf (a)} of Section \ref{subsec:2.2.1} and {\bf (b)} of Section \ref{subsec:3.1.1} hold.
Assume that there exists a finite invariant measure $\nu$ for $\M$  (see Definition \ref{def:invariantforprocess}) of Theorem \ref{th: 3.1.2} (see also Theorem \ref{theo:3.1.4}). Let $\mu=\rho\,dx$ 
be as in Theorem \ref{theo:2.2.7} or as in  Remark \ref{rem:2.2.4}. Then the followings are satisfied:
\begin{itemize}
\item[(i)]
$\M$ is non-explosive (Definition \ref{non-explosive}), hence $P_t(x, dy)$  is a probability measure on $(\R^d, \mathcal{B}(\R^d))$ for any $(x,t) \in \R^d \times (0, \infty)$ and equivalent to the Lebesgue measure (cf. Proposition \ref{prop:3.1.1}).
\item[(ii)] Any sub-invariant measure for $\M$ is finite and $\mu$ is a finite invariant measure for $\M$.
\item[(iii)] $\nu$ is unique up to a multiplicative constant. More precisely, 
if there exists another invariant measure $\pi$ for $\M$, then $\pi$ is finite and 
$$
\nu(A)= \frac{\nu(\R^d)}{\pi(\R^d)}\cdot \pi(A), \quad \text{ for all } A \in \mathcal{B}(\R^d).
$$
\item[(iv)] For any $s \in [1, \infty)$ and $f \in L^s(\R^d, \mu)$, we have
\begin{equation} \label{lslimit}
\lim_{t \rightarrow \infty}P_t f = \frac{1}{\mu(\R^d)}\int_{\R^d} f d \mu \quad \text{ in } L^s(\R^d, \mu)
\end{equation}
and for all  $x \in \R^d$, $A \in \mathcal{B}(\R^d)$
\begin{eqnarray} \label{1dmarginal}
\lim_{t \rightarrow \infty} P_t(x,A)=\lim_{t \rightarrow \infty}\P_x(X_t \in A) = \frac{\mu(A)}{\mu(\R^d)}.
\end{eqnarray}
\item[(v)] Let $A \in \mathcal{B}(\R^d)$ be such that $\mu(A)>0$ and $(t_n)_{n\ge 1}\subset (0,\infty)$ be any sequence with $\lim_{n\to \infty}t_n=\infty$. Then
\begin{equation} \label{eq:3.32*}
\P_x(X_{t_n}\in A \text{ for infinitely many } n\in \N)=1,\ \ \text{ for all } x\in \R^d.
\end{equation}
In particular, $\P_x(L_{A}=\infty)=1$  for all $x \in \R^d$ and $\M$ is recurrent in the probabilistic sense (see Definition \ref{def:3.2.2.3} and \eqref{lastexittime} for the definition of $L_A$).
\end{itemize}
\end{theorem}

\begin{proof}
(i) Since $\nu$ is finite and an invariant measure for $\M$,  it follows from \eqref{eq:3.30} that for any $t>0$ that 
$$
\int_{\R^d} (1-P_t1_{\R^d}) d\nu=0,
$$
hence $P_t 1_{\R^d}=1$, $\nu$-a.e. for any $t>0$. Thus, for some $(x_0, t_0) \in \R^d \times(0, \infty)$, $P_{t_0}1_{\R^d}(x_0)=1$ and then $(T_t)_{t>0}$ is conservative by Lemma \ref{lem:2.7}(ii). Consequently, $\M$ is non-explosive by Corollary \ref{cor:3.2.1}. \\
(ii)  By (i), Lemma \ref{lem:3.2.9}, \cite[Theorem 4.2.1(i)]{DPZB} it follows that for any $A \in \mathcal{B}(\R^d)$ and $x \in \R^d$ 
\begin{equation} \label{eq:3.33}
\lim_{t \rightarrow \infty}\P_x(X_t \in A)= \frac{\nu(A)}{\nu(\R^d)}.
\end{equation}
Now suppose that $\kappa$ is an infinite sub-invariant measure for $\M$. Since $\kappa$ is $\sigma$-finite, we can choose $A \in \mathcal{B}(\R^d)$ with $\kappa(A)<\infty$ and $\nu(A)>0$.
Then by \eqref{eq:3.33} and Fatou's lemma,
$$
\infty=\int_{\R^d} \frac{\nu(A)}{\nu(\R^d)} d\kappa \leq  \liminf_{t \rightarrow \infty}\int_{\R^d} \mathbb{P}_x(X_t \in A) \kappa(dx) \leq \kappa(A)<\infty, 
$$
which is a contradiction. Therefore, any sub-invariant measure is finite. In particular $\mu$ is finite. Since $(T_t)_{t>0}$ is conservative by (i) and $\mu$ is finite, it follows that $\mu$ is $(\overline{T}_t)_{t>0}$-invariant by Remark \ref{remark2.1.11}, so that $\mu$ is a finite invariant measure for $\M$. \\
(iii)  By (i), Lemma \ref{lem:3.2.9}, and \cite[Theorem 4.2.1(ii)]{DPZB}, $\frac{\nu}{\nu(\R^d)}$ is the unique invariant probability measure for $\M$. So, if there exists another invariant measure $\pi$ for $\M$, then $\pi$ must be finite by (ii) and therefore $\frac{\pi}{\pi(\R^d)}$ is an invariant probability measure for $\M$ which must then coincide with $\frac{\nu}{\nu(\R^d)}$.\\
(iv) By (iii), $\nu=\frac{\nu(\R^d)}{\mu(\R^d)} \mu$. Hence, \eqref{eq:3.12} (see Proposition \ref{prop:3.1.1}(i)) and \eqref{eq:3.33} implies \eqref{1dmarginal}. 
Using \eqref{eq:3.1.1} and that the strong convergence of $(P_t(x,\cdot))$ in \eqref{1dmarginal} implies weak convergence, we get
\begin{equation} \label{pointxg}
\lim_{t \rightarrow \infty} P_t f(x) = \frac{1}{\mu(\R^d)} \int_{\R^d} f d\mu, \quad x\in \R^d,\ f \in C_b(\R^d).
\end{equation}
Since $\mu$ is finite, \eqref{lslimit} follows from \eqref{pointxg} for any $f \in C_b(\R^d)$ using Lebesgue's theorem and the sub-Markovian property of $(P_t)_{t>0}$. Finally, using the denseness of $C_b(\R^d)$ in $L^s(\R^d, \mu)$ and the $L^s(\R^d, \mu)$-contraction property of $(P_t)_{t>0}$ for each $s \in [1, \infty)$, \eqref{lslimit} follows by a 3-$\varepsilon$ argument.\\
(v) By \cite[Proposition 3.4.5]{DPZB}, \eqref{eq:3.32*} holds, hence $\P_x(L_A=\infty)=1$ for all $A \in \mathcal{B}(\R^d)$ with $\mu(A)>0$ and $x \in \R^d$. Since $\mu(U)>0$ for any nonempty open set $U$ in $\R^d$, $\M$ is recurrent in the probabilistic sense.
\end{proof}

\begin{proposition} \label{prop:3.3.12}
Assume {\bf (a)} of Section \ref{subsec:2.2.1} and {\bf (b)} of Section \ref{subsec:3.1.1} hold.
Let $N_0 \in \N$. Let $g\in C^2(\R^d \setminus \overline{B}_{N_0}) \cap C(\R^d)$, $g \geq 0$, with
\begin{equation}  \label{eq:3.34*}
\lim_{r \rightarrow \infty} (\inf_{\partial B_r} g) = \infty.
\end{equation}
Assume that for some $c>0$
$$
Lg\le -c \quad \text{a.e. on }\, \R^d\setminus \overline{B}_{N_0}.
$$ 
Then $\mu$ is a finite invariant measure (see Definition \ref{def:invariantforprocess} and right before it) for $\M$ of Theorem \ref{th: 3.1.2} (see also Theorem \ref{theo:3.1.4}) and Theorem \ref{theo:3.3.8} applies.
\end{proposition}
\begin{proof}
First, $\M$ is non-explosive by Lemma \ref{lem3.2.6}, hence $(T_t)_{t>0}$ is conservative by Corollary \ref{cor:3.2.1}. By the claim of Lemma \ref{lem3.2.6}, there exists $N_1 \in \N$ with $N_1 \geq N_0+2$ and $\psi \in C^2(\R^d)$ with $\psi(x) \geq 0$ for all $x \in \R^d$ and $\psi(x)=g(x)$ for all $x\in \R^d \setminus \overline{B}_{N_1}$ such that
$$
L \psi \leq -c \quad \text{a.e. on }\, \R^d\setminus \overline{B}_{N_1}.
$$
It follows by \cite[2.3.3. Corollary]{BKRS} (see also \cite[Theorem 2]{BRS} for the original result) that $\mu$ is finite and then by Remark \ref{remark2.1.11} that $\mu$ is $(\overline{T}_t)_{t>0}$-invariant. Therefore, $\mu$ is a finite invariant measure for $\M$ and  Theorem \ref{theo:3.3.8} applies with $\nu=\mu$.
\end{proof}

\begin{corollary} \label{cor:3.2.3.7}
Assume {\bf (a)} of Section \ref{subsec:2.2.1} and {\bf (b)} of Section \ref{subsec:3.1.1} hold.
Let $N_0\in \N$ and $M>0$. Assume that either
\begin{eqnarray} \label{eq:3.35*}
-\frac{\langle A(x)x, x \rangle}{ \left \| x \right \|^2 }+ \frac12\mathrm{trace}A(x)+ \big \langle \mathbf{G}(x), x \big \rangle \leq -M \|x\|^2
\end{eqnarray}
for a.e. $x\in \R^d\setminus \overline{B}_{N_0}$ or
\begin{eqnarray} \label{eq:3.36}
\frac12\mathrm{trace}A(x)+ \big \langle \mathbf{G}(x), x \big \rangle \leq -M
\end{eqnarray}
for a.e. $x\in \R^d\setminus \overline{B}_{N_0}$.
Then $\mu$ is a finite invariant measure (see Definition \ref{def:invariantforprocess}) for $\M$ of Theorem \ref{th: 3.1.2} and Theorem \ref{theo:3.3.8} applies. In particular, if $d=2$ and $\Psi_1$, $\Psi_2$, $Q$, $A$ are as in Corollary \ref{cor:3.1.3} and
$$
\frac{|\Psi_1(x)-\Psi_2(x)|}{2}  + \langle \mathbf{G}(x),x  \rangle \leq -M\|x\|^2
$$
for a.e. $x\in \R^2\setminus \overline{B}_{N_0}$, then \eqref{eq:3.35*} is satisfied in this special situation.
\end{corollary}
\begin{proof}
Let $g(x)= \ln(\|x\|^2 \vee N_0^2)+2$, $x \in \R^d$ be as in the proof of Corollary \ref{cor:3.2.2}. Then $Lg \leq -2M$ a.e. on $\R^d\setminus \overline{B}_{N_0}$, if and only if \eqref{eq:3.35*} holds. If $f(x)=\|x\|^2$, $x \in \R^d$, then $Lf \leq -2M$ a.e. on $\R^d\setminus \overline{B}_{N_0}$, if and only if \eqref{eq:3.36} holds. Thus, the assertion follows by Proposition \ref{prop:3.3.12}. The last assertion holds, proceeding as in the proof of Corollary \ref{cor:3.1.3}.
\end{proof}
\noindent
In Theorem \ref{theo:3.3.8}, we saw that if there exists a finite invariant measure $\nu$ for $\M$, then any invariant measure for $\M$ is represented by a constant multiple of $\nu$. The following example illustrates a case where $\M$ has two infinite invariant measures which are not represented by a constant multiple of each other.
\begin{eg}\label{ex:3.8}
Define
$$
Lf  = \frac{1}{2} \Delta f+ \langle \mathbf{e}_1, \nabla f \rangle, \quad f \in C_0^{\infty}(\R^d). 
$$
Then $\mu:=dx$ is an infinitesimally invariant measure for $(L, C_0^{\infty}(\R^d))$. Hence by Theorem \ref{theorem2.1.5}, there exists a closed extension of $(L, C_0^{\infty}(\R^d))$ that generates a sub-Markovian $C_0$-semigroup $(T_t)_{t>0}$ on $L^1(\R^d, \mu)$. Then by Proposition \ref{prop:2.1.10}(iii), $(T_t)_{t>0}$ is conservative and $\mu$ is $(T_t)_{t>0}$-invariant. Let 
$$
\M =  (\Omega, \mathcal{F}, (\mathcal{F}_t)_{t \ge 0}, (X_t)_{t \ge 0}, (\mathbb{P}_x)_{x \in \R^d\cup \{\Delta\}}   )
$$ 
be the Hunt process associated with $(T_t)_{t>0}$ by Theorem \ref{th: 3.1.2}. Let $y \in \R^d$ be given. By Theorem \ref{theo:3.1.4}(i), there is a $d$-dimensional Brownian motion $((W_t)_{t\geq0}, (\mathcal{F}_t)_{t\geq0})$ on $(\Omega, \mathcal{F}, \P_y)$ such that $(\Omega, \mathcal{F}, \P_y, (\mathcal{F}_t)_{t \ge 0}, (X_t)_{t \ge 0},(W_t)_{t \geq 0})$ is a weak solution 
(see Definition \ref{def:3.48}(iv)) to
\begin{equation} \label{eq:3.34}
X_t = X_0+W_t + \int_0^t \mathbf{e}_1ds.
\end{equation}
On the other hand, $\widetilde{\mu}:=e^{2 \langle \mathbf{e}_1, x \rangle}dx$ is also an infinitesimally invariant measure for $(L, C_0^{\infty}(\R^d))$.  By Theorem \ref{theorem2.1.5}, there exists a closed extension of $(L, C_0^{\infty}(\R^d))$ that generates a sub-Markovian $C_0$-semigroup $(\widetilde{T}_t)_{t>0}$ on $L^1(\R^d, \widetilde{\mu})$. Then by Proposition \ref{prop:2.1.10}(iii), $(\widetilde{T}_t)_{t>0}$ is conservative and $\widetilde{\mu}$ is $(\widetilde{T}_t)_{t>0}$-invariant. Let 
$$
\widetilde{\M} =  (\widetilde{\Omega}, \widetilde{\mathcal{F}}, (\widetilde{\mathcal{F}}_t)_{t \ge 0}, (\widetilde{X}_t)_{t \ge 0}, (\widetilde{\mathbb{P}}_x)_{x \in \R^d\cup \{\Delta\}}   )
$$ 
be the Hunt process associated with $(\widetilde{T}_t)_{t>0}$ by Theorem \ref{th: 3.1.2}.  By Theorem \ref{theo:3.1.4}(i), there exists a $d$-dimensional  standard $(\widetilde{\mathcal{F}}_t)_{t \geq 0}$- Brownian motion $(\widetilde{W}_t)_{t\geq0}$ on the probability space $(\widetilde{\Omega}, \widetilde{\mathcal{F}}, \widetilde{\P}_y)$ such that $(\widetilde{\Omega}, \widetilde{\mathcal{F}}, \widetilde{\P}_y, (\widetilde{\mathcal{F}}_t)_{t \ge 0}, (\widetilde{X}_t)_{t \ge 0},(\widetilde{W}_t)_{t \geq 0})$ is a weak solution to \eqref{eq:3.34}. Since the SDE \eqref{eq:3.34} admits pathwise uniqueness (see Definition \ref{def:3.48}(v))
by \cite[2.9 Theorem, Chapter 5]{KaSh} (see also \cite[Proposition 1]{YW71} for the original result) and pathwise uniqueness implies the uniqueness in law (cf. Definition \ref{def:3.48}(vi)) by \cite[3.20 Theorem, Chapter 5]{KaSh}, it holds that
\begin{equation} \label{eq:3.35}
\P_y(X_t \in A) = \widetilde{\P}_y(\widetilde{X}_t \in A), \quad \text{ for all $A \in \mathcal{B}(\R^d)$ and $t>0$}. 
\end{equation}
Since $\mu$ and $\widetilde{\mu}$ are invariant measures for $\M$ and $\widetilde{\M}$, respectively, and $y \in \R^d$ is arbitrarily given, it follows from \eqref{eq:3.35} that both $\mu$ and $\widetilde{\mu}$ are invariant measures for $\M$ (and $\widetilde{\M}$). Obviously, $\mu$ and $\widetilde{\mu}$  cannot be represented by a  constant multiple of each other.
\end{eg}

\subsection{Uniqueness}
\label{sec:3.3}
In this section, we investigate pathwise uniqueness (cf. Definition \ref{def:3.48}(v)) and uniqueness in law (cf. Definition \ref{def:3.48}(vi)).\\
We will consider the following {\bf condition}:
\begin{itemize}
\item[{\bf (c)}] \ \index{assumption ! {\bf (c)}}for some $p\in (d,\infty)$, $d\ge 2$ (see beginning of Section \ref{subsec:2.2.1}), $\sigma=(\sigma_{ij})_{1 \leq i,j \leq d}$ is possibly non-symmetric 
with $\sigma_{ij} \in H^{1,p}_{loc}(\R^d) \cap C(\R^d)$ for all $1\leq i,j \leq d$ such that 
$A=(a_{ij})_{1 \leq i,j \leq d}:=\sigma \sigma^T$ satisfies \eqref{eq:2.1.2} and $\mathbf{G}=(g_1,\ldots,g_d)
\in L^p_{loc}(\R^d, \R^d)$.
\end{itemize}
If {\bf (c)} holds, then {\bf (a)} of Section \ref{subsec:2.2.1} and {\bf (b)} of Section 
\ref{subsec:3.1.1} hold. \\
Our strategy to obtain a pathwise unique and strong solution to the SDE \eqref{eq:3.39},  is to apply the Yamada--Watanabe theorem \cite[Corollary 1]{YW71} and the local pathwise 
uniqueness result \cite[Theorem 1.1]{Zh11} to the weak solution of Theorem \ref{theo:3.1.4}(i). Under the 
mere condition of {\bf (c)} and the assumption that the constructed Hunt process $\M$ in Theorem \ref{th: 3.1.2} 
is non-explosive, it is shown in Proposition \ref{prop:3.3.9} and Theorem \ref{theo:3.3.1.8} that there exists a pathwise unique and strong 
solution to the SDE \eqref{eq:3.39} (cf. Definition \ref{def:3.48}). Moreover, Proposition \ref{prop:3.3.1.9} implies that 
the local strong solution of \cite[Theorem 1.3]{Zh11} (see also \cite[Theorem 2.1]{KR} for prior work that covers the case of Brownian motion with drift)  when considered in the time-homogeneous case is non-explosive, if the 
Hunt process $\M$ of Theorem \ref{th: 3.1.2} is non-explosive. Therefore, any condition for non-explosion of $\M$ in this monograph is a new criterion for 
strong well-posedness of time-homogeneous It\^{o}-SDEs whose coefficients satisfy {\bf (c)}. 
As an example for this observation, consider the case where {\bf (c)} and the non-explosion condition \eqref{eq:3.20} 
are satisfied. Then we obtain a pathwise unique and strong solution to \eqref{eq:3.39}, under the classical-like non-explosion condition \eqref{eq:3.20} that even allows for an interplay of diffusion and drift coefficient.
Additionally, $\|\mathbf{G}\|$  is here allowed to have arbitrary growth as long as  
$\langle \mathbf{G}(x), x \rangle$ in \eqref{eq:3.20} is negative. A further example is given when $d=2$. Then the diffusion coefficient is allowed to have arbitrary 
growth in the situation of \eqref{eq:3.2.1.20} in Corollary \ref{cor:3.1.3}. In summary, one can say that Theorem \ref{theo:3.3.1.8}, Propositions \ref{prop:3.3.9} and \ref{prop:3.3.1.9}, together with further results of this work (for instance those which are mentioned in Theorem \ref{theo:3.3.1.8}) can be used to complete and to considerably improve various results from \cite{KR}, \cite{ZhXi16}, \cite{Zh05}, and \cite{Zh11}, in the time-homogeneous case (see \cite{LT18}, \cite{LT19}, in particular the introduction of \cite{LT18}). 
This closes a gap in the literature, which is described at the end of  Remark \ref{rem:3.3.1},
where we discuss related work.\\
In Section \ref{subsec:3.3.2}, under the assumption {\bf (a)} of Section \ref{subsec:2.2.1} and 
{\bf (b)} of Section \ref{subsec:3.1.1}, we investigate uniqueness in law, among 
all right processes that have a strong Feller transition semigroup (more precisely such that \eqref {eq:3.41*} holds), that have $\mu$ as a 
sub-invariant measure, and where $(L,C_0^{\infty}(\R^d))$ solves the martingale problem with respect to $\mu$. This sort of uniqueness in law is more restrictive than  
uniqueness in law in the classical sense. But under the mere assumption of 
{\bf (a)} and {\bf (b)}, classical uniqueness in law is not known to hold. Our main result in 
Section \ref{subsec:3.3.2}, Proposition \ref{prop:3.3.1.15} which is more analytic than probabilistic, is 
ultimately derived by the concept of $L^1$-uniqueness of $(L, C_0^{\infty}(\R^d))$ introduced in Definition 
\ref{def:2.1.1}(i). Therefore, as a direct consequence of Proposition \ref{prop:3.3.1.15} and Corollary 
\ref{cor2.1.2}, under the assumption that $\mu$ is an invariant measure for $\M$, we derive in Proposition 
\ref{prop:3.3.1.16} our uniqueness in law result. This result is meaningful in terms 
of being able to deal with the case of locally unbounded drift coefficients and explosive $\M$. We 
present various situations in Example \ref{ex:3.4.9} where $\mu$ is an invariant measure for $\M$, so that our 
uniqueness in law result is applicable.

\subsubsection{Pathwise uniqueness and strong solutions}
\label{subsec:3.3.1} 
\begin{definition}\label{def:3.48}
\begin{itemize}
\item[(i)]
For a filtration $(\widehat{\mathcal{F}}_t)_{t \geq 0}$ on a probability space $(\widetilde{\Omega}, \widetilde{\mathcal{F}}, \widetilde{\P})$, the {\bf augmented filtration} $(\widehat{\mathcal{F}}^{\text{aug}}_t)_{t \geq 0}$ 
of $(\widehat{\mathcal{F}}_t)_{t \geq 0}$ under $\widetilde{\P}$ is defined as
$$
\widehat{\mathcal{F}}^{\text{aug}}_t:= \sigma(\widehat{\mathcal{F}}_t \cup \widehat{\mathcal{N}}^{\widetilde{\P}}), \quad 0\leq t<\infty,
$$
where $\widehat{\mathcal{N}}^{\widetilde{\P}}:=\{ F \subset \widetilde{\Omega}: F \subset G \text{ for some } G \in \widehat{\mathcal{F}}_{\infty}:=\sigma(\cup_{t \geq 0} \widehat{\mathcal{F}}_t)  \text{ with } \widetilde{\P}(G)=0 \}$. The {\bf completion} $\widehat{\mathcal{F}}^{\text{aug}}$ of $\widetilde{\mathcal{F}}$ under $\widetilde{\P}$ is defined as
$$
\widehat{\mathcal{F}}^{\text{aug}}:= \sigma(\widetilde{\mathcal{F}} \cup \widetilde{\mathcal{N}}^{\widetilde{\P}}),
$$
where $\widetilde{\mathcal{N}}^{\widetilde{\P}}:=\{ F \subset \widetilde{\Omega}: F \subset G \text{ for some } G \in \widetilde{\mathcal{F}} \text{ with } \widetilde{\P}(G)=0 \}$.
\item[(ii)] Let $l \in \N$, $\widetilde{\sigma}=(\widetilde{\sigma}_{ij})_{1 \leq i \leq d, 1 \leq j \leq l}$ be a matrix of Borel measurable functions and $\widetilde{\mathbf{G}}=(\widetilde{g}_1, \ldots, \widetilde{g}_d)$ be a Borel measurable vector field.
Given an $l$-dimensional Brownian motion $(\widetilde{W}_t)_{t \geq 0}$ on a probability space $(\widetilde{\Omega}, \widetilde{\mathcal{F}}, \widetilde{\P})$, let $(\widehat{\mathcal{F}}_t)_{t \geq 0}:=\left(\sigma(\widetilde{W}_s\vert s \in \[0,t\])\right)_{t \geq 0}$ and $x \in \R^d$. $(\widetilde{X}_t)_{t \geq 0}$ is called a {\bf strong solution}\index{solution ! strong} to \eqref{eq:3.36*} with Brownian motion $(\widetilde{W}_t)_{t \geq 0}$ and initial condition $\widetilde{X}_0=x$, if (a)--(d) below hold:
\subitem(a) $(\widetilde{X}_t)_{t \geq 0}$ is an $\R^d$-valued stochastic process adapted to $(\widehat{\mathcal{F}}^{\text{aug}}_t)_{t \geq 0}$,
\subitem(b) $\widetilde{\P}(\widetilde{X}_0=x)=1$,
\subitem(c) $\widetilde{\P}\left(\int_0^t (\widetilde{\sigma}^2_{ij}(\widetilde{X}_s) +|\widetilde{g}_i|(\widetilde{X}_s) )ds <\infty \right)=1$ for all $1 \leq i \leq d$, $1 \leq j \leq l$\\
\text{}\quad \quad \;\; and $0 \leq t<\infty$,
\subitem(d) 
$\widetilde{\P}$-a.s. it holds that
\begin{equation} \label{eq:3.36*}
\widetilde{X}_t=\widetilde{X}_0+\int_0^t \widetilde{\sigma}(\widetilde{X}_s)d\widetilde{W}_s + \int_0^t \widetilde{\mathbf{G}}(\widetilde{X}_s)ds, \; 0\leq t<\infty,
\end{equation}
\text{} \quad \quad i.e. $\widetilde{\P}$-a.s. 
\begin{equation*}
\widetilde{X}^i_t = \widetilde{X}^i_0+ \sum_{j=1}^l \int_0^t \widetilde{\sigma}_{ij}(\widetilde{X}_s) d \widetilde{W}^j_s + \int_0^t \widetilde{g}_i(\widetilde{X}_s) ds, \;\; 1 \leq i \leq d, \; 0\leq t<\infty.
\end{equation*}
\item[(iii)] A filtration $(\widetilde{\mathcal{F}}_t)_{t \geq 0}$ on a probability space $(\widetilde{\Omega}, \widetilde{\mathcal{F}}, \widetilde{\P})$ is said to satisfy the {\bf usual conditions}, if
$$
\widetilde{\mathcal{F}}_{0} \supset \{  F \subset \widetilde{\Omega}: F \subset G \text{ for some } G \in \widetilde{\mathcal{F}} \text{ with } \widetilde{\P}(G)=0  \}
$$
and
$(\widetilde{\mathcal{F}}_t)_{t \geq 0}$ is right-continuous, i.e.
\begin{equation} \label{defrightcon}
\widetilde{\mathcal{F}}_t =  \bigcap_{\varepsilon>0} \widetilde{\mathcal{F}}_{t+\varepsilon}, \quad \forall t \geq 0.
\end{equation}

\item[(iv)]
Let $l \in \N$ and $\widetilde{\sigma}$, $\widetilde{\mathbf{G}}$, $(\widetilde{\Omega}, \widetilde{\mathcal{F}}, \widetilde{\P} )$, $(\widetilde{W}_t)_{t \geq 0}$ be as in (ii). We say that
$$
(\widetilde{\Omega}, \widetilde{\mathcal{F}}, \widetilde{\P}, (\widetilde{\mathcal{F}}_t)_{t \ge 0}, (\widetilde{X}_t)_{t \ge 0},(\widetilde{W}_t)_{t \geq 0})
$$
is a {\bf weak solution} to \eqref{eq:3.36*}\index{solution ! weak} if $(\widetilde{\mathcal{F}_t})_{t \geq 0}$ is a filtration on $(\widetilde{\Omega}, \widetilde{\mathcal{F}}, \widetilde{\P})$ satisfying the usual conditions, $(\widetilde{X}_t)_{t \geq 0}$ is an $\R^d$-valued stochastic process adapted to $(\widetilde{\mathcal{F}}_t)_{t \geq 0}$, $(\widetilde{W}_t)_{t \geq 0}$ is an $l$-dimensional standard $(\widetilde{\mathcal{F}}_t)_{t \geq 0}$-Brownian motion and (c) and (d) of (ii) hold. In particular, any strong solution as in (ii) is a weak solution as defined in (iv).
\item[(v)]
We say that {\bf pathwise uniqueness holds} for the SDE \eqref{eq:3.36*}\index{uniqueness ! pathwise}, if whenever $x \in \R^d$ and
$$
(\widetilde{\Omega}, \widetilde{\mathcal{F}}, \widetilde{\P}, (\widetilde{\mathcal{F}}_t)_{t \ge 0}, (\widetilde{X}^1_t)_{t \ge 0}, (\widetilde{W}_t)_{t \geq 0})
$$ 
and 
$$
(\widetilde{\Omega}, \widetilde{\mathcal{F}}, \widetilde{\P}, (\widetilde{\mathcal{F}}_t)_{t \ge 0}, (\widetilde{X}^2_t)_{t \ge 0}, (\widetilde{W}_t)_{t \geq 0})
$$ 
are two weak solutions to \eqref{eq:3.36*} with 
$$
\widetilde{\P}(\widetilde{X}^1_0=\widetilde{X}^2_0=x)=1,
$$ 
then 
\begin{equation*} 
\widetilde{\P}(\widetilde{X}^1_t=\widetilde{X}^2_t,\; t \geq 0 )=1.
\end{equation*}
 A weak solution to \eqref{eq:3.36*} is said to be {\bf pathwise unique}, if pathwise uniqueness holds for the SDE \eqref{eq:3.36*}.
\item[(vi)]
We say that {\bf uniqueness in law holds} for the SDE \eqref{eq:3.36*}\index{uniqueness ! in law}, if whenever $x \in \R^d$ and
$$
(\widetilde{\Omega}^1, \widetilde{\mathcal{F}}^1, \widetilde{\P}^1, (\widetilde{\mathcal{F}}^1_t)_{t \ge 0}, (\widetilde{X}^1_t)_{t \ge 0}, (\widetilde{W}^1_t)_{t \geq 0})
$$ 
and 
$$
(\widetilde{\Omega}^2, \widetilde{\mathcal{F}}^2, \widetilde{\P}^2, (\widetilde{\mathcal{F}}^2_t)_{t \ge 0}, (\widetilde{X}^2_t)_{t \ge 0}, (\widetilde{W}^2_t)_{t \geq 0})
$$ 
are two weak solutions to \eqref{eq:3.36*}, defined on possibly different probability spaces, with
$$
\widetilde{\P}^1 \circ (\widetilde{X}_0^{1})^{-1}=\widetilde{\P}^2 \circ (\widetilde{X}_0^{2})^{-1}= \delta_x,
$$
where $\delta_x$ is a Dirac measure in $x \in \R^d$, then
$$
\widetilde{\P}^1 \circ (\widetilde{X}^{1})^{-1}=\widetilde{\P}^2 \circ (\widetilde{X}^{2})^{-1}  \text{ on } \; \mathcal{B}(C([0, \infty), \R^d)).
$$
\end{itemize}
\end{definition}

\begin{proposition} 
\label{prop:3.3.9}
Let $\sigma$, $\mathbf{G}$, satisfy assumption {\bf (c)} as at the beginning of Section \ref{sec:3.3}.
Then
pathwise uniqueness holds for the SDE 
\begin{equation} \label{eq:3.39}
\widetilde{X}_t = \widetilde{X}_0+ \int_0^t \sigma(\widetilde{X}_s) dW_s +\int_0^t \mathbf{G}(\widetilde{X}_s) ds, \quad 0 \leq t<\infty.
\end{equation}
\end{proposition}

\begin{proof}
Let $n \in \N$ be such that $x \in B_n$ and $\tau_n:=\inf \{t>0 : \widetilde{X}^1_t \in \R^d \setminus B_n  \} \wedge \inf \{t>0 : \widetilde{X}^2_t \in \R^d \setminus B_n  \}$. Let $(\widetilde{\mathcal{F}}^{\text{aug}}_t)_{t \geq 0}$ be the augmented filtration of  $(\widetilde{\mathcal{F}}_t)_{t \geq 0}$ and $\widetilde{\mathcal{F}}^{\text{aug}}$ be the completion of $\mathcal{F}$ under $\P$. Then $(\widetilde{X}_t^1)_{t \geq 0}$ and $(\widetilde{X}_t^2)_{t \geq 0}$ are still adapted to $(\widetilde{\mathcal{F}}^{\text{aug}}_t)_{t \geq 0}$ and $(\widetilde{W}_t)_{t>0}$ is still a $d$-dimensional standard $(\widetilde{\mathcal{F}}^{\text{aug}}_t)_{t \geq 0}$-Brownian motion.
We can hence from now on assume that we are working on $(\widetilde{\Omega}, \widetilde{\mathcal{F}}^{\text{aug}}, \widetilde{\P}, (\widetilde{\mathcal{F}}^{\text{aug}}_t)_{t \geq 0})$. Then since $(\widetilde{X}_t^1)_{t \geq 0}$ and $(\widetilde{X}_t^2)_{t \geq 0}$ are $\P$-a.s. continuous, 
$\tau_n$ is an $(\widetilde{\mathcal{F}}^{\text{aug}}_t)_{t \geq 0}$-stopping time and $\widetilde{\P}(\lim_{n \rightarrow \infty}\tau_n=\infty)=1$. Let $\chi_n \in C_0^{\infty}(\R^d)$ be such that $0 \leq \chi_n \leq 1$, $\chi_n=1$ on $\overline{B}_n$ and $\text{supp}(\chi_n) \subset B_{n+1}$. Let $\mathbf{G}_n:=\chi_n \mathbf{G}$ and $\sigma^n=(\sigma^n_{ij})_{1 \leq i,j \leq d}$ be defined by
$$
\sigma_{ij}^n(x):= \chi_{n+1}(x)\sigma_{ij}(x)+ \sqrt{\nu_{B_{n+1}}} (1-\chi_{n}(x))\delta_{ij}, \quad x \in \R^d,
$$
where $(\delta_{ij})_{1 \leq i,j \leq d}$ denotes the identity matrix and the constant $\nu_{B_{n+1}}$ is from  \eqref{eq:2.1.2}. Then, $\mathbf{G}_n \in L^p(\R^d, \R^d)$, $\nabla \sigma^n_{ij} \in L^p(\R^d, \R^d)$ for all $1 \leq i,j \leq d$ and
$$
\nu_{B_{n+1}}^{-1} \|\xi\|^2 \leq \|(\sigma^n)^T(x) \xi\|^2 \leq 4 \nu_{B_{n+1}} \|\xi\|^2, \quad \forall x \in \R^d, \xi \in \R^d. 
$$
For $i \in \{1, 2 \}$, suppose it holds $\widetilde{\P}$-a.s. that
$$
\widetilde{X}_t^i = x+\int_0^t \sigma^n(\widetilde{X}^i_s) d \widetilde{W}_s + \int_0^t \mathbf{G}_n(\widetilde{X}^i_s)ds, \quad 0 \leq t<\tau_n.
$$
Then by \cite[Theorem 1.1]{Zh11} applied for $\sigma^n$, $\mathbf{G}_n$ and $\tau_n$, 
$$
\widetilde{\P}(\widetilde{X}^1_t = \widetilde{X}^2_t, \quad 0 \leq t<\tau_n)=1.
$$
Now the assertion follows by letting $n \rightarrow \infty$.
\end{proof}

\begin{theorem} \label{theo:3.3.1.8}
Assume {\bf (c)} as at the beginning of Section \ref{sec:3.3} and  
that $\M$ is non-explosive (cf. Definition \ref{non-explosive}). Then   
$$
(\Omega, \mathcal{F}, \P_x, (\mathcal{F}_t)_{t \ge 0}, (X_t)_{t \ge 0},(W_t)_{t \geq 0})
$$
of Theorem \ref{theo:3.1.4}(i) is for each $x \in \R^d$ a weak solution to \eqref{eq:3.39} and uniqueness in law holds for \eqref{eq:3.39} (cf. Definition \ref{def:3.48}(vi)).\\
Let further $(\widetilde{\Omega}, \widetilde{\mathcal{F}},  \widetilde{\P})$ be a probability space carrying a $d$-dimensional standard Brownian motion $(\widetilde{W}_t)_{t \geq 0}$.  Let $x \in \R^d$ be arbitrary. Then there exists a measurable map
$$ 
h^x: C([0, \infty), \R^d) \rightarrow C([0, \infty), \R^d) 
$$
such that $(Y^x_t)_{t \geq 0}:= (h^x(\widetilde{W}_t))_{t \geq 0}$ is a pathwise unique and strong solution to \eqref{eq:3.39} on the probability space $(\widetilde{\Omega}, \widetilde{\mathcal{F}},  \widetilde{\P})$  with Brownian motion $(\widetilde{W}_t)_{t \geq 0}$ and initial condition $Y^x_0=x$.
Moreover, $\P_x \circ X^{-1}= \widetilde{\P} \circ (Y^x)^{-1}$ holds, and therefore $((Y^x_t)_{t \geq 0}, \widetilde{\P})_{x \in \R^d}$ inherits all properties from $\M$ that only depend on its law. Precisely, more than strong Feller properties (Theorem \ref{theorem2.3.1}, Theorem \ref{theo:2.6}, Proposition \ref{prop:3.1.4}), irreducibility (Lemma \ref{lem:2.7}, Proposition \ref{prop:2.4.2}), Krylov-type estimates (Theorem \ref{theo:3.3}), integrability (Lemma \ref{lem:3.1.4}), moment inequalities (Proposition \ref{prop:3.2.8}, Proposition \ref{theo:3.2.8}), properties for recurrence and transience (Proposition \ref{prop:3.2.2.11}, Theorem \ref{theo:3.3.6}, Lemma \ref{lem:3.2.8}, Proposition \ref{theo:3.2.6}, Corollary \ref{cor:3.2.2.5}, Proposition \ref{cor:3.3.2.6}), ergodic properties including the uniqueness of invariant measures (Theorem \ref{theo:3.3.8}, Proposition \ref{prop:3.3.12}, Corollary \ref{cor:3.2.3.7}) are satisfied where $(X_t)_{t \geq 0}$ and $\P_x$ are replaced by $(Y^x_t)_{t \geq 0}$ and $\widetilde{\P}$, respectively.
\end{theorem}

\begin{proof}
Since $\M$ is non-explosive,  it follows from Theorem \ref{theo:3.1.4}(i) that there exists a $d$-dimensional  standard $(\mathcal{F}_t)_{t\geq0}$-Brownian motion $(W_t)_{t\geq0}$ on $(\Omega, \mathcal{F}, \P_x)$ such that $(\Omega, \mathcal{F}, \P_x, (\mathcal{F}_t)_{t \ge 0}, (X_t)_{t \ge 0},(W_t)_{t \geq 0})$ is a weak solution to \eqref{eq:3.39}. Thus, the first assertion follows from Proposition \ref{prop:3.3.9} and \cite[Proposition 1]{YW71}.
 Moreover, by Proposition \ref{prop:3.3.9} and \cite[Corollary 1]{YW71}), there exists a measurable map 
$$
h^x: C([0, \infty), \R^d) \rightarrow C([0, \infty), \R^d) 
$$
such that $(X_t)_{t \geq 0}$ and $(h^x(W_t))_{t \geq 0}$ are $\P_x$-indistinguishable and in particular, 
$$
(Y^x_t)_{t \geq 0}:= (h^x(\widetilde{W}_t))_{t \geq 0}
$$ 
is a strong solution to \eqref{eq:3.39} on the probability space $(\widetilde{\Omega}, \widetilde{\mathcal{F}},  \widetilde{\P})$ with Brownian motion $(\widetilde{W}_t)_{t \geq 0}$ and $\widetilde{\P}(Y^x_0=x)=1$. Finally, since \eqref{eq:3.39} enjoys pathwise uniqueness, using \cite[Proposition 1]{YW71}, $\P_x \circ X^{-1}= \widetilde{\P} \circ (Y^x)^{-1}$ on $\mathcal{B}(C([0, \infty), \R^d))$, which concludes the proof.
\end{proof}
\begin{proposition} \label{prop:3.3.1.9}
Assume {\bf (c)} as at the beginning of Section \ref{sec:3.3} and that $\M$ is non-explosive (cf. Definition \ref{non-explosive}). Let $x \in \R^d$ and let
$$
(\widetilde{\Omega}, \widetilde{\mathcal{F}}, \widetilde{\P}, (\widetilde{\mathcal{F}}_t)_{t \ge 0}, (\widetilde{X}^x_t)_{t \ge 0})
$$ 
be an $\R^d_{\Delta}$-valued adapted stochastic process with $\widetilde{\P}(\widetilde{X}^x_0=x)=1$. Assume that there exists an ($\widetilde{\mathcal{F}}_t)_{t \geq 0}$-stopping time $\widetilde{\zeta}$ such that $t \mapsto \widetilde{X}^x_t$  is continuous and $\R^d$-valued on $[0, \widetilde{\zeta})$ and $\widetilde{X}^x_t=\Delta$ on $\{t \geq \widetilde{\zeta} \}$, both $\widetilde{\P}$-a.s., and that for each $n \in \N$
it holds that
$$
\inf \{t>0: \widetilde{X}^x_t \in \R^d \setminus \overline{B}_n  \}<\widetilde{\zeta} \quad \text{$\widetilde{\P}$-a.s. on $\{ \widetilde{\zeta}<\infty\}$. }
$$
Let $(\widetilde{W}_t)_{t \geq 0}$ be a $d$-dimensional standard $(\widetilde{\mathcal{F}}_t)_{t \geq 0}$-Brownian motion on $(\widetilde{\Omega}, \widetilde{\mathcal{F}}, \widetilde{\P})$. If 
$$
\widetilde{X}^x_t = x+ \int_0^t \sigma(\widetilde{X}^x_s) d\widetilde{W}_s + \int_0^t \mathbf{G}(\widetilde{X}^x_s) ds, \quad 0 \leq t< \widetilde{\zeta}, \; \text{ $\widetilde{\P}$-a.s. }
$$
then $\widetilde{\P}(\widetilde{\zeta}=\infty)=1$ and $(\widetilde{X}^x_t)_{t \geq 0}$ is a strong solution to \eqref{eq:3.39} on the probability space $(\widetilde{\Omega}, \widetilde{\mathcal{F}},  \widetilde{\P})$ with Brownian motion $(\widetilde{W}_t)_{t \geq 0}$ and $\widetilde{\P}(\widetilde{X}^x_0=x)=1$. Moreover, $\P_x \circ X^{-1}= \widetilde{\P} \circ (\widetilde{X}^x)^{-1}$  \text{on } $\mathcal{B}(C([0, \infty), \R^d))$.
\end{proposition}

\begin{proof}
Without loss of generality, we may assume that $(\widetilde{\mathcal{F}}_t)_{t \geq 0}$ is right continuous and contains the augmented filtration of $(\sigma(\widetilde{W}_s; \, 0 \leq s \leq t))_{t \geq 0}$. By Theorem \ref{theo:3.3.1.8}, there exists a measurable map
$$
h^x: C([0, \infty), \R^d) \rightarrow C([0, \infty), \R^d) 
$$
such that $(Y^x_t)_{t \geq 0}:= (h^x(\widetilde{W}_t))_{t \geq 0}$ is a pathwise unique and strong solution to \eqref{eq:3.39} on the probability space $(\widetilde{\Omega}, \widetilde{\mathcal{F}},  \widetilde{\P})$ with Brownian motion $(\widetilde{W}_t)_{t \geq 0}$ and $\widetilde{\P}(Y^x_0=x)=1$, hence
$$
Y^x_t = x+ \int_0^t \sigma(Y^x_s) d\widetilde{W}_s + \int_0^t \mathbf{G}(Y^x_s) ds, \quad 0 \leq t< \infty, \; \text{ $\widetilde{\P}$-a.s. }
$$
Let $n \in \N$ be such that $x \in B_n$ and $\tau_n:=\inf \{t>0: \widetilde{X}^x_t \in \R^d \setminus \overline{B}_n  \}$. Then by the $\widetilde{\P}$-a.s. right continuity of $(\widetilde{X}^x_t)_{t \geq 0}$ and the usual conditions of $(\widetilde{\mathcal{F}}_{t})_{t \geq 0}$, we obtain that $\tau_n$ is an $(\widetilde{\mathcal{F}}_t)_{t \geq 0}$-stopping time. Since $\widetilde{\P}$-a.s.  $t \mapsto \widetilde{X}_t^x$ is continuous and $\R^d$-valued on $[0, \widetilde{\zeta})$, it follows that $\tau_n<\tau_{n+1}<\widetilde{\zeta}$, $\widetilde{\P}$-a.s. on $\{ \widetilde{\zeta}<\infty\}$. Moreover, $\widetilde{\P}$-a.s. $\lim_{n \rightarrow \infty} \tau_n=\widetilde{\zeta}$ and $\widetilde{\P}$-a.s.
$$
\widetilde{X}_t^x = x+\int_0^t \sigma(\widetilde{X}^x_s) d \widetilde{W}_s + \int_0^t \mathbf{G}(\widetilde{X}^x_s)ds, \quad 0 \leq t<\tau_{n+1}.
$$
By \cite[Theorem 1.1]{Zh11},
$$
\widetilde{\P}(Y^x_t=\widetilde{X}^x_t, \quad 0\leq t<\tau_{n+1})=1.
$$
Therefore, we obtain 
$$
Y^x_{\tau_n} = \widetilde{X}^x_{\tau_n}, \quad \text{$\widetilde{\P}$-a.s. on $\{\widetilde{\zeta}<\infty \}$}.
$$
Now suppose that $\widetilde{\P}(\widetilde{\zeta}<\infty)>0$.  Then $\widetilde{\P}$-a.s. on $\{\widetilde{\zeta}<\infty \}$
$$
\|Y^x_{\tau_n}\| = \|\widetilde{X}^x_{\tau_n}\|=n.
$$
Therefore, $\widetilde{\P}$-a.s. on $\{\widetilde{\zeta}<\infty \}$
$$
\|Y^x_{\widetilde{\zeta}}\| = \lim_{n \rightarrow \infty} \|Y^x_{\tau_n} \| = \infty,
$$
which is a contradiction since $\|Y^x_{\widetilde{\zeta}}\|<\infty$ $\widetilde{\P}$-a.s. on $\{\widetilde{\zeta}<\infty \}$. Therefore, 
$$
\widetilde{\P}(\widetilde{\zeta}=\infty)=1,
$$ 
hence by Proposition \ref{prop:3.3.9},
$$
\widetilde{\P}(Y^x_t=\widetilde{X}^x_t, \; 0\leq t<\infty)=1.
$$ 
By Theorem \ref{theo:3.3.1.8}, it follows that
$$
\P_x \circ X^{-1}= \widetilde{\P} \circ (Y^x)^{-1}=\widetilde{\P} \circ (\widetilde{X}^x)^{-1} \quad \text{on } \ \mathcal{B}(C([0, \infty), \R^d)).
$$
\end{proof}
In the following remark, we briefly mention some previous related results about pathwise uniqueness and strong solutions to 
SDEs. 

\begin{remark}\label{rem:3.3.1}
{\it The classical result developed by It\^{o} about pathwise uniqueness and existence of a strong solution (strong well-posedness) requires dispersion and drift 
coefficients to be globally Lipschitz continuous and to satisfy a linear growth condition (cf. \cite[2.9 Theorem, 
Chapter 5]{KaSh}). In \cite[Theorem 4]{Zvon}, Dini continuity that is weaker than global Lipschitz 
continuity is assumed for the drift coefficient, but the diffusion and drift coefficients should be 
globally bounded.  The result of  It\^{o} can be localized, imposing only a local Lipschitz condition together with a (global) linear growth condition (cf. \cite[IV. Theorems 2.4 and 3.1]{IW89}).\\
Strong well-posedness results for only measurable coefficients were given starting from \cite{Zvon}, \cite{Ver79}, \cite{Ver81}. 
In these works $\sigma$ is non-degenerate and $\sigma, \mathbf{G}$ are bounded. To our knowledge the first strong well-posedness results for unbounded measurable coefficients start with \cite[Theorem 2.1]{GyMa}, but the growth condition there for  non-explosion \cite[Assumption 2.1]{GyMa} does not allow for linear growth as in the classical case. 
In \cite{KR}, the authors consider the Brownian motion case with drift, covering the condition {\bf (c)}. They obtain strong well-posedness up to an explosion time and certain non-explosion conditions, which also do not allow for linear growth (see \cite[Assumption 2.1]{KR}). The main technique of \cite{Zvon}, now known as Zvonkin transformation, was employed together with Krylov-type estimates in
\cite{Zh11} in order to obtain strong well-posedness for locally unbounded drift coefficient and non-trivial dispersion coefficient up to an explosion time. The assumptions in \cite{Zh11}, when restricted to the time-homogeneous case are practically those of  {\bf (c)} (cf. \cite[Remark 3.3(ii)]{LT18} and the corresponding discussion in the introduction there), but again the non-explosion conditions are far from being classical-like linear growth conditions (see also \cite{Zh05}).
Among the references, where the technique of Zvonkin transformation together with Krylov-type estimates is used to obtain local strong well-posedness, the best non-explosion conditions up to now under the local strong well-posedness result of \cite{Zh11} can be found in \cite{ZhXi16}. In \cite{ZhXi16}  also strong Feller properties, irreducibility and further properties of the solution are studied. However, the conditions to obtain the results there are quite involved and restrictive and finally do not differ substantially from the classical results of local Lipschitz coefficients (see the discussion  in the introduction of \cite{LT18}). In summary one can say that in contrast to our results, \cite{GyMa}, \cite{KR}, \cite{Zh05}, \cite{Zh11}, \cite{ZhXi16} also cover the time-inhomogeneous case, but 
sharp results to treat SDEs with general locally unbounded drift coefficient in detail further as in Theorem \ref{theo:3.3.1.8}, similarly to classical SDEs with local Lipschitz coefficients, seem not to be at hand. The optimal local regularity assumptions to obtain local well-posedness (as in \cite{Zh11}) require a strengthening to obtain the further important properties of the solution as in Theorem 3.7 (see for instance conditions (H1), (H2), and (H1'), (H2') in \cite{ZhXi16}), contrary to the classical case (of locally Lipschitz coefficients), where important further properties of the solution can be formulated independently of the local regularity assumptions.}
\end{remark}

\subsubsection{Uniqueness in law (via $L^1$-uniqueness)}
\label{subsec:3.3.2} 
Throughout this section, we let
$$
\mu=\rho\,dx
$$ 
be as in Theorem \ref{theo:2.2.7} or as in  Remark \ref{rem:2.2.4}.
\begin{definition} \label{def:3.3.2.1}
Consider a right process
$$
\widetilde{\M} =  (\widetilde{\Omega}, \widetilde{\mathcal{F}}, (\widetilde{\mathcal{F}}_t)_{t \geq 0},(\widetilde{X}_t)_{t \ge 0}, (\widetilde{\P}_x)_{x \in \R^d\cup \{\Delta\}})
$$ 
with state space $\R^d$ (cf. Definition \ref{def:3.1.1}). 
For a measure $\nu$ on $(\R^d, \mathcal{B}(\R^d))$, we set 
$$
\widetilde{\P}_{\nu}(A):=\int_{\R^d} \widetilde{\P}_{x}(A) \nu(dx), \quad A \in \mathcal{B}(\R^d).
$$
$\widetilde{\M}$ is said to {\bf solve the martingale problem for $(L,C_0^{\infty}(\R^d))$ with respect to $\mu$}\index{solution ! to the martingale problem with respect to $\mu$}, if for all $u\in C_0^{\infty}(\R^d)$:
\begin{itemize}
\item[(i)] $u(\widetilde{X}_t) - u(\widetilde{X}_0) - \int_0^t L u(\widetilde{X}_s) \, ds$, $t \ge 0$,
is  a continuous $(\widetilde{\mathcal{F}}_t)_{t \ge 0}$-martingale under $\widetilde{\P}_{\varv \mu}$  for any $\varv \in \mathcal{B}_b^+(\R^d)$ such that $\int_{\R^d}\varv\,d\mu=1$.
\end{itemize}
\end{definition}

\begin{remark}
{\it Let $\widetilde{\M} =  (\widetilde{\Omega}, \widetilde{\mathcal{F}}, (\widetilde{\mathcal{F}}_t)_{t \geq 0},(\widetilde{X}_t)_{t \ge 0}, (\widetilde{\P}_x)_{x \in \R^d\cup \{\Delta\}})$ be a right process with state space $\R^d$ and consider the following condition: 
\begin{itemize}
\item[(i$^\prime$)]
for all $u \in C_0^{\infty}(\R^d)$, $u(\widetilde{X}_t) - u(\widetilde{X}_0) - \int_0^t L u(\widetilde{X}_s) \, ds$, $t \ge 0$, is  a continuous $(\widetilde{\mathcal{F}}_t)_{t \ge 0}$-martingale under $\widetilde{\P}_{x}$ for $\mu$-a.e. $x \in \R^d$.
\end{itemize}
If (i$^\prime$) holds, then (i) of Definition \ref{def:3.3.2.1} holds and $\widetilde{\M}$ hence solves the martingale problem for $(L, C_0^{\infty}(\R^d))$ with respect to $\mu$.  In particular, by Proposition \ref{prop:3.1.6}, $\M$ solves the martingale problem for $(L,C_0^{\infty}(\R^d))$ with respect to $\mu$. Consider the following condition:
\begin{itemize}
\item[(i$^{\prime \prime}$)] 
there exists a $d$-dimensional standard $(\widetilde{\mathcal{F}}_t)_{t\geq0}$-Brownian motion $(\widetilde{W}_t)_{t\geq0}$ on $(\widetilde{\Omega}, \widetilde{\mathcal{F}}, \widetilde{\P}_y)$ such that $(\widetilde{\Omega}, \widetilde{\mathcal{F}}, \widetilde{\P}_y, (\widetilde{\mathcal{F}}_t)_{t \ge 0}, (\widetilde{X}_t)_{t \ge 0},(\widetilde{W}_t)_{t \geq 0})$ is a weak solution to \eqref{eq:3.39}  for $\mu$-a.e. $y \in \R^d$.
\end{itemize}
By It\^{o}'s formula, if (i$^{\prime \prime}$) is satisfied, then (i$^{\prime}$) holds, hence $\widetilde{\M}$ solves the martingale problem for $(L,C_0^{\infty}(\R^d))$ with respect to $\mu$.\\
If $\mu$ is a sub-invariant measure for $\widetilde{\M}$, then by Proposition \ref{prop:3.3.1.15} below, we obtain a resolvent $(\widetilde{R}_{\alpha})_{\alpha>0}$ on $L^1(\R^d, \mu)$ associated to $\widetilde{\M}$, hence for any $f \in L^1(\R^d, \mu)$ and $\alpha>0$, it holds that
$$
\widetilde{R}_{\alpha}f(x)=\widetilde{\E}_x\left[\int_0^{\infty} e^{-\alpha t} f(\widetilde{X}_t) dt \right], \;\; \text{for $\mu$-a.e. $x \in \R^d$.}
$$
Thus, we have that $\int_0^t Lu(\widetilde{X}_s)ds$, $t \geq 0$, is $\widetilde{\P}_\mu$-a.e. independent of the Borel measurable $\mu$-version chosen for $Lu$.}
\end{remark}

\begin{proposition} \label{prop:3.3.1.15}
Suppose that condition {\bf (a)} of Section \ref{subsec:2.2.1} holds and that $(L,C_0^{\infty}(\R^d))$ is $L^1$-unique (cf. Definition \ref{def:2.1.1}(i)).
Let a right process
$$
\widetilde{\M} =  (\widetilde{\Omega}, \widetilde{\mathcal{F}}, (\widetilde{\mathcal{F}}_t)_{t \geq 0}, (\widetilde{X}_t)_{t \ge 0}, (\widetilde{\P}_x)_{x \in \R^d\cup \{\Delta\}})
$$ 
solve the martingale problem for $(L,C_0^{\infty}(\R^d))$ with respect to $\mu$ such that $\mu$ is a sub-invariant measure for $\widetilde{\M}$. Let
$$
p^{\widetilde{\M}}_t f(x):=\widetilde{\E}_x[f(\widetilde{X}_t)], \quad  f \in  \mathcal{B}_b(\R^d),  \; x \in \R^d,  t>0,
$$ 
where $\widetilde{\E}_x$ denotes the expectation with respect to $\widetilde{\P}_x$. Then $(p^{\widetilde{\M}}_t)_{t \geq 0}|_{L^1(\R^d, \mu)_b}$ uniquely extends to a sub-Markovian $C_0$-semigroup of contractions $(S_t)_{t \geq 0}$ on $L^1(\R^d, \mu)$ and 
\begin{equation} \label{eq:3.40*}
S_t f =T_tf \; \text{ in $L^1(\R^d, \mu)$,\quad  for all $f\in L^1(\R^d,\mu)$, $t\ge 0$}. 
\end{equation}
In particular, $\mu$ is an invariant measure for $\widetilde{\M}$. Moreover, if additionally assumption {\bf (b)} of Section \ref{subsec:3.1.1}  holds and 
\begin{equation} \label{eq:3.41*}
p^{\widetilde{\M}}_t(C_0^{\infty}(\R^d)) \subset C(\R^d), \quad \forall t>0,
\end{equation}
then 
$$
\widetilde {\P}_x \circ \widetilde{X}^{-1}=\P_x \circ X^{-1}  \quad \text{on } \ \mathcal{B}(C([0, \infty), \R^d))\ \text{  for all $x\in \R^d$}, 
$$
hence $\widetilde{\M}$ inherits all properties of $\M$ that only depend on its law.
\end{proposition}

\begin{proof}
Since $\mu$ is a sub-invariant measure for $\widetilde{\M}$ and $L^1(\R^d, \mu)_b$ is dense in $L^1(\R^d, \mu)$, it follows that $(p^{\widetilde{\M}}_t)_{t \geq 0}|_{L^1(\R^d, \mu)_b}$ uniquely extends to a sub-Markovian semigroup of contractions $(S_t)_{t> 0}$ on $L^1(\R^d, \mu)$. We first show the following claim. \\
{\bf Claim}: $(S_t)_{t \geq 0}$ is strongly continuous on $L^1(\R^d, \mu)$. \\
Let $f \in C_0(\R^d)$. By the right continuity and the normal property of $(\widetilde{X}_t)_{t \geq 0}$ and Lebesgue's theorem, it follows that
\begin{equation} \label{eq:3.41}
\lim_{t \rightarrow 0+} S_t f(x) = \lim_{t \rightarrow 0+} \widetilde{\E}_{x} [f(\widetilde{X}_t)]=f(x), \quad \text{ for $\mu$-a.e. $x \in \R^d$}.
\end{equation}
Now let $B$ be an open ball with $\text{supp}(f) \subset B$. By \eqref{eq:3.41} and Lebesgue's theorem, 
$$
\lim_{t \rightarrow 0+}\int_{\R^d} 1_B |S_t f| d \mu =  \int_{\R^d} 1_B  |f|  d\mu = \|f\|_{L^1(\R^d, \mu)},
$$
hence using the contraction property of $(S_t)_{t>0}$ on $L^1(\R^d, \mu)$ 
\begin{equation} \label{eq:3.42}
\int_{\R^d} 1_{\R^d \setminus B} \,|S_t f| d \mu  \leq \|f\|_{L^1(\R^d, \mu)} - \int_{\R^d} 1_B |S_t f|d \mu \rightarrow 0 \;\; \text{ as } t \rightarrow 0+.
\end{equation}
Therefore, by \eqref{eq:3.41}, \eqref{eq:3.42} and Lebesgue's theorem,
\begin{eqnarray*}
\lim_{t \rightarrow 0+} \int_{\R^d} |S_t f -f | d \mu &=& \lim_{t \rightarrow 0+} \left(\int_{\R^d} 1_B |S_t f - f| d \mu + \int_{\R^d} 1_{\R^d \setminus B} |S_t f | d \mu \right) =0. 
\end{eqnarray*}
Using the denseness of $C_0(\R^d)$ in $L^1(\R^d, \mu)$ and the contraction property of $(S_t)_{t>0}$ on $L^1(\R^d, \mu)$,  the claim follows by $3$-$\varepsilon$ argument.\\
Denote by $(A, D(A))$ the infinitesimal generator of the $C_0$-semigroup of contractions $(S_t)_{t > 0}$ on $L^1(\R^d, \mu)$.
Let $u \in C_0^{\infty}(\R^d)$ and $\varv \in \mathcal{B}_b^+(\R^d)$ with $\int_{\R^d} \varv d\mu$=1. Then by Fubini's theorem,
\begin{eqnarray*}
\int_{\R^d} (S_t u - u) \varv d\mu  &=& \widetilde{\E}_{\varv\mu} \left[ u(\widetilde{X}_t)-u(\widetilde{X}_0) \right]\\
&=&\widetilde{\E}_{\varv\mu}\left[\int_0^t Lu(\widetilde{X}_s) ds \right]=\int_{\R^d} \left(\int_0^tS_s Lu ds\right) \varv d\mu,
\end{eqnarray*}
hence we obtain $S_t u-u=\int_0^tS_s Lu ds$ in $L^1(\R^d, \mu)$. By the strong continuity of $(S_t)_{t>0}$ on $L^1(\R^d, \mu)$, we get $u \in D(A)$ and $Au =Lu$. Since $(L,C_0^{\infty}(\R^d))$ is $L^1$-unique, it follows that $(A, D(A))=(\overline{L}, D(\overline{L}))$, hence \eqref{eq:3.40*} follows. Since $\mu$ is $(\overline{T}_t)_{t>0}$-invariant by Proposition \ref{prop2.1.9}, it follows by monotone approximation that $\mu$ is an invariant measure for $\widetilde{\M}$. If \eqref{eq:3.41*} and additionally {\bf (b)} hold, then by \eqref{eq:3.40*} and the strong Feller property of $(P_t)_{t>0}$,
$$
\int_{\R^d} f(y) \, \widetilde{\P}_{x} (\widetilde{X}_t \in dy)=\int_{\R^d} f(y) \, \P_{x} (X_t \in dy), \quad \forall f \in C_0^{\infty}(\R^d), x \in \R^d, t>0.
$$
By a monotone class argument, the latter implies $\widetilde{\P}_{x} \circ \widetilde{X}_t^{-1}=\P_{x} \circ X_t^{-1}$ for all $x \in \R^d$ and $t>0$. Since the law of a right process is uniquely determined by its transition semigroup (and the initial condition), we have $\widetilde {\P}_x \circ \widetilde{X}^{-1}=\P_x \circ X^{-1}$  on $\mathcal{B}(C([0, \infty), \R^d))$ for all $x \in \R^d$ as desired. 
\end{proof}

\begin{proposition} \label{prop:3.3.1.16}
Suppose the conditions {\bf (a)} of Section \ref{subsec:2.2.1} and {\bf (b)} of Section \ref{subsec:3.1.1}  hold and that $a_{ij}$ is locally H\"{o}lder continuous on $\R^d$ for all $1 \leq i,j \leq d$, i.e. \eqref{AssumptionUniqueness} holds. Suppose that $\mu$ as at the beginning of this section is an invariant measure for $\M$ (see Definition \ref{def:invariantforprocess}) and let
$$
\widetilde{\M}=(\widetilde{\Omega}, \widetilde{\mathcal{F}}, (\widetilde{\mathcal{F}}_t)_{t \geq 0}, (\widetilde{X}_t)_{t \ge 0}, (\widetilde{\P}_x)_{x \in \R^d\cup \{\Delta\}})
$$ 
be a right process which solves the martingale problem for $(L,C_0^{\infty}(\R^d))$ with respect to $\mu$ (see Definition \ref{def:3.3.2.1}), such that $\mu$ is a sub-invariant measure for $\widetilde{\M}$ 
(Definition \ref{def:invariantforprocess}). Assume further that
$$
\widetilde{\E}_{\cdot}[f(\widetilde{X}_t)] \in C(\R^d), \quad \forall f \in C_0^{\infty}(\R^d), t>0.
$$
Then $\mu$ is an invariant measure for $\widetilde{\M}$ and 
$$
\widetilde {\P}_x \circ \widetilde{X}^{-1}=\P_x \circ X^{-1}  \quad \text{on } \ \mathcal{B}(C([0, \infty), \R^d)) \ \text{  for all $x\in \R^d$}, 
$$
hence $\widetilde{\M}$ inherits all properties of $\M$ that only depend on its law.
\end{proposition}

\begin{proof}
By Corollary \ref{cor2.1.1}, $(L, C_0^{\infty}(\R^d))$ is $L^1$-unique, if and only if $\mu$ is an invariant measure for $\M$. Therefore, the assertion follows from Proposition \ref{prop:3.3.1.15}.
\end{proof}

\begin{eg} \label{ex:3.4.9}
In (i)--(vi) below, we illustrate different kinds of situations which imply that $\mu$ is an invariant measure for $\M$, so that Proposition \ref{prop:3.3.1.16} is applicable. Throughout (i)--(iv), $a_{ij}$, $1 \leq i,j \leq d$, is assumed to be locally H\"{o}lder continuous on $\R^d$.
\begin{itemize}
\item[(i)]
By \cite[Proposition 2.5]{BCR}, $(T_t)_{t>0}$ is recurrent if and only if $(T'_t)_{t>0}$ is recurrent. Therefore, it follows from Theorem \ref{theo:3.3.6}(iii) and (v) that if $\M$ is recurrent, then $(T'_t)_{t>0}$ is conservative, hence $\mu$ is $(\overline{T}_t)_{t>0}$-invariant by Remark 2.4(i). Thus, under the assumptions of Proposition \ref{theo:3.2.6} or Proposition \ref{cor:3.3.2.6}, we obtain that $\mu$ is an invariant measure for $\M$ by Remark \ref{rem:3.2.3.3}.
\item[(ii)]
Consider the situation of Remark \ref{rem:2.2.4} and let additionally $\nabla (A+C^T) \in L^q_{loc}(\R^d, \R^d)$.  Note that this implies {\bf (a)} of Section \ref{subsec:2.2.1} and {\bf (b)} of Section \ref{subsec:3.1.1}.
Then, under the assumption of Proposition \ref{prop:2.1.10}(i), or Proposition \ref{prop:3.2.9}(i) or (ii) (in particular, Example \ref{exam:3.2.1.4}(i)), it follows that $\mu$ is an invariant measure for $\M$.
\item[(iii)]
Under the assumption of Example \ref{exam:3.2.1.4}(ii), it follows that the Hunt process $\M'$ associated with $(T'_t)_{t>0}$ is non-explosive, hence $(T'_t)_{t>0}$ is conservative by Corollary \ref{cor:3.2.1}, so that $\mu$ is an invariant measure for $\M$ by Remark 2.4(i) and Remark \ref{rem:3.2.3.3}.
\item[(iv)] Suppose that {\bf (a)} of Section \ref{subsec:2.2.1} and {\bf (b)} of Section \ref{subsec:3.1.1} hold and that Proposition \ref{prop:3.2.9}(i) or (ii) is verified. Then $\M$ is non-explosive and $\mu$ is an invariant measure for $\M$, by Proposition \ref{prop:3.2.9} and Remark.\ref{rem:3.2.3.3}.
\item[(v)] Suppose that $A=id$, $\mathbf{G}=\nabla \phi$ and $\mu=\exp(2\phi) dx$, where $\phi \in H^{1,p}_{loc}(\R^d)$ for some $p\in (d,\infty)$, and that \eqref{eq:3.2.27} holds.Then $\M$ is non-explosive and $\mu$ is an invariant measure for $\M$, by Remarks \ref{rem:3.2.1 1} and \ref{rem:3.2.3.3}.
\item[(vi)] Let $A=id$ and $\mathbf{G}= (\frac{1}{2}-\frac{1}{2}e^{-x_1}) \mathbf{e}_1$. Then $(L, C_0^{\infty}(\R^d))$ is written as
$$
Lf = \frac12  \Delta f +(\frac12-\frac12 e^{-x_1} ) \partial_1 f, \; \quad \forall f \in C_0^{\infty}(\R^d),
$$
and $\mu=e^{x_1}dx$ is an infinitesimally invariant measure for $(L, C_0^{\infty}(\R^d))$. Let $(\overline{L}', D(\overline{L}'))$ be the infinitesimal generator of $(T'_t)_{t>0}$ on $L^1(\R^d, \mu)$ as in Remark \ref{remark2.1.7}(ii). Then $C_0^{\infty}(\R^d) \subset D(\overline{L}')$ and it holds that
$$
\overline{L}' f= \frac12  \Delta f +(\frac12+\frac12 e^{-x_1} ) \partial_1 f, \; \quad \forall f \in C_0^{\infty}(\R^d).
$$
By Remark \ref{rem:2.1.12}(ii) and Proposition \ref{prop:2.1.4.1.4}, $\mu$ is not $(\overline{T}'_t)_{t>0}$-invariant, hence $(T_t)_{t>0}$ is not conservative. But by Proposition \ref{prop:2.1.10}, $\mu$ is $(\overline{T}_t)_{t>0}$-invariant, hence $\mu$ is an invariant measure for $\M$.  Thus, Proposition \ref{prop:3.3.1.16} is applicable, even though $\M$ is explosive.
\end{itemize}
\end{eg}
\newpage
\subsection{Comments and references to related literature}\label{Comments3}
The classical probabilistic techniques that we use in Chapter \ref{chapter_3} can be found for instance in 
\cite{IW89}, \cite{D96}, \cite{KaSh}. Beyond that, in Section \ref{subsec:3.1.1}, the idea for the construction of the Hunt process $\M$ whose starting points are all points of $\R^d$ originates from \cite{AKR}, that originally only covers the case of an underlying symmetric Dirichlet form. More precisely, using the theory of generalized Dirichlet forms (\cite{WSGDF}) and their stochastic counterpart (\cite{Tr2} and \cite{Tr5}), we  extend the method of \cite{AKR} in order to obtain Theorem \ref{th: 3.1.2}. In Section \ref{subsec:3.1.3}, the identification of $\M$ as a weak solution to an SDE is done via a representation theorem for semi-martingales (\cite[II. Theorem 7.1, 7.1']{IW89}).\\
The Krylov-type estimates in Theorem \ref{theo:3.3} (see also Remark \ref{rem:ApplicationKrylovEstimates}), which result from the application of regularity theory of PDEs seem to be new, even in the classical case of locally Lipschitz continuous coefficients.\\
Concerning Section \ref{subsec:3.2.1}, providing sufficient conditions for non-explosion in terms of Lyapunov functions goes back at least to \cite[Theorem 3.5]{kha12}.  In \cite[Example 5.1]{BRSta}, a procedure is explained on how to extend that method to Lyapunov functions that are considered as $\alpha$-superharmonic functions outside an arbitrarily large compact set. This procedure is used to obtain Lemma \ref{lem3.2.6} about non-explosion of $\M$. Corollary \ref{cor:3.2.2} on a Lyapunov condition for non-explosion is an improved version of \cite[Theorem 4.2]{LT18}.\\
Various results about recurrence and transience in Section \ref{subsec:3.2.2} are obtained by combining results and methods of 
\cite{Ge}, \cite{Pi}, \cite{GT2}. Proposition \ref{cor:3.2.2.5} on a Lyapunov condition for recurrence is an improved version of \cite[Theorem 4.13]{LT18}.  \\
Doob's theorem on regular semigroups \cite[Theorem 4.2.1]{DPZB}, resp. the Lyapunov condition for finiteness of $\mu$ in \cite[Theorem 2]{BRS} are crucial for the results on ergodicity in Theorem \ref{theo:3.3.8}, respectively the finiteness of $\mu$ in Proposition \ref{prop:3.3.12} in Section \ref{subsec:3.2.3}. Corollary \ref{cor:3.2.3.7} on a Lyapunov condition for ergodicity is an improved version of \cite[Proposition 4.17]{LT18}. 
The uniqueness of weak
solutions of SDEs is then applied in Example \ref{ex:3.8} to show non-uniqueness
of invariant measures.
\\
In Section \ref{subsec:3.3.1}, Proposition \ref{prop:3.3.9} on pathwise uniqueness which is a direct consequence of \cite[Theorem 1.1]{Zh11} together with
the Yamada--Watanabe theorem (\cite[Corollary 1, Proposition 1]{YW71}) are crucial to obtain global strong existence in Theorem \ref{theo:3.3.1.8}. 
However, Theorem \ref{theo:3.3.1.8}  not only draws on \cite[Corollary 1, Proposition 1]{YW71}, \cite[Theorem 1.1]{Zh11}, since together with the weak existence result and various other results on properties of the weak solution presented in this monograph, Theorem \ref{theo:3.3.1.8}  actually discloses new results for the existence of a strong solution to time-homogeneous It\^o-SDEs with rough coefficients and its various properties.\\
The idea to derive uniqueness in law via $L^1$-uniqueness in Section \ref{subsec:3.3.2} 
can be found in \cite{AR} (see also \cite{Eberle}).

%
%
%
\section{Conclusion and outlook}\label{conclusionoutlook}

In this book, we studied the existence, uniqueness and stability of solutions to It\^{o} SDEs with non-smooth coefficients, using functional analysis, PDE-techniques and stochastic analysis. Theories that played important roles in developing the contents of this book were elliptic and parabolic regularity theory for PDEs and generalized Dirichlet form theory. 
In order to study the existence and various properties of solutions to It\^{o}-SDEs, we could use the functional analytic characterization of a generator and additional analytic properties of the corresponding semigroups and resolvents. 
Thus, without restricting the local regularity assumptions on the coefficients that ensure the local uniqueness of solutions, 
we could derive strong Feller properties and irreducibility of the semigroup as well as Krylov-type estimates for the solutions to the SDEs. Subsequently, we verified that the solutions of the SDEs with non-smooth coefficients can be further analyzed in very much the same way as the solutions to classical SDEs with Lipschitz coefficients. In particular, through the theory of elliptic PDEs, we could explore the existence of an infinitesimally invariant measure that is not only a candidate for the invariant measure but also a reference measure for our underlying $L^r$-space. Thus, investigating the conservativeness of the adjoint semigroups, the existence of invariant measures could be characterized and we could present various criteria for recurrence and ergodicity, as well as uniqueness of invariant probability measures.\\
Let us provide some outlook to further related topics that can now be investigated based on the techniques developed in this book. \\ \\
{\bf 1. The time-inhomogeneous case and other extensions}\\ \\
The way of constructing weak solutions to SDEs by methods as used in  this book is quite robust and was already successfully applied in the degenerate case (see \cite{LT19de}) and to cases with reflection (\cite{ShTr13a}). We may hence think of applying it also in the time-dependent case. As mentioned in the introduction, the local well-posedness result \cite[Theorem 1.1]{Zh11} also holds in
the time-dependent case (and including also the case $d=1$) with some trade-off between the integrability assumptions in time and space.
In particular, the corresponding time-dependent Dirichlet form theory is already well-developed (see \cite{O04, O13, WS04, RuTr4}).
Our method to construct weak solutions independently and separately from local well-posedness, and thereby to extend existing literature, may also work well in the time-inhomogeneous case, if an adequate regularity theory can be developed or exploited. Moreover, we may also think of developing the time-homogeneous case $d=1$. As it allows explicit computations with stronger regularity results and there always exists a symmetrizing measure\index{measure ! symmetrizing} under mild regularity assumptions on the coefficients, one can always apply symmetric Dirichlet form theory (see for instance \cite[Remark 2.1]{GT1}, \cite[Lemma 2.2.7(ii), Section 5.5]{FOT}). Therefore, in the time-homogeneous case in $d=1$, we expect to obtain weak existence results under quite lower local regularity assumptions on the coefficients than are needed for local well-posedness.\\ \\
{\bf 2. Relaxing the local regularity conditions on the coefficients}\\ \\
By introducing a function space called $VMO$, it is possible to relax the condition {\bf (a)} of 
Section \ref{subsec:2.2.1}. 
 For $g \in L^1_{loc}(\R^d)$, let us write $g \in VMO$ (cf. \cite{BKRS}) if there exists a positive continuous function $\omega$ on $[0, \infty)$ with $\omega(0)=0$ such that
\begin{equation*} \label{vmoine}
\sup_{z \in \R^d, r <R} r^{-2d} \int_{B_r(z)} \int_{B_r(z)} |g(x)-g(y)|dx dy \leq \omega(R), \quad \forall R>0.
\end{equation*}
Given an open ball $B$ and $f \in L^1(B)$, we write $f \in VMO(B)$ if there exists an extension $\widetilde{f} \in L^1_{loc}(\R^d)$ of $f \in L^1(B)$ such that $\widetilde{f} \in VMO$. For $f \in L^1_{loc}(\R^d)$, we write $f \in VMO_{loc}$ if for each open ball $B$, $f|_{B} \in VMO(B)$. Obviously, $C(\R^d) \subset VMO_{loc}$. By the Poincar\'{e} inequality (\cite[Theorem 4.9]{EG15}) and an extension result (\cite[Theorem 4.7]{EG15}), it holds that $H^{1,d}_{loc}(\R^d) \subset VMO_{loc}$. 
Note that if the assumption $\hat{a}_{ij} \in C(\overline{B})$  for all $1 \leq i,j \leq d$ in Theorem \ref{Theorem2.2.2} is replaced by $\hat{a}_{ij} \in VMO(B) \cap L^{\infty}(B)$ for all $1 \leq i,j \leq d$, Theorem \ref{Theorem2.2.2} remains true, since it is a consequence of \cite[Theorem 1.8.3]{BKRS} which merely imposes $\hat{a}_{ij} \in VMO(B)$. Therefore, by replacing assumption {\bf (a)} of Section \ref{subsec:2.2.1} with the following assumption:
\begin{itemize}
\item[{\bf ($\tilde{\mathbf{a}}$)}] \ \index{assumption ! {\bf (a)}}$a_{ji}= a_{ij}\in H_{loc}^{1,2}(\R^d) \cap VMO_{loc}\cap L^{\infty}_{loc}(\R^d)$, $1 \leq i, j \leq d$, $d\ge 2$, and $A = (a_{ij})_{1\le i,j\le d}$ satisfies \eqref{eq:2.1.2}.
$C = (c_{ij})_{1\le i,j\le d}$, with $-c_{ji}=c_{ij} \in H_{loc}^{1,2}(\R^d) \cap VMO_{loc} \cap L^{\infty}_{loc}(\R^d)$, $1 \leq i,j \leq d$, $\mathbf{H}=(h_1, \dots, h_d) \in L_{loc}^p(\R^d, \R^d)$ for some $p\in (d,\infty)$,
\end{itemize}
we can achieve analogous results to those derived in this book. Regarding an analytic approach to a class of degenerate It\^{o}-SDEs allowed to have discontinuous coefficients, a systematic study was conducted in \cite{LT19de}. Further studies to relax the assumptions of \cite{LT19de} are required. \\ \\
{\bf 3. Extending the theory of symmetric Dirichlet forms to non-symmetric cases}\\ \\
In the general framework of symmetric Dirichlet forms, many results in stochastic analysis have been derived in \cite{FOT}. However, in the general framework of non-symmetric and non-sectorial Dirichlet forms, it is necessary to confirm in detail whether or not the results of \cite{FOT} can be applied. In particular, the semigroup $(P_t)_{t>0}$ studied in this book is possibly non-symmetric with respect to $\mu$ and may not be an analytic semigroup in $L^2(\R^d, \mu)$, hence the corresponding Dirichlet form is in general non-symmetric and non-sectorial. The absolute continuity condition of $\M$, i.e. $P_t(x, dy) \ll  \mu$ for each $x \in \R^d$ and $t>0$, is crucially used in \cite{FOT} to strengthen results that are valid up to a capacity zero set, to results that hold for every (starting) point in $\R^d$. In our case, under the assumption {\bf (a)} of Section \ref{subsec:2.2.1} and {\bf (b)} of Section \ref{subsec:3.1.1}, the absolute continuity condition of $\M$ is fulfilled, so that we expect to derive similar results, related to every starting point in $\R^d$, such as those in \cite{FOT}. For instance, adapting the proof of \cite[Theorem 4.7.3]{FOT}, we expect to obtain the following result under the assumption that $\M$ is recurrent: given $x \in \R^d$ and $f \in L^1(\R^d, \mu)$ with $f \in L^{\infty}(B_{r}(x))$ for some $r>0$, it holds
$$
\lim_{t \rightarrow \infty}\frac{1}{t} \int_0^t f(X_s) ds = c_f,\qquad  \text{$\P_x$-a.s},
$$
where $c_f = \frac{1}{\mu(\R^d)} \int_{\R^d}f d\mu$ if $\mu(\R^d)<\infty$ and $c_f = 0$ if $\mu(\R^d)=\infty$. Concretely, under the assumption of Theorem \ref{theo:3.3.8} or Proposition \ref{prop:3.3.12}, we may obtain that $\mu$ is not only a finite invariant measure but for any $x \in \R^d$ and $f \in L^{\infty}(\R^d, \mu)$ it holds
$$
\lim_{t \rightarrow \infty} \frac{1}{t}\int_0^t f(X_s) ds =\frac{1}{\mu(\R^d)} \int_{\R^d}f d\mu, \qquad \text{$\P_x$-a.s.}
$$
\\
{\bf 4. Further exploring infinitesimally invariant measures using numerical approximations
}\\ \\
In this book, the existence of an infinitesimally invariant measure $\rho dx$ for $(L, C_0^{\infty}(\R^d))$ whose coefficients satisfy condition {\bf (a)} of Section \ref{subsec:2.2.1} follows from Theorem \ref{theo:2.2.7}. In addition, from Theorem \ref{theo:2.2.7} we know that $\rho$ has the local regularity properties $\rho \in H^{1,p}_{loc}(\R^d) \cap C(\R^d)$ for some $p \in (d, \infty)$ and $\rho(x)>0$ for all $x \in \R^d$. However, we do not know the concrete behavior of $\rho$ for sufficiently large $\|x\|$.  Of course, we can start with an explicitly given $\rho$ and consider a partial differential operator whose infinitesimally invariant measure is $\rho dx$ as in Remark \ref{rem:2.2.4}. But this approach is restrictive in that it may not deal with arbitrary partial differential operators. Indeed, having concrete information about $\rho$ is important since in the 
Krylov-type estimate of Theorem \ref{theo:3.3}, the product of the constants left to the norm of $f$ in \eqref{kryest1}, \eqref{kryest2} depends on $\rho$. In addition, a certain volume growth on $\mu$ is required for the conservativeness and recurrence criteria in Propositions \ref{prop:3.2.9}, \ref{cor:3.3.2.6}, and in Theorem \ref{theo:3.3.8} 
the asymptotic behavior of $P_t f$ as $t \rightarrow \infty$ is determined $\frac{1}{\mu(\R^d)}\int_{\R^d} f d\mu$. 
Recently, it was shown in \cite{LT21in} that if $\M$ is recurrent and {\bf (a$^{\prime}$)} of Section \ref{subsec:2.2.1} is assumed, then an infinitesimally invariant measure for $(L, C_0^{\infty}(\R^d))$ is unique up to a multiplicative constant. Therefore, in the case where $\M$ is recurrent, if one can find explicitly an infinitesimally invariant measure $\mu=\rho dx$ for $(L, C_0^{\infty}(\R^d))$ or if one can estimate the error by finding an approximation for $\rho$ solving numerically the elliptic PDE of divergence type \eqref{eq:2.2.0a}, then it will lead to a useful supplement to this book. \\ \\ \\
{\bf 5. Uniqueness and stability of classical solutions to the Cauchy problem}\\ \\
Consider the Cauchy problem
\begin{equation} \label{clascauchypro}
\partial_t u_f = \frac12 \sum_{i,j=1}^d a_{ij} \partial_{ij} u_f +\sum_{i=1}^d g_i \partial_i u_f \ \ \text{ in \ $\R^d\times (0,\infty)$,}\quad  u_f(\cdot,0) = f\ \ \text{ in \ $\R^d$}.
\end{equation}
For $f \in C_b(\R^d)$, $u_f$ is said to be a {\it classical solution} to \eqref{clascauchypro} if $u_f \in C^{2,1}(\R^d \times (0, \infty)) \cap C_b(\R^d \times [0, \infty))$ and $u_f$ satisfies \eqref{clascauchypro}. There is an interesting connection between the uniqueness of classical solutions to \eqref{clascauchypro} and existence of a global weak solution to \eqref{itosdeweakglo}. Under the assumption that $\M$ is non-explosive and that {\bf (a)} of Section \ref{subsec:2.2.1} and {\bf (b)} of Section \ref{subsec:3.1.1} hold, every classical solution $u_f$ to \eqref{clascauchypro} is represented as (cf. for instance the proof of \cite[Proposition 4.7]{LT21in})
\begin{equation} \label{eqclastrans}
u_f(x,t) = \E_x[f(X_t)]=P_t f(x), \quad \; \text{ for all $(x,t) \in \R^d \times [0, \infty)$}.
\end{equation}
Remarkably, under the assumptions of Theorem \ref{theo:3.3.8} (or those of Proposition \ref{prop:3.3.12}), every classical solution $u_f$ to \eqref{clascauchypro} enjoys by \eqref{eqclastrans} and Theorem \ref{theo:3.3.8}(iv) and its proof the following asymptotic behavior: 
\begin{equation} \label{assymen1}
\lim_{t \rightarrow \infty} u_f(x,t) =  \int_{\R^d} f dm\quad \text{for each $x \in \R^d$}
\end{equation}
and 
\begin{equation} \label{assymen2}
\lim_{t \rightarrow \infty} u_f(\cdot, t) =  \int_{\R^d} f dm, \quad \text{ in $L^r(\R^d, m)$}, \quad \text{for each $r \in [1, \infty)$,}
\end{equation}
where $m=\mu(\R^d)^{-1}\mu$ is the unique probability invariant measure for $\M$. Actually, in \cite[Chapter 2.2]{LB07} under the assumption that the $a_{ij}$ and $g_i$ are locally H\"{o}lder continuous of order $\alpha \in (0,1)$ for all $1 \leq i,j \leq d$ and that $A$ is locally uniformly strictly elliptic, it is shown that there exists a classical solution $u_f\in C_b(\R^d \times [0, \infty))  \cap C^{2+\alpha, 1+\frac{\alpha}{2}}(\R^d \times (0, \infty))$ to \eqref{clascauchypro}. Therefore, under the assumption {\bf (a$^{\prime}$)} of Section \ref{subsec:2.2.1} and that the $g_i$ are locally H\"{o}lder continuous of order $\alpha \in (0,1)$ for any $1 \leq i \leq d$, the classical solution $u_f$ to \eqref{clascauchypro} induced by \cite[Chapter 2.2]{LB07} satisfies \eqref{eqclastrans} and enjoys the asymptotic behavior \eqref{assymen1} and \eqref{assymen2}.

\newpage

%
%

\subsection*{Notations and conventions}
\addcontentsline{toc}{section}{Notations and conventions}
\begin{itemize}
\item[]
\item[] \centerline{\bf Vector spaces and norms}
\item[$\|\cdot\|$]{the Euclidean norm on the $d$-dimensional Euclidean space $\R^d$}
\item[$\langle\cdot,\cdot\rangle$]{the Euclidean inner product in $\R^d$}
\item[$|\cdot |$]{the absolute value in $\R$}
\item[$\| \cdot \|_{\mathcal{B}}$]{the norm associated with a Banach space $\mathcal{B}$} 
\item[$\mathcal{B}'$]{the dual space of a Banach space $\mathcal{B}$}
\item[]
\item[] \centerline{\bf Sets and set operations}
\item[$\R^d$]{the $d$-dimensional Euclidean space} 
\item[$\R^d_{\Delta}$]{the  one-point compactification of $\R^d$ with the point at infinity \lq\lq$\Delta$\rq\rq}
\item[$(\R^d_{\Delta})^{S}$]{set of all functions from $S$ to $\R^d_{\Delta}$, where $S \subset [0,\infty)$}
\item[$\overline{V}$]{the closure of $V\subset \R^d$}
\item[$B_r(x)$]{for $x\in \R^d$, $r>0$, defined as $\{y\in \R^d : \|x-y\|<r\}$} 
\item[$\overline{B}_r(x)$]{defined as $\{y\in \R^d : \|x-y\|\le r\}$}
\item[$B_r$]{short for $B_r(0)$}
\item[$R_{x}(r)$]{the open cube in $\R^d$ with center $x \in \R^d$ and edge length $r>0$}
\item[$\overline{R}_{x}(r)$]{the closure of $R_{x}(r)$}
\item[$A+B$]{defined as $\{a+b: a\in A, b\in B\}$, for sets $A,B$ with an addition operation}
\item[]
\item[] \centerline{\bf Measures and $\sigma$-algebras}
\item[]{In this monograph, any measure is always non-zero and positive and if a measure is defined on a subset of $\R^d$, then it is a Borel measure, i.e. defined on the Borel subsets.} \\
\item[$\mu = \rho \, dx$]{denotes the infinitesimally invariant measure (see \eqref{eq:2.1.4}, Theorem \ref{theo:2.2.7} and Remark \ref{rem:2.2.4})}
\item[$dx$]{the Lebesgue measure on $\mathcal{B}(\R^d)$}
\item[$dt$]{the Lebesgue measure on $\mathcal{B}(\R)$}
\item[$\mathcal{B}(\R^d)$]{the Borel subsets of $\R^d$ or the space of Borel measurable functions $f:\R^d\to \R$}
\item[$\mathcal{B}(\mathbb{R}^d_{\Delta})$]{defined as $\{A\subset \mathbb{R}^d_{\Delta} : A\in \mathcal{B}(\mathbb{R}^d) \text{ or } A=A_0\cup\{\Delta\}, \ A_0\in \mathcal{B}(\mathbb{R}^d)\}$}
\item[$\mathcal{B}(X)$]{smallest $\sigma$-algebra containing the open sets of a topological space $X$}
\item[a.e.]{almost everywhere}
\item[supp$(\nu)$]{the support of a measure $\nu$ on $\R^d$}
\item[supp$(u)$]{for a measurable function $u:\R^d\to \R$ defined as supp$(|u|dx)$}
\item[$\delta_x$]{Dirac measure at $x \in \R^d_{\Delta}$}
\item[$P_t(x, dy)$]{the sub-probability measure defined by $P_t(x, A) = P_t 1_A (x)$, $A \in \mathcal{B}(\R^d), \\
(x,t) \in \R^d \times (0, \infty)$ (see Proposition \ref{prop:3.1.1})}
\item[]
\item[] \centerline{\bf Derivatives of functions, vector fields}
\item[$\partial_t f$]{(weak) partial derivative in the time variable $t$}
\item[$\partial_i f$]{(weak) partial derivative in the $i$-th spatial coordinate}
\item[$\nabla f$]{(weak) spatial gradient, $\nabla f:=(\partial_1 f, \ldots, \partial_d f)$}
\item[$\partial_{ij}f$]{second-order (weak) partial derivatives, $\partial_{ij} f:=\partial_i \partial_j f$}
\item[$\nabla^2 f$]{(weak) Hessian matrix, $(\nabla^2 f)=(\partial_{ij}f)_{1 \leq i,j \leq d}$}
\item[$\Delta f$]{(weak) Laplacian, $\Delta f=\sum_{i=1}^{d}\partial_{ii} f$}
\item[$\text{div}\mathbf{F}$]{(weak) divergence of the vector field $\mathbf{F}=(f_1,\dots,f_d)$, defined as $\sum_{i=1}^d\partial_i f_i$}
\item[$(\nabla B)_i$]{for $1 \leq i \leq d$ and a matrix $B=(b_{ij})_{1 \leq i,j \leq d}$ of functions,  $(\nabla B)_i$ is the divergence of the $i$-th row of $B$, i.e. defined as $\sum_{j=1}^d\partial_j b_{ij}$}
\item[$\nabla B$]{defined as  $((\nabla B)_1, \ldots, (\nabla B)_d)$}
\item[$B^T$]{for a matrix $B$, the transposed matrix is denoted by $B^T$}
\item[$\text{trace}(B)$]{trace of a matrix of functions $B=(b_{ij})_{1 \leq i,j \leq d}$, $\text{trace}(B)=\sum_{i=1}^{d}b_{ii}$}
\item[$A$]{diffusion matrix $A = (a_{ij})_{1 \leq i,j \leq d}$}   
\item[$\mathbf{G}$]{in Section \ref{sec2.1} the drift $\mathbf{G}$ satisfies $\mathbf{G} = (g_1 , \ldots , g_d ) \in L^2_{loc}(\R^d, \R^d, \mu)$ (cf. \eqref{eq:2.1.3}). From Section \ref{subsec:2.2.1} on the drift satisfies $\mathbf{G} = (g_1 , \ldots , g_d )=\frac12 \nabla(A+C^T)+\mathbf{H}$ 
(see assumption {\bf (a)} in Section \ref{subsec:2.2.1} and \eqref{form of G}, but also Remark \ref{rem:2.2.7})}

\item[$\beta^{\rho,B}$]{logarithmic derivative $\beta^{\rho, B} = (\beta_1^{\rho, B}, \ldots , \beta_d^{\rho, B})$ (of $\rho$ associated with $B=(b_{ij})_{1 \leq i,j \leq d}$), where $\beta_i^{\rho, B}  = \frac 12 \left( \sum_{j=1}^d \partial_j b_{ij} + b_{ij} \frac{\partial_j \rho}{\rho}\right)$, i.e. $\beta^{\rho, B} = \frac12 \nabla B+ \frac{1}{2\rho} B \nabla \rho$ (see \eqref{eq:2.1.5} and Remark \ref{rem:2.2.4})}
\item[$\mathbf{B}$]{$\mathbf{B} = \mathbf{G} - \beta^{\rho, A}$, divergence zero vector field with respect to $\mu$ (see \eqref{eq:2.1.5a}, \eqref{eq:2.1.6})}  
\item[]
\item[] \centerline{\bf Function spaces and norms}
\item[]We always choose the continuous version of a function, if it has one.\\
\item[$q$]{the real number $q$ is given throughout by
$$
q:= \frac{pd}{p+d} 
$$ 
for an arbitrarily chosen real number $p\in (d, \infty)$}
\item[$\mathcal{B}(\R^d)$]{the Borel subsets of $\R^d$ or the space of Borel measurable functions $f:\R^d\to \R$}
\item[$\mathcal{B}^+(\R^d)$]{defined as $\{f \in \mathcal{B}(\R^d): f(x) \geq 0 \text{ for all } x \in \R^d  \}$}
\item[$\mathcal{B}_b(\R^d)$]{defined as $\{f \in \mathcal{B}(\R^d): \text{ $f$ is pointwise uniformly bounded} \}$}
\item[$\mathcal{B}(\R^d)_0$]{defined as $\{f \in \mathcal{B}(\R^d): \text{supp}(|f| dx) \text{ is a compact subset of } \R^d \}$}
\item[$\mathcal{B}_b^+(\R^d)$]{defined as $\mathcal{B}^+(\R^d) \cap \mathcal{B}_b(\R^d)$}
\item[$\mathcal{B}_b(\R^d)_0$]{defined as $\mathcal{B}_b(\R^d) \cap \mathcal{B}(\R^d)_0$}
\item[$L^r(U, \nu)$]{the space of $r$-fold integrable functions on $U$ with respect to $\nu$, equipped with the norm $\|f\|_{L^r(U, \nu)}:=(\int_U |f|^r d\nu)^{1/r}$, where $\nu$ is a measure on $\R^d$, $r \in [1, \infty)$ and $U \in \mathcal{B}(\R^d)$}
\item[$(f,g)_{L^2(U, \mu)}$]{inner product on $L^2 (U , \mu )$, defined as $\int_U f g d\mu$, $f, g \in L^2(U,\mu)$, where $U \in \mathcal{B}(\R^d)$}
\item[$L^{\infty}(U, \nu)$]{the space of $\nu$-a.e. bounded measurable functions on $U$, equipped with the norm $\|f\|_{L^{\infty}(U, \nu)}:= \inf\{c>0: \nu(\{|f|>c\})=0 \}$, where  $U \in \mathcal{B}(\R^d)$} 
\item[$\mathcal{A}_0$, $\mathcal{A}_b$, $\mathcal{A}_{0,b}$]
If $\mathcal{A} \subset L^s (V ,\mu )$ is an arbitrary subspace, where $V$ is open subset of $\R^d$, $s \in [1, \infty]$, denote by $\mathcal{A}_0$ the subspace of 
all elements $u\in \mathcal{A}$ such that $\text{supp}(|u|\mu )$ is a compact subset of $V$, and by $\mathcal{A}_b$ the 
subspace of all essentially bounded elements in $\mathcal{A}$, and $\mathcal{A}_{0,b}:=\mathcal{A}_0\cap \mathcal{A}_b$
\item[$L^r(U)$]{defined as $L^r(U, dx)$, $r \in [1, \infty]$, where $U \in \mathcal{B}(\R^d)$}
\item[$L_{loc}^r(\R^d, \nu)$]{defined as $\{ f \in \mathcal{B}(\R^d): f1_K \in L^r(\R^d, \nu)$ \text{ for any compact subset}}
\item[]{$K$ \text{of } $\R^d \}$, $r \in [1, \infty]$} 
\item[$L^r_{loc}(\R^d)$]{defined as $L^r_{loc}(\R^d, dx)$, where $r \in [1, \infty]$}
\item[$L^r(U, \R^d, \nu)$]{defined as $\{(\mathbf{F}=(f_1,\ldots,f_d) \in \mathcal{B}(\R^d)^d: \|\mathbf{F}\| \in L^r(\R^d, \nu)  \}$, equipped with the norm $\|\mathbf{F}\|_{L^r(U, \R^d, \nu)}:=\| \|\mathbf{F}\| \|_{L^r(U, \nu)}$, where $r \in [1, \infty]$ and $U \in \mathcal{B}(\R^d)$}   
\item[$L^r(U, \R^d)$]{defined as $L^r(U, \R^d, dx)$, where $r \in [1, \infty]$ and $U \in \mathcal{B}(\R^d)$}
\item[$L_{loc}^r(\R^d, \R^d, \nu)$]{defined as $\{ \mathbf{F} =(f_1,\ldots,f_d) \in \mathcal{B}(\R^d)^d: 1_K  \mathbf{F} \in L^r(\R^d, \R^d, \nu)$}
\item[]{\text{for any compact subset} $K$ \text{of } $\R^d \}$, $r \in [1, \infty]$ }
\item[$L_{loc}^r(\R^d, \R^d)$]{defined as $L_{loc}^r(\R^d, \R^d,dx)$, where $r \in [1, \infty]$}
\item[$C([0, \infty), \R^d)$]
{the space of $\R^d$-valued continuous functions on $[0,\infty)$ equipped with the metric $d$, where for $\omega, \omega' \in C([0, \infty), \R^d)$
$$
d(\omega, \omega')=\sum_{n=1}^{\infty} 2^{-n} \left(1 \wedge \sup_{t \in [0,n]} |\omega(t)-\omega'(t)|  \right)
$$
} 
\item[$C([0, \infty), \R_{\Delta}^d)$]
{the space of $\R_{\Delta}^d$-valued continuous functions on $[0,\infty)$}
\item[$C(U)$]{the space of continuous functions on $U$, where $U \in \mathcal{B}(\R^d)$}
\item[$C_b(U)$]{the space of bounded continuous functions on $U$ equipped with the norm $\|f\|_{C_b(U)}:=\sup_{U} |f|$, where $U \in \mathcal{B}(\R^d)$}
\item[$C^k(U)$]{the set of $k$-times continuously differentiable functions on $U$, where $k \in \N \cup \{ \infty \}$ and $U$ is an open subset of $\R^d$}
\item[$C^k_b(U)$]{defined as $C^k(U)\cap C_b(U)$, where $k \in \N \cup \{ \infty \}$ and $U$ is an open subset of $\R^d$}
\item[$C_0(U)$]{defined as $\{f \in C(U): \text{supp}(|f|dx) \text{ is a compact subset of } U \}$, where $U$ is an open subset of $\R^d$}
\item[$C_0^k(U)$]{defined as $C^k(U)\cap C_0(U)$ , $k \in \N \cup \{ \infty \}$, $U\subset \R^d$ open
\item[$C_{\infty}(\R^d)$]{defined as $\{f \in C_b(\R^d): \exists\lim_{\|x\| \rightarrow \infty} f(x)= 0  \}$ equipped with the norm $\|f\|_{C_b(\R^d)}$}}
\item[$C^{0,\beta}(\overline{V})$]{defined as $ \{f \in C(\overline{V}): \text{h\"ol}_\beta (f,\overline{V}) < \infty \}$, where $V$ is an open subset of $\R^d$, $\beta \in (0,1)$ and 
$$
\text{h\"ol}_\beta (f,\overline{V})  := \sup\left\{\frac{|f(x)-f(y)|}{\|x-y\|^\beta}: x,y \in \overline{V}, x\not=y\right\} 
$$ 
equipped with the norm
$$ 
\|f\|_{C^{0,\beta}(\overline{V})} := \sup_{x \in \overline{V}} |f(x)| + \mathrm{\text{h\"{o}l}}_\beta (f,\overline{V})
$$
}
\item[$C^{\gamma; \frac{\gamma}{2}}(\overline{Q})$]{defined as $\{f \in C(\overline{Q}): \text{ph\"ol}_\gamma (f,\overline{Q}) < \infty\}$, where $Q$ is an open subset of $\R^d \times \R$, $\gamma \in (0,1)$ and
$$
\text{ph\"ol}_\gamma (f,\overline{Q}) := \sup\left\{\frac{|f(x,t)-f(y,s)|}{\left(\|x-y \|+\sqrt{|t-s|}\right)^{\gamma}}: \;(x,t), (y,s) \in \overline{Q}, \;(x,t) \not=(y,s)\right \}
$$
equipped with the norm 
$$
\|f\|_{C^{\gamma; \frac{\gamma}{2}}(\overline{Q})} := \sup_{(x,t) \in \overline{Q}} |f(x,t)| +\text{ph\"ol}_\gamma(f,\overline{Q})
$$
}

\item[$H^{1,r}(U)$]{defined as $\{ f \in L^r(U): \partial_i f \in L^r(U), \text{ for all $i=1,\ldots, d$} \}$  equipped with the norm  $\|f\|_{H^{1,r}(U)}:=(\|f\|^r_{L^r(U)}+\sum_{i=1}^d\|\partial_i f \|_{L^r(U)}^r)^{1/r}$, if $r \in [1, \infty)$ and $\|f\|_{H^{1,\infty}(U)}:=\|f\|_{L^{\infty}(U)}+\sum_{i=1}^d \|\partial_i f\|_{L^{\infty}(U)}$, if $r=\infty$, where $U$ is an open subset of $\R^d$} 
\item[$H^{1,r}_0(U)$]{the closure of $C_0^{\infty}(U)$ in $H^{1,r}(U)$, where $r \in [1, \infty)$ and $U$ is an open subset of $\R^d$}
\item[$H_{loc}^{1,r}(\R^d)$]{defined as $\{f \in L^r_{loc}(\R^d): f|_B \in H^{1,r}(B) \text{ for any open ball $B$ in $\R^d$} \}=\{f \in L^r_{loc}(\R^d): f\chi\in H^{1,r}(\R^d ) \text{ for any } \chi\in C_0^\infty (\R^d ) \}$}, where $r\in [1,\infty]$
\item[$H_0^{1,2}(V, \mu)$]{ the closure of  
$C_0^\infty (V)$ in $L^2 (V,\mu )$ w.r.t. the norm 
$$ 
\|u\|_{H_0^{1,2} (V,\mu )} 
:= \left( \int_V u^2\, d\mu + \int_V \|\nabla u\|^2\, d\mu \right)^{\frac 12},
$$
where $V$ is an open subset of $\R^d$}
\item[$H^{1,2}_{loc}(V, \mu)$]{the space of all elements $u$ such that 
$u\chi\in H_0^{1,2}(V,\mu)$ for all $\chi\in C_0^\infty (V)$, where $V$ is an open subset of $\R^d$}
\item[]
\item[] \centerline{\bf Operators}
\item[$id$]{identity operator on a given space}
\item[$L^A f$]{defined as $\frac12 \text{trace}(A \nabla^2 f)=\frac12\sum_{i,j=1}^{d} a_{ij}\partial_{ij} f$, $f \in C^2(\R^d)$ (see \eqref{eq:2.1.3bis})}
\item[$Lf$]{defined as $L^A f + \langle \mathbf{G}, \nabla f \rangle=\frac 12 \sum_{i,j=1}^d a_{ij} \partial_{ij} f+ \sum_{i=1}^d g_i\partial_i f$, $f\in C^{2}(\R^d)$ (see \eqref{eq:2.1.3bis2}) and as $L^0 u + \langle \mathbf{B}, \nabla u\rangle,  u\in D(L^0)_{0,b}$ (see \eqref{defLfirst}). The definitions are consistent, since they coincide on $D(L^0)_{0,b}\cap C^{2}(\R^d)=C^{2}_0(\R^d)$ by Remark \ref{cnulltwocoincide}}
\item[$L'f$]{defined as $ L^A  f+\langle 2 \beta^{\rho, A}- \mathbf{G}, \nabla f \rangle= L^A f + \langle \beta^{\rho, A} - \mathbf{B}, 
\nabla f\rangle$, $f\in C^{2}(\R^d)$ (see \eqref{eq:2.1.3bis2'} and as $L^0 u - \langle \mathbf{B}, \nabla u\rangle,  u\in D(L^0)_{0,b}$ (see \eqref{defL'first} and \eqref{eq:2.1.5a}, \eqref{eq:2.1.6}). The definitions are consistent, since they coincide on $D(L^0)_{0,b}\cap C^{2}(\R^d)=C^{2}_0(\R^d)$ by Remark \ref{cnulltwocoincideprime}}
\item[$(\mathcal{E}^{0,V}, D(\mathcal{E}^{0,V}))$]
symmetric Dirichlet form defined as the closure of 
$$
\mathcal{E}^{0,V}(u,v) = \frac12 \int_{V} \langle A \nabla u, \nabla v \rangle d\mu, \quad u,v \in C_0^{\infty}(V)
$$
in $L^2(V, \mu)$, where $V$ is an open subset of $\R^d$. (If $V$ is relatively compact, then $D(\mathcal{E}^{0,V})=H^{1,2}_0(V, \mu)$.)
\item[$(L^{0,V}, D(L^{0,V}))$]{the generator associated with $(\mathcal{E}^{0,V}, D(\mathcal{E}^{0,V}))$}
\item[$(T_t^{0,V})_{t>0}$]{the sub-Markovian $C_0$-semigroup of contractions on $L^2(V, \mu)$ generated by $(L^{0,V}, D(L^{0,V}))$}
\item[$(\mathcal{E}^{0}, D(\mathcal{E}^0))$]{defined as $(\mathcal{E}^{0, \R^d}, D(\mathcal{E}^{0, \R^d}))$}
\item[$\mathcal{E}^{0}_{\alpha}(\cdot,\cdot)$]{defined as $\cE^0(\cdot,\cdot)+\alpha(\cdot,\cdot)_{L^2(\R^d,\mu)}\, , \ \alpha>0$}
\item[$(L^{0}, D(L^{0}))$]{the generator associated with $(\mathcal{E}^{0}, D(\mathcal{E}^{0}))$ (see \eqref{DefGeneratorDF})}
\item[$(T^{0}_t)_{t>0}$]{the sub-Markovian $C_0$-semigroup of contractions on $L^2(\R^d, \mu)$ generated by $(L^{0}, D(L^{0}))$ }
\item[$(\overline{L}^V, D(\overline{L}^V))$]{the $L^1$-closed extension of $(L, C_0^{\infty}(V))$ generating the sub-Markovian
 $C_0$-semigroup $(\overline{T}^V_t)_{t>0}$ on $L^1 (V,\mu )$}, where $V$ is a bounded open subset of $\R^d$ (see Proposition \ref{prop:2.1})
\item[$(\overline{L}^{V, \prime}, D(\overline{L}^{V, \prime}))$]{the $L^1$-closed extension of $(L^{\prime}, C_0^{\infty}(V))$ generating the sub-Markovian $C_0$-semigroup $(\overline{T}^{V, \prime}_t)_{t>0}$ on $L^1 (V,\mu )$}, where $V$ is a bounded open subset of $\R^d$ (see Remark \ref{rem2.1.3})
\item[$(\overline{T}^V_t)_{t>0}$]{the sub-Markovian $C_0$-semigroup of contractions on $L^1(V, \mu)$ generated by $(\overline{L}^V, D(\overline{L}^V))$}
\item[$(\overline{T}^{V, \prime}_t)_{t>0}$]{the sub-Markovian $C_0$-semigroup of contractions on $L^1(V, \mu)$ generated by $(\overline{L}^{V, \prime}, D(\overline{L}^{V, \prime}))$} 
\item[$(\overline{G}^V_{\alpha})_{\alpha>0}$]{the sub-Markovian $C_0$-resolvent of contractions on $L^1(V, \mu)$ generated by $(\overline{L}^V, D(\overline{L}^V))$}
\item[$(\overline{G}^{V, \prime}_\alpha)_{\alpha>0}$]{the sub-Markovian $C_0$-resolvent of contractions on $L^1(V, \mu)$ generated by $(\overline{L}^{V, \prime}, D(\overline{L}^{V, \prime}))$ } 

\item[($\overline{L}, D(\overline{L}))$]{the $L^1$-closed extension of $(L, C_0^{\infty}(\R^d))$ generating the  sub-Markovian$C_0$-semigroup $(\overline{T}_t)_{t>0}$ on $L^1 (\R^d,\mu )$}  (see Theorem \ref{theorem2.1.5})
\item[($\overline{L}', D(\overline{L}'))$]{the $L^1$-closed extension of $(L', C_0^{\infty}(\R^d))$ generating the sub-Markovian $C_0$-semigroup $(\overline{T}'_t)_{t>0}$ on $L^1 (\R^d,\mu )$ (see Remark \ref{remark2.1.7}(ii))}
\item[$(\overline{T}_t)_{t>0}$]{the sub-Markovian $C_0$-semigroup of contractions on $L^1(\R^d, \mu)$ generated by $(\overline{L}, D(\overline{L}))$}
\item[$(\overline{T}'_t)_{t>0}$]{the sub-Markovian $C_0$-semigroup of contractions on $L^1(\R^d, \mu)$ generated by $(\overline{L}', D(\overline{L}'))$}

\item[$(\overline{G}_{\alpha})_{\alpha>0}$]{the sub-Markovian $C_0$-resolvent of contractions on $L^1(\R^d, \mu)$ generated by $(\overline{L}, D(\overline{L}))$}

\item[$(\overline{G}'_{\alpha})_{\alpha>0}$]{the sub-Markovian $C_0$-resolvent of contractions on $L^1(\R^d, \mu)$ generated by $(\overline{L}', D(\overline{L}'))$} 

\item[$(T_t)_{t>0}$]{the semigroup corresponding to $(\overline{T}_t)_{t>0}$ on all $L^r(\R^d,\mu )$-spaces, $r\in [1,\infty]$ (cf. Definition \ref{definition2.1.7})}
\item[$(P_t)_{t>0}$]{the regularized semigroup of $(T_t)_{t>0}$ (cf. Proposition \ref{prop:3.1.1})}

\item[$(T'_t)_{t>0}$]{the semigroup corresponding to $(\overline{T}'_t)_{t>0}$ on all $L^r(\R^d,\mu )$-spaces, $r\in [1,\infty]$ (cf. Definition \ref{definition2.1.7})}
\item[$(G_{\alpha})_{\alpha>0}$]{the resolvent associated with $(T_t)_{t>0}$ on all $L^r(\R^d,\mu )$-spaces, $r \in [1,\infty]$}
\item[$(R_{\alpha})_{\alpha>0}$]{the regularized resolvent of $(G_{\alpha})_{\alpha>0}$ (cf. Proposition \ref{prop:3.1.2})}

\item[$(G_{\alpha}')_{\alpha>0}$]{the resolvent associated with $(T'_t)_{t>0}$ on all $L^r(\R^d,\mu )$-spaces, $r \in [1,\infty]$}

\item[$(L_r,D(L_r))$]{the generator of $(G_{\alpha})_{\alpha>0}$ on $L^r(\R^d,\mu )$, $r \in [1,\infty)$ (cf. Definition \ref{definition2.1.7})}
\item[$(L_r',D(L_r'))$]{the generator of $(G_{\alpha}')_{\alpha>0}$ on $L^r(\R^d,\mu )$, $r \in [1,\infty)$ (cf. Definition \ref{definition2.1.7})}

\item[]
\item[] \centerline{\bf Stochastic processes, stopping times and the like}
\item[$\M$]{the Hunt process $\M =  (\Omega, \mathcal{F}, (\mathcal{F}_t)_{t \ge 0}, (X_t)_{t \ge 0}, (\mathbb{P}_x)_{x \in \R^d\cup \{\Delta\}}   )$ whose transition semigroup is $(P_t)_{t \ge 0}$ (see Theorem \ref{th: 3.1.2} and also Theorem \ref{theo:3.1.4})}
\item[a.s.]{almost surely}
\item[$\E_x$, resp. $\tilde \E_x$]{expectation w.r.t. the probability measure $\P_x$, resp. $\tilde{\P}_x$}
\item[$\sigma(\widetilde{X}_s \vert s \in I)$]{smallest $\sigma$-algebra such that all $\widetilde{X}_s, s \in I$ are measurable, where $I \subset [0, \infty)$ with $I \in \mathcal{B}(\R)$}
\item[$\sigma(\mathcal{S})$]{the smallest $\sigma$-algebra which contains every set of some collection of sets $\mathcal{S}$}
\item[$\sigma_{A}$]{$\sigma_{A}:=\inf\{t>0\,:\, X_t\in A\}$, $A\in \mathcal{B}(\R^d)$}
\item[$\sigma_n$]{$\sigma_n:=\sigma_{\R^d\setminus B_n},n\ge 1$}
\item[$D_A$]{$D_A:=\inf\{t\ge0\,:\, X_t\in  A\}$, $A\in \mathcal{B}(\R^d)$}
\item[$D_n$]{$D_n:=D_{\R^d\setminus B_n}$, $n\ge 1$}
\item[$L_A$] the last exit time from $A\in \mathcal{B}(\R^d)$, $L_A:=\sup \{ t\geq 0: X_t\in A\}$, $\sup \emptyset:=0$
\item[$\langle X\rangle_t$, $\langle X, Y \rangle_t$]{the quadratic variation up to time $t$ of a continuous stochastic process $(X_t)_{t \geq 0}$, resp. the covariation of two continuous stochastic processes $(X_t)_{t \geq 0}$ and $(Y_t)_{t \geq 0}$.}
\item[$\zeta$, resp. $\widetilde{\zeta}$]{lifetime of a stochastic process $(X_t)_{t \geq 0}$, resp. $(\widetilde{X}_t)_{t \geq 0}$ (see Theorem \ref{th: 3.1.2}), Definition \ref{def:3.1.1})}
\item[$\vartheta_t$]{the shift operator, i.e. $X_s\circ\vartheta_t=X_{s+t}$, $s,t\ge 0$}
\item[]
\item[] \centerline{\bf Miscellanea}
\item[$\mathbf{e}_1$]{$\mathbf{e}_1:=(1,0, \ldots, 0) \in \R^d$}
\item[w.r.t]{with respect to}
\item [$a \wedge b$]{minimum value of $a$ and $b$, $a \wedge b=\frac{|a+b|-|a-b|}{2}$}
\item[$a \vee b$]{maximum value of $a$ and $b$, $a \vee b=\frac{|a+b|+|a-b|}{2}$}
\item[$a^+$]{defined as $a \vee 0$}
\item[$a^-$]{defined as $-a \vee 0$}
\end{itemize}

\printindex

\newpage
\text{}\\
Haesung Lee\\
Department of Mathematics and Computer Science, \\
Korea Science Academy of KAIST, \\
105-47 Baegyanggwanmun-ro, Busanjin-gu,\\
Busan 47162, South Korea\\
E-mail: fthslt14@gmail.com\\ \\
Wilhelm Stannat\\
Institut f\"ur Mathematik\\
Technische Universit\"at Berlin\\
Stra\ss{}e des 17. Juni 136, \\
D-10623 Berlin, Germany \\
E-mail: stannat@math.tu-berlin.de\\ \\
Gerald Trutnau\\
Department of Mathematical Sciences and \\
Research Institute of Mathematics of Seoul National University,\\
1 Gwanak-ro, Gwanak-gu \\
Seoul 08826, South Korea  \\
E-mail: trutnau@snu.ac.kr\\ \\

\end{document}